\documentclass[oneside,a4paper,english, 11pt]{amsart}

\input{artmymacros.sty}

\title{Serre weights and Breuil's lattice conjecture in dimension three}

\author{Daniel Le}
\address{Department of Mathematics,
University of Toronto,
40 St. George Street,
Toronto, ON M5S 2E4, Canada}
\email{le@math.toronto.edu}

\author{Bao V.~Le Hung}
\address{Department of Mathematics,
Northwestern University, 
2033 Sheridan Road\\
Evanston, IL 60208, USA}
\email{lhvietbao@googlemail.com}

\author{Brandon Levin}
\address{Department of Mathematics,
University of Arizona, 
617 N. Santa Rita Avenue,
P.O. Box 210089,
Tucson, Arizona 85721, USA}
\email{bwlevin@math.arizona.edu}

\author{Stefano Morra}
\address{LAGA, UMR 7539, CNRS, Universit\'e Paris 13 - Sorbonne Paris Cit\'e, 
Universit\'e de Paris 8,
99 avenue Jean Baptiste Cl\'ement,
93430 Villetaneuse,
France }
\email{morra@math.univ-paris13.fr}

\begin{document}

\begin{abstract} We prove in generic situations that the lattice in a tame type induced by the completed cohomology of a $U(3)$-arithmetic manifold is purely local, i.e., only depends on the Galois representation at places above $p$.  This is a generalization to $\GL_3$ of the lattice conjecture of Breuil. In the process, we also prove the geometric Breuil--M\'ezard conjecture for (tamely) potentially crystalline deformation rings with Hodge-Tate weights $(2,1,0)$ as well as the Serre weight conjectures of \cite{herzig-duke} over an unramified field extending the results of \cite{LLLM}.  We also prove results in modular representation theory about lattices in Deligne--Luzstig representations for the group $\GL_3(\F_q)$.
\end{abstract}

\maketitle

\tableofcontents

\section{Introduction}

One of the most important developments in the Langlands program in recent years has been the $p$-adic local Langlands correspondence for $\GL_2(\Qp)$.  Unfortunately, extending this correspondence even to $\GL_2(K)$ has proven to be exceedingly difficult and all evidence suggests that the desired correspondence will be much more complicated.  On the other hand, there has been some progress on several avatars of the $p$-adic local Langlands correspondence, namely, (generalized) Serre weight conjectures, geometric Breuil--M\'ezard conjecture, and Breuil's lattice conjecture.  These conjectures inform our understanding of what the sought after $p$-adic correspondence should look like.    In this paper, we prove versions of each of these three conjectures for $\GL_3(K)$ when $K/\Qp$ is unramified. 

\subsection{Breuil's lattice conjecture}

Motivated by Emerton's local-global compatibility for completed cohomology, \cite{CEGGPS} constructs a candidate for one direction of the $p$-adic local Langlands correspondence for $\GL_n(K)$.  Namely, they associate to any continuous $n$-dimensional $\overline{\Q}_p$-representation $\rho$ of $\Gal(\overline{K}/K)$ an admissible Banach space representation $V(\rho)$ of $\GL_n(K)$  by patching completed cohomology.  However, the construction depends on a choice of global setup, and one expects it to be a deep and difficult problem to show that the correspondence $\rho \mapsto V(\rho)$ is purely local.
 
In \cite{breuil-buzzati}, Breuil formulates a conjecture on lattices in tame types cut out by completed cohomology of Shimura curves which is closely related to the local nature of $V(\rho)$.  This conjecture was proven subsequently in the ground-breaking work of Emerton--Gee--Savitt \cite{EGS}.   Our first main theorem is a generalization of Breuil's conjecture to three-dimenional Galois representations and the completed cohomology of $U(3)$-arithmetic manifolds.

Specifically, let $p$ be a prime, $F/F^+$ a CM extension (unramified everywhere), and $r:G_F \ra \GL_3(\overline{\Q}_p)$ a Galois representation. Let $\lambda$ be the Hecke eigensystem corresponding to $r$, which appears in the cohomology of a $U(3)$-arithmetic manifold.  Choose a place $v|p$ of $F^+$ which splits in $F$, and let $\tld{H}$ be the integral $p$-adically completed cohomology with infinite level at $v$. 
One expects completed cohomology to realize a global $p$-adic Langlands correspondence generalizing the case of ${\GL_2}_{/\Q}$. 
That is, by letting $\tld{v}$ denote a place of $F$ above $v$ and $G_{F_{\tld{v}}}$ be the absolute Galois group of $F_{\tld{v}}$, the $\GL_3(F_{\tld{v}})$-representation on the Hecke eigenspace $\tld{H}[\lambda]$ corresponds to $r|_{G_{F_{\tld{v}}}}$ via a hypothetical $p$-adic local Langlands correspondence (when the level outside $v$ is chosen minimally). 
In particular, the globally constructed object $\tld{H}[\lambda]$ should depend only on $r|_{G_{F_{\tld{v}}}}$.

Suppose that $r$ is tamely potentially crystalline with Hodge--Tate weights $(2,1,0)$ at each place above $p$. 
For simplicity, in the introduction we suppose that $r$ is unramified away from $p$, although our results hold if $r$ is minimally split ramified. 
Assume that each place $v|p$ in $F^+$ splits in $F$ and fix a place $\tld{v}| v$ for all $v|p$ in $F^+$.
Let $\sigma(\tau)$ be the tame type corresponding to the Weil--Deligne representations associated to $r|_{G_{F_{\tld{v}}}}$ for all $v|p$ under the inertial local Langlands correspondence.
Throughout the introduction, the tame type $\sigma(\tau)$ is assumed to be sufficiently generic.
If $r$ is modular, then by classical local-global compatibility, $\tld{H}[\lambda][1/p]$ contains $\sigma(\tau)$ with multiplicity one.

\begin{thm}[Breuil's conjecture, cf.~Theorem \ref{thm:lattice}] \label{mainthm1} Assume that $p$ is unramified in $F^+$, and that $\overline{r}$ satisfies Taylor--Wiles hypotheses and is semisimple at places above $p$. 
Assume the level of $\tld{H}$ outside $p$ is minimal with respect to $r$.
Then, the lattice 
\[
\sigma(\tau)^0 := \sigma(\tau) \cap \tld{H}[\lambda] \subset \sigma(\tau)
\]
depends only on the collection $\{ r|_{G_{F_{\tld{v}}}} \}_{v|p}$.  
\end{thm}

Let $\overline{H}$ be the mod $p$ reduction of $\tld{H}$, so that $\overline{H}$ is the mod $p$ cohomology with infinite level at places above $p$ of a $U(3)$-arithmetic manifold.
We prove the following ``mod $p$ multiplicity one'' result (cf.~Theorem \ref{thm:modpmultone}).
\begin{thm}[Theorem \ref{thm:modpmultone}]
\label{thm2} 
Keep the assumptions of Theorem $\ref{mainthm1}$.
Let $\sigma(\tau)^{\sigma}$ be a lattice in $\sigma(\tau)$ such that its mod $p$ reduction $\overline{\sigma}(\tau)^\sigma$ has an irreducible upper alcove cosocle.
Then $\Hom_{K_p}(\overline{\sigma}(\tau)^\sigma, \overline{H}[\lambda])$ is one dimensional. 
\end{thm}

These theorems should be compared to Theorem 8.2.1 and Theorem 10.2.1 in \cite{EGS} in dimension two. In the special case where $p$ is split in $F^+$ and $\overline{r}$ is irreducible above $p$, both theorems were proven by the first author in \cite{Le}.

The main ingredients used in \cite{EGS} are the Taylor--Wiles patching method, the geometric Breuil--M\'ezard conjecture for potentially Barsotti--Tate Galois deformation rings (building on work of \cite{breuil-buzzati}), and a classification of lattices in tame types (extending \cite{breuil-buzzati, BP}).
When we began this project, only the first of these tools was available in the case of $\GL_3$.  The analogue of potentially Barsotti--Tate Galois deformation rings are potentally crystalline deformation rings with Hodge--Tate weights $(2,1,0)$.  In \cite{LLLM}, we develop a technique for computing these Galois deformation rings when the descent data is tame and sufficiently generic.   
We discuss in \S \ref{intro:BM} the geometric Breuil--M\'ezard conjecture for these rings.  The representation theoretic results are discussed in \S \ref{intro:reptheory}.      

It is worth mentioning several key differences which distinguish our situation from \cite{EGS}.  Breuil's conjecture for $\GL_2$ gave an explicit description of the lattice $\sigma(\tau)^0$ in terms of the Dieudonn\'e module of $r|_{G_{F_{\tld{v}}}}$.  We prove abstractly that $\sigma(\tau)^0$ is ``purely local'' (Theorems \ref{thm:gauge} and \ref{thm:lattice}), but without giving any explicit description of the lattice. The lattice $\sigma(\tau)^0$ is determined by the parameters of the Galois deformation ring but in a complicated way.    

Let $\sigma(\tau)^{\sigma}$ be a lattice in $\sigma(\tau)$ whose reduction has irreducible cosocle $\sigma$. To prove Theorem \ref{thm2}, we show that a certain (minimal) patched module $M_{\infty}(\sigma(\tau)^{\sigma})$ is free of rank one over the local Galois deformation ring (with patching variables) $R_{\infty}(\tau)$ (Theorem \ref{thm:cyclic}).
In fact, this result is also a key step in our proof of Theorem \ref{mainthm1}.
In \emph{loc. cit.}, the analogue of this result is Theorem 10.1.1 where they show that the patched module of any lattice with irreducible cosocle is free of rank 1. In our situation, it is no longer true that all such patched modules are cyclic.   Rather, this is only true when the cosocle $\sigma$ is upper alcove in every embedding.  As a consequence of this, one can deduce that the isomorphism class as an $R_{\infty}(\tau)$-module of $M_{\infty}(\sigma(\tau)^{\sigma})$ is purely local for any lattice $\sigma(\tau)^{\sigma}$, however, it need not be free.  

For the proof that $M_{\infty}(\sigma(\tau)^{\sigma})$ is free of rank one when $\sigma$ is upper alcove, we induct on the complexity of the deformation ring.  The simplest deformation rings resemble those for $\GL_2$ and so we follow the strategy similar to \cite{EGS}.  For the most complicated deformation rings, we build up $M_{\infty}(\sigma(\tau)^{\sigma})$ from its subquotients relying on the description of the submodule structure of reduction $\ovl{\sigma}(\tau)^{\sigma}$ discussed in \ref{intro:reptheory} and crucially intersection theory results for components of mod $p$ fiber of the Galois deformation ring.  
 
\subsection{Serre weight and Breuil--M\'ezard conjectures}\label{intro:BM}

There is an analogous global context for a mod $p$ Galois representation $\rbar:G_F \ra \GL_3(\F)$ whose corresponding Hecke eigensystem $\mathfrak{m}$ appears in the mod $p$ cohomology with infinite level at $v$ of a $U(3)$-arithmetic manifold.
One expects $\ovl{H}[\mathfrak{m}]$ to correspond to $\rbar|_{G_{F_{\tld{v}}}}$ via a hypothetical mod $p$ local Langlands correspondence.
Furthermore, if we let $W_v(\rbar)$ be the set of Jordan--H\"{o}lder factors in the $\GL_3(\cO_{F_{\tld{v}}})$-socle\footnote{
strictly speaking the Hecke eigenspace $\ovl{H}[\mathfrak{m}]$ is a smooth representation of $G(F^+_v)$, where $G_{/{F^+_v}}$ is a reductive group having a reductive model $\cG_v$ over $\cO_{F^+_v}$ together with an isomorphism $\iota_{\tld{v}}: \cG_v(\cO_{F^+_v})\stackrel{\sim}{\ra}\GL_3(\cO_{F_{\tld{v}}})$.} of $\ovl{H}[\mathfrak{m}]$ and $\rbar|_{G_{F_{\tld{v}}}}$ is tamely ramified, \cite{herzig-duke,GHS} predict that $W_v(\rbar) = W^?(\rbar|_{I_{F_{\tld{v}}}})$, where $I_{F_{\tld{v}}}$ denotes the inertia subgroup of $G_{F_{\tld{v}}}$  and $W^?(\rbar|_{I_{F_{\tld{v}}}})$ is a set which is explicitly defined in terms of $\rbar|_{I_{F_{\tld{v}}}}$.
We have the following version of the weight part of Serre's conjecture (cf. Theorem \ref{SWC2}).

\begin{thm}[The weight part of Serre's conjecture]
Assume that $p$ is unramified in $F^+$, and that $\overline{r}$ satisfies Taylor--Wiles hypotheses, has split ramification, and is semisimple and sufficiently generic at places above $p$. 
Then $W_v(\rbar) = W^?(\rbar|_{I_{F_{\tld{v}}}})$.
\end{thm}

In \cite{LLLM}, we prove this theorem with the additional assumption that $p$ is split in $F^+$.  The strategy is to show the numerical Breuil--M\'{e}zard conjecture for the simplest deformation rings (where the shape has length at least two) using \cite[\S 6.2]{LLLM}.  The key new tool is a more conceptual and robust combinatorial technique for computing the intersection between the predicted weights $W^?(\rbar|_{I_{F_{\tld{v}}}})$ and the  Jordan--H\"{o}lder factors of a type, and which is developed in \S 2.  This allows us to inductively prove that all predicted weights are modular. 
 
Using a patching functor which is constructed globally, we show that the generic fibers of tamely potentially crystalline deformation rings of Hodge--Tate weight $(2,1,0)$ are connected for $\rhobar$ generic and deduce the full numerical Breuil--M\'{e}zard conjecture for these Galois deformation rings. 
Using the numerical formulation, we prove the following geometric version of the Breuil--M\'{e}zard conjecture (cf. \cite{EG}, Proposition \ref{prop:identify:cmpt}).

\begin{thm}[Proposition \ref{prop:identify:cmpt}]
Assume that $\rbar|_{G_{F_{\tld{v}}}}$  is semisimple and sufficiently generic.
There is a unique assignment $\sigma \mapsto \mathfrak{p}(\sigma)$ taking Serre weights $\sigma \in W^?(\rbar|_{G_{F_{\tld{v}}}})$ to prime ideals in the unrestricted framed deformation ring $R^{\Box}_{\rbar|_{G_{F_{\tld{v}}}}}$ such that the special fiber $\Spec(\ovl{R}^{\tau}_{\rbar|_{G_{F_{\tld{v}}}}})$ of the potentially crystalline framed deformation ring $R^{\tau}_{\rbar|_{G_{F_{\tld{v}}}}}$ of Hodge--Tate weight $(2,1,0)$ and tame type $\tau$, is the reduced underlying subscheme of
\[
\underset{\sigma\in W^?(\rbar|_{I_{F_{\tld{v}}}}) \cap \JH(\ovl{\sigma(\tau)})}{\bigcup} \Spec(R^{\Box}_{\rbar|_{G_{F_{\tld{v}}}}}/\mathfrak{p}(\sigma)).
\]
Moreover, this is compatible with any patching functor.
\end{thm}

\subsection{Representation theory results} \label{intro:reptheory}

In order to deduce Breuil's lattice conjecture from the Breuil-M\'ezard conjecture we need (and prove) new results on integral structures in Deligne-Lusztig representations, which may be of independent interest.
The main theorem (Theorem \ref{thm structure}) is a classification of integral lattices with irreducible cosocle in tame types, by means of an \emph{extension graph}, which plays a key role in the proofs of Theorems \ref{mainthm1} and \ref{thm2}.

We now briefly describe the extension graph.
In \S \ref{sec:gph:1}, we introduce a graph on the set of $p$-regular Serre weights (with fixed central character), with vertices corresponding to $p$-regular Serre weights and adjacency between vertices described in a combinatorially explicit way. We then show in Lemma \ref{lem:dic} that two vertices are adjacent if and only if the corresponding Serre weights have a non-trivial $\GL_3(\Fq)$-extension between them, justifying the terminology.
This gives a natural notion of \emph{graph distance} $\mathrm{d}_{\mathrm{gph}}$ between two $p$-regular Serre weights.
Theorem \ref{thm structure} states the following.

\begin{thm}[Theorem \ref{thm structure}] \label{intro:submod}
Assume that $R$ is a sufficiently generic Deligne--Lusztig representation of $\GL_3(\F_q)$.
\emph{(}In particular, the Jordan--H\"older factors of $\ovl{R}$ occur with multiplicity one.\emph{)}
If $\sigma$ is a Jordan--H\"older factor of $\ovl{R}$, let $R^\sigma$ be the unique lattice up to homothety with cosocle $\sigma$.
If $\sigma, \sigma'\in \JH(\ovl{R})$ and that $\dgr{\sigma}{\sigma'} = d$, then:
\begin{enumerate}
\item \label{intro:loewy} $\sigma'$ is a direct summand of the $d$-th layer of the cosocle filtration of $\ovl{R}^\sigma$;
\item if $\sigma''\in \JH(\ovl{R})$ is such that $\dgr{\sigma}{\sigma''} = d+1$ and $\dgr{\sigma'}{\sigma''}=1$ then $\ovl{R}^\sigma$ has a subquotient which is isomorphic to the unique non-split extension of $\sigma'$ by $\sigma''$; and
\item \label{intro:sat} if $R^{\sigma'} \subseteq R^\sigma$ is a saturated inclusion of lattices, then $p^dR^\sigma\subseteq R^{\sigma'}$ is also a saturated inclusion of lattices.
\end{enumerate}
\end{thm}

The argument is involved, using a mixture of local and global techniques, but we can distinguish two main steps in its proof.
In the first step (\S \ref{subsub:Weyl} and \ref{subsub:emb:arg}), we prove the first two items of Theorem \ref{intro:submod} in the case when $\sigma$ is a lower alcove weight of \emph{defect zero} (cf. Definition \ref{defn:obv} and Theorem \ref{thm:soc:simple}).
The proof uses methods from the modular representation theory of algebraic groups, embedding $\ovl{R}^{\sigma}$ in a Weyl module with non-$p$-restricted  highest weights. The key local argument is a careful study of the restriction of algebraic representations to rational points (Proposition \ref{prop:nonsplit}), which lets us constrain the submodule structure of (part of) the $\GL_3(\F_q)$-restriction of an algebraic Weyl module in terms of the extension graph.
This method does not work for all weights $\sigma\in\JH(\ovl{R})$, as the corresponding lattices will not always have simple socle, and thus can not be embedded into a Weyl module. %

In the second step (\S \ref{subsec:dist}), we reduce the theorem for the remaining lattices to the case treated in the first step. We relate the first two items and the last item of Theorem \ref{intro:submod}.
The last item, a statement in characteristic zero, is amenable to an inductive analysis.
First, we show that for a fixed weight $\sigma \in \JH(\ovl{R})$, item (\ref{intro:sat}) of Theorem \ref{intro:submod} actually implies the other two items (cf. Proposition  \ref{prop:dcosoc}).
This crucially uses Theorem \ref{intro:submod}(\ref{intro:sat}) in the case $d=1$ (cf. Proposition \ref{prop saturation 1}), which is proved using the computation of deformation rings in \cite{LLLM}, combinatorics of \S \ref{sec:weights}, and the Kisin--Taylor--Wiles patching method.
This argument follows the suggestion in \cite[\S B.2]{EGS} that tamely potentially crystalline deformation rings strongly reflect aspects of local representation theory through global patching constructions.

Next, we show that the first two items of Theorem \ref{intro:submod} applied to $R$ and its dual in the case of lower alcove defect zero weights imply Theorem \ref{intro:submod}(\ref{intro:sat}) in the case of lower alcove defect zero weights.
From this starting point, an inductive argument proves Theorem \ref{intro:submod}(\ref{intro:sat}), thus concluding the proof of Theorem \ref{intro:submod}.

\emph{Acknowledgments:} We would like to thank Christophe Breuil, James Humphreys, and Cornelius Pillen for many helpful conversations.  We would also like to thank Matthew Emerton, Toby Gee, and Florian Herzig for their support, guidance, and for comments on an earlier draft of this paper.  
Part of this work was carried out while the authors were visiting the Institut Henri Poincar\'e and the Mathematisches Forschungsinstitut Oberwolfach, and we would like to thank these institutions for their hospitality. BLH acknowedges support from the National Science Foundation under grant Nos.~DMS-1128155, DMS-1802037 and the Alfred P. Sloan Foundation.
DL was supported by the National Science Foundation under agreements Nos.~DMS-1128155 and DMS-1703182 and an AMS-Simons travel grant.
Finally, the authors express their utmost gratitude to the anonymous referee, for his or her meticulous and patient reading of several versions of this paper. 
The reports were invaluable in helping us to improve the quality, precision, and clarity of this paper.

\subsection{Notation} \label{subsection:Notation}

If $F$ is any field, we write $G_F\defeq \Gal(\overline{F}/F)$ for the absolute Galois group, where $\overline{F}$ is a separable closure of $F$. 
If $F$ is a number field and $v$ is a prime of $F$ we write $G_{F_v}$ for the decomposition group at $v$ and $I_{F_v}$ for the inertia subgroup of $G_{F_v}$.
If $F$ is a $p$-adic field, we write $I_F$ to denote the inertia subgroup of $G_F$.

We fix once and for all an algebraic closure $\overline{\Q}$ of $\Q$. All number fields are considered as subfield of our fixed $\overline{\Q}$. Similarly, if $\ell\in \Q$ is a prime, we fix algebraic closures $\overline{\Q}_\ell$ as well as embeddings $\overline{\Q}\iarrow\overline{\Q}_{\ell}$. All finite extensions of $\Q_{\ell}$ will thus be considered as subfields in $\overline{\Q}_{\ell}$. Moreover, the residue field of $\overline{\Q}_{\ell}$ is denoted by $\overline{\F}_\ell$.

Let $p>3$ be a prime. 
For $f>0$ we let $K$ be the unramified extension of $\Qp$ of degree $f$. 
We write $k$ for its residue field (of cardinality $q=p^f$) and  $\cO_K=W(k)$ for its ring of integers. 
For $r\geq1$ we set $e_r\defeq p^{fr}-1$ and fix a compatible system of roots $\pi_r \defeq (-p)^{\frac{1}{e_r}}\in \overline{K}$. 
We write $e$ for $e_1$ and $\pi$ for $\pi_1$.
Define the extension $L = K(\pi)$ and set $\Delta_0\defeq \Gal(L/K)$. 
The choice of the root $\pi$ let us define a character
\begin{align*}
\teich{\omega}_{\pi}:\Delta_0&\rightarrow W(k)\s\\
g&\mapsto \frac{g(\pi)}{\pi}
\end{align*}
whose associated residual character is denoted by $\omega_{\pi}$. 
In particular, for $f=1$, $\omega_{\pi}$ is the mod $p$ cyclotomic character, which will be simply denoted by $\omega$. 
If $F_w/\Qp$ is a finite extension and $W_{F_w}\leq G_{F_w}$ denotes the Weil group we normalize Artin's reciprocity map $\mathrm{Art}_{F_w}: F_w\s\ra W_{F_w}^{\mathrm{ab}}$ in such a way that uniformizers are sent to geometric Frobenius elements.

\vspace{2mm}

Let $E\subset \overline{\Q}_p$ be a finite extension of $\Qp$, which will be our coefficient field. 
We write $\cO$ for its ring of integers, fix an uniformizer $\varpi\in \cO$ and let $\mathfrak{m}_E=(\varpi)$. 
We write $\F\defeq \cO/\mathfrak{m}_E$ for its residue field. 
We will always assume that $E$ is sufficiently large. 
In particular, we will assume that any embedding $\sigma: K\iarrow \overline{\Q}_p$ factors through $E\subset \overline{\Q}_p$. %

\vspace{2mm}

We fix an embedding $\sigma_0: K\into E$. The embedding $\sigma_0$ induces maps $\cO_K \iarrow \cO$ and $k \iarrow \F$; we will abuse notation and denote these all by $\sigma_0$. We let $\phz$ denote the $p$-th power Frobenius on $k$ and set $\sigma_i \defeq \sigma_0 \circ \phz^{-i}$. The choice of $\sigma_0$ gives $\omega_f\defeq \sigma_0 \circ \widetilde{\omega}_{\pi}:I_K \ra \cO^{\times}$, a fundamental character of niveau $f$.   
We fix once and for all a sequence $\un{p}\defeq (p_n)_{n\in \N}$ where $p_n\in\ovl{\Q}_p$ satisfies $p_{n+1}^p=p_n$ and $p_0=-p$.
We let $K_\infty\defeq \underset{n\in\N}{\bigcup}K(p_n)$ and $G_{K_\infty}\defeq \Gal(\ovl{\Q}_p/K_\infty)$.

\vspace{2mm}

Let $\rho: G_K\rightarrow \GL_n(E)$ be a $p$-adic, de Rham Galois representation. 
For $\sigma: K\iarrow E\subset\overline{\Q}_p$, we define $\mathrm{HT}_\sigma(\rho)$ to be the multiset of $\sigma$-labeled Hodge-Tate weights of $\rho$, i.e. the set of integers $i$ such that $\dim_E\big(\rho\otimes_{\sigma,K}\C_p(-i)\big)^{G_K}\neq 0$ (with the usual notation for Tate twists).  
In particular, the cyclotomic character has Hodge--Tate weights 1 for all embedding $\sigma\into E$. 
For $\mu=(\mu_j)_j\in X^*(\un{T})$ we say that $\rho$ has Hodge--Tate weighs $\mu$ if
\[
\mathrm{HT}_{\sigma_j}(\rho)=\{\mu_{1,j},\mu_{2,j},\dots,\mu_{n,j}\}.
\]
The \emph{inertial type} of $\rho$ is the isomorphism class of $\mathrm{WD}(\rho)|_{I_K}$, where $\mathrm{WD}(\rho)$ is the Weil--Deligne representation attached to $\rho$ as in \cite{CDT}, Appendix B.1 (in particular, $\rho\mapsto\mathrm{WD}(\rho)$ is \emph{covariant}).
An inertial type is a morphism $\tau: I_K\ra \GL_n(E)$ with open kernel and which extends to the Weil group $W_K$ of $G_K$.
We say that $\rho$ has type $(\mu,\tau)$ if $\rho$ has Hodge--Tate weights $\mu$ and inertial type given by (the isomorphism class of) $\tau$.

Let $G\defeq {\GL_3}_{/ \F}$, denote by $T\subseteq G$ the torus of diagonal matrices, and write $W$ (resp. $W_a$, resp. $\tld{W}$) for the Weyl group (resp. the affine Weyl group, resp. the extended affine Weyl group) of $G$. 
We let $X^*(T)$ denote the group of characters of $T$, which we identify with $\Z^3$ in the usual way.
Let $R$ (resp. $R^\vee$) denote the set of roots (resp. coroots) of $G$ and $\La_R\subseteq X^*(T)$ the root lattice.
We then have
\begin{equation}
\label{eqn:iso:weyl}
W_a=\La_R\rtimes W(G),\qquad\qquad \tld{W}=X^*(T)\rtimes W(G).
\end{equation}
Let $\eps_1'$ and $\eps'_2$ be $(1,0,0)$ and $(0,0,-1)$ respectively. 
Let ${G}^{\mathrm{der}}$ be ${\SL_3}_{/\F}$ with torus ${T}^{\mathrm{der}}$.
Let ${\Lambda}_W = X^*({T}^{\mathrm{der}})$ denote the weight lattice for ${G}^{\mathrm{der}}$.
Let $X^0({T})$ denote the kernel of the restriction map $X^*({T}) \surj {\Lambda}_W$. 
We write $\eps_{1}$ and $\eps_{2}$ for the images of $\eps'_{1}$ and $\eps'_{2}$ in $\Lambda_W$, respectively.
We define in a similar fashion the various Weyl groups ${W}^{\der}$, ${W}_a^{\der}$, $\tld{{W}}^{\der}$ for ${G}^{\der}$.
Note that we have canonical isomorphisms $W\cong W^{\der}$ and $W_a\cong W_a^{\der}$.

Let $\mathcal{S}$ be a finite set.
For each $\tld{v}\in \mathcal{S}$, let $F_{\tld{v}}$ be an unramified extension of $\Q_p$ of degree $f_{\tld{v}}$, and let $k_{\tld{v}}$ be the residue field of $F_{\tld{v}}$.
Let $\cJ$ be the set of ring homomorphisms $\prod_{\tld{v} \in \mathcal{S}} k_{\tld{v}} \ra \F$.
If we fix an embedding $\sigma_{\tld{v},0}: k_{\tld{v}} \into \F$ and set $\sigma_{\tld{v},i}$ to be $\sigma_{\tld{v},0}\circ \varphi^{-i}$, then $\cJ$ is naturally identified with the set of pairs $(\tld{v},i_{\tld{v}})$ with $\tld{v}\in \mathcal{S}$ and $i_{\tld{v}} \in \Z/f_{\tld{v}}$.
In applications, $\mathcal{S}$ will be a finite set of places dividing $p$ of a number field $F$. 
Sometimes, $\mathcal{S}$ will have cardinality one, in which case we might drop the subscripts from $f_{\tld{v}}$ and $k_{\tld{v}}$ and denote the single unramified extension $F_{\tld{v}}$ of $\Q_p$ by $K$.

Let $\un{G}_0$ be the algebraic group $\prod_{\tld{v}\in \mathcal{S}} \Res_{k_{\tld{v}}/\FF_p} \GL_3$ with $\un{T}_0$ the diagonal torus and center $\un{Z}_0$. 
Let $\un{G}$ be the base change $\un{G}_0 \times_{\F_p} \F$, and similarly define $\un{T}$ and $\un{Z}$.
There is a natural isomorphism $\un{G} \cong \prod_{i\in \cJ} {\GL_3}_{/\F}$.
One has similar isomorphisms for $\un{T}$, $\un{Z}$, $X^*(\un{T})$, $\un{R}$, $\un{R}^\vee$ where $\un{R}$ (resp. $\un{R}^\vee$) denotes the set of roots (resp. coroots) of $\un{G}$.
If $\mu \in X^*(\un{T})$, then we correspondingly write $\mu = \sum_{i\in \cJ} \mu_i=\sum_{\tld{v}\in\cS}\mu_{\tld{v}}$.
We use similar notation for similar decompositions.
Again we identify $X^*(\un{T})$ with $(\Z^3)^\cJ$ in the usual way and let $\eps'_{1,i}$ and $\eps'_{2,i}$ be $(1,0,0)$ and $(0,0,-1)$, respectively, in the $i$-th coordinate and $0$ otherwise.
Let $\un{R}^+\subseteq \un{R}$ (resp. $\un{R}^{\vee,+}\subseteq \un{R}^\vee$) be the subset of positive roots (resp. coroots) of $\un{G}$ with respect to the upper triangular Borel in each embedding.
We define dominant (co)characters with respect to these choices.
Let $X^*_+(\un{T})$ be the set of dominant weights.
We denote by $X_1(\un{T}) \subset X^*_+(\un{T})$ be the subset of weights $\lambda\in X^*_+(\un{T})$  satisfying $0\leq \langle \lambda,\alpha^\vee\rangle\leq p-1$ for all simple roots $\alpha\in \un{R}^+$.
We call $X_1(\un{T})$ the set of $p$-restricted weights. Let $\eta'_i\in X^*(\un{T})$ be $(1,0,-1)$ in the $i$-th coordinate and $0$ otherwise, and let $\eta'$ be $\sum_{i\in\cJ} \eta'_i \in X^*(\un{T})$.
Let $\eta_i\in X^*(\un{T})$ be $(2,1,0)$ in the $i$-th coordinate and $0$ otherwise, and let $\eta$ be $\sum_{i\in\cJ} \eta_i \in X^*(\un{T})$.
Then $\eta$ is a lift of the half sum of the positive roots of $\un{G}$.

Let $\un{W}$ be the Weyl group of $\un{G}$ with longest element $w_0$.
Let $\un{W}_a$ and $\widetilde{\un{W}}$ be the affine Weyl group and extended affine Weyl group, respectively, of $\un{G}$.
Let $\un{\La}_R \subset X^*(\un{T})$ denote the root lattice of $\un{G}$.
As above we have identifications $\un{W}\cong W^\cJ$, $\un{W}_a\cong W_a^\cJ$, $\tld{\un{W}}\cong \tld{W}^\cJ$ and isomorphisms analogous to (\ref{eqn:iso:weyl}).

The Weyl groups $\un{W}$, $\tld{\un{W}}$, $\un{W}_a$ act naturally on $X^*(\un{T})$.
The image of $\lambda\in X^*(\un{T})$ via the injection $X^*(\un{T})\into \tld{\un{W}}$ is denoted by $t_\lambda$.
Our convention is that the \emph{dot action} is always a $p$-dot action i.e.~ $t_\lambda w \cdot \mu = t_{p\lambda} w (\mu+\eta) - \eta$.

Recall that for $(\alpha,n)\in \un{R}^+\times \Z$, we have the root hyperplane $H_{\alpha,n}\defeq \{\lambda:\ \langle\lambda+\eta,\alpha^\vee\rangle=np\}$. An alcove (or sometimes $p$-alcove) is a connected component of  
the complement $X^*(\un{T})\otimes_{\Z}\R\ \setminus\ \big(\bigcup_{(\alpha,n)}H_{\alpha,n}\big)$.
We say that an alcove $\un{C}$ is $p$-restricted if $0<\langle\lambda+\eta,\alpha^\vee\rangle<p$ for all simple roots $\alpha\in \un{R}^+$ and $\lambda\in \un{C}$.
If $\un{C}_0 \subset X^*(\un{T})\otimes_{\Z}\R$ denotes the dominant base alcove we let 
\[\tld{\un{W}}^+\defeq\{\tld{w}\in \tld{\un{W}}:\tld{w}\cdot \un{C}_0 \textrm{ is dominant}\}.\]
and
\[\tld{\un{W}}^+_1\defeq\{\tld{w}\in \tld{\un{W}}^+:\tld{w}\cdot \un{C}_0 \textrm{ is } p\textrm{-restricted}\}.\]
Let $\tld{w}_h  =( \tld{w}_{h,i}) \in \tld{\un{W}}^+_1$ be the element $w_0 t_{-\eta}$.
The discussion in this paragraph also applies for $G$, from which we define the dominant base alcove $C_0$, $\tld{{W}}^+$, and $\tld{{W}}_1^+$ for $G$.

There is a Frobenius action on $X^*(\un{T})$, denoted by $\pi$ and an induced action of $\pi$ on $\un{W}$ defined by $\pi(w)(\pi(\lambda))=\pi(w(\lambda))$.
If $\lambda\in X^*(\un{T})$ then $\pi(\lambda)_{\tld{v},i_{\tld{v}}}=\lambda_{\tld{v},i_{\tld{v}}-1}$ under the standard identification and similarly $\pi(s)_{\tld{v},i_{\tld{v}}} = s_{\tld{v},i_{\tld{v}}-1}$ for an element $s$ of one of the Weyl groups above.

\vspace{2mm}

Let $\un{G}_0^{\mathrm{der}}$ be $\prod_{\tld{v}\in \mathcal{S}} \Res_{k_{\tld{v}}/\FF_p} \SL_3$ with torus $\un{T}_0^{\mathrm{der}}$.
Let $\un{G}^{\mathrm{der}}$ be the base change $\un{G}_0^{\mathrm{der}} \times_{\F_p} \F$, and similarly define $\un{T}^{\mathrm{der}}$.
Let $\un{\Lambda}_W = X^*(\un{T}^{\mathrm{der}})$ denote the weight lattice for $\un{G}^{\mathrm{der}}$.
Let $X^0(\un{T})$ denote the kernel of the restriction map $X^*(\un{T}) \surj \un{\Lambda}_W$. 
We write $\eps_{1,i}$ and $\eps_{2,i}$ for the images of $\eps_{1,i}'$ and $\eps_{2,i}'$ respectively.
We define in a similar fashion the various Weyl groups $\un{W}^{\der}$, $\un{W}_a^{\der}$, $\tld{\un{W}}^{\der}$, $\tld{\un{W}}^{\der,+}$, $\tld{\un{W}}^{\der,+}_1$ for $\un{G}^{\der}$ (and $\tld{{W}}^{\der,+}$, $\tld{{W}}^{\der,+}_1$ for $G$), with analogous product decompositions.

\vspace{2mm}

Let $\alpha, \beta, \gamma^+$ denote the generators for the affine Weyl group $W_a$ of $\GL_3$ given by reflection over the walls of $C_0$ with $\alpha, \beta \in W$ and $\gamma^+$ is the affine reflection. Note that $\alpha, \beta$ satisfy $\alpha(\eps'_2) = \eps'_2$ and $\beta(\eps'_1) = \eps'_1$.   

 \vspace{2mm}

Let $S_3$ denote the symmetric group on $\{1,2,3\}$. We fix an injection $S_3\hookrightarrow \GL_3(\Z)$ sending $s$ to the permutation matrix whose $(k, m)$-entry is $\delta_{k,s(m)}$ and $\delta_{k,s(m)}\in\{0,1\}$ is the Kronecker $\delta$ specialized at $\{k,s(m)\}$. We will abuse notation and simply use $s$ to denote the corresponding permutation matrix. 
We consider the embedding  $X^*(T)\into\GL_3(\F(\!(v)\!))$ defined by $\lambda \mapsto v^{\lambda}$ where then $v^{(a,b,c)}$ is the diagonal matrix with entries $v^a, v^b, v^c$ respectively.
In this way, we get an embedding $\widetilde{W} \hookrightarrow \GL_3(\F(\!(v)\!))$.
Finally for $m\geq 0$ and a collection $(B_j)_{j=0,\dots,m}$ of square matrices of the same size, we write $\prod_{j=0}^{m}B_j=B_0\cdot B_1\cdot\dots\cdot B_m$.
\vspace{2mm}

Let $V$ be a representation of a finite group $\Gamma$ over an $E$-vector space.
We write $\JH(\ovl{V})$ to denote the set of Jordan--H\"older factors of the mod $\varpi$-reduction of an $\cO$-lattice in $V$. 
This set is independent of the choice of the lattice.

If $R$ is a ring, we let $\Irr(\Spec(R))$ denote the set of minimal primes of $R$.

\section{Extension graph}
\label{sec:graph}

In this section, we give a description of the extension graph of generic irreducible representations of the group $\rG\defeq\un{G}_0(\F_p)$ and describe the constituents of the mod $p$ reduction of \emph{generic} Deligne--Lusztig representations of $\rG$.
This also gives a description of the set $W^?(\rhobar)$ of predicted Serre weights defined by Herzig \cite{herzig-duke}.   
In \S \ref{sec:weights}, we use these descriptions to prove the Serre weight conjectures.  The distance in the extension graph also plays an important role in computing the cosocle filtration for the reductions of lattices in tame types in \S \ref{subsec:dist}.  

\subsection{Definition and properties of the extension graph}
\label{sec:gph:1}

The surjection $X^*(T) \ra \Lambda_W$ identifies $p$-restricted alcoves of $\GL_3$ and $\SL_3$.
Let $C_0$ and $C_1$ denote the $p$-restricted lower and upper alcoves, respectively, of the weight space of $\GL_3$,
so that $\mathcal{A} \defeq \{C_0,C_1\}^\cJ$ is naturally identified with the set of $p$-restricted alcoves for $\un{G}$ (or $\un{G}^\der$ when convenient). 
Our alcoves are fundamental domains with respect to the \emph{dot action}, so the base alcove will have a vertex at $-\eta$.
For notational convenience, we sometimes write $0$ and $1$ instead of $C_0$ and $C_1$ so that $\mathcal{A}$ is identified with $\{0,1\}^\cJ$. 
We recall some notions from \S \ref{subsection:Notation}.
Let $X_1(\un{T})$ denote the set of $p$-restricted weights of $\un{G}$ and $\un{C}_0\subseteq X(\un{T})\otimes\R$ denote the lowest alcove $(C_0)_{i\in \cJ}$. %
We consider $\widetilde{\un{W}}^{\mathrm{der}} = \un{\Lambda}_W \rtimes \un{W}$ acting \emph{via the $p$-dot action} on $\un{\La}_W\otimes\R$.
Let $\widetilde{\un{W}}^{\mathrm{der},+}_1 \subset \widetilde{\un{W}}^{\mathrm{der}}$ denote the set of elements which take $\un{C}_0$ to an element of $\mathcal{A}$.
There are six elements of $\widetilde{W}^{\der,+}_1$, namely $\mathrm{id}, (1 2)t_{-({\eps}_1-{\eps}_2)}, (2 3)t_{-({\eps}_2-{\eps}_1)}, (1 2 3)t_{-{\eps}_2}, (1 3 2)t_{-{\eps}_1},$ and $(1 3)t_{-({\eps}_1+{\eps}_2)}$. 

Observe that the inclusion $\un{\Lambda}_W \into \widetilde{\un{W}}^{\mathrm{der}}$ (resp.~$X^*(\un{T}) \into \tld{\un{W}}$) induces an isomorphism $\iota^{\der}:\un{\La}_W/\un{\La}_R \risom \widetilde{\un{W}}^{\der}/\un{W}^{\der}_a$ (resp.~$\iota:X^*(\un{T})/\un{\Lambda}_R (\cong X^*(\un{Z})) \risom \tld{\un{W}}/\un{W}_a$).
Let $\mathcal{P}^{\mathrm{der}} \subset \un{\Lambda}_W \times \widetilde{\un{W}}^{\der,+}_1$ be the subset of pairs $(\omega,\tld{w})$ with $\iota^{\der}(-\pi^{-1}(\omega)+\un{\Lambda}_R) = \tld{w}\un{W}^{\mathrm{der}}_a$.
We similarly define $\mathcal{P} \subset X^*(\un{T}) \times \tld{\un{W}}^+_1$.
Note that restriction gives a natural surjection $\mathcal{P}\onto \mathcal{P}^{\der}$.

\begin{lemma}
\label{lem:eq:Weyl}
The map 
\begin{align*}
\beta: \mathcal{P}^{\der} &\ra \un{\Lambda}_W\times \mathcal{A} \\
(\omega,\tld{w}) &\mapsto (\omega,\pi(\tld{w}) \cdot \un{C}_0)
\end{align*}
is a bijection.
\end{lemma}
\begin{proof}
The map $\beta$ is a bijection because $\iota^{\der}$ is a bijection and $\un{W}^{\der}_a$ acts simply transitively on the set of alcoves for $\un{G}^{\mathrm{der}}$.
\end{proof}

Let $\mu$ be an element of $\un{C}_0/(p-\pi)X^0(\un{T})$. 
We will often have some lift of $\mu$ in $X^*(\un{T})$ in mind, but what we write will not depend on the choice of this lift.
We define a map
\begin{align}\label{eqn:tr}
\mathcal{P}^{\der} &\ra X^*(\un{T})/(p-\pi)X^0(\un{T}) \\
(\omega,\tld{w}) &\mapsto \tld{w}'\cdot(\mu-\eta+\omega'),\nonumber
\end{align}
where $(\omega',\tld{w}')\in \mathcal{P}$ is a lift of $(\omega,\tld{w})$.
The map (\ref{eqn:tr}) does not depend on the choice of lift. 
Then we define
\[\Trns'_\mu: \un{\Lambda}_W \times \mathcal{A} \ra X^*(\un{T})/(p-\pi)X^0(\un{T})\]
to be the composition of $\beta^{-1}$ with (\ref{eqn:tr}).

\begin{defn}\label{defn:regular}
We say that a weight $\lambda \in X^*(\un{T})$ is {\it $p$-regular} if $\langle \lambda+\eta,\alpha^\vee \rangle \notin p\Z$ for all positive coroots $\alpha^\vee \in \un{R}^{\vee, +}$.

\vspace{1mm}

\noindent
Note that a $p$-regular element belongs to a unique alcove.
\end{defn}

\vspace{6mm}

\noindent
Define $\un{\Lambda}^\mu_W$ to be the set 
\[
\left\{
\omega \in \un{\Lambda}_W  :  \omega+\mu-\eta\in\un{C}_0
\right\}
\]
(taking the image of $\mu-\eta$ in $\un{\Lambda}_W$).
Let $\Trns_\mu$ be the restriction of $\Trns_\mu'$ to $\un{\Lambda}_W^\mu \times \mathcal{A}$.
We establish some basic properties of $\Trns_\mu$.
\begin{prop}
\label{prop:inj:trns}
\begin{enumerate}
\item \label{item:regalcove} \emph{(}Any lift of\emph{)} $\Trns_\mu(\omega,a)$ in $X^*(\un{T})$ is $p$-regular and is in alcove $\pi^{-1}(a)$.
\item \label{item:inj} The map $\Trns_\mu$ is injective.
\end{enumerate}
\end{prop}
\begin{proof} 
For (\ref{item:regalcove}), suppose that $(\omega',\tld{w}')\in \mathcal{P}$ is a lift of $\beta^{-1}(\omega,a)$.
Then $\mu-\eta+\omega'$ is $p$-regular in alcove $\un{C}_0$ by definition of $\un{\La}_W^\mu$ so that $\Trns_\mu(\omega,a) = \tld{w}'\cdot(\mu-\eta+\omega')$ is $p$-regular in alcove $\pi^{-1}(a) = \tld{w}\cdot \un{C}_0$.

For (\ref{item:inj}), suppose that $\Trns_\mu(\omega,a) = \Trns_\mu(\nu,b)$.
Let $(\omega',\tld{w}')$ and $(\nu',\tld{x}')\in \mathcal{P}$ be lifts of $\beta^{-1}(\omega,a)$ and $\beta^{-1}(\nu,b)$, respectively.
Then $\tld{w}'\cdot(\mu-\eta+\omega') \equiv \tld{x}'\cdot (\mu-\eta+\nu') \pmod {(p-\pi)X^0(\un{T})}$, from which we conclude from the definition of $\mathcal{P}^{\mathrm{der}}$ that $\omega'|_{\un{Z}} \equiv \nu'|_{\un{Z}} \pmod{3X^*(\un{Z})}$.
Combining this with the fact that $a=b$ from (\ref{item:regalcove}) and using $\iota$, we see that $\tld{w}' \equiv \tld{x}' \pmod {X^0(\un{T})}$.
This implies that $\omega = \nu$.
\end{proof}

\begin{prop}\label{connected}
Let $\mu \in X^*(\un{T})$ be a character. 
We have:
\begin{align}
\label{eq:connected}
\Trns_\mu\big(\un{\La}_W^\mu\times\cA\big)=
\left\{
\begin{aligned}
\lambda\in X_1(\un{T})/(p-\pi)X^0(\un{T}):&\quad\textrm{$\lambda$ is $p$-regular and}\\
&\quad(\lambda-\mu+\eta)|_{\un{Z}}\in(p-\pi)X^*(\un{Z})
\end{aligned}
\right\}
\end{align}
\end{prop}
\begin{proof}
The inclusion of the left hand side of (\ref{eq:connected}) in the right follows from Proposition \ref{prop:inj:trns}(\ref{item:regalcove}) and the formula for $\Trns_\mu$.
We show now the reverse inclusion. 
Let $\lambda$ be an element in the right hand side of (\ref{eq:connected}).
Then we let 
\begin{itemize}
\item $a\in \mathcal{A}$ denote the unique alcove that contains $\pi (\lambda)$;
\item $\omega'\in X^*(\un{T})$ be such that $(\lambda-\mu+\eta)|_{\un{Z}} = (1-p\pi^{-1})\omega'|_{\un{Z}}$;
\item $\omega\in \un{\Lambda}_W$ be the image of $\omega'$; 
\item $(\omega',\tld{w}') \in \mathcal{P}$ be a lift of $\beta^{-1}(\omega,a)$, and
\item $\nu$ be $\lambda - \tld{w}'\cdot (\mu-\eta+\omega') \in \un{\Lambda}_R$.
\end{itemize}
Then $(w'^{-1}(\nu)+\omega',\tld{w}') \in \mathcal{P}$ is a lift of $\beta^{-1}(w'^{-1}(\nu)+\omega,a)$, so that $\Trns_\mu(w'^{-1}(\nu)+\omega,a) = \lambda$ in $X^*(\un{T})/(p-\pi)X^0(\un{T})$.
\end{proof}

\begin{prop} \label{changeofcoordinates}
Let $\lambda-\eta$ be a lift of  $\Trns_\mu(\varepsilon,\un{C}_0)$ and $\beta^{-1}(\varepsilon,\un{C}_0) = (\varepsilon,\tld{w})$. 
Then 
\[\Trns_\lambda(\nu,a) = \Trns_\mu(w^{-1}(\nu)+\varepsilon,a)\]
for $(\nu,a)\in \un{\La}_W^{\lambda}\times \cA$, where $w \in \un{W}$ is the image of $\tld{w}$.
\end{prop}
\begin{proof}
Let $(\nu',\tld{x}') \in \mathcal{P}$ be a lift of $\beta^{-1}(\nu,a)$.
Similarly, let $(\varepsilon',\tld{w}') \in \mathcal{P}$ be a lift of $(\varepsilon,\tld{w})$.
Then 
\begin{align*}
\Trns_\lambda(\nu,a) &= \tld{x}'\cdot (\lambda-\eta+\nu')\\
&= \tld{x}'\cdot(\tld{w}' \cdot (\mu-\eta+\varepsilon')+\nu')\\
&= \tld{x}'\tld{w}' \cdot (\mu-\eta+\varepsilon'+w^{-1}(\nu')) \\
&= \Trns_\mu(w^{-1}(\nu)+\varepsilon,a).
\end{align*}
\end{proof}

Recall that a \emph{Serre weight} is an irreducible $\ovl{\F}_p$-representation of $\rG$.
Each Serre weight is obtained by restriction to $\rG$ from an irreducible algebraic representation of $\un{G}$ of highest weight $\lambda\in X_1(\un{T})$, and this process gives a bijection between from $X_1(\un{T})/(p-\pi)X^0(\un{T})$ to the set of Serre weights of $\rG$ (as described in \cite[Theorem 3.10]{herzig-duke}, cf.~also the beginning of \S~\ref{sec:inj:env} below). 
If $\lambda\in X_1(\un{T})$, we write $F(\lambda)$ for the Serre weight corresponding to $\lambda$.
We say that a Serre weight $F$ is $p$-\emph{regular} if $F\cong F(\lambda)$ where $\lambda\in X_1(\un{T})$ is $p$-regular (cf. Definition \ref{defn:regular}).
Given $\mu\in\un{C}_0$ and $(\omega,a)\in \un{\La}_W^\mu\times \mathcal{A}$, we get a corresponding $p$-regular Serre weight $F(\Trns_\mu(\omega,a))$.
Proposition \ref{prop:inj:trns}(\ref{item:inj}) and \ref{connected} show that $F(\Trns_\mu(-))$ induces a bijection between the set $\un{\Lambda}_W^\mu\times \mathcal{A}$ and the set of $p$-regular Serre weights of $\rG$ with the same central character as $F(\mu-\eta)$.

\begin{defn}
\label{df:adj}
We say that $(\omega, a)$ and $(\nu, b)$ in $\un{\Lambda}^{\mu}_{W}\times \mathcal{A}$  are \emph{adjacent} if  there exists $j \in \cJ$  such that both $a_i = b_i$ and $\omega_i = \nu_i$ for $i \in \cJ$ with $i \neq j$, $a_j \neq b_j$, and 
\begin{equation}\label{eqn:adj}
\omega_{j} - \nu_{j} \in  \{0, \pm (\eps_{1,j}-\eps_{2,j}) , \pm \eps_{1,j}, \pm \eps_{2,j}\}.
\end{equation}
If $F$ and $F'$ are $p$-regular Serre weights we say that they are \emph{adjacent} if we can write $F=F(\Trns_\mu(\omega,a))$ and $F'=F(\Trns_\mu(\nu,b))$ for some $p$-regular $\mu\in \un{C}_0/(p-\pi)X^0(\un{T})$ and some $(\omega, a)$ and $(\nu, b)$ which are adjacent in $\un{\Lambda}^{\mu}_{W}\times \mathcal{A}$. 
By Proposition \ref{changeofcoordinates}, this definition does not depend on the choice of the lowest alcove weight $\mu\in \un{C}_0$.
Indeed if $F=F(\Trns_\mu(\omega,a))$, $F'=F(\Trns_\mu(\nu,b))$ are adjacent $p$-regular Serre weights and $\lambda\in \un{C}_0$ such that $F(\lambda-\eta)$ and $F(\mu-\eta)$ have the same central character, then $F(\Trns_\lambda(w(\omega-\varepsilon),a))=F(\Trns_\mu(\omega,a))$ and $F(\Trns_\lambda(w(\nu-\varepsilon),b))=F(\Trns_\mu(\nu,b))$ for some $w\in \un{W}$ and $\varepsilon \in \un{\Lambda}_W$, and $(w(\omega-\varepsilon),a)$, $(w(\nu-\varepsilon),b)$ are again adjacent since the set in (\ref{eqn:adj}) is $\un{W}$-invariant.

\end{defn}

\begin{rmk} Geometrically, $\un{\Lambda}_W^\mu$ is the intersection of $\un{\Lambda}_W$ with a translate of a $p$-alcove, and two pairs $(\omega,a)$ and $(\nu,b)$ are adjacent if and only if $\omega$ and $\nu$ are either equal or neighbors (i.e.~ differ by a $\un{W}$-conjugate of a fundamental weight), and either $\omega\neq \nu$ and $a$ and $b$ have different labels in the unique component where $\omega$ and $\nu$ differ, or $\omega=\nu$ and $a$ and $b$ differ in exactly one component.
Note that for any $(\omega,a)\in \un{\La}_{W}^{\mu}$ with all its neighbors in $\un{\La}_{W}^{\mu}$, there are $7^{\# \cJ}$ adjacent vertices. 
\end{rmk}
Definition \ref{df:adj} endows $\un{\Lambda}^{\mu}_{W} \times \mathcal{A}$  with a graph structure with edges given by adjacency.
By the above description, this graph is connected and thus it is endowed with a metric.
By Proposition \ref{connected}, any $\mu\in \un{C}_0$ thus endows the set of $p$-regular Serre weights with the same central character as $F(\mu-\eta)$ with a metric. 
By Proposition \ref{changeofcoordinates}, the metric on a fixed set is independent of the choice of $\mu$.
\begin{defn} \label{defn:dgph}
Given $p$-regular Serre weights $F,\, F'$ with the same central character, we denote their distance by  $\dgr{F}{F'}$.
In particular, $F$ and $F'$ are adjacent if and only if $\dgr{F}{F'} = 1$.
\end{defn}
We conclude this section by showing the relation between $p$-regular Serre weights and $\rG$-extensions.
We start by recalling the following definition
\begin{defn}
\label{dfn:deep:S}
Let $\lambda\in X^*(T)$ be a weight. 
We say that $\lambda$ lies $n$-deep in its alcove if for all $\alpha^\vee\in R^{\vee,+}$,
there exist integers $m_{\alpha}\in \Z$ such that $pm_{\alpha}+n<\langle\lambda+\eta,\alpha^\vee\rangle< p(m_{\alpha}+1)-n$.
We have the analogous definition for weights in $X^*(\un{T})$.
\end{defn}
For instance, a dominant weight $\lambda \in X^*(\un{T})$ lies $n$-deep in alcove $\un{C}_0$ if  $n<\langle\lambda+\eta,\alpha^\vee\rangle< p-n$ for all $i=0,\dots,f-1$ and all positive coroots $\alpha^\vee\in \un{R}^{\vee,+}$.

A Serre weight $F$ is said to be \emph{$n$-deep} if we can write $F\cong F(\mu)$ for some $\mu\in X_1(\un{T})$ which is $n$-deep.
We call the graph on the $p$-regular Serre weights defined by adjacency the \emph{extension graph}.
The terminology is justified by the following theorem.
 
\begin{thm}  Let $F,\, F'$ be Serre weights which are both $6$-deep.  Then $\Ext^1_{\rG}(F, F') \neq 0$ if and only if $F$ and $F'$ are adjacent.  
\end{thm}
\begin{proof}
This is Lemma \ref{lem:dic}, whose proof only uses modular representation theory and does not rely on any results from other sections.
\end{proof} 

\begin{rmk}\label{rmk:bipartite}
We partition the set of vertices $(\omega,a)$ into two sets, according to the class $\Sigma_i a_i \mod 2$ (by interpreting $a_i \in \{0,1\}$).This makes the extension graph into a bipartite graph. 
\end{rmk}

\subsection{Types and Serre weights}
\label{sec:types:SW}

Now suppose that $\mathcal{S}$ contains exactly one element $\tld{v}$, and let $K$ be $F_{\tld{v}}$.
An \emph{inertial type} is a representation $\tau: I_K\ra\GL_3(E)$ with open kernel which extends to the Weil group of $K$.
An inertial type is \emph{tame} if it factors through the tame inertial quotient. 
All our tame inertial types are defined over $\cO$ and we use $\ovl{\tau}:I_K\ra\GL_3(\F)$ to denote the reduction to the residue field.

Tame inertial types have a combinatorial description which we will now recall (cf.~ \cite[(6.15)]{herzig-duke} or \cite[Definition 8.2.2]{GHS}). Let $(w, \mu) \in \un{W} \times X^*(\un{T})$.  
As in \cite[(4.1)]{herzig-duke} or the paragraph preceding \cite[Definition 10.1.12]{GHS}, for any $(\nu, \sigma) \in X^*(\un{T}) \rtimes \un{W}$, define
\begin{equation} \label{sigmaconj}
^{(\nu, \sigma)} (w, \mu) = (\sigma w \pi(\sigma)^{-1}, \sigma(\mu) + p \nu - \sigma w \pi(\sigma)^{-1}\pi(\nu))
\end{equation}
and we write $(w, \mu) \sim (w', \mu')$ if there exists $(\nu, \sigma)$ such that $^{(\nu, \sigma)} (w, \mu) = (w', \mu')$.
The following describes all isomorphism classes of tame inertial types for $K$.

\begin{defn} \label{defn:tau} Define an inertial type $\tau(w, \mu):I_K \ra \GL_3(\cO)$ %
as follows:  If $w = (s_0, \ldots, s_{f-1})$, then set $s_{\tau} = s_0 s_{f-1} s_{f-2} \cdots s_1 \in W$  and  $\bm{\alpha} \in X^*(\un{T})$ such that $\bm{\alpha}_0 = \mu_0$ and $\bm{\alpha}_{j} = s_1^{-1} s_2^{-1} \ldots s_j^{-1}(\mu_j)$ for $1 \leq j \leq f-1$.  Let $r$ denote the order of $s_{\tau}$, and set $f' = fr$.  Then, 
\begin{equation} \label{eq:def:type}
\tau(w,\mu) \defeq \bigoplus_{1 \leq i \leq 3} \omega_{f'}^{\sum_{0 \leq k \leq r-1} \bf{a}^{(0)}_{s_{\tau}^{k}(i)} p^{fk}} 
\end{equation}
where $\bf{a}^{(0)} := \sum_{j =0}^{f-1}  \bm{\alpha}_{j} p^j\in \Z^3$. Note that $(w, \mu) \sim ((s_{\tau}, 1,\ldots, 1), \bm{\alpha})$ and $\tau(w, \mu) \cong \tau((s_{\tau}, 1,\ldots, 1), \bm{\alpha})$ by construction. 
\end{defn} 

\begin{defn}
\label{defi:gen}
Let $\tau$ be a tame inertial type. 
\begin{enumerate}
\item
\label{def:gen:reg}
We say that $\tau$ is \emph{regular} if the characters appearing in the right hand side of (\ref{eq:def:type}) are pairwise distinct.
\item 
\label{defi:gen:type}
We say that $\tau$ is \emph{$n$-generic} if there is an isomorphism $\tau\cong \tau(s,\lambda + \eta)$  for some $s \in \un{W}$ and $\lambda \in X^*_+(\un{T})$ which is $n$-deep in alcove $\un{C}_0$. 
\item
\label{def:gen:Gal}
We say that $\rhobar: G_K\ra\GL_3(\F)$ is \emph{$n$-generic} if $\rhobar^{\mathrm{ss}}|_{I_K}\cong\taubar$ for a tame inertial type $\tau$ which is $n$-generic. 
\item 
\label{def:LApres}
A \emph{lowest alcove presentation} of $\tau$ is a pair $(s, \mu) \in \un{W} \times X^*(\un{T})$ where $\mu \in \un{C}_0$ such that $\tau \cong \tau(s, \mu + \eta)$ (which by definition exists exactly when $\tau$ is 0-generic).
\end{enumerate}
\end{defn}

\begin{rmk}  
The notion of genericity given in \cite[Definition 2.1]{LLLM} is slightly different than Definition \ref{defi:gen}(\ref{defi:gen:type}) above. 
In particular, an inertial type which is $n$-generic in the sense of Definition \ref{defi:gen}(\ref{defi:gen:type}) is \emph{a fortiori} $n$-generic in the sense of \cite{LLLM}.
Furthermore our notion of genericity differs from that of \cite[Definition 10.1.12]{GHS} by a shift by $\eta$.  
\end{rmk}

\begin{rmk}
If $\tau$ is $1$-generic then $\tau$ is regular.
\end{rmk}

Now let $\mathcal{S}$ have arbitrary (finite) cardinality rather than one.
For $s\in \un{W}$ and $\mu\in X^*(\un{T})$, we can define $s_{\tld{v}}$ and $\mu_{\tld{v}}$ in the evident way.
Then we set $\tau_\cS(s,\mu)$ to be the collection $(\tau_{\tld{v}}(s_{\tld{v}},\mu_{\tld{v}}))_{\tld{v}\in \cS}$.

The notions of regular and $n$-generic are extended in the evident way to a collection $\tau_{\cS}(s,\mu)$ of tame inertial types.
Similarly, there is an evident notion of $n$-generic for a collection $\rhobar_\cS$ of Galois representations $\rhobar_{\tld{v}}: G_{F_{\tld{v}}} \ra \GL_3(\F)$ for $\tld{v} \in \mathcal{S}$.

We now introduce the Deligne--Lusztig representations which are relevant for this paper.
See also \cite[\S 9.2]{GHS} and \cite[\S 2.2]{LLL}, though we note that our context is slightly more general than \cite[\S 2.2]{LLL} since $\mathcal{S}$ may have size greater than one.
Let $(s,\mu)\in \un{W}\times X^*(\un{T})$ be a \emph{good pair} (\cite[\S 2.2]{LLL}).
Following \cite[Proposition 9.2.1 and 9.2.2]{GHS}, we can attach to $(s,\mu)$ a genuine Deligne--Lusztig representation $R_s(\mu)$ of $\rG\defeq \un{G}_0(\Fp)$ defined over $E$ (taking $E$ to be sufficiently large). 
Note that $R_s(\mu)$ is denoted as $R(s,\mu)$ in \cite{GHS}.

\begin{defn}
Let $(s,\mu)\in\un{W}\times X^*(\un{T})$ be a good pair and let $n\geq 0$.
We say that $R_{s}(\mu)$ is $n$-generic if there exists an isomorphism $R_s(\mu)\cong R_{s'}(\mu')$ where $\mu'-\eta$ is $n$-deep in alcove $\un{C}_0$ (note that the evident generalization of \cite[Lemma 2.2.3]{LLL} implies that $(s',\mu')$ is good).
By \cite[Proposition 9.2.1]{GHS} and the evident generalization of \cite[Proposition 2.2.4]{LLL}, when $n\geq 0$ then $\tau_{\cS}(s,\mu)$ is $n$-generic (cf. Definition \ref{defi:gen}(\ref{defi:gen:type})) if and only if $R_s(\mu)$ is $n$-generic. %
\end{defn}

By inflation, we consider $R_s(\mu)$ as a smooth $\prod_{\tld{v}\in \mathcal{S}} \GL_3(\cO_{F_{\tld{v}}})$-representation.
We recall some basic results on $R_{s}(\mu)$.
Let $\sigma(\tau_{\tld{v}})$ be a smooth $\GL_3(\cO_{F_{\tld{v}}})$-representation associated to $\tau_{\tld{v}}$ as in \cite[Theorem 3.7]{CEGGPS}, and let $\sigma(\tau_\mathcal{S})$ be the $\prod_{\tld{v}\in \mathcal{S}} \GL_3(\cO_{F_{\tld{v}}})$-representation $\otimes_{\tld{v}\in \mathcal{S}} \sigma(\tau_{\tld{v}})$.

\begin{prop}
\label{prop:basic:Ktypes}
Let $\mu-\eta\in X^*(\un{T})$ be $1$-deep in alcove $\un{C}_0$. Then:
\begin{enumerate}
\item $($\cite[Theorem 6.8]{DeligneLusztig}$)$
 $R_s(\mu)$ is irreducible.

\item $($\cite[Corollary 2.3.5]{LLL}$)$
\label{item2:basic:Ktypes} 
 $\sigma(\tau_\cS)$ can be taken to be $R_s(\mu)$.
\end{enumerate}
\end{prop}

We conclude this section with some background on Serre weights associated to semisimple Galois representations.
Let $\ovl{\tau}_\cS$ be a collection of tame representations $\ovl{\tau}_{\tld{v}}: I_{F_{\tld{v}}}\ra\GL_3(\F)$ which extend to $G_{F_{\tld{v}}}$. 
By \cite[Proposition 9.2.1]{GHS}, one attaches to $\ovl{\tau}_{\tld{v}}$ an $E$-valued representation $V(\ovl{\tau}_{\tld{v}})$ of $\GL_3(k_{\tld{v}})$.
Let $V(\ovl{\tau}_{\cS})$ be the $\rG$-representation $\otimes_{\tld{v}\in \cS} V(\ovl{\tau}_{\tld{v}})$.
When $\ovl{\tau}_{\cS}$ is $\ovl{\tau}_\cS(s,\mu)$ for a good pair $(s,\mu)$, $V(\taubar_\cS)$ is isomorphic to $R_s(\mu)$ (\cite[Proposition 9.2.3]{GHS}).
We write $\rhobar_\cS|_{I_{F_\cS}}$  for the restriction to inertia of the collection $\rhobar_\cS$.

Let $\tld{w}_h\defeq w_0t_{-\eta}\in\tld{\un{W}}$.
Recall the self-bijection $\mathcal{R}$ on $p$-regular Serre weights defined in \cite[\S 9.2]{GHS}:
\[
\mathcal{R}(F(\lambda))\defeq F(\tld{w}_h\cdot \lambda).
\]
\begin{defn}\cite[Definition 9.2.5]{GHS} \label{defn:serrewts}
Let $\rhobar_\cS$ be a collection of $2$-generic semisimple Galois representations $\rhobar_{\tld{v}}:G_{F_{\tld{v}}}\ra\GL_3(\F)$.
The set of predicted weights for $\rhobar_\cS$ is defined to be
$$
W^?(\rhobar_{\cS})\defeq \left\{
\mathcal{R}(F)\,:\,F\in \JH(\ovl{V(\rhobar_\cS|_{I_{F_\cS}})})
\right\}.
$$
If $\tau_\cS$ is a collection of tame inertial types $\tau_{\tld{v}}: I_{F_{\tld{v}}} \ra\GL_3(\cO)$, we furthermore define
$$
W^?(\rhobar_\cS, \tau_\cS) \defeq W^?(\rhobar_\cS) \cap \mathrm{JH}(\ovl{\sigma(\tau_\cS)}).
$$
\end{defn}

\begin{rmk}
The condition that $\rhobar_\cS|_{I_{F_{\cS}}}$ is $2$-generic is to ensure that the elements of $\JH(\ovl{V(\rhobar_\cS|_{I_{F_{\cS}}})})$ are all $0$-deep, so that $\mathcal{R}$ is defined (cf.~Lemma \ref{lem:deep:type}).
\end{rmk}

Recall from \cite[Definition 7.1.3]{GHS} that there is a subset $W_{\mathrm{obv}}(\rhobar)\subseteq W^?(\rhobar)$ of \emph{obvious 
Serre weights of $\rhobar$}.
The set $W_{\mathrm{obv}}(\rhobar_{\mathcal{S}})$ is defined in the evident way.

\subsection{Combinatorics of types and Serre weights}\label{subsec:comb}

Recall the notation $\rhobar_{\mathcal{S}}$ from the previous section.
We will always assume that $\rhobar_{\tld{v}}$ is $2$-generic and semisimple for all $\tld{v}\in \mathcal{S}$ in what follows.
The following proposition describes $W_{\mathrm{obv}}(\rhobar_{\mathcal{S}})$ in terms of the extension graph. 

\begin{prop}\label{prop:obv}
Assume that $\rhobar_{\cS}$ is a collection of semisimple Galois representations. 
Suppose that $\rhobar_{\cS}|_{I_{F_{\cS}}}\cong\taubar_{\cS}(s,\lambda)$, %
where $\lambda-\eta$ is $3$-deep in $\un{C}_0$.
Then $W_\mathrm{obv}(\rhobar_{\mathcal{S}})$ is the set
\[\{F\big(\Trns_\lambda(s(\omega),\pi(\tld{w})\cdot\un{C}_0)\big) : \tld{w}=wt_{-\pi^{-1}\omega} \in \widetilde{\un{W}}^{+,\mathrm{der}}_1\}\]
\end{prop}
\begin{proof}
This follows from \cite[Corollary 2.2.13 and (2.6)]{LLL}.
\end{proof}

\begin{defn} \label{weightlabel} Define
$$
\Sigma_0 \defeq  \begin{Bmatrix} (\eps_1+ \eps_2, 0), (\eps_1 - \eps_2, 0), (\eps_2 - \eps_1, 0) \\
 (0, 1) , (\eps_1, 1), (\eps_2, 1) \\
 (0, 0), (\eps_1, 0), (\eps_2, 0) 
 \end{Bmatrix}.$$
 Define   
$\Sigma\defeq \Sigma_0^{\cJ} \subset \un{\Lambda}_W \times \mathcal{A}$.
\end{defn}

The set of predicted weights $W^{?}(\rhobar_{\cS})$ and the Jordan--H\"{o}lder factors of a Deligne--Lusztig representation $R_s(\mu)$ 
will be described in terms of the ``triangle'' $\Sigma_0$ inside $\Lambda_W \times \{C_0,C_1\}$ (see Figure \ref{Triangle}). 
\begin{figure}[h] 
\centering
\includegraphics[scale=.7]{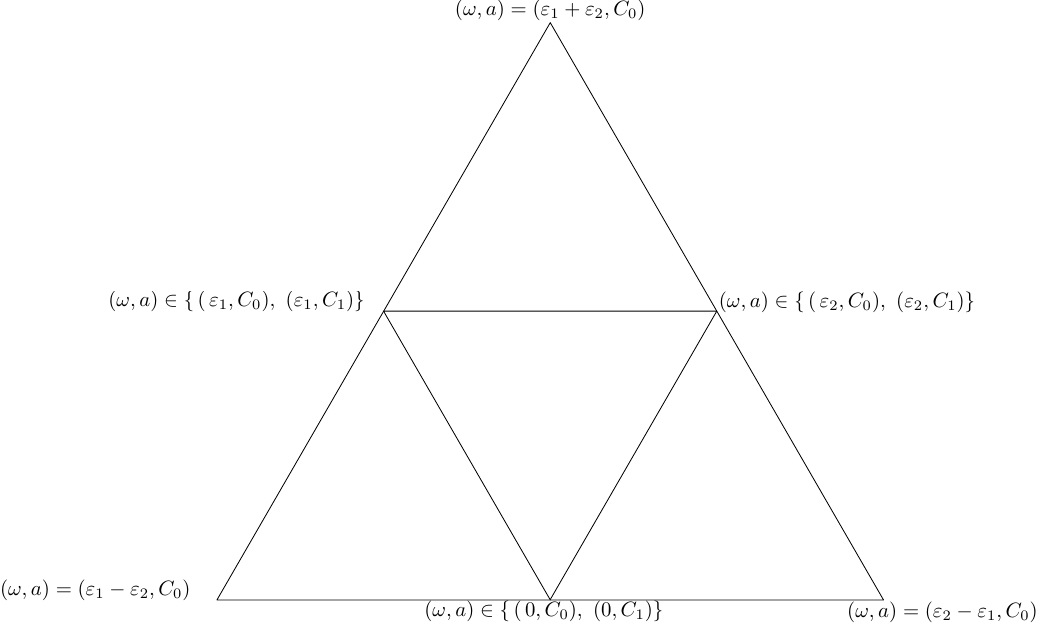}
\caption{The set $\Sigma_0$.}
\label{Triangle}
\end{figure}

\begin{defn} \label{obvweights} Define the following subsets of $\Sigma_0$:
\begin{align*}
\Sigma_0^{\mathrm{obv}} &\defeq  \begin{Bmatrix} (\eps_1+ \eps_2, 0), (\eps_1 - \eps_2, 0), (\eps_2 - \eps_1, 0) \\
 (0, 1) , (\eps_1, 1), (\eps_2, 1) \\
 \end{Bmatrix},\\
\Sigma_0^{\mathrm{inn}}&\defeq  \Sigma_0\setminus \Sigma_0^{\mathrm{obv}}.
\end{align*}
 We say that $(\omega, a) \in \Sigma$ is \emph{obvious} (resp. \emph{inner}) in component $i$ if $(\omega_i, a_i) \in \Sigma_0^{\mathrm{obv}}$ (resp. $(\omega_i, a_i) \in \Sigma_0^{\mathrm{inn}}$).  
\end{defn}

We define an involution $r$ of $\un{\Lambda}^{\mu}_W \times \mathcal{A}$ by $r(\omega,a) = (\omega,a')$ with $a_i \neq a_i'$ for all $i\in \cJ$. 

\begin{prop} \label{prop:serrewts}
Let $\mu\in X^*(\un{T})$.
Suppose that $V(\rhobar_\cS|_{I_{F_\cS}}) = R_s(\mu+\nu)$ where $\nu \in \un{\Lambda}_R$ and $\mu+\nu-\eta$ is $3$-deep in alcove $\un{C}_0$.
Then $W^{?}(\rhobar_\cS)$ is $F(\Trns_\mu(t_\nu sr(\Sigma)))$.
\end{prop}
\begin{proof}
By Proposition \ref{prop:obv}, $W_\mathrm{obv}(\rhobar_\cS)$ is given by 
\[\{F\big(\Trns_{\mu+\nu}(s(\omega),\pi(\tld{w})\cdot \un{C}_0)\big):\tld{w}=wt_{-\pi^{-1}\omega} \in \widetilde{\un{W}}^{+,\mathrm{der}}_1\}.\]
Since $\nu\in\un{\La}_R$, we have $\Trns_{\mu+\nu}(s(\omega),\pi(\tld{w})\cdot \un{C}_0)=\Trns_\mu(s(\omega)+\nu,\pi(\tld{w})\cdot \un{C}_0)$.
By the generalization of Jantzen's formula in \cite[Appendix, Theorem 3.4]{herzig-duke}, the reduction of $R_s(\mu + \nu)$ is given in terms of reductions of Weyl modules.
The fact that $\mu+\nu-\eta$ is $3$-deep implies that these Weyl modules decompose in a generic pattern.
Hence $W^{?}(\rhobar_{\cS})$ is given by \cite[Proposition 10.1.8]{GHS}, which is
\[F(\{\Trns_\mu(t_\nu s(\omega), a)\,|\,\tld{w}=wt_{-\pi^{-1}\omega} \in \widetilde{\un{W}}^{+,\mathrm{der}}_1, \pi (\tld{w})\cdot \un{C}_0\uparrow a\}).\]
This set is precisely $F(\Trns_\mu(t_\nu sr(\Sigma)))$.
\end{proof}

\begin{prop} \label{prop:typedecomp}
If $\nu\in X^*(\un{T})$ such that $\nu+\eta \in \un{\Lambda}_R$ and $\mu+\nu-\eta$ is $3$-deep in alcove $\un{C}_0$, %
then $\JH(\overline{R_s(\mu+\nu)})=F(\Trns_\mu(t_\nu s(\Sigma)))$.
\end{prop}
\begin{proof}
Since $t_{\un{1}}\circ\mathcal{R}\circ \Trns_\mu=\Trns_\mu\circ r$, $F(\Trns_\mu(\omega,a)) \in \JH(\overline{R_s(\lambda-\un{1})})$ if and only if $F(\Trns_\mu(r(\omega,a)))\in \mathcal{R}(\JH(\ovl{R_s(\lambda)})) = W^?(\rhobar_{\cS})$ where $V(\rhobar|_{I_{F_{\cS}}}) = R_s(\lambda)$.
The result then follows from Proposition \ref{prop:serrewts}.
\end{proof}

\begin{defn} \label{typeaction}
For any $\widetilde{w} \in \tld{\un{W}}$, we set $R_{\widetilde{w}}(\mu) \defeq  R_w(\mu+\widetilde{w}(0))$, where $w\in \un{W}$ is the projection of $\widetilde{w}$.
\end{defn}

There is an action of $\tld{\un{W}}$ on $\un{\Lambda}_W \times \mathcal{A}$ where it acts through $\tld{\un{W}}^\der$ on just the first factor (via the \emph{usual} action, not the dot action).
By Propositions \ref{prop:serrewts} and \ref{prop:typedecomp}, we have the following:

\begin{prop} 
\label{prop:equivariance}
Let $\tld{w}\in \un{W}_a t_\eta$ and assume that $\mu+\tld{w}^{-1}(0)-\eta$ is $3$-deep in $\un{C}_0$.
Then the set $\JH(\ovl{R_{\widetilde{w}^{-1}}(\mu)})$ is $F(\Trns_\mu(\widetilde{w}^{-1}(\Sigma)))$.
If $\tld{w} \in \un{W}_a$ and $V(\rhobar_{\cS}|_{I_{F_{\cS}}}) = R_{\widetilde{w}}(\mu)$, then the set $W^?(\rhobar_{\cS})$ is $F(\Trns_\mu(\widetilde{w}(r(\Sigma))))$.
\end{prop}

For the rest of this section, fix a character $\mu$ such that $\mu - \eta$ is $3$-deep in $\un{C}_0$ and an element $s \in \un{W}$.
For $(s\omega, a) \in \un{\Lambda}_W^\mu \times \mathcal{A}$, we let $\sigma_{(\omega, a)}^{(s,\mu)}=F(\Trns_\mu(s\omega,a))$.
When the pair $(s,\mu)$ is clear from the context, we will simply write $\sigma_{(\omega,a)}$ instead of $\sigma_{(\omega,a)}^{(s,\mu)}$ to lighten notation.
Let $\widetilde{w} \in \un{W}_a t_{\eta}$ such that $s\widetilde{w}^{-1}(\Sigma) \subset \un{\Lambda}_W^\mu \times \mathcal{A}$ and $\mu+s\tld{w}^{-1}(0)-\eta$ is $3$-deep.
Then the map
\begin{align}
\label{eq:map:sigma}
\tld{w}^{-1}(\Sigma) &\ra \mathrm{JH}(\ovl{R_{s\tld{w}^{-1}}(\mu)}) \\
\tld{w}^{-1}(\omega,a) &\ra \sigma_{(\tld{w}^{-1}(\omega),a)}^{(s,\mu)}\nonumber
\end{align}
is a bijection by Proposition \ref{prop:equivariance}.
Similarly, if $V(\rhobar_\cS|_{I_{F_\cS}}) = R_s(\mu)$, then the map
\begin{align}
\label{eq:map:sigma1}
r(\Sigma) &\ra W^?(\rhobar_{\cS}) \\
r(\omega,a) &\ra \sigma_{r(\omega,a)}^{(s,\mu)}\nonumber
\end{align}
is a bijection.

\begin{defn} \label{defn:obv} We say that $\sigma\in \mathrm{JH}(\ovl{R_{s\tld{w}^{-1}}(\mu)})$ is an \emph{outer} (resp. \emph{inner}) weight for $R\defeq R_{s \tld{w}^{-1}}(\mu)$ in component $i\in \cJ$ if $\sigma=\sigma_{(\tld{w}^{-1}(\omega),a)}$ with $(\omega_i,a_i)\in \Sigma^{\mathrm{obv}}_0$ (resp. $(\omega_i,a_i)\in \Sigma^{\mathrm{inn}}_0$). 
We define the \emph{defect of $\sigma$ with respect to $R$} as
\begin{equation}
\label{df:defc}
\Def_{R}(\sigma)\defeq \#\big\{i\in \cJ \  \text{such that $\sigma$ is an inner weight of $R$ in component $i$} \big\}.
\end{equation}
Similarly, if $V(\rhobar_\cS|_{I_{F_{\cS}}}) = R_s(\mu)$, we say that $\sigma\in W^{?}(\rhobar_\cS)$ is an \emph{obvious} (resp.~ \emph{shadow}) weight for $\rhobar_\cS$ in component $i \in \cJ$ if $\sigma=\sigma_{r(\omega,a)}$ with $(\omega_i,a_i)\in \Sigma^{\mathrm{obv}}_0$ (resp. $(\omega_i,a_i)\in \Sigma^{\mathrm{inn}}_0$). We define the \emph{defect of $\sigma$ with respect to $\rhobar_\cS$} as
\begin{equation}
\label{df:defc:rhobar}
\Def_{\rhobar_\cS}(\sigma)\defeq \#\big\{i\in \cJ \  \text{such that $\sigma$ is a shadow weight of $\rhobar_\cS$ in component $i$} \big\}.
\end{equation}
\end{defn}

We give the set $\Sigma_0$ the structure of a graph in Table \ref{TableExtGraph}.
We write $\Sigma_0^\cJ$ for the corresponding product graph.
The above bijections respect the notion of adjacency for $\JH(\ovl{R_{s\tld{w}^{-1}}(\mu)})$, $W^?(\rhobar_\cS)$, and $\Sigma_0^\cJ$.

\begin{table}[h]
\caption{\textbf{The graph $\Sigma_0$}}\label{TableExtGraph}
\centering
\adjustbox{max width=\textwidth}{
\begin{tabular}{ c }
$
\xymatrix{
&&&
\sigma_{(\eps_1 + \eps_2, 0)} \ar@{-}[dr]\ar@{-}[dl]&
&&&
\\
&&
\sigma_{(\eps_1, 1)} \ar@{-}[drrr] \ar@{-}[dr] \ar@{-}[dl] \ar@{-}[dll]
&
&
\sigma_{(\eps_2, 1)}\ar@{-}[dr]\ar@{-}[dl]\ar@{-}[dlll]\ar@{-}[drr]
&&
\\
\sigma_{(\eps_1 - \eps_2, 0)} \ar@{-}[drrr] & \sigma_{(\eps_1, 0)}\ar@{-}[drr]
&
&
\sigma_{(\underline{0}, 0)} \ar@{-}[d]
&
&
\sigma_{(\eps_2, 0)} \ar@{-}[dll]&  \sigma_{(\eps_2 - \eps_1, 0)}\ar@{-}[dlll]
\\
&&&\sigma_{(\underline{0}, 1)}&&&
}
$\\
\end{tabular}
}
\end{table}

We now describe the intersections of $\JH(\ovl{R_{s\tld{w}^{-1}}(\mu)})$ and $W^?(\rhobar_\cS)$ for $\rhobar_\cS|_{G_{F_\cS}} \cong \taubar_{\cS}(s,\mu)$ in terms of the action of $\tld{\un{W}}$ on $\un{\Lambda}_W$ using Proposition \ref{prop:equivariance}.  
We first introduce a subset of $\tld{\un{W}}$, the admissible set, which appears in Corollary \ref{cor:intersections} and again in \S \ref{sec:weights}. 

Recall from \S \ref{subsection:Notation} that $\tld{W}$ denotes the extended affine Weyl group of $G\defeq {\GL_3}_{/ \F}$.
The choice of the dominant base alcove endows $\tld{W}$ with a Bruhat order which is denoted by $\leq$. 
We also have the natural generalization of the Bruhat order $\leq$ on $\un{\tld{W}}$ associated to the choice of the dominant base alcove.

\begin{defn} \label{adm} 
Let $\lambda \in X^*(T)$. 
We define
\[
\Adm (\lambda) = \{ \tld{w} \in \widetilde{W} \mid \tld{w} \leq t_{s(\lambda)} \text{ for some } s \in W \}.
\] 
Similarly, for a weight $\mu =(\mu_j)_j \in X^*(\un{T})$ define
\[
\Adm (\mu) = \prod_j \Adm (\mu_j) \subset \widetilde{\un{W}}
\]
\end{defn}
\begin{rmk} 
For combinatorics of weights as in this section, we use the dominant base alcove and the corresponding affine reflection $\gamma^+=t_{(1,0,-1)}(13)$ as a generator for $W^{\der}_{a}$.  When working on the Galois side, as in \S \ref{sec:weights}, the \emph{anti-dominant} base alcove appears naturally, and the analogous set $\Adm^\vee(\eta)$ for the anti-dominant base alcove plays an important role (see Definition \ref{affineadjoint} and the discussion before).  The affine generator for the antidominant choice of base alcove is $\gamma=(13)t_{(1,0,-1)}$. Note that in \cite{LLLM} everything is in terms of anti-dominant and so what we call $\Adm(\eta)$ there we denote by $\Adm^{\vee}(\eta)$ here.  
\end{rmk}

For $\tld{w}\in \tld{W}$, define $\Sigma_{\tld{w}}\defeq \Sigma_0\cap\tld{w}\big(r(\Sigma_0)\big)$.
Similarly, for $\tld{w}\in{\tld{\un{W}}}$, let $\Sigma_{\tld{w}}\defeq \Sigma\cap\tld{w}\big(r(\Sigma)\big)$.
One can directly compute these sets using Figure \ref{Triangle}. 
For example, using Figure \ref{Admissiblepic}, one can check that $\Sigma_{\widetilde{w}}$ is empty if and only if $\tld{w}$ is not in $\Adm(\eta)X^0(\un{T})$ (in essence, \cite[Proposition 4.4.2]{LLL} is a generalization of this fact).
Otherwise, see Table \ref{Intersections} for a description of $\Sigma_{\tld{w}}$. 

\begin{cor} \label{cor:intersections}
Let $\widetilde{w} \in \un{W}_a t_{\eta}$.
Assume that both $\mu - \eta$ and $\mu + s \tld{w}^{-1}(\underline{0}) - \eta$ are $3$-deep in $\un{C}_0$.
The set $\mathcal{R}(\JH(\ovl{R_{s}(\mu)}))\cap \JH(\ovl{R_{s\widetilde{w}^{-1}}(\mu)})$ is the set of weights $\sigma_{r(\omega,a)}^{(s,\mu)}$ where $(\omega, a) \in \Sigma_{\widetilde{w}^{-1}}$. 
Similarly, the set $\JH(\ovl{R_{s}(\mu-\un{1})})\cap \mathcal{R}(\JH(\ovl{R_{s\widetilde{w}}(\mu-\un{1})}))$ is the set of weights $\sigma_{(\omega,a)}^{(s,\mu)}$ where $(\omega, a) \in \Sigma_{\widetilde{w}}$.
\end{cor} 
\begin{proof}  By Proposition \ref{prop:equivariance}, one can reduce the case where $s = \Id$.  In that case, (\ref{eq:map:sigma}) and (\ref{eq:map:sigma1}) imply that the set of Serre weights in the intersection $\mathcal{R}(\JH(\ovl{R_{s}(\mu)}))\cap \JH(\ovl{R_{s\widetilde{w}^{-1}}(\mu)})$ is $\Trns_\mu(\tld{w}^{-1}(\Sigma)\cap r(\Sigma))$.
We conclude by noting that $\tld{w}^{-1}(\Sigma)\cap r(\Sigma) = r(\Sigma_{\tld{w}^{-1}})$.
The statement for $\JH(\ovl{R_{s}(\mu-\un{1})})\cap \mathcal{R}(\JH(\ovl{R_{s\widetilde{w}}(\mu-\un{1})}))$ is proved similarly.
\end{proof}

\begin{table}[h]
\caption{Intersections} \label{Intersections} 
\begin{center}
\adjustbox{max width=\textwidth}{%
\begin{tabular}{| c | c || c | c |}
\hline
& & & \\
$\widetilde{w}t_{-\un{1}}$  & $\Sigma_{\widetilde{w}}$ & $\widetilde{w}t_{-\un{1}}$ & $\Sigma_{\widetilde{w}}$\\
& & & \\
\hline
& & & \\
$\gamma^+ \alpha \beta \alpha = t_{(1,0,-1)}$ & $\{ (\eps_1 + \eps_2, 0)\}$    &$ \alpha \beta  \alpha \gamma^+=t_{(-1,0,1)}$& $\{ (0, 1) \}$ \\
& & & \\
\hline
& & & \\
$\beta \gamma^+ \beta \alpha=t_{(1,-1,0)}$ &  $\{ (\eps_1,1) \}$    &$\alpha \beta \gamma^+ \beta=t_{(-1,1,0)}$ & $\{ (\eps_2-\eps_1 , 0)\}$  \\
& & & \\
\hline
& & & \\
$\alpha \gamma^+ \alpha \beta=t_{(0,1,-1)}$& $\{  (\eps_2,1) \}$   &$\beta \gamma^+ \alpha \gamma^+=t_{(0,-1,1)}$ &  $\{ (\eps_1 - \eps_2, 0)\}$ \\ 
& & & \\
\hline
& & & \\
$\beta \alpha \gamma^+$ & $\{(\eps_1 -  \eps_2, 0), (\underline{0},1) \}$ &$\alpha \beta \gamma^+$ & $\{(\eps_2 - \eps_1, 0), (0,1) \}$  \\
& & &\\
\hline
& & &\\
$\alpha \beta \alpha$ & $\{(\underline{0},1), (\underline{0},0) \}$ &   & \\
& & &\\
\hline 
& & &\\
$\beta \alpha$ &  $\begin{Bmatrix}(\underline{0},1), & (\underline{0},0) \\ (\eps_1 - \eps_2, 0), & (\eps_1, 1) \end{Bmatrix}$ & $\alpha \beta$ & $\begin{Bmatrix}(\underline{0},1), & (\underline{0},0) \\ (\eps_2 - \eps_1, 0), & (\eps_2, 1) \end{Bmatrix}$\\
& & &\\
\hline
& & &\\
$\alpha$ & $\begin{Bmatrix}(\underline{0},1), & (\underline{0},0), & (\eps_2, 1) \\ (\eps_2, 0), & (\eps_2 - \eps_1, 0), & (\eps_1, 1) \end{Bmatrix}$& $\mathrm{Id}$ & $ \begin{Bmatrix} (\underline{0},1) & (\eps_1, 1) & (\eps_2, 1)  \\
 (\underline{0},0) & (\eps_1, 0) & (\eps_2, 0)   \end{Bmatrix}$
 \\
& & &\\
\hline
\end{tabular}}
\end{center}
\caption*{Not all elements of $\Adm(\eta)t_{-\un{1}}$ appear in the table. There is an order three symmetry of $\Adm(\eta)t_{-\un{1}}$ induced by the outer automorphism of $\tld{W}^{\der}$, and we include at least one representative for each orbit.  
} 
\end{table}

\begin{figure}[h] 
\centering
\includegraphics[scale=.4]{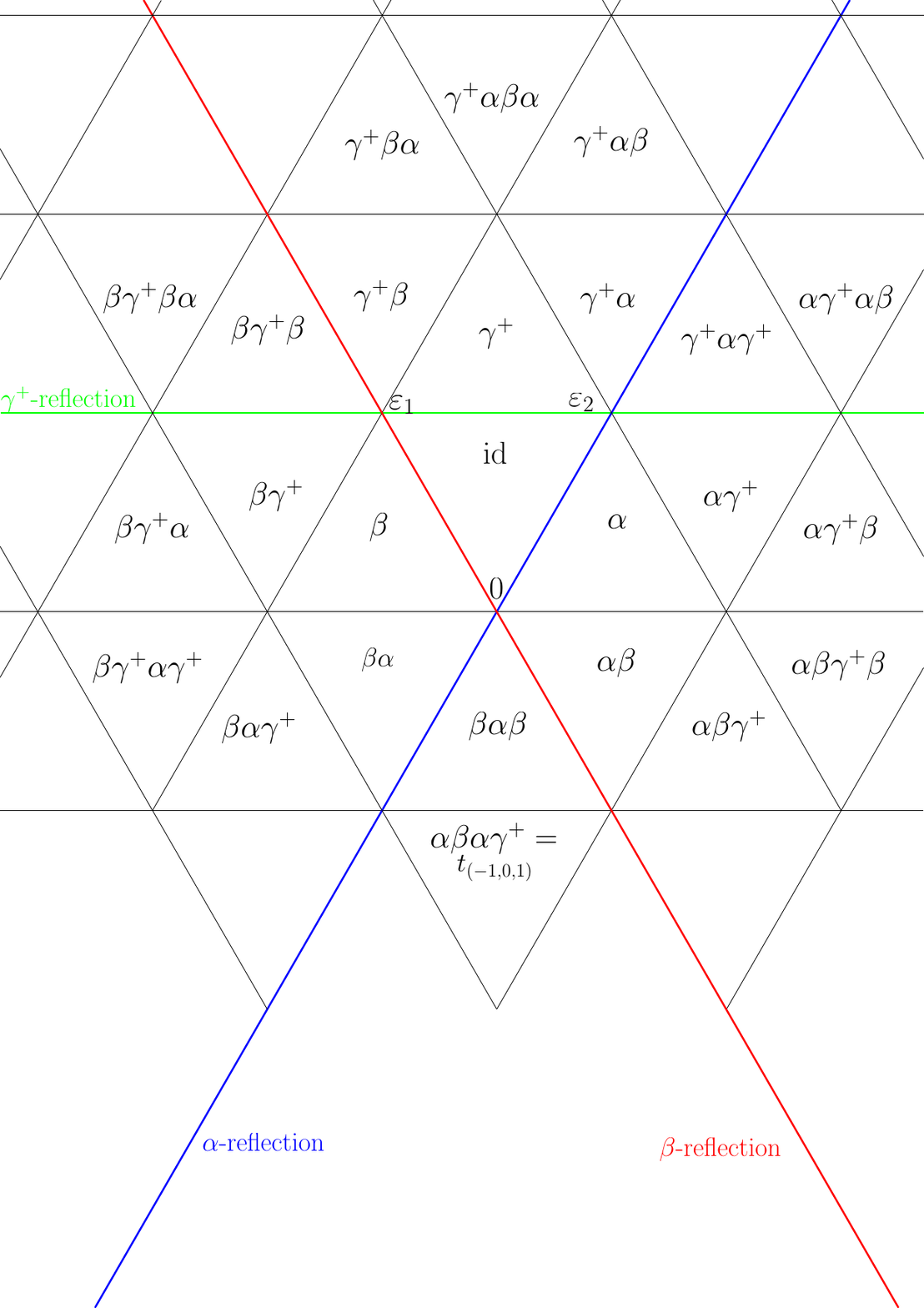}
\caption{The standard apartment of $\mathrm{SL}_3$. Labeled alcoves are alcoves in $\Adm (\eta)$.}
\label{Admissiblepic}
\end{figure}

\begin{lemma} \label{lemma:elimpic}
If $(\omega,a)\in \un{\Lambda}_W\times \mathcal{A} \setminus r(\Sigma)$, there exists $\widetilde{w}\in \un{W}_a t_{\eta}$ such that 
$(\omega,a)\in \widetilde{w}^{-1}(\Sigma)$ and $\widetilde{w}\not\in\Adm(\eta)$. 
\end{lemma}
\begin{proof}
We can work one component at a time, and thus reduce to the case where $\#\cJ=1$.
This can then be checked directly using Figure \ref{Admissiblepic}.
\end{proof}

\begin{cor} \label{separateweights}
Let $\lambda\in X_1(\un{T})$ be 5-deep and $\mu-\eta \in \un{C}_0$ be 3-deep. 
If $F(\lambda) \notin \mathcal{R}(\JH(\ovl{R_s(\mu)}))$,  then there exists a $3$-generic Deligne--Lusztig representation $R_{s'}(\mu')$ such that $F(\lambda) \in \JH(\ovl{R_{s'}(\mu')})$ but 
\[
 \mathcal{R}(\JH(\ovl{R_s(\mu)})) \cap  \JH(\ovl{R_{s'}(\mu')}) = \emptyset.
 \]
\end{cor}
\begin{proof}
If the central characters do not match, i.e.,  $(\lambda - \mu+\eta)|_{\un{Z}} \notin (p- \pi)X^*(\un{Z})$, then any Deligne-Lusztig representation which contains $F(\lambda)$ works.
Otherwise, by Proposition \ref{connected}, there exists $(s\omega, a) \in \un{\Lambda}_{W}^{\mu} \times \mathcal{A}$ such that 
$\sigma_{(\omega, a)}^{(s,\mu)} \cong F(\lambda)$.
Moreover, since $F(\lambda) \notin \mathcal{R}(\JH(\ovl{R_s(\mu)}))$, we have $(\omega,a) \notin r(\Sigma)$.
By Lemma \ref{lemma:elimpic}, there exists $\tld{w} \in \un{W}_a t_{\eta} \setminus \Adm(\eta)$ such that $(\omega,a) \in \tld{w}^{-1}(\Sigma)$.
We take $s' = sw^{-1}$ and $\mu' = \mu+s\tld{w}^{-1}(0)$ so that $R_{s'}(\mu')=R_{s\tld{w}^{-1}}(\mu)$.
Since $(\omega,a) \in \tld{w}^{-1}(\Sigma)$, $\omega - \tld{w}^{-1}(0)$ is in $\un{W}(\ovl{\Sigma})$, where $\ovl{\Sigma}$ denotes the image of $\Sigma$ in $\un{\Lambda}_W$.
A direct computation shows that since $\lambda$ is $5$-deep, $\mu+s\tld{w}^{-1}(0)-\eta$ is $3$-deep in $\un{C}_0$.
By Proposition \ref{prop:equivariance}, we have that  $F(\lambda) \in \JH(\ovl{R_{s'}(\mu')})$.
By Corollary \ref{cor:intersections} and the remark before it, we have that $\mathcal{R}(\JH(\ovl{R_s(\mu)})) \cap  \JH(\ovl{R_{s'}(\mu')}) = \emptyset$.
\end{proof}

\section{Serre weight conjectures} \label{sec:weights}

In \S 3.1-3.3, we improve results of \cite{LLLM} on Kisin modules and \'etale $\phz$-modules with descent data.  Most of these results are generalized to $\GL_n$ in \cite{LLL}.  %
In \S \ref{subsec:et:dd}, we prove a key structure result relating the Frobenius of a Kisin module with tame descent data to the Frobenius of the corresponding \'etale $\phi$-module (Proposition \ref{prop:phif}). This plays an important role in computing the Galois representation associated to semisimple Kisin modules (Proposition \ref{prop:phiV}) and later establishing an explicit geometric Breuil-M\'ezard conjecture (see \S \ref{subsec:GBM}).  
The remainder of \S \ref{Semisimple Kisin modules} is devoted to bounding reductions of potentially crystalline representations of type $(\eta, \tau)$ for the purpose of weight elimination (Theorem \ref{fixedpoints}).

In \S \ref{sec:SandSW}, we connect the results of \S \ref{sub:backgr}-\ref{Semisimple Kisin modules} with the combinatorics of Serre weights from \S \ref{subsec:comb}.  In particular, Proposition \ref{typematching} describes the intersection of the set of predicted weights of a semisimple $\rhobar$ with the set of Jordan--H\"older constituents of a Deligne--Lusztig representation, as a function of its shape.  In \S \ref{sec:SWC}, we prove the Serre weight conjectures for semisimple $\rhobar$ by matching the number of predicted Serre weights to the number of irreducible components of the Galois deformation rings which were computed in \cite[Theorem \ref{thm:BM}]{LLLM}.  In \S \ref{subsec:GBM}, we prove a geometric Breuil-M\'ezard conjecture which assigns to each (predicted) Serre weight a prime ideal of the universal framed deformation ring (Proposition \ref{prop:identify:cmpt}).  In \S \ref{sec:match}, we make this assignment explicit for certain tamely potentially crystalline deformation rings.  The explicit ideals play an important role in the proof of Breuil's lattice conjecture in \S \ref{sec:BC}.

Throughout this section, we assume that $\cS=\{\tld{v}\}$ unless otherwise stated, and write $K\defeq F_{\tld{v}}$. We drop the subscript $\tld{v}$ from notation in this situation. The set $\cJ$ is identified with the set of embeddings $K\into E$ and also with $\Z/f\Z$ using our chosen embedding $\sigma_0$ (cf.~Section  \ref{subsection:Notation}). For $i\in \cJ$, we will use the word component $i$ and embedding $i$ interchangeably.%

\subsection{Background}
\label{sub:backgr}

In this subsection, we recall some basic facts on Kisin modules with tame descent data.
We start with the following involution on $\tld{\un{W}}$, which naturally appears when passing from Kisin modules to $G_{K_\infty}$-representations.
We let $\tld{\un{W}}^{\vee}$ (resp. $\tld{W}^\vee$) be the partially ordered group which is identified with $\tld{\un{W}}$ (resp. $\tld{W}$) as a group, and whose Bruhat order is defined by the \emph{antidominant} base alcove (and still denoted as $\leq$).
\begin{defn} \label{affineadjoint}
Define a bijection $\tld{w}\mapsto \tld{w}^*$ between $\tld{\un{W}}^\vee$ and $\tld{\un{W}}$ as follows:
\begin{enumerate}
\item For $w = (w_j) \in \un{W}$ define $w^* = (w^*_j) \in \un{W}$ by $w^*_j = w_{f-1-j}^{-1}$.
\item For $\nu = (\nu_j) \in X^*(\un{T})$ define $\nu^* = (\nu^*_j) \in X^*(\un{T})$ by $\nu^*_j = \nu_{f-1-j}$.
\item For $\widetilde{w} = wt_\nu \in \widetilde{\un{W}}^\vee$ define $\widetilde{w}^* \in \widetilde{\un{W}}$  by $\widetilde{w}^* = t_{\nu^*} w^*$.
\end{enumerate}
Note that $\widetilde{w} \mapsto \widetilde{w}^*$ is an antihomomorphism of groups. 
For $j\in\cJ$ we write $\tld{w}^*_j$ for the $j$-th component of $\tld{w}^*$.
\end{defn}

Let $\tau$ be a tame inertial type which we always assume to be 1-generic. Fix a lowest alcove presentation $(s, \mu)$ for $\tau$ (Definition \ref{defn:tau}).

If $s = (s_0, \ldots, s_{f-1})$ and $\mu = (\mu_j)_{0 \leq j \leq f-1} \in X^*(\un{T})$, we take $s_\tau \defeq s_0 s_{f-1} s_{f-2} \cdots s_1 \in W$ and  $\bm{\alpha}_{(s, \mu)} \in X^*(\un{T})$ such that $\bm{\alpha}_{(s,\mu), j} = s_1^{-1} s_2^{-1} \ldots s_j^{-1}(\mu_j + \eta_j)$ for $1 \leq j \leq f-1$ and $\bm{\alpha}_{(s, \mu), 0} = \mu_0 + \eta_0$. Let $r \in \{1,2,3\}$ be the order of $s_{\tau}$, and write $K'$ for the unramified extension of $\Qp$ of degree $f'\defeq fr$.
 
\begin{rmk} \label{rmk:compare}
In \cite{LLLM}, the notion of lowest alcove presentation does not appear. Everything is written for presentations of the form $\tau((s_{\tau}, 1, \ldots, 1), \bm{\alpha}_{(s, \mu)})$ (see, for example, the beginning of \cite[\S 2.1, \S 6.1]{LLLM}). 
In the notation of \emph{loc.~cit.},  $\bm{\alpha}_{(s, \mu), j} = (a_{1, j}, a_{2,j}, a_{3,j})$.  The element \[s_{\mathrm{or}} \defeq (s_1^{-1} s_2^{-1} \ldots s_{f-1}^{-1}, s_1^{-1} s_2^{-1} \ldots s_{f-2}^{-1},   \ldots, s_1^{-1},  1) \in \un{W},\] has the property that $s^*_{\mathrm{or}}(\bm{\alpha}_{(s, \mu)}) = \mu + \eta$, and hence ${}^{(0, s^*_{\mathrm{or}})}(((s_{\tau}, 1, \ldots, 1), \bm{\alpha}_{(s, \mu)}))=(s,\mu+\eta)$.
The element $s_{\mathrm{or}}$ is called the \emph{orientation} of $\bm{\alpha}_{(s, \mu)}$ (\cite[Definition 2.6 and equation (2.2)]{LLLM}).  
\end{rmk}

If $r =1$, we say that $\tau$ is a \emph{principal series type}. Otherwise, we write $\tau'$ for the base change of $\tau$ to $K'/K$ (which is just $\tau$ considered as a principal series type for $G_{K'}$). We record the relevant data for $\tau'$.  
Define $\bm{\alpha}'_{(s,\mu)} \in  X^*(T)^{\Hom(k', \F)} \cong X^*(\un{T})^{r}$ (using a choice of embedding $\sigma_0':k'\into \F$ extending $\sigma_0$) by 
\[
\bm{\alpha}'_{(s,\mu), j + kf} \defeq  s_{\tau}^{-k} (\bm{\alpha}_{(s, \mu), j}) \text{ for } 0 \leq j \leq f-1, 0 \leq k \leq r-1.
\]   

If $\tau_{K'}(w', \mu')$ is the analogous construction of tame types over $K'$ for $(w', \mu') \in (\un{W} \times X^*(\un{T}))^{r}$ as in Definition \ref{defn:tau}, then $\tau' \cong \tau_{K'}(1, \bm{\alpha}'_{(s,\mu)})$ by direct comparison using \eqref{eq:def:type}.  The \emph{orientation} $s'_{\mathrm{or}} \in \un{W}^{r}$ of $\bm{\alpha}'_{(s,\mu)}$ in the sense of \cite[Definition 2.6]{LLLM} is given by 
\begin{equation} \label{primeorient}
s'_{\mathrm{or}, j + kf} \defeq s_{\tau}^{k+1} s_{\mathrm{or}, j}\text{ for } 0 \leq j \leq f-1, 0 \leq k \leq r-1 
\end{equation} 
 (compare with \cite[Proposition 6.1]{LLLM}).

Let $L'=K'(\pi_r)=K'((-p)^{\frac{1}{p^{rf}-1}})$ and $\Delta'\defeq \Gal(L'/K')\subseteq \Delta\defeq \Gal(L'/K)$.
Note that $\tau$ defines a $\cO$-valued representation of $\Delta'$.
For any complete local Noetherian $\cO$-algebra $R$ with residue field $\F'$ finite over $\F$, let $\fS_{L', R} \defeq (W(k') \otimes_{\Zp} R)[\![u']\!]$. We endow $\fS_{L', R}$ with an action of $\Delta$ as follows: for any $g$ in $\Delta'$, $g(u') = \frac{g(\pi_r)}{\pi_r} u'$ and $g$ acts trivially on the coefficients; if $\sigma \in\Gal(L'/\Qp)$ is the lift of Frobenius on $W(k')$ which fixes $\pi_r$, then $\sigma^f$ generates $\Gal(K'/K)$ acting in natural way on $W(k')$ and trivially on both $u'$ and $R$. Set $v = (u')^{p^{rf}-1}$, and note that 
\[
(\fS_{L', R})^{\Delta = 1} = (W(k) \otimes_{\Zp} R)[\![v]\!].
\]
As usual, $\varphi:\fS_{L', R} \ra \fS_{L', R}$ acts as $\sigma$ on $W(k')$, trivially on $R$, and sends $u'$ to $(u')^{p}$.

If $\tau$ is a principal series types, let $Y^{[0, h], \tau}(R)$ be the category of Kisin modules over $L'$ with tame descent of type $\tau$ and height in $[0,h]$ as defined in \cite[\S 3]{CL}.
More generally, we have the following:

\begin{defn} \label{defn:Kisin} An element $(\fM, \phi_{\fM}, \{ \widehat{g} \}) \in Y^{[0, h], \tau}(R)$ is a Kisin module $(\fM, \phi_{\fM})$ over $\fS_{L', R}$ (\cite[Definition 2.3]{LLLM}) with height less than $h$ together with a semilinear action of $\Delta$ which commutes with $\phi_{\fM}$ such that for each $0 \leq j \leq f' - 1$ 
\[
\fM^{(j)} \mod u' \cong \tau^{\vee} \otimes_{\cO} R 
\]  
as $\Delta'$-representations. In particular, the semilinear action induces an isomorphism $\iota_{\fM}:(\sigma^f)^*(\fM) \cong \fM$ (see \cite[\S 6.1]{LLLM}) as elements of $Y^{[0,h], \tau'}(R)$.
\end{defn}   
\begin{rmk}
As explained in \cite[\S 6.1]{LLLM}, the data of an extension of the action of $\Delta'$ to an action of $\Delta$ is equivalent to the choice of an isomorphism $\iota_{\fM}:(\sigma^f)^*(\fM) \cong \fM$ satisfying an appropriate cocycle condition. We will use both point of view interchangeably.
\end{rmk}
\begin{rmk} \label{tauvee} The appearance of $\tau^{\vee}$ in the definition is due to the fact that we are using the contravariant functors to Galois representations to be consistent with \cite{LLLM} as opposed to the covariants versions which appear in \cite{CL, EGH}.  In \cite{LLLM}, we didn't use the notation $\tau^{\vee}$. Instead, we included it in our description of descent data by having a minus sign in the equation before Definition 2.1 of \emph{loc.~cit.}   
The notion of Kisin module with tame descent data of type $\tau$ here is consistent with what appears in \emph{loc.~cit.}
\end{rmk}

\begin{defn} \label{def:eigen} An \emph{eigenbasis} of $\fM \in Y^{[0,h], \tau}(R)$ is an eigenbasis $\beta'=({\beta'}^{(j')})_{j'}$ for $\fM$ considered as a Kisin module with descent data of type $\tau'$ in the sense of \cite[Definition 2.8]{LLLM}, which is compatible with the isomorphism $\iota_{\fM}$: by letting  ${\beta'}^{(j')}=({f'_1}^{(j')},{f'_2}^{(j')}, {f'_3}^{(j')})$ then
\begin{equation*}
\xymatrix{
\fM^{(j')}&\big((\sigma^f)^*(\fM)\big)^{(j'+f)}\ar_-{\sim}[l]\ar^-{\iota_{\fM}}[r]&\fM^{(j'+f)}\\
\{{f'_1}^{(j')},{f'_2}^{(j')}, {f'_3}^{(j')}\}\ar@{|->}[rr]&&
\{{f'_1}^{(j'+f)},{f'_2}^{(j'+f)}, {f'_3}^{(j'+f)}\}
}
\end{equation*}
for all $0\leq j'\leq f'-1$, where the isomorphism on the left hand side is obtained from \cite[Lemma 6.2]{LLLM}.
For short, the compatibility above will be written as $\iota_{\fM}({\beta'}^{(j')})={\beta'}^{(j'+f)}$. (See also \cite[Definition 3.2.8]{LLL}.) 
\end{defn} 

\subsection{\'Etale $\phz$-modules with descent data} 
\label{subsec:et:dd}

The main result of this subsection is Proposition \ref{prop:phif}, describing the Frobenius action on \'etale $\phz$-modules with tame descent data.

Let  $\cO_{\cE,K'}$ (resp. $\cO_{\cE,L'}$) be the $p$-adic completion of $(W(k')[\![v]\!])[1/v]$ (resp. of $(W(k')[\![u']\!])[1/u']$). 
Recall from \cite[\S 2.3]{LLLM}, that for a complete local Noetherian $\cO$-algebra $R$ we have the category $\Phi\text{-}\Mod^{\text{\'et}}_{K'}(R)$ (resp.  $\Phi\text{-}\Mod^{\text{\'et}}_{dd,L'}(R)$) of \'etale $(\phz,\cO_{\cE,K'}\widehat{\otimes}_{\Zp}R)$-modules (resp. \'etale $(\phz,\cO_{\cE,L'}\widehat{\otimes}_{\Zp}R)$-modules with descent data from $L'$ to $K'$).
There is an analogous definition of $\Phi\text{-}\Mod^{\text{\'et}}_{K}(R)$ and  $\Phi\text{-}\Mod^{\text{\'et}}_{dd,L}(R)$.
Given $(\fM,\iota_{\fM})\in Y^{[0,h],\tau}(R)$, the element $\fM \otimes\cO_{\cE,L'}$ is naturally an object $\Phi\text{-}\Mod^{\text{\'et}}_{dd,L'}(R)$. 

We define an \'etale $\phz$-module $\cM \in \Phi\text{-}\Mod^{\text{\'et}}_{K}(R)$ by 
\[
\cM \defeq (\fM\otimes\cO_{\cE,L'})^{\Delta=1} %
\]
with the induced Frobenius.  This defines a functor from $Y^{[0,h],\tau}(R)$ to $\Phi\text{-}\Mod^{\text{\'et}}_{K}(R)$.  %
Finally, recall  the functor $\eps_0:\Phi\text{-}\Mod^{\text{\'et}}_{K}(R)\rightarrow \Phi^f\text{-}\Mod^{\text{\'et}}_{W(k)}(R)$ from \cite[\S 2.3]{LLLM}. It is obtained by considering the $f$-fold composite of the partial Frobenii which acts on $\cM^{(0)}$.

Let $R$ be a local, Artinian $\cO$-algebra with finite residue field $\F$. We have the usual functor $\bV^*_K$ from $\Phi\text{-}\Mod^{\text{\'et}}_{K}(R)$ to representations of $G_{K_{\infty}}$ over $R$.
Recall from \cite[\S 6.1]{LLLM} the functor
\begin{equation*}
T_{dd'}^*:Y^{[0,h], \tau}(R)\ra \Rep_R(G_{K_{\infty}})
\end{equation*}  
which is defined as $(\fM,\iota_{\fM})\mapsto \bV_K^*(\cM)$.
From now onwards, \emph{we write $T_{dd}^*$ for the functor which was written as $T_{dd'}^*$ in \cite[\S 6.1]{LLLM}.}

We can now state the main result of this section, which is a sharpening of \cite[Proposition 2.26]{LLLM}.

\begin{prop}
\label{prop:phif} 
Let $\tau$ be a $1$-generic type with lowest alcove presentation $(s, \mu)$. 
Let $\fM \in Y^{[0,h], \tau}(R)$ and $\beta$ be an eigenbasis of $\fM$. Let $s_{\orient}=(s_{\orient,j})_j$ be the 
orientation for $\tau$ and write $A^{(j)}=\Mat_{\beta}\big(\phi^{(j)}_{\fM, s_{\orient,j+1}(3)}\big)$ for $j\in\{0,\dots,f-1\}$.
Let $\cM = (\fM \otimes \cO_{\cE, L'})^{\Delta=1} \in \Phi\text{-}\Mod^{\text{\'et}}_K(R)$.

Then there is a basis $\mathfrak{f} = (\mathfrak{f}^{(j)})_j$ for $\cM$ such that
\[
\Mat_{\mathfrak{f}}(\phi_{\cM}^{(j)}) = A^{(j)}  s^*_j v^{\mu^*_j + \eta^*_j}. 
\]
\end{prop}

\begin{proof}
The proof is a direct computation using  \cite[Proposition 2.26 and Proposition 6.1]{LLLM}.

Let $\bm{\alpha}'_{(s, \mu)}$ and $(s'_{\orient,j'})$ be as in the discussion after Remark \ref{rmk:compare}. Let ${\beta'}^{(j')} = ({f_1'}^{(j')}, {f_2'}^{(j')}, {f_3'}^{(j')})$ be an eigenbasis for $\fM$ and write ${A}^{(j')}\defeq \Mat_{\beta'}\big(\phi^{(j')}_{\fM', s'_{\orient,j'+1}(3)}\big)$ for $0 \leq j' < f'-1$.

For any $0 \leq j' \leq f' -1$, set
\[
\bf{a}_{(s,\mu)}^{\prime \, (j')} = \sum_{i=0}^{f'-1} \bm{\alpha}'_{(s, \mu), -j' + i} p^i \in \Z^3
\]
where $-j' + i$ is taken modulo $f'$ (compare with \cite[\S 2.1]{LLLM}). Then \[{\mathfrak{f}'}^{(j')} \defeq \{(u')^{\bf{a}_{(s,\mu), 1}^{\prime \, (j')}}{f_1'}^{(j')}, (u')^{\bf{a}_{(s,\mu), 2}^{\prime \, (j')}} {f_2'}^{(j')}, (u')^{\bf{a}_{(s,\mu), 3}^{\prime \, (j')}}{f_3'}^{(j')}\}\] is a basis for ${\cM'}^{(j')}\defeq({\fM}^{(j')}[1/u'])^{\Delta'=1}$ for $0 \leq j' < f'-1$.

 For any $0 \leq j' \leq f'-1$, a direct computation (using \cite[Proposition 2.13]{LLLM}) shows that the matrix for $\phi_{\cM'}^{(j')}:{\cM'}^{(j')} \ra {\cM'}^{(j'+1)}$ with respect to the bases above is:
\[
s'_{\mathrm{or}, j'+1} A^{(j')} (s'_{\mathrm{or}, j'+1})^{-1} (u')^{p \bf{a}^{\prime \, (j')}_{(s,\mu)} - \bf{a}^{\prime \, (j'+1)}_{(s,\mu)}}.    
\]
Since $p \bf{a}^{\prime \, (j')}_{(s,\mu)} - \bf{a}^{\prime \, (j'+1)}_{(s, \mu)} = (p^{f'} -1) \bm{\alpha}'_{(s, \mu), f'-1-j'}$, this is same as 
\begin{equation} 
\label{phi-mod-dd}
s'_{\mathrm{or}, j'+1} A^{(j')} (s'_{\mathrm{or}, j'+1})^{-1}  v^{\bm{\alpha}'_{(s, \mu), f'-1-j'}}.  
\end{equation}
Define $\tld{\beta}$ by $\tld{\beta}^{(j')} =  {\mathfrak{f}'}^{(j')} s'_{\mathrm{or}, j'}$ (reordering the basis vectors). Let $j' = j + if$ for $0 \leq j \leq f-1$. Then the matrix for $\phi_{\cM'}^{(j')}$ with respect to $\tld{\beta}$ is given by 
\begin{equation} \label{eq:final}
 A^{(j')} (s'_{\mathrm{or}, j'+1})^{-1}  s'_{\mathrm{or},j'} v^{(s'_{\mathrm{or},j'})^{-1} (\bm{\alpha}'_{(s, \mu), f'-1-j'})} =  A^{(j)} s^*_j v^{\mu^*_j + \eta^*_j}
\end{equation}
using \eqref{primeorient}, Remark \ref{rmk:compare} and the fact that $A^{(j')}$ only depends on $j'$ mod $f$ (cf. \cite[Proposition 6.9]{LLLM}). 

Finally, recall that the eigenbasis $\beta'$ is required to be compatible with $\iota_{\fM}:(\sigma^f)^*(\fM) \cong \fM$ as in Definition \ref{def:eigen}, which gives $\iota^{(j')}_\fM(\beta'^{(j')})=\beta'^{(j'+f)}s_\tau^{-1}$. Hence $\tld{\beta}$ descends to a collection of ordered bases $(\mathfrak{f}^{(j)})_{0 \leq j \leq f-1}$ of the $\cM^{(j)} = ((\fM[1/u'])^{\Delta = 1})^{(j)}$ such the matrix for the partial Frobenius map $\cM^{(j)} \ra \cM^{(j+1)}$ is given by \eqref{eq:final}.

\end{proof}

\begin{prop} 
\label{prop:phif:1} 
Let $\fM \in Y^{[0,h],\tau}(R)$ and let $\cM = (\fM \otimes \cO_{\cE, L'})^{\Delta=1} \in \Phi\text{-}\Mod^{\text{\'et}}_K(R)$ be the associated \'etale $\phz$-module over $K$.

If $\mathfrak{f}$ is a basis of  $\cM$ and $B^{(j)} \defeq\Mat_{\mathfrak{f}}(\phi_{\cM}^{(j)})$ are the matrices of the partial Frobenii, then the \'etale $\phz^f$-module $\eps_0(\cM)$ is described by 
\begin{equation*}
\Mat_{\mathfrak{f}^{(0)}}(\phi^{f}_{\cM^{(0)}})=\prod^{f-1}_{j=0} \phz^j (B^{(f-j-1)}).
\end{equation*}
\end{prop} 

\begin{lemma} \label{phiV} Let $\rhobar: G_K\ra \GL_3(\F)$ be semisimple and $0$-generic.  
Let $\cM\in \Phi^f\text{-}\Mod^{\mathrm{\acute{e}t}}_{W(k)}(\F)$ such that $\bV^*_K(\cM) \cong \rhobar|_{G_{K_{\infty}}}$.  
Assume there exists $s_0\in W$, $\lambda = (\lambda_j) \in X^*_1(\un{T})$ and a basis for $\cM$ such that the $\phi^f$-action on $\mathcal{M}$ is given by $D s_0^{-1} v^\mu$ where $D \in T(\F)$ and $\mu \defeq \sum_{j=0}^{f-1}p^j\lambda_j\in X^*(T)$.
Then by letting $s=(s_0,1,\dots,1)\in \un{W}$,
\[
V(\rhobar|_{I_K}) = R_s(\lambda).
\]
\end{lemma}
\begin{proof}
Let $k$ be the order of $s_0$ and write $\mu=(\mu_1,\mu_2,\mu_3)\in X^*(T)$. 
Then 
\[\phi^{fk}(e_i) = \left(\prod_{m=0}^{k-1} D_{s_0^{-m}(i)} \right) v^{\sum_{m=0}^{k-1}p^{f(k-1-m)}\mu_{s_0^{-m}(i)}} e_i= \left(\prod_{m=0}^{k-1} D_{s_0^{m+1}(i)}\right) v^{\sum_{m=0}^{k-1} p^{fm} \mu_{s_0^{m+1}(i)}} e_i.\]
Hence $\rhobar|_{I_K}$ is isomorphic to $\bigoplus_{i=1}^3 \omega_{kf}^{\sum_{m=0}^{k-1} p^{fm} \mu_{s_0^{m+1}(i)}}$.
The lemma now follows from  (\ref{eq:def:type}).
\end{proof}

\subsection{Semisimple Kisin modules}
\label{Semisimple Kisin modules}

The statement of \cite[Theorem 2.21]{LLLM} gives a classification of Kisin modules with descent datum which can arise as reductions of potentially crystalline representations with Hodge--Tate weights $(2,1,0)^f$.  In this subsection, we identify Kisin modules which correspond to semisimple Galois representations in an explicit way (Theorem \ref{fixedpoints}).   

For $\F'/\F$ finite, let $\cI(\F') \subset \GL_3(\F'[\![v]\!])$ be the Iwahori subgroup of elements which are upper triangular modulo $v$.  
Recall from \S \ref{sub:backgr} the partially ordered group $\un{\tld{W}}^\vee$ (resp. $\tld{W}^\vee$) which is identified with $\tld{\un{W}}$ (resp. $\tld{W}$) as a group, but whose Bruhat order is defined by the \emph{antidominant} base alcove (and still denoted as $\leq$).
For any character $\mu\in X^*(\un{T})$, we define $\Adm^\vee(\mu)$ as in Definition \ref{adm} but for antidominant base alcove.

Let $h$ be a positive integer. 
As before, let $\tau$ be a $1$-generic tame type with fixed lowest alcove presentation $(s, \mu)$.  
Recall from \cite[Definition 2.17]{LLLM} the notion of \emph{shape}:
\begin{defn}
\label{defn:shape}  Let $\tau$ be a principal series type,  and let $\tld{w} = (\widetilde{w}_0, \widetilde{w}_1, \ldots, \widetilde{w}_{f-1}) \in \tld{\un{W}}^\vee$. %
A Kisin module $\fM \in Y^{[0,h], \tau}(\F')$ has \emph{shape} $\tld{w}$ if for any eigenbasis $\beta$, the matrices $ \big(A^{(j)} \big)_j = \big(\Mat_{\beta}\big(\phi_{\fM, s_{\orient,j+1}(3)}^{(j)} \big) \big)_j$ have the property that $A^{(j)} \in \Iw(\F') \widetilde{w}_j \Iw(\F')$. 

For a non-principal series type, we define the \emph{shape} of a Kisin module $\fM \in Y^{[0,h], \tau}(\F')$ in terms of its base change $\mathrm{BC}(\fM)$, see \cite[Definition 3.2.11]{LLL} or \cite[Definition 6.10]{LLLM} for details.
\end{defn} 

For any $\lambda \in X^*(\un{T})$ effective (i.e. $\lambda_j=(a_{i,j})_i$ with $a_{i,j}\geq 0$ for all $i,\,j$) %
and $h\gg0$, let $Y^{\lambda, \tau}\subset Y^{[0,h], \tau}$ denote the closed substack defined in \cite[\S 5]{CL} (cf. also \cite[\S 2.2]{LLLM}).
Then for any finite extension $\F'/\F$, $Y^{\lambda, \tau}(\F') \subset Y^{[0,h], \tau}(\F')$ is the subgroupoid consisting of Kisin modules with shape in $\mathrm{Adm}^\vee(\lambda)$.%

\begin{defn} \label{shaperhobar} Let $\tau$ be a 3-generic type with lowest alcove presentation $(s, \mu)$ and let $\rhobar: G_K\ra \GL_3(\F)$.  Assume there exists $\overline{\fM}_{\rhobar}\in Y^{\eta, \tau}(\F)$ such that $T^*_{dd}(\overline{\fM}_{\rhobar}) \cong \rhobar|_{G_{K_{\infty}}}$.  We define $\tld{w}(\rhobar,\tau)\in\Adm^{\vee}(\eta)$ to be the shape of $\overline{\fM}_{\rhobar}$. 

This is well-defined because there exists at most one Kisin module $\overline{\fM}_{\rhobar}$ with height in $[0,2]$ such that $T^*_{dd}( \overline{\fM}_{\rhobar}) \cong \rhobar|_{G_{K_{\infty}}}$ (\cite[Theorem 3.2]{LLLM} for a principal series type, and \S 6.2 and the discussion before Lemma 6.13 in \emph{loc.~cit.}~for the general case). 
\end{defn}

\begin{rmk}
More generally if $\#\cS>1$, for $\tld{w}=(\tld{w}_{\tld{v}})_{\tld{v}\in\cS}$ we define $\tld{w}^*$ as the collection $(\tld{w}^*_{\tld{v}})_{\tld{v}\in\cS}$, where each $\tld{w}_{\tld{v}}^*\in\tld{\un{W}}_{\tld{v}}$ is as in Definition \ref{affineadjoint}.
Similarly if $\tau_\cS$ is a collection of $3$-generic tame inertial type and $\rhobar_{\cS}$ is a collection of Galois representations, we let $\tld{w}(\rhobar_{\cS},\tau_\cS)\defeq (\tld{w}(\rhobar_{\tld{v}},\tau_{\tld{v}}))_{\tld{v}\in\cS}$.
\end{rmk}

We now introduce the notion of semisimple Kisin module of shape $\tld{w}$.  
 
\begin{defn} \label{Kissemisimple}  Let $\overline{\fM} \in Y^{[0,h], \tau}(\F')$, where $\F'/\F$ is a finite extension.  We say that $\overline{\fM}$ is \emph{semisimple} of shape $\tld{w} = (\widetilde{w}_j)$ if there exists an eigenbasis $\beta$ of $\overline{\fM}$ such that 
$$
A^{(j)} \in T(\F'[\![v]\!]) \widetilde{w}_j
$$ 
for $0 \leq j \leq f'-1$.
 \end{defn} 
 
\begin{prop}  If $\overline{\fM}$ is semisimple of shape $\tld{w}$, then $T^*_{dd}(\overline{\fM})$ is semisimple.  
\end{prop}
  \begin{proof}
If $\overline{\fM}$ is semisimple, we can choose a basis such that $A^{(j)} \in T(\F'[\![v]\!]) \widetilde{w}_j$ where $\F'/\F$ is a finite extension. 
Using this basis, we see by Proposition \ref{prop:phif} and \ref{prop:phif:1} that the matrix for the \'etale $\phz^f$-module $\eps_0(\ovl{\fM}[1/u']^{\Delta = 1})$ lies in $T(\F'(\!(v)\!)) s$ for some $s \in W$. 
If $s$ has order $d$, then $T^*_{dd}(\overline{\fM})$ restricted to the unramified extension of $K_{\infty}$ of degree $d$ is a direct sum of characters (as $p\nmid d$) and so $T^*_{dd}(\overline{\fM})$ is semisimple. 
  \end{proof}
  
\begin{prop} \label{simpleform} If $\overline{\fM}$ is semisimple of shape $\tld{w} = (\widetilde{w}_j)$, then there exists an eigenbasis $\beta$ such that 
\begin{equation*}
A^{(j)} = \widetilde{w}_j \text{ for } 0 \leq j \leq f-2,\qquad   A^{(f-1)} \in T(\F') \widetilde{w}_{f-1} .
\end{equation*}
\end{prop}
\begin{proof}
See \cite[Proposition 3.2.16]{LLL}.
\end{proof}

\begin{defn} 
For any $\tld{w}=(w_j t_{\lambda_j}) \in \widetilde{\un{W}}^\vee$, define $\cM(\tld{w})\in \Phi\text{-}\Mod^{\text{\'et}}_K(\F)$ to be the free \'etale $\phz$-module over $\cO_{\cE,K} \otimes \F$ of rank 3 such that the matrix of
\[
\phi_{\cM}^{(j)}:\cM(\tld{w})^{(j)} \rightarrow \cM(\tld{w})^{(j+1)}
\]
is given by $w_j v^{\lambda_{j}}$ (with respect to the standard basis).  

\end{defn}

\begin{prop} \label{prop:phiV} Let $(s, \mu)$ be a lowest alcove presentation of $\tau$.  If $\overline{\fM} \in Y^{[0,h], \tau}(\F')$ is semisimple of shape $\tld{w} \in \un{\widetilde{W}}^{\vee}$, then 
\[
T^*_{dd}(\overline{\fM})|_{I_K} \cong \bV^*_K(\cM(\tld{w} s^* t_{\mu^*+\eta^*}))|_{I_K} \cong \taubar(w, \nu + \eta)
\]  
where $\tld{w} s^* t_{\mu^*} = w^* t_{\nu^*}$. $($Note that since both $T^*_{dd}(\overline{\fM})$ and $\bV^*_K(\cM(\tld{w} s^* t_{\mu^*+\eta^*}))$ are tame, they canonically extend to $G_K$.$)$
\end{prop}
\begin{proof}
This is the special case when $n = 3$ of  \cite[Corollary 3.2.17]{LLL}.  For sake of completeness, we include an argument here. 

The first isomorphism follows from Proposition \ref{simpleform} combined with Proposition \ref{prop:phif}.  

Let $w, \nu$ satisfy $\tld{w} s^* t_{\mu^*} = w^* t_{\nu^*}$.  One easily checks that the $\phi^f$-action on $\eps_0(\cM(w^* t_{\nu^* + \eta^*}))\in \Phi^f\text{-}\Mod^{\text{\'et}}_{W(k)}(\F)$ is given by $\prod_{j=0}^{f-1} w^{-1}_j v^{p^j(\nu_j + \eta_j)}$ using Proposition \ref{prop:phif:1}.

Let $s\in\un{W}$.
The $\phi^f$-action of $\eps_0(\cM((st_{\nu + \eta} w \pi(s)^{-1})^*))\in \Phi^f\text{-}\Mod^{\text{\'et}}_{W(k)}(\F)$ is given by 
\[\prod_{j=0}^{f-1} s_{j-1} w^{-1}_j s_{j}^{-1} v^{p^js_{j}(\nu_j + \eta_j)} = s_{f-1} \left(\prod_{j=0}^{f-1} w^{-1}_j v^{p^j(\nu_j + \eta_j)}\right) s_{f-1}^{-1}.\]
We conclude that $\eps_0(\cM(w^* t_{\nu^* + \eta^*})) \cong \eps_0(\cM((st_{\nu + \eta} w \pi(s)^{-1})^*))$. Therefore, $\bV^*_K(\cM(w^* t_{\nu^* + \eta^*})) \cong \bV^*_K(\cM((st_{\nu + \eta} w \pi(s)^{-1})^*)).$

As $(w,\nu + \eta)\sim (sw\pi(s)^{-1},s(\nu + \eta))$,  by \cite[Proposition 2.2.4]{LLL} and \cite[Lemma 4.2]{herzig-duke}, we reduce to the case where $w_i = 1$ for $i\neq f-1$.  The second isomorphism then follows from Lemma \ref{phiV}.

\end{proof}

\begin{prop} \label{unramtwist} Let $\rhobar, \rhobar'$ be semisimple representations of $G_{K}$.  Assume there exists $\overline{\fM}$ semisimple of shape $\tld{w}$ such that $T^*_{dd} (\overline{\fM}) \cong \rhobar|_{G_{K_{\infty}}}$.   Then $\rhobar'|_{I_K} \cong \rhobar|_{I_K}$ if and only if there exists a semisimple $\overline{\fM}'$ of shape $\tld{w}$ such that $T^*_{dd}(\overline{\fM}') \cong \rhobar|_{G_{K_{\infty}}}$.     
\end{prop}  
\begin{proof}
This follows from counting unramified twists using Proposition \ref{simpleform}. 
\end{proof}

The remainder of the subsection is devoted to results used for weight elimination. 
\begin{lemma} \label{reducetoss} If $\rhobar:G_{K}\to \GL_3(\F)$ admits a potentially crystalline lift of type $(\lambda,\tau)$, then so does $\rhobar^{\mathrm{ss}}$ $($after possibly replacing $E$ with a ramified extension$)$. 
\end{lemma}
\begin{proof} Fix a characteristic $0$ lift $\rho$ of $\rhobar$ which is potentially crystalline of type $(\lambda,\tau)$. Then by enlarging the coefficient ring of $\rho$, we can always find a lattice whose reduction is semisimple by \cite[Lemma 5(2)]{Enns}.
\end{proof}

The following two Theorems are key inputs in weight elimination by giving an upper bound on the semisimple representations which are reductions of potentially crystalline representations of type $(\lambda, \tau)$.
\begin{thm} \label{thm:potcris}  Let $\rhobar:G_K \ra \GL_3(\F)$.  
Assume that $\rhobar$ has a potentially crystalline lift of type $(\lambda, \tau)$ where $\lambda\in X^*(\un{T})$ is effective.
Assume that either $(1)$ $\tau$ is a regular principal series type or $(2)$ $\lambda=\eta$ and $\tau$ is $3$-generic.
Then, for a sufficiently large $h$, there is a Kisin module $\overline{\fM} \in  Y^{[0,h], \tau}(\F)$ of shape  $\tld{w} = (\widetilde{w}_0, \widetilde{w}_1, \ldots, \widetilde{w}_{f-1}) \in \Adm^\vee(\lambda)$ such that $T^*_{dd}(\overline{\fM}) \cong \rhobar|_{G_{K_{\infty}}}$. In particular, $\overline{\fM} \in  Y^{\lambda, \tau}(\F) $.

\end{thm}
\begin{proof}
See \cite[Theorem 3.2.20]{LLL}.
\end{proof}

\begin{thm} \label{fixedpoints} Let $\rhobar$ be a semisimple representation of $G_{K}$ and assume $\tau$ is a $1$-generic tame type. Assume that either $(1)$ $\rhobar$ is a direct sum of characters or $(2)$ $\lambda = \eta$ and $\tau$ is $3$-generic. 

If there exists a Kisin module $\overline{\fM} \in  Y^{\lambda, \tau}(\F)$ such that $T^*_{dd}(\overline{\fM}) \cong \rhobar|_{G_{K_{\infty}}}$, then there exists a finite extension $\F'/\F$ and a semisimple Kisin module $\overline{\fM}' \in Y^{\lambda, \tau}(\F')$ such that $($after extending scalars$) 
$ $T^*_{dd}(\overline{\fM}') \cong \rhobar|_{G_{K_{\infty}}}$.  Furthermore, we can take $\F' = \F$ in case $(2)$.
\end{thm}

\begin{rmk} \label{rmk:Nonzero def ring}Theorems \ref{thm:potcris} and \ref{fixedpoints} together say that for a fixed 3-generic type $\tau$ the set of semisimple $\rhobar|_{I_K}$ which arise as reductions of potentially crystalline representations of type $(\eta, \tau)$ are in bijection with a subset of $\mathrm{Adm}^\vee((2,1,0))^f$.  In fact, it turns out this admissibility condition exactly captures those semisimple $\rhobar$ which are reductions of crystalline representations of type $(\eta, \tau)$. Checking this is equivalent to checking that the potentially crystalline deformation ring of type $(\eta,\tau)$ of the semisimple $\rhobar$ corresponding to $\tld{w}\in \mathrm{Adm}^\vee((2,1,0))^f$ is non-zero. When $\tld{w}$ is such that $\ell(\tld{w}_i) >1$ for all $i$, this non-triviality follows from \cite[Table 7]{LLLM}. When $K=\Qp$, \cite[\S 8]{LLLM} shows the non-triviality for all admissible $\tld{w}$. For general unramified $K$ and admissible $\tld{w}$, the non-triviality of the deformation ring will follow from Theorem \ref{thm:BM} below.    
\end{rmk}

\begin{proof}[Proof of Theorem \ref{fixedpoints}]
We first treat the case where $\rhobar$ is a direct sum of characters. %
Let $\cM_{dd} = \overline{\fM}[1/u']$ be an \'etale $\phz$-module with descent datum for $L'/K$ (i.e., a semilinear action of $\Delta$). 
Since there is an equivalence of categories between \'etale $\phz$-modules over $L'$ with descent datum to $K$ and $G_{K_{\infty}}$-representations (cf. \cite[pg. 24]{LLLM} for principal series case and \S 6.1 in general), if $\rhobar$ is a direct sum of characters then $\cM_{dd} = \cM_1 \oplus \cM_2 \oplus \cM_3$ where each $\cM_i$ is stable under $\phi_{\cM_{dd}}$ and the descent datum.  

Let $Y^{\lambda, \tau}_{\cM_{dd}}$ be the Kisin variety parametrizing lattices in $\cM_{dd}$ which lie in $Y^{\lambda, \tau}$. \cite[Definition 3.1]{LLLM} defines $Y^{\lambda, \tau}_{\cM_{dd}}$ in the principal series case.  In general, $Y^{\lambda, \tau}_{\cM_{dd}}$ is the closed subscheme of fixed points of $Y^{\lambda, \tau'}_{\cM_{dd}}$ under the natural action of $\sigma^f$.
The torus $T = \Gm^3$ acts on $\cM_{dd}$ by scaling individually in each factor.  
As a consequence, we get an algebraic action of $T$ on the projective variety $Y^{\lambda, \tau}_{\cM_{dd}}$.

Any such action has a fixed point over some finite extension $\F'/\F$.  
Let $\overline{\fM}' \subset Y^{\lambda, \tau}_{\cM_{dd}}(\F')$ be a $T$-fixed point.   
Let $\chi_i:T \ra \Gm$ denote the projection onto each coordinate and set $\overline{\fM}'_i \defeq {(\overline{\fM}')}^{\chi_i}$.   Then 
\begin{equation} \label{a1}
\overline{\fM}' = \overline{\fM}'_1 \oplus \overline{\fM}'_2 \oplus \overline{\fM}'_3.
\end{equation}
Since the $T$-action commutes with $\phi_{\cM_{dd}}$ and $\Delta$, each $\overline{\fM}'_i$ is stable under both, hence $\overline{\fM}'_i$ is a rank one Kisin module with descent datum.  Any choice of eigenbasis which respects this decomposition shows that $\overline{\fM}'$ is semisimple. Because $\overline{\fM}'$ is in $Y^{\lambda,\tau}(\F')$, it is semisimple with an admissible shape $\tld{w} \in \Adm^\vee(\lambda)$.

Now suppose that $\tau$ is 3-generic, but $\rhobar$ is not necessarily a direct sum of characters. 
In this case, the Kisin module $\overline{\fM}$ of type $(\eta,\tau)$ is unique by \cite[Theorem 3.2]{LLLM}. We make a base change to the unramified extension $\breve{K}/K$ of degree $6$. 
Let $\breve{\tau}$ be the base change of $\tau$ to $\breve{K}$, and let $\breve{\ovl{\fM}}=\BC(\ovl{\fM})$ be the base change of the Kisin module $\ovl{\fM}$.
Since $\breve{\ovl{\fM}}$ is the unique Kisin module of type $(\eta,\breve{\tau})$, by the above argument, it must be semisimple. 

Recall the notion of gauge basis \cite[Definition 2.22]{LLLM}. 
Fix a gauge basis $\beta$ of $\ovl{\fM}$ and let $\breve{\beta}$ be the induced gauge basis on $\breve{\ovl{\fM}}$.
The eigenbasis $\breve{\beta}_2$ for $\breve{\ovl{\fM}}$ which puts the partial Frobenii in the form as in Proposition \ref{simpleform} is also a gauge basis.   
By \cite[Theorem 4.16]{LLLM}, $\breve{\beta}$ and $\breve{\beta}_2$ differ by embedding-wise scaling by torus elements.   
We conclude that the matrices for the partial Frobenii with respect to $\breve{\beta}$ and hence $\beta$ are monomial. 
Since the only monomial matrices $\mathcal{I}(\F)\tld{w}\mathcal{I}(\F)$ are $T(\F[[v]])\tld{w}$, we see that $\ovl{\fM}$ is semisimple.

\end{proof}

\subsection{Shapes and Serre weights}
\label{sec:SandSW}

We continue to assume $\cS=\{\tld{v}\}$.
We compute $V(\rhobar|_{I_K})$ in terms of shape for a semisimple $\rhobar$.  
This will effectively allow us to determine $W^?(\rhobar)\cap \JH(\ovl{\sigma(\tau)})$ for a $3$-generic tame type $\tau$ (Proposition \ref{typematching}) via the combinatorics of \S 2 (especially Corollary \ref{cor:intersections}). 

\begin{prop} \label{typereflection} Let $(s, \mu)$ be a lowest alcove presentation of $\tau$. Let $\overline{\fM} \in Y^{[0,h], \tau}(\F')$ be semisimple of shape $\widetilde{w} = (\widetilde{w}_j) \in\tld{\un{W}}^\vee$. 
Then:
\[
V(T_{dd}^*(\overline{\fM})|_{I_K}) = R_{s \widetilde{w}^*}(\mu + \eta).
\]
\end{prop}
\begin{proof}
By Proposition \ref{prop:phiV}, $T_{dd}^*(\overline{\fM})|_{I_K} \cong \taubar(w, \nu + \eta)$ where $\tld{w} s^* t_{\mu^*} = w^* t_{\nu^*}$.   By \cite[Proposition 9.2.3]{GHS}, 
\[
V(T_{dd}^*(\overline{\fM})|_{I_K})  = R_w(\nu + \eta)
\]
Finally, an easy calculation shows that $R_{s \tld{w}^*}(\mu + \eta) = R_w(\nu + \eta)$ using Definition \ref{typeaction}.  
\end{proof}

\begin{prop} \label{typematching}  
Let $\tau$ be a $3$-generic type with lowest alcove presentation $(s, \mu)$
and let $\overline{\fM} \in Y^{[0,h], \tau}(\F')$ be semisimple of shape $\widetilde{w} = (\widetilde{w}_j) \in\tld{\un{W}}^\vee$.
Let $\rhobar:G_{K} \ra \GL_3(\F)$ be semisimple and $3$-generic and assume that $\rhobar|_{G_{K_\infty}}\cong T^*_{dd}(\ovl{\fM})$. Then
$$
W^?(\rhobar, \tau) \defeq 
W^?(\rhobar)\cap \JH(\ovl{\sigma(\tau)})=
 \{\sigma_{(\omega,a)}^{(s, \mu + \eta)}: (\omega,a) \in \Sigma_{\widetilde{w}^*}\} \subset \JH(\overline{R_{s} (\mu + \eta)})
$$   
\end{prop}
\begin{proof}
This follows from Corollary \ref{cor:intersections} and Proposition \ref{typereflection}. %
\end{proof}

\begin{rmk} \label{rmk:expadj} There is an explicit list of elements of $\Adm^\vee(2,1,0)$ given in \cite[Table 1]{LLLM}.  The effect of the involution $\tld{w}\mapsto \tld{w}^*$ is, in addition to reversing the order of components, to reverse the order of the word and to turn $\gamma = (13) v^{(1,0, -1)}$ into $\gamma^+$.  In particular, $(\cdot)^*$ defines a bijection between $\Adm^\vee(2,1,0)$ and $\Adm(2,1,0)$.   
\end{rmk}

\subsubsection{Type elimination results}

We assume throughout that $\rhobar:G_K \ra \GL_3(\F)$ is 6-generic.%

\begin{prop} \label{telim} Let $\tau$ be a $1$-generic tame inertial type.  
If $\rhobar$ is $6$-generic and arises as the reduction of a potentially crystalline representation of type $(\eta, \tau)$  then 
$$
 W^?(\rhobar^{\mathrm{ss}}, \tau) \neq \emptyset.
$$
\end{prop}
\begin{proof}
First, assume that $\tau$ is 3-generic.  By Theorem \ref{thm:potcris} combined with  Lemma \ref{reducetoss}, there exists $\overline{\fM} \in  Y^{\eta, \tau}(\F)$ such that $T^*_{dd}(\overline{\fM}) \cong \rhobar^{\mathrm{ss}}|_{G_{K_{\infty}}}$.   
In fact, by Theorem \ref{fixedpoints}, we can take $\overline{\fM}$ to be semisimple of shape $\tld{w} = (\widetilde{w}_j) \in \Adm^\vee(\eta)$.  
By Proposition \ref{typematching} and Table \ref{Intersections}, we conclude that 
$$
W^?(\rhobar^{\mathrm{ss}}, \tau) \neq \emptyset.
$$
If $\tau$ is not 3-generic, then Proposition \ref{telim2} below shows that $\rhobar$ does not arise as the reduction of potentially crystalline representation of type $(\eta, \tau)$ for any such $\tau$. 
\end{proof}

\begin{prop} \label{telim2}  
Let $n\geq 4$ and assume that $\rhobar:G_K \ra \GL_3(\F)$ is $n$-generic.
Assume that $\rhobar$ arises as the reduction of a potentially crystalline representation of type $(\eta, \tau)$ where $\tau$ is a $1$-generic tame inertial type.
Then $\tau$ is $(n-3)$-generic.
\end{prop}
\begin{proof} 
By Lemma \ref{reducetoss}. we can assume that $\rhobar$ is semisimple.
The result follows now from \cite[Proposition 3.3.2]{LLL}. 
\end{proof}

\subsection{Serre weight conjectures}
\label{sec:SWC}

We are now ready to prove an abstract version of the Serre weight conjecture as well as a numerical Breuil--M\'ezard statement.
In this section, we allow $\cS$ to have arbitrary (finite) cardinality.
\subsubsection{Setup and summary of results}
\label{sub:sub:setup.}

Recall from \S \ref{subsection:Notation} that $\cS$ is a finite set and $F_{\tld{v}}$ is a finite extension of $\QQ_p$ for each $\tld{v}\in \cS$, and $k_{\tld{v}}$ is the residue field of $F_{\tld{v}}$.
For Definition \ref{minimalpatching} below, we will not require that $F_{\tld{v}}$ an unramified extension of $\Q_p$.
In applications, $F$ will be a fixed global field, $\cS$ will be a set of places in $F$, and $F_{\tld{v}}$ will be the completion at $\tld{v}$ of $F$.
Let $\rhobar_\cS$ be a collection $(\rhobar_{\tld{v}})_{\tld{v}\in \cS}$ where $\rhobar_{\tld{v}}:G_{F_{\tld{v}}} \ra \GL_3(\F)$ is a continuous Galois representation.
Let $R_{\widetilde{v}}^\square=R_{\rhobar_{\widetilde{v}}}^\square$ denote the unrestricted universal framed deformation ring of $\rhobar_{\tld{v}}$.
Fix a natural number $h$ and let 
\[R_\infty = \Big( \widehat{\underset{\tld{v}\in \cS}{\bigotimes}} R_{\widetilde{v}}^\square \Big)[\![x_1,x_2,\ldots, x_h]\!] \textrm{ and } X_\infty = \Spf R_\infty.\]

If $\tau_{\widetilde{v}}$ is an inertial type for $G_{F_{\widetilde{v}}}$, then let $R_{\widetilde{v}}^{\tau_{\widetilde{v}}}=R_{\rhobar_{\widetilde{v}}}^{\tau_{\widetilde{v}}}$ be the universal framed deformation ring of $\rhobar_{\tld{v}}$ of inertial type $\tau_{\widetilde{v}}$ and (parallel) Hodge--Tate weights $(2,1,0)$.
If $\tau_\cS = (\tau_{\widetilde{v}})_{\tld{v}\in \cS}$, then let
\[R_\infty(\tau_\cS) = \widehat{\underset{\tld{v}\in \cS}{\bigotimes}} R_{\tld{v}}^{\tau_{\tld{v}}}\otimes_{\widehat{\underset{\tld{v}\in \cS}{\bigotimes}} R_{\widetilde{v}}^\square} R_\infty \textrm{ and } X_\infty(\tau_\cS) = \Spf R_\infty\,(\tau_\cS).\]
Let $d+1$ be the dimension of $X_\infty(\tau_\cS)$ (the dimension is independent of $\tau_\cS$ by \cite[Theorem 3.3.4]{KisinPSS}).
We denote by $\overline{R}_{\widetilde{v}}^\square$, $\overline{R}_\infty$, etc. the reduction of these objects modulo $\varpi$.

Let $K_{\tld{v}}$ be the group $\GL_3(\cO_{F_{\tld{v}}})$ and $K$ be the product $\prod_{\tld{v}\in \cS} K_{\tld{v}}$.
Results towards the inertial local Langlands correspondence (cf.~ Proposition \ref{prop:basic:Ktypes}(\ref{item2:basic:Ktypes})) associate a $\ovl{\Q}_p$-valued $K_{\tld{v}}$-representation $\sigma(\tau_{\tld{v}})$ to a $1$-generic tame inertial type $\tau_{\tld{v}}$ (and $\sigma(\tau_{\tld{v}})$ can be realized over $\cO$).

\begin{defn}\label{minimalpatching}
Let $\Rep_{K}(\cO)$ denote the category of continuous $K$-representations over finitely generated $\cO$-modules and $\Mod(X_\infty)$ the category of coherent sheaves over $X_\infty$.

A \emph{weak minimal patching functor for $\rhobar_{\cS} = (\rhobar_{\tld{v}})_{\tld{v}\in\cS}$} is defined to be a nonzero covariant exact functor 
$M_{\infty}:\Rep_{K}(\cO)\ra \Coh(X_{\infty})$ satisfying the following axioms:
\begin{enumerate}
	\item For each $\tld{v}\in\cS$, let $\tau_{\tld{v}}$ be an inertial type for $G_{F_{\tld{v}}}$. Let $\sigma(\tau_\cS)$ be the $K$-representation $\underset{\tld{v}\in\cS}{\bigotimes}\sigma(\tau_{\tld{v}})$. 
If $\sigma(\tau_\cS)^{\circ}$ is an $\cO$-lattice in $\sigma(\tau_\cS)$, then $M_{\infty}(\sigma(\tau_\cS)^{\circ})$ is $p$-torsion free and is maximally Cohen--Macaulay over $R_\infty(\tau_\cS)$; \label{support}
	\item If $M_\infty(R_1(\mu))$ is nonzero, then $\rhobar_{\tld{v}}$ has a semistable lift of type $\tau_{\tld{v}}(1,\mu_{\tld{v}})$ for all $\tld{v}\in \cS$. \label{item:semistable}
	\item if $\sigma$ is an irreducible $\prod_{\tld{v}\in\cS}\GL_3(k_{\tld{v}})$-representation over $\F$, the module $M_{\infty}(\sigma)$ is either $0$ or maximal Cohen--Macaulay over its support, which is equidimensional of dimension $d$; and \label{dimd}
	\item the locally free sheaf $M_\infty(\sigma(\tau_\cS)^\circ)[1/p]$ (being maximal Cohen--Macaulay over the regular ring $R_\infty(\tau_\cS)[\frac{1}{p}]$) has rank at most one on each connected component. %
\end{enumerate}
\end{defn}

Assume that a weak minimal patching functor $M_\infty$ for $\rhobar_\cS$ exists. 
Following \cite[Definition 3.2.6]{GHS}, let $W^{\mathrm{BM}}(\rhobar_{\cS})$ be the set of irreducible $\rG$-representations (recall that $\rG \defeq \underset{\tld{v}\in \cS}{\prod} \GL_3(k_{\tld{v}})$) $\sigma$ over $\F$ such that $M_\infty(\sigma)$ is nonzero (note that \emph{a priori} this set depends on the choice of $M_{\infty}$). 
For a finitely generated $R_\infty$-module $M$ with scheme-theoretic support $\Spec A$ of dimension at most $d$, define $e(M)$ to be $d!$ times the coefficient of the degree $d$-term of the Hilbert polynomial of $M$ (considered as an $A$-module).
In particular $e(M)$ is the Hilbert--Samuel multiplicity of $M$ as an $A$ module when $\dim A=d$, and is $0$ otherwise.

We now assume that for all $\tld{v}\in \cS$, $F_{\tld{v}}$ is an unramified extension of $\Q_p$.
Recall Definition \ref{defn:serrewts} of $W^?(\rhobar_\cS)$.

\begin{thm} \label{SWC}
Suppose that $\rhobar_\cS$ is semisimple and $10$-generic, and that $M_\infty$ is a weak patching functor for $\rhobar_{\cS}$.
Then for all Serre weights $\sigma$, $e(M_\infty(\sigma)) = 1$ if $\sigma \in W^?(\rhobar_{\cS})$ and $e(M_\infty(\sigma)) = 0$ otherwise.
In particular, $W^{\mathrm{BM}}(\rhobar_{\cS}) = W^?(\rhobar_{\cS})$.
\end{thm}
\begin{thm} \label{thm:BM}\label{thm:conngenfib} 
Let $K/\Q_p$ be a finite unramified extension of degree $f$.
Let $\rhobar:G_K \rightarrow \GL_3(\F)$ be a continuous, $10$-generic, and semisimple Galois representation, and let $\tau$ be a tame inertial type. 
If $\tau$ is not $1$-generic, $R_{\rhobar}^{\tau}$ is $0$.
If $\tau$ is $1$-generic, then the number of irreducible components of $\overline{R}_{\rhobar}^{\tau}\defeq R_{\rhobar}^{\tau}/\varpi$ is equal to the number of elements in $W^?(\rhobar, \tau)$.
The ring $R_{\rhobar}^{\tau}$ is a normal domain and is Cohen--Macaulay.
Moreover $\ovl{R}_{\rhobar}^{\tau}$ is reduced and its components are formally smooth of the same dimension.
\end{thm}

The proofs of Theorems \ref{SWC} and \ref{thm:BM} appear in \S \ref{subsubsec:proofs}.

\subsubsection{Types, weights, and the Breuil--M\'ezard philosophy}
In this subsection, we describe the basic inductive argument towards the proofs of Theorems \ref{SWC} and \ref{thm:BM}.

Recall that we have a length function $\ell: \tld{W}^\vee\rightarrow \Z$. 

\begin{lemma} \label{lemma:BM}
Let $K/\Q_p$ be a finite unramified extension of degree $f$.
Let $\rhobar:G_K \rightarrow \GL_3(\F)$ be a semisimple Galois representation.	
Let $\tau$ be $5$-generic with lowest alcove presentation $(s, \mu)$.  Assume that $V(\rhobar|_{I_K}) = R_{s \widetilde{w}^*}(\mu + \eta)$ where $\tld{w} = \tld{w}(\rhobar, \tau) = (\widetilde{w}_0,\ldots, \widetilde{w}_{f-1})\in\Adm^{\vee}(\eta)$ with $\ell(\tld{w}_j)>1$ for every $j$. Then $\overline{R}_{\rhobar}^{\tau}\neq 0$ and the number of irreducible components of $\overline{R}_{\rhobar}^{\tau}$ is equal to the number of elements in $W^?(\rhobar, \tau)$.
The ring $R_{\rhobar}^{\tau}$ is a normal domain and is Cohen--Macaulay.
Moreover $\ovl{R}_{\rhobar}^{\tau}$ is reduced and its components are formally smooth of the same dimension.
\begin{proof}
By Proposition \ref{typereflection}, our hypotheses implies that there exists a Kisin module $\overline{\fM}_{\rhobar}\in Y^{\eta,\tau}(\F)$ such that $T^*_{dd}(\overline{\fM}_{\rhobar}) \cong \rhobar|_{G_{K_{\infty}}}$, and that $\tld{w}=\tld{w}(\rhobar,\tau)$.
In \cite[\S 5.3, \S 6]{LLLM}, we gave an explicit presentation for $R_{\rhobar}^{\tau}$ whenever the length of $\widetilde{w}_j$ is at least 2 and the type $\tau$ is $5$-generic.  
For each shape, there were at most two cases depending on $\overline{\fM}_{\rhobar}$.  By Proposition \ref{unramtwist} and Proposition \ref{fixedpoints}, whenever $\rhobar$ is semisimple, $\overline{\fM}_{\rhobar}$ is semisimple in the sense of Definition \ref{defn:shape} and so $\overline{\fM}_{\rhobar}$ lies in the special locus for the given shape.  

We claim that 
\[
\# \Irr\big(\Spec(\overline{R}_{\rhobar}^{\tau})\big) = \prod_j  2^{4 - \ell(\widetilde{w}_j)}
\]
and that $\overline{R}_{\rhobar}^{\tau}$ is reduced and Cohen--Macaulay. 
Indeed, from Table 7 in \emph{loc.~cit.}~,  we see that $\overline{R}_{\rhobar}^{\tau}$ is the completion of a tensor product over $\bF$ of rings of the form $ \bF[X]$, $\bF[X,Y]/XY$ at the maximal ideal generated by the variables $X,Y$. Since such a tensor product is Cohen-Macaulay, so is its completion by \cite[\href{https://stacks.math.columbia.edu/tag/07NX}{Tag 07NX}]{stacks-project}. Since such a tensor product is reduced and excellent (being a finite type $\bF$-algebra), its completion is reduced by \cite[\href{https://stacks.math.columbia.edu/tag/07NZ}{Tag 07NZ}]{stacks-project}.
Finally, we also note that since each irreducible components of such a tensor product is smooth, the completion of the irreducible components stay irreducible and are formally smooth.

This implies the remaining properties (see the proof of \cite[Corollary 8.9]{LLLM}. Note that the reducedness of $\overline{R}_{\rhobar}^{\tau}$ implies ${R}_{\rhobar}^{\tau}$ is a normal domain).

Finally, we compare this with the size of $W^?(\rhobar, \tau)$, which is $\prod_{j\in \cJ} \#\Sigma_{(\widetilde{w}^*_j)}$  by Proposition \ref{typematching} using Table \ref{Intersections}.%
\end{proof}
\end{lemma}

\begin{cor}\label{cor:globalBM}
Suppose that for each $\tld{v}\in \cS$, $\rhobar_{\tld{v}}:G_{F_{\tld{v}}} \rightarrow \GL_3(\F)$ is a semisimple Galois representation and $\tau_{\tld{v}}$ is a $5$-generic tame inertial type satisfying the hypotheses of Lemma \ref{lemma:BM}.
Let $\tau_\cS$ be $(\tau_{\tld{v}})_{\tld{v}\in \cS}$.
Then $e(\ovl{R}_\infty(\tau_\cS))$ is equal to $\# W^?(\rhobar_\cS, \tau_\cS)$.
\end{cor}
\begin{proof}
This follows immediately from Lemma \ref{lemma:BM} and \cite[Lemma 2.2.14]{EG} and properties of tensor products of representations.
\end{proof}

\begin{prop}[Weight elimination] \label{WE}  
Suppose that for each $\tld{v}$, $\rhobar_{\tld{v}}:G_{F_{\tld{v}}} \ra \GL_3(\F)$ is a $10$-generic Galois representation.   
Then
$$
W^{\mathrm{BM}}(\rhobar_{\cS}) \subset W^?(\rhobar^{\mathrm{ss}}_{\cS}).
$$
\end{prop} 
\begin{proof}
Suppose that $\sigma$ is a Serre weight such that $\sigma \in W^{\mathrm{BM}}(\rhobar_\cS) \setminus W^?(\rhobar_\cS^{\mathrm{ss}})$.
Note that $W^{\mathrm{BM}}(\rhobar_\cS)$ satisfies the (evident generalization of) \cite[Axiom (WE)]{Enns} by \ref{minimalpatching}(\ref{item:semistable}).
By \cite[Theorem 8]{Enns} and \cite[Remark 2.2.8]{LLL}, $\rhobar_\cS$ is $10$-generic implies that $\sigma$ is $3$-generic.
Then $\sigma\in\JH(\ovl{\sigma(\tau_\cS)})$ for some collection $\tau_\cS\defeq (\tau_{\tld{v}})_{\tld{v}\in\cS}$ of $1$-generic tame types (for example taking a tame principal series type containing $\sigma$). 
Since $\rhobar_\cS$ is $10$-generic, we see by Proposition \ref{telim2} and Definition \ref{minimalpatching}(\ref{support}) that $M_\infty(\ovl{\sigma(\tau_\cS)})\neq 0$ implies that $\tau_\cS$ is $7$-generic.
By Lemma \ref{lem:deep:type} below (whose proof only uses modular representation theory and is independent of the results in \S \ref{sec:weights}), $\sigma$ is $5$-deep.

By Corollary \ref{separateweights}, $\tau_\cS$ above can be taken so that $W^?(\rhobar^{\mathrm{ss}}_\cS,\tau_\cS) = \emptyset$ and $\sigma\in \JH(\ovl{\sigma(\tau_\cS)})$.
By Proposition \ref{telim}, $\rhobar^{\mathrm{ss}}_{\tld{v}}$ is not the reduction of a potentially crystalline representation of type $(\eta_{\tld{v}}, \tau_{\tld{v}})$ for some $\tld{v}\in \cS$.
By Proposition \ref{reducetoss}, $\rhobar_{\tld{v}}$ is also not the reduction of a potentially crystalline representation of type $(\eta_{\tld{v}}, \tau_{\tld{v}})$ for some $\tld{v}\in \cS$.
Then $X_\infty(\tau_\cS) = \emptyset$, and $M_\infty(\sigma(\tau_\cS)^\circ) = 0$ for any $\cO$-lattice $\sigma(\tau_\cS)^\circ$ in $\sigma(\tau_\cS)$.
By exactness of $M_\infty$, $M_\infty(\sigma) = 0$ which is a contradiction.
\end{proof}

\begin{rmk} 
\label{rmk:loc:alg:WE}
Instead of appealing to \cite{Enns}, one can show that provided that $\rhobar_\cS$ is generic enough, every element of $W^{\mathrm{BM}}(\rhobar_{\cS})$ appears as a Jordan--H\"older factor of the reduction of some $1$-generic tame type as follows.
Even though not all Serre weights come from the reduction $\sigma(\tau_\cS)$ for a collection $\tau_\cS$ of $1$-generic tame types, they do occur in the reduction of $V_\lambda \otimes \sigma(\tau_\cS)$ where $\lambda$ is sufficiently large (independent of $p$).
The results of \S \ref{Semisimple Kisin modules} hold also for potentially crystalline representations of type $(\lambda_{\tld{v}} + \eta_{\tld{v}},\tau_{\tld{v}})_{\tld{v}\in\cS}$ (possibly with a stronger genericity hypothesis), and then the same argument as in the proof of Proposition \ref{WE} gives the result.
\end{rmk}

Given a collection $\tau_\cS$ of tame inertial types and a $\cO$-lattice $\sigma(\tau_\cS)^\circ$ in $\sigma(\tau_\cS)$,
we write $\ovl{\sigma}(\tau_\cS)^\circ$ to denote the reduction of $\sigma(\tau_\cS)^\circ$ modulo $\varpi$.

\begin{lemma}\label{lemma:multone}
Suppose that for each $\tld{v}\in \cS$, $\rhobar_{\tld{v}}:G_{F_{\tld{v}}} \rightarrow \GL_3(\F)$ is a semisimple Galois representation and $\tau_{\tld{v}}$ is a $5$-generic tame inertial type satisfying the hypotheses of Lemma \ref{lemma:BM}.
Let $\tau_\cS = (\tau_{\tld{v}})_{\tld{v}\in \cS}$.
Let $\sigma(\tau_\cS)^{\circ}$ be an $\cO$-lattice in $\sigma(\tau_\cS)$.
Then either $M_\infty(\ovl{\sigma}(\tau_\cS)^\circ)=0$ or $e(M_\infty(\ovl{\sigma}(\tau_\cS)^\circ))$ is equal to $e(\overline{R}_\infty(\tau_\cS))$, and the alternative does not depend on the lattice.
\end{lemma}
\begin{proof}
The generic fiber of $X_\infty(\tau_\cS)$ is irreducible since its special fiber is reduced (see Lemma \ref{lemma:BM}).
Then since $M_\infty(\sigma(\tau_\cS)^\circ)$ is maximal Cohen--Macaulay over $X_\infty(\tau_\cS)$, either $M_\infty(\sigma(\tau_\cS)^\circ)$ has full support on $X_\infty(\tau_\cS)$ or $M_\infty(\sigma(\tau_\cS)^\circ)$ is $0$.
The lemma now follows from the proof of \cite[Proposition 7.14]{LLLM}.
\end{proof}

\begin{lemma} \label{lemma:weight-combinatorics}
Suppose that $\rhobar_{\tld{v}}:G_{F_{\tld{v}}} \rightarrow \GL_3(\F)$ is semisimple and $7$-generic for all $\tld{v}\in \cS$.
If $\sigma \in W^?(\rhobar_{\cS})$ has defect $\delta$ %
then there are tame inertial types $\tau_{\tld{v}}$ and $\tau'_{\tld{v}}$ satisfying the hypotheses of Corollary \ref{cor:globalBM} such that if $\tau_\cS=(\tau_{\tld{v}})_{\tld{v}\in \cS}$ and $\tau_\cS'=(\tau'_{\tld{v}})_{\tld{v}\in \cS}$, we have
\begin{enumerate}
\item $\sigma \in W^?(\rhobar_{\cS},\tau_\cS) \subset W^?(\rhobar_{\cS},\tau_\cS')$; \label{type-inclusion}
\item $\#W^?(\rhobar_{\cS},\tau_\cS) = 2^\delta$ and $\#W^?(\rhobar_{\cS},\tau_\cS') = 2^{\delta+1}$; \label{minimal-type}
\item all Serre weights in $W^?(\rhobar_{\cS},\tau_\cS)$ and $W^?(\rhobar_{\cS},\tau_\cS')$ have defect at most $\delta$ and if $\delta>0$, $\sigma$ is the unique Serre weight in $W^?(\rhobar_{\cS},\tau_\cS')$ with defect $\delta$; \label{lower-defect}
\item if $\delta=0$, 
for any $\sigma'\in W^?(\rhobar_{\cS})$ of defect $0$ which is  adjacent to $\sigma$, one can choose the $\tau_\cS'$ above so that $W^?(\rhobar_{\cS},\tau_\cS')=\{\sigma,\ \sigma'\}$. 
\label{obvious-neighbor}
\end{enumerate}
\end{lemma}%
\begin{proof}
This is essentially a consequence of results in \S \ref{subsec:comb}. 
Suppose that $V(\rhobar_\cS|_{I_{F_\cS}}) = R_s(\mu)$ with $\mu-\eta$ $7$-deep in alcove $\un{C}_0$, and label $W^?(\rhobar_{\cS})$ by $r(\Sigma)$ as in Proposition \ref{prop:serrewts}.  
Let $\tld{w} = (\tld{w}_{\tld{v}})_{\tld{v}\in \cS} = (\tld{w}_i)_{i\in \cJ}\in \Adm(\eta)$ and $\tau_\cS$ be the tame inertial type with lowest alcove presentation $(sw^{-1},\mu + s\tld{w}^{-1}(0)-\eta)$, where $w \in \un{W}$ is the image of $\tld{w}$.  
Then $\mu + s\tld{w}^{-1}(0)-\eta$ is $5$-deep in alcove $\un{C}_0$ so that $\tau_\cS$ is $5$-generic.
Corollary \ref{cor:intersections} says that %
\[
W^?(\rhobar_\cS, \tau_\cS) = \{\sigma_{r(\omega,a)}^{(s,\mu)}:r(\omega,a)\in \prod_{i\in \cJ} r(\Sigma_{\widetilde{w}_i^{-1}})\}
\]
where $\Sigma_{\widetilde{w}_i^{-1}}=\Sigma_0\cap\tld{w_i}^{-1}\big(r(\Sigma_0)\big)$. 
If $\ell(\tld{w}_i)>1$ for all $i$, then $\tau_\cS$ satisfies the hypotheses of Corollary \ref{cor:globalBM} (noting that $\ell(\tld{w}) = \ell(\tld{w}^*)$ by the proof of \cite[Lemma 2.1.3]{LLL}, where the lengths are as elements of $\un{\tld{W}}^\vee$ and $\un{\tld{W}}$, respectively). 

Let $(\omega,a)=((\omega_i,a_i))_i$ be such that $\sigma=\sigma^{(s,\mu)}_{r(\omega,a)}$.
We will construct the required types by appropriately choosing $(\tld{w}_i)_i$.
If $(\omega_i, a_i)\in \Sigma_0^{\mathrm{obv}}$, we can find an element $\tld{w}_i \in \Adm((2,1,0))$ such that $\ell(\tld{w}_i) = 4$ and $\Sigma_{\widetilde{w}_i^{-1}} = \{(\omega_i, a_i)\}$.
If $(\omega_i, a_i)\in \Sigma_0^{\mathrm{inn}}$, we can find an element $\widetilde{w}_it_{-\un{1}} \in \{ \alpha \beta \alpha, \beta \gamma \beta, \alpha \gamma \alpha \}$ such that $\Sigma_{\widetilde{w}_i^{-1}}$ contains exactly $\{(\omega_i, a_i), r(\omega_i, a_i) \}$. This choice of $ (\tld{w}_i)_{i\in \cJ}$ gives a type $\tau_\cS=(\tau_{\tld{v}})_{\tld{v}\in\cS}$
such that $\sigma \in W^?(\rhobar_{\cS},\tau_\cS)$, $\#W^?(\rhobar_{\cS},\tau_\cS) = 2^\delta$, and all weights in $W^?(\rhobar_{\cS},\tau_\cS)$ have defect at most $\delta$.  

To construct $\tau_\cS'$ satisfying (\ref{minimal-type}) and (\ref{lower-defect}), we proceed similarly. 
We first deal with the case $\delta>0$.
In this case, we can find a $j_0\in \cJ$ such that $\ell(\widetilde{w}_{j_0}) = 3$, and consider any $\widetilde{w}_{j_0}' \leq \widetilde{w}_{j_0}$ of length $2$.   
Then
\[
\Sigma_{\widetilde{w}_{j_0}^{-1}} \subset    \Sigma_{\widetilde{w}_{j_0}'^{-1}}
\]
and has size 4 (see Table \ref{Intersections}).  Furthermore, $\Sigma_{\widetilde{w}_{j_0}'^{-1}} \backslash  \Sigma_{\widetilde{w}_{j_0}^{-1}} \subset \Sigma^{\mathrm{obv}}_0$. 
Let $\tld{w}'$ be such that $\tld{w}'_{i}=\tld{w}_{i}$ if $i\neq j_0$ and $\tld{w}'_{j_0}$ is the element chosen above.
Then the type $\tau_\cS'$ such that $\sigma(\tau_\cS')=R_{s(\tld{w}')^{-1}}(\mu)$ satisfies items (\ref{minimal-type}) and (\ref{lower-defect}).

Finally, assume that $\delta=0$. Let $\sigma'\in W^?(\rhobar_{\cS})$ be a defect $0$ weight adjacent to $\sigma$, and write $\sigma'=\sigma_{r(\omega',a')}^{(s,\mu)}$. Then there is a unique $j_0\in \cJ$ such that $(\omega'_{j_0},a'_{j_0})\neq (\omega_{j_0},a_{j_0})$. 
There are six possible pairs $\{ (\omega_j, a_j), (\omega_j', a'_j) \} \in \Sigma^{\mathrm{obv}}_0$ which are adjacent in the Table \ref{TableExtGraph}, each of which is $\Sigma_{\widetilde{w}_j^{-1}}$ for some length three non-shadow element $\widetilde{w}'_j$ (see \cite[Table 1]{LLLM}). We let $\tau_\cS'$ be the inertial type corresponding to $\tld{w}'$ such that $\tld{w}'_i=\tld{w}_i$ for $i\neq j_0$ and $\tld{w}'_{j_0}$ chosen as in the previous sentence. This gives the type satisfying items (\ref{minimal-type}) and (\ref{lower-defect}) and (\ref{obvious-neighbor}). 
\end{proof}
\begin{rmk}\label{rmk:minimaltype}%
From Table \ref{Intersections}, we see that the type $\tau_\cS$ constructed in Lemma \ref{lemma:weight-combinatorics} is uniquely characterized by requiring that $\sigma \in W^?(\rhobar_{\cS},\tau_\cS)$ and $\#W^?(\rhobar_{\cS},\tau_\cS) = 2^\delta$.
We call it the \emph{minimal type} of $\sigma$ with respect to $\rhobar$.
\end{rmk}
In what follows, our $\rhobar_\cS=(\rhobar_{\tld{v}})_{\tld{v}\in \cS}$ will be assumed to be 10-generic, so that Proposition \ref{WE} applies. 
We observe that if $\tau_\cS$ is a $3$-generic tame type, then $\ovl{\sigma}(\tau_\cS)^\circ$ is multiplicity free for any choice of lattice $\sigma(\tau_\cS)^\circ$. 
Then 
\[
e(M_\infty(\ovl{\sigma}(\tau_\cS)^\circ)) = \sum_{\sigma \in \JH(\ovl{\sigma(\tau_\cS)})} e(M_\infty(\sigma))
\]
by Definition \ref{minimalpatching}, and in fact
\[
e(M_\infty(\ovl{\sigma}(\tau_\cS)^\circ)) = \sum_{\sigma \in W^?(\rhobar_\cS^{\mathrm{ss}},\tau_\cS)} e(M_\infty(\sigma)).
\]
Finally, observe that if $\rhobar_\cS$ is $10$-generic and $W^?(\rhobar_\cS,\tau_\cS)$ is nonempty, then $\tau_\cS$ is $7$-generic by Proposition \ref{telim2}.%

\begin{lemma} \label{lemma:defect-0-case}
Suppose that $\rhobar_{\tld{v}}:G_{F_{\tld{v}}} \rightarrow \GL_3(\F)$ is semisimple and $10$-generic for all $\tld{v}\in \cS$.
If there exists $\sigma \in W^{\mathrm{BM}}(\rhobar_{\cS})$ with defect $0$, then for all $\sigma' \in W^?(\rhobar_{\cS})$, $e(M_\infty(\sigma')) = 1$. In particular, $W^{\mathrm{BM}}(\rhobar_{\cS}) = W^?(\rhobar_{\cS})$.
\begin{proof}

We first prove the lemma assuming $\sigma'$ has defect $0$ by induction on %
$d\defeq \dgr{\sigma}{\sigma'}$.
By Lemma \ref{lemma:weight-combinatorics}, one can choose a $1$-generic tame type $\tau_\cS$ such that $W^?(\rhobar_{\cS},\tau_\cS) = \{\sigma\}$. Note that $\tau_\cS$ is then $7$-generic under our assumptions.
Then $e(M_\infty(\sigma)) = 1$ by Corollary \ref{cor:globalBM} and Lemma \ref{lemma:multone}.
This establishes the case $d=0$.

Suppose that $\sigma' \in W^?(\rhobar_{\cS})$ has defect $0$ and that $d>0$.
Then there is a $\sigma'' \in W^?(\rhobar_{\cS})$ adjacent to $\sigma'$, with defect $0$, and such that $\dgr{\sigma}{\sigma''}=d-1$.
We choose a type $\tau_\cS'$ as in Lemma \ref{lemma:weight-combinatorics}(\ref{obvious-neighbor}) for the adjacent weights $\sigma'$ and $\sigma''$ and an $\cO$-lattice $\sigma(\tau_\cS')^{\circ}$ in $\sigma(\tau_\cS')$.
Then by inductive hypothesis $M_\infty(\sigma'')$, and hence $M_\infty(\ovl{\sigma}(\tau_\cS')^{\circ})$, is nonzero.
By Corollary \ref{cor:globalBM} and Lemma \ref{lemma:multone}, $e(M_\infty(\ovl{\sigma}(\tau_\cS')^{\circ}) = 2$.
We deduce from the inductive hypothesis that $e(M_\infty(\sigma')) = 1$.

We now prove the general case of the lemma by induction on the defect.
Suppose that $\sigma' \in W^?(\rhobar_{\cS})$ has defect $\delta>0$.
We choose $\tau_\cS$ as in Lemma \ref{lemma:weight-combinatorics} (and an $\cO$-lattice $\sigma(\tau_\cS)^0$ in $\sigma(\tau_\cS)$) taking $\sigma$ to be $\sigma'$.
By Lemma \ref{lemma:weight-combinatorics}(\ref{minimal-type}) and (\ref{lower-defect}), $W^?(\rhobar_{\cS},\tau_\cS)$ contains a weight of lower defect.
The inductive hypothesis implies $M_\infty(\ovl{\sigma}(\tau_\cS)^{\circ})$ is nonzero.
By Lemma \ref{lemma:multone} and Corollary \ref{cor:globalBM}, $e(M_\infty(\ovl{\sigma}(\tau_\cS)^{\circ})) = 2^\delta$.
By Lemma \ref{lemma:weight-combinatorics}(\ref{lower-defect}) and induction, $e(M_\infty(\ovl{\sigma}(\tau_\cS)^{\circ})) - e(M_\infty(\sigma'))$ is the number of weights in $W^?(\rhobar_{\cS},\tau_\cS) \setminus\{\sigma'\}$, which is $2^\delta-1$.
We conclude that $e(M_\infty(\sigma')) = 1$.
\end{proof}
\end{lemma}

\begin{lemma} \label{lemma:lower-defect}
Suppose that $\rhobar_{\tld{v}}:G_{F_{\tld{v}}} \rightarrow \GL_3(\F)$ is semisimple and $10$-generic for all $\tld{v}\in \cS$.
If $\sigma \in W^{\mathrm{BM}}(\rhobar_{\cS})$ has defect $\delta>0$, then there exists $\sigma' \in W^{\mathrm{BM}}(\rhobar_{\cS})$ with defect less than $\delta$.
\begin{proof}
Choose $\tau_\cS$ and $\tau_\cS'$ as in Lemma \ref{lemma:weight-combinatorics} and fix lattices $\sigma(\tau_\cS)^{\circ}$ and $\sigma(\tau_\cS')^{\circ}$.
Then $M_\infty(\sigma(\tau_\cS)^{\circ})$ and $M_\infty(\sigma(\tau_\cS')^{\circ})$ are nonzero. 
Hence by Lemmas \ref{lemma:multone} and \ref{lemma:weight-combinatorics}(\ref{minimal-type}) and Corollary \ref{cor:globalBM}, $e(M_\infty(\ovl{\sigma}(\tau_\cS')^{\circ})) - e(M_\infty(\ovl{\sigma}(\tau_\cS')^{\circ})) = 2^\delta$.
By Proposition \ref{lemma:weight-combinatorics}(\ref{type-inclusion}), $e(M_\infty(\ovl{\sigma}(\tau_\cS')^{\circ})) - e(M_\infty(\ovl{\sigma}(\tau_\cS)^{\circ}))$ is the sum of $e(M_\infty(\sigma'))$ as $\sigma'$ runs over the Serre weights in $W^?(\rhobar_{\cS},\tau_\cS') \setminus W^?(\rhobar_{\cS},\tau_\cS)$.
By Lemma \ref{lemma:weight-combinatorics}(\ref{lower-defect}), the Serre weights in $W^?(\rhobar_{\cS},\tau_\cS') \setminus W^?(\rhobar_{\cS},\tau_\cS)$ have defect less than $\delta$.
We conclude that there must be a Serre weight $\sigma' \in W^{\mathrm{BM}}(\rhobar_{\cS})$ of defect less than $\delta$. 
\end{proof}
\end{lemma}

\subsubsection{Proofs}\label{subsubsec:proofs}

\begin{proof}[Proof of Theorem \ref{SWC}]
Since $M_\infty$ is non-zero, there is a Serre weight $\sigma \in W^{\mathrm{BM}}(\rhobar_{\cS})$ $\subset W^?(\rhobar_{\cS})$ by Proposition \ref{WE}. 
By induction on the defect using Lemma \ref{lemma:lower-defect}, we can assume without loss of generality that the defect of $\sigma$ is $0$.
The theorem now follows from Lemma \ref{lemma:defect-0-case}.
\end{proof}

\begin{rmk}\label{rmk:descendsupport}
Our axioms for $M_\infty$ imply that if $\sigma \in \JH(\ovl{\sigma(\tau_\cS)})$, then the support of $M_\infty(\sigma)$ is a (possibly empty) union of irreducible components of $\Spec \ovl{R}_\infty(\tau_\cS)$. As $\Spec R_\infty(\tau_\cS)$ is the preimage of 
\[
\Spec \widehat{\bigotimes}_{\tld{v}\in \cS} R_{\tld{v}}^{\tau_{\tld{v}}} \subset \Spec \widehat{\bigotimes}_{\tld{v}\in \cS} R_{\tld{v}}^\Box
\]
inside $\Spec R_\infty$, we have shown that if $\sigma \in W^{\mathrm{BM}}(\rhobar_\cS)$, then the scheme-theoretic support of $M_\infty(\sigma)$ is irreducible with generic point given by the preimage in $R_\infty$ of a prime ideal in $\widehat{\bigotimes}_{\tld{v}\in \cS} R_{\tld{v}}^\Box$.
\end{rmk}

We now give two examples of weak minimal patching functors using the setup from \cite[\S 7.1]{LLLM}.
Recall the definitions from {\it loc.~cit.}~ of $F/F^+$, $\Sigma_p^+$, $G_{/F^+}$, and $\iota_w$ (see also \S \ref{subsec:global}).
Suppose that $\rbar:G_F \ra \GL_3(\F)$ is 
\begin{itemize}
\item automorphic (of some weight) in the sense of \cite[Definition 7.1]{LLLM};
\item satisfies the Taylor--Wiles hypotheses in the sense of \cite[Definition 7.3]{LLLM}; and
\item 
if $\rbar$ is ramified at a finite place $w\notin \Sigma_p$ of $F$ then $w|_{F^+}$ splits in $F$ (we say that $\rbar$ has \emph{split ramification} outside of $p$).
\end{itemize}
Then \cite[Proposition 7.15]{LLLM} constructs a weak minimal patching functor for $\rbar$ in the sense of \cite[Definition 7.11]{LLLM}, which we will denote $\tld{M}_\infty$.
Let $h$ be the corresponding integer.

For each $v\in \Sigma_p^+$, choose a place $\tld{v}$ of $F$ lying above $v$.
Let $\cS_p$ be the set $\{\tld{v}:v\in \Sigma_p^+\}$.
Let $\rhobar_{\cS_p}$ be $(\rhobar_{\tld{v}})_{\tld{v}\in \cS_p}$ where we define $\rhobar_{\tld{v}} \defeq \rbar|_{G_{F_{\tld{v}}}}: G_{F_{\tld{v}}} \ra \GL_3(\F)$.
Define $R_\infty$, $R_\infty(\tau_{\cS_p})$, etc. ~as before with respect to $h$ above.
Let $K$ be $\prod_{\cS_p}K_{\tld{v}}$ as before.
 and the proof of \cite[Proposition 4.2.6]{LLL}.

\begin{prop}\label{prop:patchfunctor}
Let $M_\infty:\Rep_{K}(\cO)\ra \Coh(X_{\infty})$ be the functor $\tld{M}_\infty \circ \prod_{\tld{v}\in \cS_p}\iota_{\tld{v}}$.
Then $M_\infty$ is a weak minimal patching functor for $\rhobar_{\cS_p}$.
\end{prop}
\begin{proof}
Definition \ref{minimalpatching}(\ref{item:semistable}) follows from the proof of \cite[Proposition 4.2.6]{LLL}.
The remaining properties follow easily from definitions and \cite[Proposition 7.15]{LLLM}.
\end{proof}

Now suppose that $p>3$ and $K/\QQ_p$ is a finite extension.
Let $\rhobar:G_K\ra \GL_3(\F)$ be a continuous Galois representation with a potentially diagonalizable lift of type $(\eta,\tau)$, such that $R^{\tau}_{\rhobar}$ is formally smooth.
For example, if $\rhobar$ is $6$-generic and semisimple, 
we can take $\tau$ so that if $\tld{w}(\rhobar,\tau) = (\tld{w}_j)_j$, then $\ell(\tld{w}_j) = 4$ for all $j$.
We now construct a weak minimal patching functor for $\rhobar$ (here $\#\cS = 1$ in the notation of Definition \ref{minimalpatching}).

\cite[Corollary A.7]{EG} constructs a CM extension $F/F^+$, a choice of places $\cS_p$ above $\Sigma_p^+$ as before, and an automorphic Galois representation $\rbar: G_F \ra \GL_3(\F)$ satisfying the above itemized properties such that there is an isomorphism $K \cong F_{\tld{v}}$ and $\rbar|_{G_{F_{\tld{v}}}} \cong\rhobar$ for all $\tld{v}\in \cS_p$.
Fix a place $\tld{v}\in \cS_p$ and let $R_\infty$ be $R_{\rhobar}^\square \widehat{\otimes} \widehat{\bigotimes}_{v'\in \Sigma_p^+, v'\neq v} R_{\rhobar_{\tld{v}'}}^\tau[\![x_1,\ldots,x_h]\!]$.
Note that $R_\infty$ is as defined in \S \ref{sub:sub:setup.} (by increasing $h$, and the assumption of formal smoothness of $R_{\rhobar}^{\tau}$) and we identify $R_\infty$ with $R_{\rbar|_{G_{F_{\tld{v}}}}}^\square \widehat{\otimes} \widehat{\bigotimes}_{v'\in \Sigma_p^+, v'\neq v} R_{\rbar|_{G_{F_{\tld{v}'}}}}^{\tau}[\![x_1,\ldots,x_h]\!]$.
Let $X_\infty = \Spf R_\infty$ as usual.
We abusively let $K$ be the group $\GL_3(\cO_K)$ (the meaning of each $K$ will be clear from the context).
Fix a lattice $\sigma(\tau)^\circ$ in the $K$-module $\sigma(\tau)$.

Let $\tld{M}_\infty$ be the weak minimal patching functor for $\rbar$ constructed above.
The following proposition follows from the construction of $\rbar$ and Proposition \ref{prop:patchfunctor}. 

\begin{prop}\label{prop:singlepatch} Assume the above setup.
Let $M_\infty:\Rep_{K}(\cO)\ra \Coh(X_{\infty})$ be the functor 
\[\tld{M}_\infty \circ \underset{v\in \Sigma_p^+} \prod \iota_{\tld{v}} \circ\left(-\otimes \widehat{\bigotimes}_{v\in \Sigma_p^+, v'\neq v} \sigma(\tau)^\circ\right).\]
Then $M_\infty$ is a patching functor for $\rhobar$.
\end{prop}

\begin{cor}\label{cor:multipatch}
Suppose that $\rhobar_\cS$ is a collection of continuous Galois representations satisfying the properties of $\rhobar$ in Proposition \ref{prop:singlepatch}%
Then there exists a patching functor for $\rhobar_\cS$.
\end{cor}
\begin{proof}
We can take the completed tensor product of the patching functors constructed in Proposition \ref{prop:singlepatch} for each $\rhobar_{\tld{v}}$.
\end{proof}

\begin{proof}[Proof of Theorem \ref{thm:BM}] 
Suppose that $\tau$ is not $1$-generic.
We claim that $R^\tau_{\rhobar} = 0$.
It suffices to show that after restriction to $G_{K'}$ for any unramified extension $K'/K$, $\rhobar$ does not have a potentially crystalline lift of type $\tau$ and Hodge--Tate weights $(2,1,0)$.
Moreover, after such a restriction, $\tau$ is still not $1$-generic and $\rhobar$ is $10$-generic.
We can then assume without loss of generality that $\tau$ is principal series.
By \cite[Remark 2.2.8]{LLL}, $\tau$ is not $2$-generic and $\rhobar$ is $10$-generic in the sense of \cite[Definition 2.2]{Enns}.
Then $\rhobar$ does not have a potentially crystalline lift of type $\tau$ and Hodge--Tate weights $(2,1,0)$ by \cite[Proposition 7]{Enns}.

If $\tau$ is $1$-generic and $R^{\tau}_{\rhobar}\neq 0$, then Proposition \ref{telim2} implies that $\tau$ is $7$-generic.
Proposition \ref{typematching} implies that $W^?(\rhobar,\tau) \neq \emptyset$.
Suppose that $\tau$ is $1$-generic and $W^?(\rhobar,\tau)\neq \emptyset$. 
By Proposition \ref{prop:singlepatch}, a patching functor $M_\infty$ for $\rhobar$ exists.
For any $\cO_E$-lattice $\sigma(\tau)^\circ \subset \sigma(\tau)$, $M_\infty(\sigma(\tau)^\circ) \neq 0$ by Theorem \ref{SWC}.
This implies that $R^\tau_{\rhobar}$ is nonzero.

We now show that if $\tau$ is $1$-generic and $R^\tau_{\rhobar}\neq 0$ then 
\begin{equation}
\label{eq:BM:mult}
e(\ovl{R}_{\rhobar}^\tau) = e(M_\infty(\ovl{\sigma}(\tau)^\circ)).
\end{equation}
As $R^\tau_{\rhobar}\neq 0$ we deduce from Proposition \ref{telim2} that $\tau$ is 7-generic. The proof of (\ref{eq:BM:mult}) is now obtained by a direct generalization of the arguments of \cite[\S 8]{LLLM}, so we only explain the key details.
There is a ring $\tld{R}_{\rhobar}^\tau$ of the same dimension as $\ovl{R}_{\rhobar}^\tau[\![X_1,\cdots X_{3f}]\!]$ which is a power series ring over the completed tensor product over $\F$ of rings $\widetilde{R}$ in \cite[\S 8]{LLLM} and $\ovl{R}^{\mathrm{expl},\nabla}_{\ovl{\fM},\tld{w}}$ for some $\ovl{\fM}$ and $\tld{w}$ from \cite[Table 7]{LLLM}, and which admits a surjection to $\ovl{R}_{\rhobar}^\tau[\![X_1,\cdots X_{3f}]\!]$.
Then we have 
\[
e(\tld{R}_{\rhobar}^\tau) = e(M_\infty(\ovl{\sigma}(\tau)^\circ))
\]
for any $\cO$-lattice $\sigma(\tau)^\circ \subset \sigma(\tau)$ by computing both sides using \cite[Table 7, Propositions 8.5 and 8.12]{LLLM} and Theorem \ref{SWC} (using that $\rhobar$ is $10$-generic).
The proof of \cite[Proposition 7.14]{LLLM} implies 
\[
e(M_\infty(\ovl{\sigma}(\tau)^\circ)) \leq e(\ovl{R}_{\rhobar}^\tau).
\] 
Then we must have that $e(\tld{R}_{\rhobar}^\tau) = e(\ovl{R}_{\rhobar}^\tau[\![X_1,\cdots X_{3f}]\!])=e(\ovl{R}_{\rhobar}^\tau)$ by the above surjection.
Furthermore, $\tld{R}_{\rhobar}^\tau$ is reduced and Cohen--Macaulay by the same argument as in the proof of Lemma \ref{lemma:BM}.%
Then $\ovl{R}_{\rhobar}^\tau[\![X_1,\cdots X_{3f}]\!]$ is isomorphic to $\tld{R}_{\rhobar}^\tau$ by \cite[Lemma 8.8]{LLLM} and the desired properties of $\ovl{R}_{\rhobar}^\tau$ hold because they do for $\tld{R}_{\rhobar}^\tau$.
The desired ring theoretic properties of $R_{\rhobar}^\tau$ follow from the proof of \cite[Corollary 8.9]{LLLM}.
\end{proof}

\begin{rmk}\label{rmk:inj}
Since the number of irreducible components of $\ovl{R}_{\rhobar}^\tau$ is equal to $\# W^?(\rhobar,\tau)$, $\ovl{R}_{\rhobar}^\tau$ is reduced and $M_\infty(\ovl{\sigma}(\tau)^\circ)$ has full support over $\ovl{R}_\infty(\tau)$, the proof of Theorem \ref{thm:BM} shows that the irreducible support of $M_\infty(\sigma)$ must be different for each $\sigma \in W^?(\rhobar,\tau)$.
\end{rmk}

\subsection{The geometric Breuil--M\'{e}zard conjecture}
\label{subsec:GBM}

We now show that weak minimal patching functors can be used to assign components in deformation rings to Serre weights.

\begin{prop} 
\label{prop:identify:cmpt}
\begin{enumerate}
\item \label{item:assignment} Let $\rhobar$ be as in Theorem \ref{thm:conngenfib}.
There is a unique assignment $\sigma \mapsto \mathfrak{p}(\sigma)$ for $\sigma \in W^?(\rhobar)$ such that $\mathfrak{p}(\sigma)\subset R_{\rhobar}^{\Box}$ is a prime ideal and
\begin{equation}\label{eqn:union}
\Spec(\ovl{R}^{\tau}_{\rhobar})=\underset{\sigma\in W^?(\rhobar,\tau)}{\bigcup} \Spec(R_{\rhobar}^{\Box}/\mathfrak{p}(\sigma))
\end{equation}
for any tame type $\tau$ where the right-hand side is given the reduced scheme-structure.
Moreover, the image of $\mathfrak{p}(\sigma)$ in $\ovl{R}^{\tau}_{\rhobar}$ is a minimal prime ideal and $(\ref{eqn:union})$ is the decomposition of $\ovl{R}^{\tau}_{\rhobar}$ into irreducible components.
\item If $M_\infty$ is a weak minimal patching functor for $\rhobar_{\cS} = (\rhobar_{\tld{v}})_{\tld{v}\in \cS}$, then the scheme-theoretic support of $M_\infty(\otimes_{\tld{v}\in \cS}\sigma_{\tld{v}})$ is $\Spec(R_\infty/\sum_{\tld{v}\in \cS}\mathfrak{p}(\sigma_{\tld{v}})R_\infty)$. \label{item:patchcomp}
\end{enumerate}
\end{prop}
\begin{proof}
We first prove uniqueness.
Suppose there is such an assignment.
This closely follows the procedure of induction on the defect with respect to $W^?(\rhobar)$ in the proof of Theorem \ref{SWC}.
If the defect of $\sigma \in W^?(\rhobar)$ is $0$, then by letting $\tau$ be the minimal type of $\sigma$ with respect to $\rhobar$ we have $\#W^?(\rhobar,\tau)=1$ by Lemma \ref{lemma:weight-combinatorics}.
Then we must have $\mathfrak{p}(\sigma) = \mathrm{Ann}_{R_{\rhobar}^{\Box}}\ovl{R}^{\tau}_{\rhobar}$.
If the defect of $\sigma$ is $\delta>0$, then choose $\tau$ be the minimal type of $\sigma$ with respect to $\rhobar$ as in Lemma \ref{lemma:weight-combinatorics}. Then by induction and Lemma \ref{lemma:weight-combinatorics}(\ref{lower-defect}), there is a unique component of $\Spec(\ovl{R}^{\tau}_{\rhobar})$ whose defining ideal is not $\mathfrak{p}(\sigma')$ for some $\sigma'\in W^?(\rhobar,\tau)$ of lower defect.
Then $\mathfrak{p}(\sigma)$ must be this defining ideal.

We now show existence of an assignment.
By Proposition \ref{prop:singlepatch}, there is a weak minimal patching functor $M_\infty$ for $\rhobar$, which we fix.
By Remark \ref{rmk:descendsupport}, the generic point of the scheme-theoretic support of $M_\infty(\sigma)$ is of the form $\mathfrak{p}'(\sigma) R_\infty$ for some prime ideal $\mathfrak{p}'(\sigma) \subset {R}^{\Box}_{\rhobar}$.
We claim that $\sigma \mapsto \mathfrak{p}'(\sigma)$ is an assignment satisfying (\ref{eqn:union}).
Indeed, since the generic fiber of $R^{\tau}_{\rhobar}$ is connected by \ref{thm:conngenfib}, $\Spec \ovl{R}_\infty(\tau)$ is the scheme-theoretic support of $M_\infty(\ovl{\sigma}(\tau)^{\circ})$. 
On the other hand, $M_\infty(\ovl{\sigma}(\tau)^{\circ})$ is filtered by $M_\infty(\sigma)$ for $\sigma \in W^?(\rhobar,\tau)$ so that the support of $M_\infty(\ovl{\sigma}(\tau)^{\circ})$ is 
\begin{equation} \label{eqn:unioncheck}
\underset{\sigma\in W^?(\rhobar,\tau)}{\bigcup} \Spec(R_\infty/\mathfrak{p}'(\sigma)R_\infty).
\end{equation}
(\ref{eqn:unioncheck}) is a decomposition of $\ovl{R}_\infty(\tau)$ into irreducible components by Remark \ref{rmk:inj}. Finally, we observe that this statement descends to $\ovl{R}^{\tau}_{\rhobar}$.

We now show part (\ref{item:patchcomp}).
Suppose that $M_\infty$ is a weak minimal patching functor for $\rhobar_{\cS}$.
Let $\sigma = \otimes_{\tld{v}\in \Sigma}\sigma_{\tld{v}} \in W^?(\rhobar_{\cS})$.
Again by the proof of Theorem \ref{SWC}, the scheme-theoretic support of $M_\infty(\sigma)$ is $\Spec R_\infty/(\sum_{\tld{v}\in \cS}\mathfrak{p}(\sigma)_{\tld{v}} R_\infty)$ for some prime ideals $\mathfrak{p}(\sigma)_{\tld{v}} \subset R^\Box_{\rhobar_{\tld{v}}}$.
We will show that $\mathfrak{p}(\sigma)_{\tld{v}} = \mathfrak{p}(\sigma_{\tld{v}})$, where $\mathfrak{p}$ is the assignment in part (\ref{item:assignment}).
We induct on $\delta = \mathrm{Def}_{\rhobar_\cS}(\sigma)$.
If $\delta = 0$, then one can choose $\tau_\cS$ as in Lemma \ref{lemma:weight-combinatorics} so that $\ovl{R}_\infty(\tau_\cS) = R_\infty/\mathfrak{p}(\sigma)$.
We conclude that $\mathfrak{p}(\sigma)_{\tld{v}} = \mathfrak{p}(\sigma_{\tld{v}})$ for all ${\tld{v}}\in \cS$ in this case.
Suppose that $\delta > 0$.
Again choose $\tau_\cS$ as in Lemma \ref{lemma:weight-combinatorics}.
Then by the inductive hypothesis, for any weight $\sigma' =\otimes_{\tld{v}\in \cS} \sigma'_{\tld{v}} \in W^?(\rhobar_{\cS},\tau_\cS)$ with $\sigma' \not\cong \sigma$, the scheme-theoretic support of $M_\infty(\sigma')$ is $R_\infty/(\sum_{\tld{v}\in \cS}\mathfrak{p}(\sigma'_{\tld{v}}) R_\infty)$.
Since the generic fiber of $\Spec(\widehat{\otimes}_{\tld{v}\in \cS} R_{\rhobar_{\tld{v}}}^{\tau_{\tld{v}}})$ is connected by \ref{thm:conngenfib}, by item (\ref{item:assignment}) the scheme-theoretic support of $M_\infty(\ovl{\sigma}(\tau_\cS)^\circ)$ is 
\begin{equation}
\underset{\sigma'\in W^?(\rhobar_{\cS},\tau_\cS)}{\bigcup} \Spec(R_\infty/(\sum_{\tld{v}\in \cS} \mathfrak{p}(\sigma'_{\tld{v}}) R_\infty)).
\end{equation}
From this, we see that since $e(M_\infty(\sigma)) = 1$, the scheme-theoretic support of $M_\infty(\sigma)$ is forced to be $\Spec R_\infty/(\sum_{\tld{v}\in \cS}\mathfrak{p}(\sigma_{\tld{v}}) R_\infty)$. 
\end{proof}

We have the following refinement of Proposition \ref{prop:identify:cmpt}(\ref{item:patchcomp}).

\begin{lemma}
\label{lem:supp} Assume $\rhobar$ is $10$-generic.
Let $\tau_\cS= (\tau_{\tld{v}})_{\tld{v}\in \cS}$ be a collection of tame inertial types and let $V$ be a subquotient of $\ovl{\sigma}(\tau_\cS)^\circ$ for some $\cO$-lattice $\sigma(\tau_\cS)^\circ$ in $\sigma(\tau_\cS)$.
Define the closed subscheme $\ovl{X}_\infty(V)\into \Spec \ovl{R}_\infty$ to be the reduced subscheme underlying $\bigcup_{\sigma\in \JH(V)} \supp(M_{\infty}(\sigma))$.
Then the scheme-theoretic support of $M_{\infty}(V)$ is $\ovl{X}_\infty(V)$.
In particular if $M_\infty(V)$ is a cyclic $R_\infty$-module, then $M_\infty(V) \cong R_\infty/I(V)$ where
\begin{equation*}
I(V)\defeq \underset{\sigma\in \JH(V)}{\bigcap}\mathrm{Ann}_{R_{\infty}}(M_{\infty}(\sigma)).
\end{equation*}
\end{lemma}
\begin{proof}
Since $\rhobar$ is $10$-generic, there is nothing to prove unless $\tau_S$ is $7$-generic.
The proof now follows exactly as in the second paragraph of the proof of \cite[Proposition 8.1.1]{EGS}.
We recall the argument. 
The support of $M_\infty(V)$ is (the topological space) $\ovl{X}_\infty(V)$ since $M_\infty(V)$ is filtered by $M_\infty(\sigma)$ for $\sigma \in \JH(V)$.
It suffices to show that the scheme-theoretic support of $M_\infty(V)$ is reduced.
The scheme-theoretic support of $M_\infty(V)$ is generically reduced (since the same is true for $M_\infty(\ovl{\sigma}(\tau_\cS)^\circ)$). 
Now since each $M_\infty(\sigma)$ is maximal Cohen-Macaulay over $\ovl{R}_\infty(\tau_\cS)$ (by Definition \ref{minimalpatching}(\ref{dimd}) and the fact that $\dim \ovl{R}_\infty(\tau_{\cS})=d$), and the maximal Cohen--Macaulay property is preserved under extension (by the characterization of depth in terms of $\Ext$ groups), $M_\infty(V)$ is maximal Cohen--Macaulay over $\ovl{R}_\infty(\tau_\cS)$.
This guarantees that the scheme-theoretic support of $M_\infty(V)$ has no embedded associated primes, and hence is reduced.
\end{proof}

\subsubsection{Matching components}
\label{sec:match}

Recall that Proposition \ref{prop:identify:cmpt} gives a canonical parametrization of the irreducible components of the special fiber of the potentially crystalline deformation ring $R_{\rhobar}^{\tau}$ in terms of $W^?(\rhobar, \tau)$. 
Given a lowest alcove presentation $(s,\mu-\eta)$ of $\tau=\tau(s,\mu)$ we will define in item (\ref{it:def:rings}) below explicit rings $\ovl{R}_{\ovl{\fM},\tld{w}}^{\expl,\nabla}$, building on \cite[\S 5.3.1, \S 8]{LLLM}. 
The rings $\ovl{R}_{\ovl{\fM},\tld{w}}^{\expl,\nabla}$ will be formally smooth modifications of $\ovl{R}_{\rhobar}^{\tau}$.
Thus, Proposition \ref{prop:identify:cmpt} gives a bijection between minimal primes of $\ovl{R}_{\ovl{\fM},\tld{w}}^{\expl,\nabla}$ and $W^?(\rhobar, \tau)$. On the other hand, by Corollary \ref{cor:intersections}, the data $(s,\mu)$ gives a description of $W^?(\rhobar, \tau)$ as $\sigma^{(s,\mu)}_{(\omega,a)}$, for $(\omega,a)\in \Sigma_{\tld{w}}$. In this subsection, we will make the bijection between minimal primes of 
 $\ovl{R}_{\ovl{\fM},\tld{w}}^{\expl,\nabla}$ and $\Sigma_{\tld{w}}$ explicit. This will be needed in Section \ref{sec:BC}, where we need to check relations between ideals corresponding to various subquotients of $\ovl{\sigma}(\tau)^0$ for certain lattices $\sigma(\tau)^0$.

We begin by recalling the relationship between $\ovl{R}_{\ovl{\fM},\tld{w}}^{\expl,\nabla}$ and $\ovl{R}^{\tau}_{\rhobar}$.
For the rest of this section, we assume $\rhobar$ is $10$-generic. Recall that we have chosen a lowest alcove presentation $(s,\mu-\eta)$ of $\tau=\tau(s,\mu)$. We assume that $R^\tau_{\rhobar}\neq 0$, and thus $\tau$ is $7$-generic and there exists a unique $\ovl{\fM}\in Y^{\eta,\tau}(\F)$ such that $T^*_{dd}(\overline{\fM}) \cong \rhobar|_{G_{K_{\infty}}}$. Let $\tld{w}=\tld{w}(\rhobar,\tau)$ be the shape of $\ovl{\fM}$ and let $\overline{\cM} \defeq (\ovl{\fM}[1/u'])^{\Delta = 1}$. We also recall the notion of gauge basis \cite[Definitions 4.15, 6.11]{LLLM}, and we fix a gauge basis $\ovl{\beta}$ for $\ovl{\fM}$. We also fix a framing (i.e.~a basis) for $\rhobar$. Recall from \cite[Definition 2.11 and \S 6]{LLLM} (see also the discussion after \cite[Definition 3.2.8]{LLL}, which is more aligned with the notation of this paper) the notation $A^{(i)}\defeq \Mat_{\beta}(\phi^{(i)}_{\fM,s_{\orient,i+1}(3)})$, for an eigenbasis $\beta$ of a Kisin module $\fM$ with descent data of type $\tau$ and $s_{\orient}\in \un{W}$ the orientation of $\tau=\tau(s,\mu)$. We say that $A^{(i)}$ is the matrix of the $i$-th partial Frobenius $\phi^{(i)}_{\fM,s_{\orient,i+1}(3)}$ with respect to the eigenbasis $\beta$.

We have the following canonical diagram (cf \cite[Diagram (5.9)]{LLLM}:
\begin{equation}
\label{diag:f.s.prime}
\xymatrix{
\Spf(\ovl{R}^{\tau,\ovl{\beta},\Box}_{\ovl{\fM},\rhobar})\ar^{f.s.}[rr]
\ar_{f.s.}[d]\ar@{}[drr]|{\ulcorner}&&\Spf(\ovl{R}_{\ovl{\fM},\tld{w}}^{\expl,\nabla})\ar^{f.s.}[d]\ar@{^{(}->}[r]\ar@{}[dr]|{\ulcorner}&\ovl{D}_{\ovl{\fM}}^{\tau,\ovl{\beta}}\ar^{f.s.}[d]
\\
\Spf (\ovl{R}^{\tau}_{\rhobar})=\Spf(\ovl{R}^{\tau,\Box}_{\ovl{\fM},\rhobar})\ar^{f.s.}[rr]\ar@{^{(}->}[d]\ar@{}[drr]|{\ulcorner}&&\left[\Spf(\ovl{R}_{\ovl{\fM},\tld{w}}^{\expl,\nabla})/\widehat{\bf{G}}_m^{3f}\right]\ar@{^{(}->}^{\imath_{\tau}}[d]\ar@{^{(}->}[r]&
 \ovl{Y}^{\eta,\tau}_{\ovl{\fM}} \ar^{\imath'_{\tau}}@{^{(}->}[dl]
\\
\Phi\text{-}\Mod^{\text{\'et},\Box}_{\ovl{\cM}}\ar^{f.s.}[rr]
&&\Phi\text{-}\Mod^{\text{\'et}}_{\ovl{\cM}}
}
\end{equation}
where $f.s.$ stands for a formally smooth morphism. %
We explain the diagram:
\begin{enumerate}
\item $\ovl{R}^{\tau,\ovl{\beta},\Box}_{\ovl{\fM},\rhobar}$, $\ovl{R}^{\tau,\Box}_{\ovl{\fM},\rhobar}$, are defined in \cite[\S 4.3, \S 5.2, \S 6.2]{LLLM}. They parametrize various deformation problems of $\rhobar$ and $\ovl{\fM}$ with extra data such as framings on Galois representations and gauge bases on $\ovl{\fM}$. We note that in \emph{loc.~cit.}, we used the symbol $\mu$ for what we call $\eta$ in this paper.
\item $\Phi\text{-}\Mod^{\text{\'et}}_{\ovl{\cM}}$ (resp. $\Phi\text{-}\Mod^{\text{\'et},\Box}_{\ovl{\cM}}$) denotes the groupoid of \'{e}tale $\phz$-modules deforming $\ovl{\cM}$ (resp. deformations with a basis on the associated $G_{K_\infty}$-representation).
\item $\ovl{Y}^{\eta,\tau}_{\ovl{\fM}}$ is the groupoid whose values on a local Artinian $\F$-algebra $A$ is given by the groupoid of pairs $(\fM_A,\jmath_A)$ where $\fM_A\in Y^{\eta,\tau}(A)$ and $\jmath_A:\fM_A\otimes_A\F\stackrel{\sim}{\ra}\ovl{\fM}$ is an isomorphism in $Y^{\eta,\tau}(\F)$. %
The groupoid $\ovl{D}_{\ovl{\fM}}^{\tau,\ovl{\beta}}$ parametrizes the same data plus the data of a gauge basis lifting $\ovl{\beta}$. By \cite[Theorem 6.12 ]{LLLM}, there is an action of $\widehat{\bf{G}}_m^{3f}$ on $\ovl{D}_{\ovl{\fM}}^{\tau,\ovl{\beta}}$ by scaling the gauge basis, and one has $\left[\ovl{D}_{\ovl{\fM}}^{\tau,\ovl{\beta}}/\widehat{\bf{G}}_m^{3f}\right]\cong \ovl{Y}^{\eta,\tau}_{\ovl{\fM}}$.
By \cite[Theorem 4.17 and \S 6.2]{LLLM}, $\ovl{D}_{\ovl{\fM}}^{\tau,\ovl{\beta}}$ is representable by $\ovl{R}^{\tau,\ovl{\beta}}_{\ovl{\fM}}=\widehat{\bigotimes}_{i}(R_{\tld{w}_i}^{\mathrm{expl}})^{p\text{-flat, red}}/\varpi$. Over $\ovl{R}^{\tau,\ovl{\beta}}_{\ovl{\fM}}$, we have a universal pair $(\fM^{\univ},\beta^{\univ})$, and hence the universal matrices of partial Frobenii $A^{(i),\univ}$.
By construction, $(R_{\tld{w}_i}^{\mathrm{expl}})^{p\text{-flat, red}}/\varpi$ is a quotient of the power series ring over $\F$ generated by (suitable modifications of) the coefficients of the entries of $A^{(i),\univ}$ subject to certain ``finite height'' equations. The map $\imath'_{\tau}$ is the map sending $\fM^{\univ}$ to $(\fM^{\univ}[\frac{1}{u'}])^{\Delta=1}$
\item 
\label{it:def:rings}
The ring $\ovl{R}_{\ovl{\fM},\tld{w}}^{\expl,\nabla}= \widehat{\bigotimes}_{i}\ovl{R}_{\ovl{\fM},\tld{w}_i}^{\expl,\nabla}$. We recall the description of each component ring from \cite{LLLM} (see also Table \ref{Table:intsct} below). 
We let $(a,b,c)\in \Fp^3$ be the mod $p$ reduction of $s^{-1}_{f-1-i}(\mu_{f-1-i})\in X^*(T)\cong \Z^3$.
\begin{enumerate}
\item \label{expl:length>1}When $\ell(\tld{w}_i)>1$: $\ovl{R}_{\ovl{\fM},\tld{w}_i}^{\expl,\nabla}$ is the quotient of the power series ring over $\F$ generated by (suitable modifications of) of the coefficients of the entries of $A^{(i),\univ}$ by an explicit list of relations given by \cite[\S 5.3 and Table 7]{LLLM}. In this case, we even have the rings  $R_{\ovl{\fM},\tld{w}_i}^{\expl,\nabla}$ and diagram (\ref{diag:f.s.prime}) can be lifted to a diagram over $\cO$ with the same properties, cf.~\cite[Diagram (5.9)]{LLLM}.

Note that, strictly speaking, \cite[Table 7]{LLLM} only has entries for $\tld{w}_i$ belonging to a certain set of representatives under the action of the outer automorphisms of $\tld{W}_a$.
\item
\label{expl:alpha}
When $\ell(\tld{w}_i)=1$: By symmetry, we may assume $\tld{w}_i=\alpha t_{\un{1}}$. 
The matrix $A^{(i),\univ}$ has the form
\begin{align*}
A^{(i),\univ}=\begin{pmatrix}
c_{11}&c_{12}+vc_{12}^*& c_{13}\\
vc_{21}^*&c_{22}+vd_{22}&c_{23}\\
vc_{31}&vc_{32}&c_{33}+vc_{33}^*
\end{pmatrix}.
\end{align*}
Set $\tld{c}_{32}\defeq \frac{c_{32}c_{21}^*-d_{22}c_{31}}{c_{21}^*}$. 
We \emph{define} $\ovl{R}_{\ovl{\fM},\tld{w}_i}^{\expl,\nabla}$ to be the quotient of the power series ring 
\[
\F[\![c_{11},c_{12},c_{13}, c_{22},c_{23}, c_{31}, \tld{c}_{32},c_{33},d_{22}, c^{*}_{12}-[\ovl{c}^*_{12}], c^{*}_{21}-[\ovl{c}^*_{21}], c^{*}_{33}-[\ovl{c}^*_{33}]]\!]
\] by the relations:
\begin{align*}
&c_{11}c_{23}=0,&&
c_{33}^*c_{11}\tld{c}_{32}=c_{13}c_{31}\tld{c}_{32},\\
&c_{11}d_{22}c^*_{33}=\frac{b-c}{a-b}c^*_{21}c_{13}\tld{c}_{32},&&
c_{13}c_{23}\tld{c}_{32}=0,\\
&c_{23}c_{31}\tld{c}_{32}=0,&&
(a-b)c_{13}c_{31}d_{22}+(c-b)c_{13}\tld{c}_{32}c_{21}^*+(-1-a+c)c_{23}c_{31}c_{12}^*=0,\\
&c_{12}c^*_{33}=\frac{a-c}{a-b}c_{13}\tld{c}_{32},
&&
c_{22}c_{33}^*=\frac{(-1-a+c)}{(-1-a+b)}c_{23}\tld{c}_{32},
\\
&c_{21}^*c_{33}=c_{31}c_{23}.
&&
\end{align*}
These equations comes from \cite[Proposition 8.11]{LLLM} and its proof, by restoring the units $c_{12}^*, c_{21}^*$ and $c_{33}^*$ (which were set to be $1$ in \emph{loc.~cit.}); the first $6$ equations above are deduced from the equations appearing in the statement of \emph{loc.~cit.}~(where $d_{22}$ above is denoted by $c'_{22}$ in \emph{loc.~cit.}), the seventh and eighth equations above appear in the proof of \emph{loc.~cit.}, and the last equation above is implicit in \emph{loc.~cit.}~(where we solved $c_{33}$ using the $p$-saturation of the $2$ by $2$ minor condition).
In particular, $\ovl{R}_{\ovl{\fM},\tld{w}_i}^{\expl,\nabla}$ is a formal power series ring over the ring $\tld{R}$ appearing in \cite[Proposition 8.11]{LLLM}.
\item 
\label{expl:id}
When $\ell(\tld{w}_i)=0$ the matrix $A^{(i),\univ}$ has the form:
\begin{align*}
A^{(i),\univ}=\begin{pmatrix}
c_{11}+vc_{11}^*&c_{12}& c_{13}\\
vc_{21}&c_{22}+vc_{22}^*&c_{23}\\
vc_{31}&vc_{32}&c_{33}+vc_{33}^*
\end{pmatrix}.
\end{align*}
We \emph{define} $\ovl{R}_{\ovl{\fM},\tld{w}_i}^{\expl,\nabla}$ to be the quotient of $\F[\![c_{ij},\  c^{*}_{kk}-[\ovl{c}_{kk}^*],\, 1\leq i,j,k\leq 3]\!]$ by the relations
\begin{align*}
&c_{ii}c_{jj}=0\,\text{for $i\neq j$},&&
c_{11}c_{23}=0&&c_{31}c_{22}=0,&&
c_{33}c_{12}=0,\\
&c_{12}c_{23}=c_{22}c_{13},&&
c_{11}c_{32}=c_{12}c_{31},&&
c_{21}c_{33}=c_{31}c_{23}&&
\end{align*}
and
\begin{align*}
&(-1-a+c)c^*_{22}c_{33}+(-1-a+b)c_{22}c^*_{33}-(-1-a+c)c_{23}c_{32}=0
\\
&(a-b)c^*_{33}c_{11}+(-1-b+c)c_{33}c^*_{11}-(a-b)c_{13}c_{31}=0
\\
&(b-c)c^*_{11}c_{22}+(a-c)c_{11}c^*_{22}-(b-c)c_{12}c_{21}=0\\
&c_{11}c_{22}^*c_{33}^*+c_{22}c_{11}^*c_{33}^*+c_{33}c_{11}^*c_{22}^*-c_{11}^*c_{23}c_{32}-c_{22}^*c_{13}c_{31}-c_{33}^*c_{12}c_{21}+c_{13}c_{32}c_{21}=0.
\end{align*}
These equations comes from \cite[Corollary 8.4]{LLLM}, by restoring the units $c_{kk}^*$ (which were set to be $1$ in \emph{loc.~cit.}) and we added the equation $c_{12}c_{33}=0$, which was missing in \emph{loc.~cit.}~and which is obtained by the $2$ by $2$ minor condition on $A^{(i),\univ}$.
In particular, $\ovl{R}_{\ovl{\fM},\tld{w}_i}^{\expl,\nabla}$ is a formal power series ring over the ring $\tld{R}$ appearing in \cite[Corollary 8.4]{LLLM}.
\end{enumerate}
\end{enumerate}
Note that each object with a superscript $\Box$ receives a $\widehat{\GL}_3$-action, corresponding to changing the framing on the Galois representation.
 
We now justify the diagram:
\begin{itemize}
\item We claim that there is a canonical isomorphism $\left[\Spf(\ovl{R}^{\tau,\ovl{\beta},\Box}_{\ovl{\fM},\rhobar})/\widehat{\GL}_3\right]\cong \Spf (\ovl{R}_{\ovl{\fM},\tld{w}}^{\expl,\nabla})$. When $\ell(\tld{w}_i)>1$ for all $i$, this is \cite[Theorem 5.12, Theorem 6.14 and Table 7]{LLLM}. When $\ell(\tld{w}_i)\leq 1$ for some $i$, this follows from the same arguments as \cite[\S 8]{LLLM} together with Theorem \ref{SWC}, as explained in the proof of Theorem \ref{thm:BM}. This justifies the existence and the properties of the first row of Diagram (\ref{diag:f.s.prime}). It is clear from the construction that the $\widehat{\bf{G}}_m^{3f}$-action on $\ovl{D}_{\ovl{\fM}}^{\tau,\ovl{\beta}}$ preserves $\Spf (\ovl{R}_{\ovl{\fM},\tld{w}}^{\expl,\nabla})$. Let $(\fM_{\tld{w},\tau},\beta_{\tld{w},\tau})$ be the restriction of $(\fM^{\univ},\beta^{\univ})$ to $\Spf (\ovl{R}_{\ovl{\fM},\tld{w}}^{\expl,\nabla})$. Thus the map $\imath_{\tau}$ sends $\fM_{\tld{w},\tau}$ to $\cM_{\tld{w},\tau}\defeq (\fM_{\tld{w},\tau}[\frac{1}{u'}])^{\Delta=1}$.
\item The second row is obtained from the first row by quotiening by the $\widehat{\bf{G}}_m^{3f}$-action coming from scaling the gauge basis, hence inherits all properties from the first row. The top squares are Cartesian.
\item The second column is obtained from the first column by quotiening by the $\widehat{\GL}_3$-action coming from changing the framing of the Galois representation, hence the bottom square is Cartesian.
\end{itemize}
The following Proposition finishes our justification of the diagram:
\begin{prop}
\label{prop:tg:inj}
Assume that $\tau$ is $3$-generic and that $\ovl{\fM}\in \ovl{Y}^{\eta,\tau}(\F)$ is semisimple of shape $\tld{w}=(\tld{w}_j)$.
Let $\ovl{\beta}$ be a gauge basis for $\ovl{\fM}$.
Then the map 
\[
\imath'_\tau:  \ovl{Y}^{\eta,\tau}_{\ovl{\fM}}\rightarrow
\Phi\text{-}\Mod^{\text{\'et}}_{\ovl{\cM}}
\]
is a monomorphism.
\end{prop}
\begin{proof}
We need to prove that the map on the groupoids of $\F[\varepsilon]/(\varepsilon^2)$-points
\[
\ovl{Y}^{\eta,\tau}_{\ovl{\fM}}(\F[\varepsilon]/(\varepsilon^2))\rightarrow
\Phi\text{-}\Mod^{\text{\'et}}_{\ovl{\cM}}(\F[\varepsilon]/(\varepsilon^2))
\]
induced by $\imath'_{\tau}$ is fully faithful. But this follows from \cite[Proposition 3.2.18]{LLL}, noting that the right-hand side is equivalent to $\mathrm{Rep}_{\F[\eps]/\eps^2}(G_{K_\infty})_{\rhobar}$ in \emph{loc.~cit.}
\end{proof}

Diagram (\ref{diag:f.s.prime}) gives a bijection between the set of minimal primes $\Irr\big(\Spec(\ovl{R}^{\tau}_{\rhobar})\big)$ and $\Irr(\Spec(\ovl{R}_{\ovl{\fM},\tld{w}}^{\expl,\nabla}))=\prod_i \Irr(\Spec \big(\ovl{R}_{\ovl{\fM},\tld{w}_i}^{\expl,\nabla})\big)$. By Proposition \ref{prop:identify:cmpt}, this set is in bijection with the set $\fp(\sigma)$ for $\sigma\in W^?(\rhobar, \tau)$.
On the other hand, by Proposition \ref{typematching}, we have a bijection
\begin{align*}
\prod_i \Sigma_{\widetilde{w}^*_i}&\stackrel{\sim}{\ra}W^?(\rhobar, \tau) \\
(\omega,a)&\mapsto \sigma^{(s,\mu)}_{(\omega,a)}
\end{align*}
The following Theorem is the main result of this subsection, which computes the above bijections in terms of the explicit rings:
\begin{thm} 
\label{thm:matching}
Assume $\rhobar$ is $10$-generic. Via the above bijections, we have
\begin{align*}
\prod_{i=0}^{f-1}\Sigma_{\tld{w}^*_{i}}&\stackrel{\sim}{\longrightarrow}
\prod_i \Irr\big(\Spec (\ovl{R}_{\ovl{\fM},\tld{w}_i}^{\expl,\nabla})\big)\overset{\sim}{\longrightarrow}\Irr\big(\Spec(\ovl{R}^{\tau}_{\rhobar})\big)\\
\big((\omega_i,a_i)\big)_i&\longmapsto  \big((\fC_{(\omega_{f-1-i},a_{f-1-i})}\ovl{R}_{\ovl{\fM},\tld{w}_i}^{\expl,\nabla}\big)_i \longmapsto \mathfrak{p}(\sigma^{(s,\mu)}_{(\omega, a)}),\nonumber
\end{align*}
where $\fC_{(\omega_i,a_i)}$ is the minimal prime of $\ovl{R}_{\ovl{\fM},\tld{w}_{f-1-i}}^{\expl,\nabla}$ is given by Table \ref{Table:intsct}.
\end{thm}
\begin{rmk}
\label{rmk:chg:var}
Note that Table \ref{Table:intsct} only gives the ideals $\fC_{(\omega_i,a_i)}\subseteq \ovl{R}_{\ovl{\fM},\tld{w}_{f-1-i}}^{\expl,\nabla}$ for a set of representatives for $\tld{w}_{f-1-i}\in \Adm^\vee((2,1,0))$ for the action of the outer automorphisms of $\tld{W}^\vee$.
This is sufficient because this action corresponds to changing the lowest alcove presentation of $\tau$.
\end{rmk}
We now describe the strategy of proof. The main idea is that by the proof of Proposition \ref{prop:identify:cmpt}, the prime $\fp(\sigma)$ can be characterized by the relation (\ref{eqn:union}) for a specific type $\tau'$ (inducting on defect of $\sigma$). Namely, one can use the minimal type $\tau'$ of $\sigma$ with respect to $\rhobar$. For this type, each component of $\tld{w}'=\tld{w}(\rhobar,\tau')$ has length $\geq 3$. Furthermore, the minimal type $\tau'$ has the property that $\Spf\ovl{ R}^{\tau'}_{\rhobar}\subset \Spf \ovl{R}^{\tau}_{\rhobar}$ inside $\Spf \ovl{R}^{\Box}_{\rhobar}$. Thus it suffices to determine the subset of $\Irr\big(\Spec (\ovl{R}_{\ovl{\fM},\tld{w}_i}^{\expl,\nabla})\big)$ which occur in $\ovl{ R}^{\tau'}_{\rhobar}$. This is achieved by ``matching'' the universal \'{e}tale $\phz$-module over a union of irreducible components of $\ovl{R}_{\ovl{\fM},\tld{w}}^{\expl,\nabla}$ with the universal \'{e}tale $\phz$-module living over $\ovl{R}^{\tau'}_{\rhobar}$ (or rather $\ovl{R}_{\ovl{\fM}',\tld{w}'}^{\expl,\nabla}$).

The precise formulation of this matching mechanism is given by the following:
\begin{lemma}\label{lem:matching} Let $\tau=\tau(s,\mu)$ and $\tau'=\tau(s',\mu')$. Assume that $R^{\tau}_{\rhobar}, R^{\tau'}_{\rhobar}\neq 0$, and the running hypothesis that $\rhobar$ is $10$-generic. Consider diagram (\ref{diag:f.s.prime}) for $R^{\tau}_{\rhobar}$, $R^{\tau'}_{\rhobar}$, constructed using the above presentations. We decorate the objects that occur in the diagram for $\tau'$ with the same symbol as those in the diagram for $\tau$ but with a superscript ${}^\prime$ added \emph{(}so we have, e.g.~$\tld{w}'$, $\fM_{\tld{w}',\tau'}$, etc.\emph{)}. Assume that there are ideals $I$, \emph{(}resp. $I'$\emph{)} of $\ovl{R}_{\ovl{\fM},\tld{w}}^{\expl,\nabla}$ \emph{(}resp. $\ovl{R}_{\ovl{\fM}',\tld{w}'}^{\expl,\nabla}$\emph{)} such that
\begin{itemize}
\item $I$, $I'$ are intersections of minimal primes.
\item There is an isomorphism $\ovl{R}_{\ovl{\fM},\tld{w}}^{\expl,\nabla}/I\cong \ovl{R}_{\ovl{\fM}',\tld{w}'}^{\expl,\nabla}/I'$.%
\item There exists an isomorphism of \'{e}tale $\phz$-modules between the base change of $\cM_{\tld{w},\tau}$ to $\ovl{R}_{\ovl{\fM},\tld{w}}^{\expl,\nabla}/I$ and the base change of $\cM_{\tld{w}',\tau'}$ to $\ovl{R}_{\ovl{\fM}',\tld{w}'}^{\expl,\nabla}/I'$ compatible with the above ring homomorphism. 
\end{itemize}
Let $J$ \emph{(}rep. $J'$\emph{)} denote the intersection of minimal primes in $\ovl{R}^{\tau}_{\rhobar}$ \emph{(}resp. $\ovl{R}^{\tau'}_{\rhobar}$\emph{)} which correspond to $I$ \emph{(}resp. $I'$\emph{)}. Then, $J$ and $J'$ induce the same ideals in $\ovl{R}^{\Box}_{\rhobar}$.
\end{lemma}
\begin{proof} Our hypotheses implies there is an isomorphism
\begin{equation}
\xymatrix{
\Spf(\ovl{R}_{\ovl{\fM},\tld{w}}^{\expl,\nabla}/I)\ar[rd]\ar^{\sim}[rr]&
&\Spf(\ovl{R}_{\ovl{\fM}',\tld{w}'}^{\expl,\nabla}/I')\ar[dl]
\\
&\Phi\text{-}\Mod^{\text{\'et}}_{\ovl{\cM}}
&}
\end{equation}
Pulling this back along the map $\Phi\text{-}\Mod^{\text{\'et},\Box}_{\ovl{\cM}}\rightarrow \Phi\text{-}\Mod^{\text{\'et}}_{\ovl{\cM}}$ gives a commutative diagram
\begin{equation}
\xymatrix{
\Spf(\ovl{R}^{\tau,\ovl{\beta},\Box}_{\ovl{\fM},\rhobar})/I^\Box)\ar_{f.s}[d]\ar^{\sim}[rr]&
&\Spf(\ovl{R}^{\tau',\ovl{\beta}',\Box}_{\ovl{\fM}',\rhobar}/I^{\prime\,\Box})\ar_{f.s}[d]
\\
\Spf(\ovl{R}^{\tau}_{\rhobar}/J)\ar@{^{(}->}[d]\ar@{^{(}->}[rdd] & & \Spf(\ovl{R}^{\tau'}_{\rhobar}/J')\ar@{_{(}->}[d] \ar@{_{(}->}[ldd]
\\
\Spf( \ovl{R}^{\Box}_{\rhobar})\ar@{^{(}->}[rd]&&\Spf( \ovl{R}^{\Box}_{\rhobar})\ar@{_{(}->}[dl]\\
&\Phi\text{-}\Mod^{\text{\'et},\Box}_{\ovl{\cM}}
&}
\end{equation}
where we use \cite[Proposition 3.12]{LLLM} to see that the natural map $\Spf(\ovl{R}^{\Box}_{\rhobar}) \to \Phi\text{-}\Mod^{\text{\'et},\Box}_{\ovl{\cM}}$ is a monomorphism (as $\rhobar$ $10$-generic implies $\ad(\rhobar)$ is cyclotomic free).
But this implies that $\Spf(\ovl{R}^{\tau}_{\rhobar}/J)$ and $\Spf(\ovl{R}^{\tau'}_{\rhobar}/J')$ define the same subfunctor of 
$\Phi\text{-}\Mod^{\text{\'et},\Box}_{\ovl{\cM}}$, hence also the same subfunctor of $\Spf(\ovl{R}^{\Box}_{\rhobar})$.
\end{proof}
In practice, we will apply the Lemma by matching matrices of partial Frobenii: 
\begin{cor}\label{cor:matching} Keep the notations and setting of Lemma \ref{lem:matching}. Assume that there exists $\tld{z}=(\tld{z}_i)_i\in {\un{W}}^\vee_a$ with corresponding $\tld{z}^*=(\tld{z}^*_i)_i\in \un{W}_a$ and $z$ is the image of $\tld{z}$ in $\un{W}^\vee$, such that $s'=sz^*$ and $\mu'=\mu+s\tld{z}^*(0)$.
Write $(I_i)_i$ for the collection of ideals in $\ovl{R}_{\ovl{\fM},\tld{w}_i}^{\expl,\nabla}$ giving rise to $I$, and write $A^{(i)}$ for the matrix of the $i$-th partial Frobenius of $\fM_{\tld{w},\tau}$ with respect to $\beta_{\tld{w},\tau}$, and let $(I'_i)_i$ and $A'^{(i)}$ be the analogous objects for $\tau'$. Assume that:
\begin{itemize} 
\item The isomorphism $\ovl{R}_{\ovl{\fM},\tld{w}}^{\expl,\nabla}/I\cong \ovl{R}_{\ovl{\fM}',\tld{w}'}^{\expl,\nabla}/I'$ is induced by a collection of isomorphisms $\ovl{R}_{\ovl{\fM},\tld{w}_i}^{\expl,\nabla}/I_i\cong \ovl{R}_{\ovl{\fM}',\tld{w}'_i}^{\expl,\nabla}/I'_i$.
\item For each $i$, 
\[A^{(f-1-i)} \mod I_{f-1-i} = A'^{(f-1-i)}\tld{z}_{f-1-i} \mod I'_{f-1-i}\] 
via the above isomorphisms.
\end{itemize} 
Then the same conclusion as Lemma \ref{lem:matching} holds.
\end{cor}
\begin{proof} This follows from Lemma \ref{lem:matching}, Proposition \ref{prop:phif} and the fact that $\tld{z}s^*t_{\mu^*}=s'^*t_{\mu'^*}$.
\end{proof}

\begin{rmk}
\label{matching gauge} 
Note that the second condition in Corollary \ref{cor:matching} implies in particular that 
\begin{equation}
\label{eq:equal:modM}
\big(\ovl{A}^{(i)}\big)_i=\big(\ovl{A}^{\prime\,(i)} \tld{z}_i \big)_i
\end{equation} 
where $\ovl{A}^{(i)}, \ovl{A}^{\prime \,(i)}$ denote the reductions modulo the maximal ideal.  We now explain how in the situations where we apply Corollary \ref{cor:matching} we can always arrange this.  

If we have $\rhobar$ semisimple and $\tau$, $\tau', \tld{z}$ as in Proposition \ref{prop:minimal type formula}, then  
\begin{equation}
\label{eq:equal:modM0}
\tld{w}=\tld{w}' \tld{z}
\end{equation}
Furthermore, for any choices of gauge bases $\beta$, $\beta'$, for $\fM_{\tld{w},\tau}$ and $\fM_{\tld{w}',\tau'}$ respectively, one has
$\ovl{A}^{(i)}\in T(\bF) \tld{w}_i$ and $\ovl{A}^{\prime\,(i)}\in T(\bF) \tld{w}'_i$ by Proposition \ref{simpleform}. Let $\fM'_{\tld{w}, \tau}$ denote the element of $Y^{\eta, \tau}(\F)$ with eigenbasis $\beta_{\fM'}$ such that the partial Frobenii is given by $(\ovl{A}^{\prime\,(i)} \tld{z}_i)_i$. (Note that $\fM'_{\tld{w}, \tau} \in Y^{\eta, \tau}(\F)$ because it has shape $\tld{w} \in \Adm^{\vee}(\eta)$.)    Since $\tld{z}s^*t_{\mu^*}=s'^*t_{\mu'^*}$,  $(\fM'_{\tld{w}, \tau})^{\Delta = 1} \cong (\fM_{\tld{w}', \tau'})^{\Delta = 1}$ by Proposition \ref{prop:phif}. Thus,  $T^*_{\mathrm{dd}}(\fM'_{\tld{w}, \tau}) \cong \rhobar$.  By the triviality of Kisin variety (see \cite[Theorem 3.2]{LLLM}), $\fM'_{\tld{w}, \tau} \cong \fM_{\tld{w}, \tau}$ and $\beta_{\fM'}$ defines a gauge basis of $\fM_{\tld{w}, \tau}$. Thus, by replacing $\beta_{\fM}$ by $\beta_{\fM'}$ (this can be done by scaling by an element of $\un{T}(\F)$ by \cite[Proposition 3.2.22]{LLL}), we can then ensure that condition (\ref{eq:equal:modM}) holds.
In all the matching the we perform below, we will always assume that the gauge bases have been chosen so that condition (\ref{eq:equal:modM}) holds.
\end{rmk}

The following Proposition describes all the types $\tau'$ for which we need to perform some matching with $\tau$.
\begin{prop}\label{prop:minimal type formula} Assume $\rhobar$ is $10$-generic. Let $\tau=\tau(s,\mu)$ be a tame inertial type such that $R^{\tau}_{\rhobar}\neq 0$, and let $\tld{w}=\tld{w}(\rhobar,\tau)=(\tld{w}_i)_i$. Let $(\omega,a)=\big((\omega_i,a_i)\big)_i\in \prod_i \Sigma_{\tld{w}^*_i}$ and $\sigma=\sigma^{(s,\mu)}_{(\omega,a)}\in W^?(\rhobar,\tau)$. Let $\tau'$ be the minimal type of $\sigma$ with respect to $\rhobar$, cf.~Remark \ref{rmk:minimaltype}. Then:
\begin{itemize}
\item For $\tld{z}^*=(\tld{z}^*_i)_i\in \un{\tld{W}}_a$ as given by Table \ref{Table:intsct}, we have $\tau'=\tau(s',\mu')$ where
\[s'=sz^*,\quad \mu'=\mu+s\tld{z}^*(0)\] 
\item $\tld{w}(\rhobar, \tau') \tld{z} = \tld{w}$; and%
\item  $W^?(\rhobar,\tau')\subset W^?(\rhobar,\tau)$.
\end{itemize}
\end{prop}
\begin{proof} This is immediate by computing the pairwise intersections among $\Sigma_0$, $\tld{w}_i^*(r(\Sigma_0))$ and $\tld{z}_i^*(\Sigma_0)$.
\end{proof}

\begin{proof}[Proof of Theorem \ref{thm:matching}]
By Proposition \ref{prop:minimal type formula} and the proof of Proposition \ref{prop:identify:cmpt}, for each $\sigma\in W^?(\rhobar,\tau)$ with minimal type $\tau'$ with respect to $\rhobar$, we need to show that the intersection of minimal primes 
\begin{equation*}
\underset{(\omega,a)\in \Sigma \cap \tld{z}^*(\Sigma)}{\bigcap} \Bigg(\sum_{i=0}^{f-1} \fC_{(\omega_i,a_i)}\ovl{R}_{\ovl{\fM},\tld{w}}^{\expl,\nabla}\bigg)=\sum_{i=0}^{f-1}\bigg(\underset{(\omega_i,a_i)\in \Sigma_0 \cap \tld{z}_i^*(\Sigma_0)}{\bigcap}\fC_{(\omega_i,a_i)}\ovl{R}_{\ovl{\fM},\tld{w}}^{\expl,\nabla}\bigg) 
\end{equation*}
of $\ovl{R}_{\ovl{\fM},\tld{w}}^{\expl,\nabla}$ corresponds to the intersection of primes $\underset{\sigma'\in W^?(\rhobar,\tau')}{\bigcap} \fp(\sigma')$ of $\ovl{R}^{\Box}_{\rhobar}$. 
But this follows from Corollary \ref{cor:matching} (which allows us to work for each $i$ separately) and the explicit computations in Subsection \ref{subsec:explicitcomputation} below.
\end{proof}

\begin{table}[h]
\caption{}
\label{Table:intsct}
\centering
\adjustbox{max width=\textwidth}{
\begin{tabular}{| c | c | c | c | c | }
\hline
$\tld{w}_{f-1-i}t_{-\un{1}}$&$A^{(f-1-i)}$&$(\omega_i,a_i)\in \Sigma_{\tld{w}^*_i}$&$\fC_{(\omega_i, a_i)}$&$\tld{z}_i^*$\\
\hline
&&&&\\
\multirow{2}{*}{$\beta\alpha\gamma$}&\multirow{2}{*}{$\begin{matrix}\begin{pmatrix}vc_{11}&v c_{12}^*&0\\v^2c_{21}^*&v c_{22}&0\\v(c_{31}+v d_{31})&vc_{32}&c_{33}^*\end{pmatrix}\\
c_{11}c_{22}=0;\\
(-1-a+c)c_{12}^*c_{31}-(-1-b+c)c_{32}c_{11}=0
\end{matrix}$}&$(\eps_1+\eps_2,0)$&$(c_{11})$ &$\beta\gamma^+\beta$ \\
&&&&\\
\cline{3-5}&&&&\\
&&$(\eps_2,1)$&$(c_{22})$&$\alpha$ 
\\
&&&&\\
\hline
\hline
&&&&\\
\multirow{2}{*}{$\alpha\beta\gamma$}&\multirow{2}{*}{$\begin{matrix}\begin{pmatrix}v^2c_{11}&0&0\\v(c_{21}+vd_{21})&c_{22}&c_{23}^*\\v(c_{21}c_{33}(c_{23}^*)^{-1}+v d_{31})&vc^*_{32}&c_{33}\end{pmatrix}\\
c_{22}c_{33}=0;\\
(-1-a+c)c_{21}c^*_{32}+(b-c)d_{31}c_{22}=0
\end{matrix}$}&$(\eps_1+\eps_2,0)$&$(c_{22})$ &$\alpha\gamma^+\alpha$ 
\\
&&&&\\
\cline{3-5}&&&&\\
&&
$(\eps_1,1)$&$(c_{33})$&$\beta$ \\
&&&&\\
\hline
\hline
&&&&\\
\multirow{2}{*}{$\alpha\beta\alpha$}&\multirow{2}{*}{$
\begin{matrix}
\begin{pmatrix}
c_{11}&c_{11}c_{32}(c^{*}_{31})^{-1}&d_{33}c_{11}(c^{*}_{31})^{-1}+vc_{13}^{*}\\
0&v c^{*}_{22}& vc_{23} \\
v c^{*}_{31}&v c_{32}&v d_{33}
\end{pmatrix}\\
\\
c_{11}((a-b)c_{23}c_{32}-(a-c)c^{*}_{33}d_{33})=0
\end{matrix}
$}&$(0,0)$&$(c_{11})$ &$\gamma^+$ 
\\
&&&&\\
\cline{3-5}&&&&\\
&&
$(0,1)$&$((a-b)c_{23}c_{32}-(a-c)c^{*}_{33}d_{33})$&$\id$ \\
&&&&\\
\hline
\hline
&&&&\\
\multirow{4}{*}{$\alpha\beta$}&\multirow{4}{*}{$\begin{matrix}\begin{pmatrix}c_{13}c_{12}(c_{32}^*)^{-1}&c_{12}&c_{13}+v c_{13}^*\\v c_{21}^*&c_{22}&c_{23}+vd_{23}\\v c_{31}&vc^*_{32}&c_{31}c_{23}(c_{21}^*)^{-1}+v d_{33}\end{pmatrix}\\
\\
c_{12}c_{23}-c_{22}c_{13}=0;\\
c_{22}c_{31}=0;\\
c_{32}^*c_{13}-d_{33}c_{12}=0;\\
c_{12}((a-b)c_{31}d_{23}-(b-c)d_{33}c_{21}^*)=0;\\
(-1-a+c)c_{23}c_{32}^*=(-1-a+b)c_{22}d_{33}
\end{matrix}$}&$(\eps_1,1)$&$(c_{12},c_{31})$ &$\gamma^+\beta$ \\&&&&\\
\cline{3-5}&&&&
\\
&&$(\eps_1-\eps_2,0)$&$(c_{31},d_{33})$&$\beta\gamma^+\alpha\beta$ \\
&&&&\\
\cline{3-5}&&&&
\\
&&$(0,0)$&$(c_{12},c_{22})$&$\alpha\gamma^+$ \\&&&&\\
\cline{3-5}&&&&
\\
&&$(0,1)$&$(c_{22},(b-c)d_{33}c_{21}^*+(a-b)c_{31}d_{23})$&$\alpha$ \\
&&&&\\
\hline
\hline
&&&&\\
\multirow{4}{*}{$\beta\alpha$}&\multirow{4}{*}{$\begin{matrix}\begin{pmatrix}c_{11}&(c_{31}^*)^{-1}c_{11}c_{32}+v c_{12}^*&c_{13}\\0&vd_{22}&vc_{23}^*\\v c^*_{31}&vc_{32}&c_{33}+v d_{33}\end{pmatrix}\\
\\
c_{11}c_{33}=0;\\
d_{22}(c_{13}c_{31}^*-c_{11}d_{33})=0;\\
c_{11}((a-b)c_{32}c^*_{23}-(a-c)d_{22}d_{33})=0;\\
(1+a-c)c_{33}c_{23}^*c_{12}^*=c_{13}((a-b)c_{32}c^*_{23}-(a-c)d_{22}d_{33})
\end{matrix}$}&$(\eps_2,1)$&$(d_{22},c_{11})$ &$\gamma^+\alpha$ \\
&&&&\\
\cline{3-5}&&&&
\\
&&$(\eps_2-\eps_1,0)$&$(d_{22},c_{32})$&$\alpha\gamma^+\beta\alpha$ \\
&&&&\\
\cline{3-5}&&&&
\\
&&$(0,0)$&$(c_{11},c_{13})$&$\beta\gamma^+$ \\
&&&&\\
\cline{3-5}&&&&
\\
&&$(0,1)$&$((a-b)c_{32}c^*_{23}-(a-c)d_{22}d_{33},c_{13}c_{31}^*-c_{11}d_{33})$&$\beta$ \\
&&&&\\
\hline
\hline
\multirow{6}{*}{$\alpha$}&\multirow{6}{*}{$\begin{matrix}\begin{pmatrix}
c_{11}&c_{12}+vc_{12}^*& c_{13}\\
vc_{21}^*&c_{22}+vd_{22}&c_{23}\\
vc_{31}&vc_{32}&(c_{21}^*)^{-1}c_{31}c_{23}+vc_{33}^*
\end{pmatrix}\\
\\
\text{see \S \ref{subsec:explicitcomputation} for the relations among coefficients}\end{matrix}$}&$(\eps_1,1)$&$(c_{11},c_{13},c_{31})$&$\beta\gamma^+\beta$ \\
\cline{3-5}&&$(\eps_2,0)$&$(c_{11},c_{31},c_{32}c_{21}^*-d_{22}c_{31})$
&$\gamma^+\alpha\beta$\\
\cline{3-5}&&$(\eps_2,1)$&$ (c_{11},c_{32}c_{21}^*-d_{22}c_{31}, (a-b)c_{13}d_{22}+(-1-a+c)c_{23}c^*_{12})$
&$\gamma^+\alpha$\\
\cline{3-5}&&$(\eps_2-\eps_1,0)$&$(c_{23},d_{22},c_{32}c_{21}^*-d_{22}c_{31})$
&$\alpha\beta\gamma^+\beta\alpha$\\
\cline{3-5}&&$(0,0)$&$(c_{11},c_{13},c_{23})$
&$\beta\alpha\gamma^+$\\
\cline{3-5}&&$(0,1)$&$ (c_{11}c_{33}^*-c_{13}c_{31},c_{23}, (a-b)c_{31}d_{22}+(c-b)(c_{32}c_{21}^*-d_{22}c_{31}))$
&$\beta\alpha$\\
\hline
\hline
\multirow{6}{*}{$\id$}&\multirow{6}{*}{$
\begin{matrix}
\begin{pmatrix}
c_{11}+vc_{11}^*&c_{12}& c_{13}\\
vc_{21}&c_{22}+vc_{22}^*&c_{23}\\
vc_{31}&vc_{32}&c_{33}+vc_{33}^*
\end{pmatrix}
\\
\text{see \S \ref{subsec:explicitcomputation} for the relations among coefficients}\end{matrix}
$}&$(\eps_1,0)$&$(c_{ii}, \text{\tiny{$i=1,2,3$}}, \ c_{21},c_{31},c_{23})$&
$\beta\gamma^+\beta\alpha$\\
\cline{3-5}&&$(\eps_1,1)$&$
(c_{31}, c_{33}, c_{11}, (-1-a+c)c_{32}c_{13} - (-1-a+b)c_{12}c^*_{33}, c_{21}c_{13} - c_{23}c^*_{11})
$&
$\beta\gamma^+\beta$ \\
\cline{3-5}&&$(\eps_2,0)$&$(c_{ii}, \text{\tiny{$i=1,2,3$}},\ c_{12},c_{31},c_{32})$&
$\alpha\gamma^+\alpha\beta$ \\
\cline{3-5}&&$(\eps_2,1)$&$(c_{12}, c_{22}, c_{11}, (a-b)c_{21}c_{13} -(-1-b+c)c_{23}c^*_{11}, c_{21}c_{32} - c_{31}c^*_{22})
$&
$\alpha\gamma^+\alpha$ \\
\cline{3-5}&&$(0,0)$&$(c_{ii}, \text{\tiny{$i=1,2,3$}},\ c_{13},c_{23},c_{12})$&
$\alpha\beta\alpha\gamma^+$\\
\cline{3-5}&&$(0,1)$&$
(c_{23}, c_{33}, c_{22}, (a-b)c_{21}c_{32} -(a-c)c_{31}c^*_{22}, c_{32}c_{13} - c_{12}c^*_{33})$&
$\alpha\beta\alpha$\\
\hline
\end{tabular}}
\caption*{The table records data relevant to Theorem \ref{thm:matching}. %
The first column records the components of the shape $\tld{w}=\tld{w}(\rhobar,\tau)$. The second column records the form of the matrix of partial Frobenius $A^{(f-1-i)}$, and the presentation of the ring $\ovl{R}_{\ovl{\fM},\tld{w}_{f-1-i}}^{\expl,\nabla}$ in terms of the entries of $A^{(f-1-i)}$. The fourth column records the prime ideal $\fC_{(\omega_i,a_i)}$ in the statement of Theorem \ref{thm:matching}. The last column records the element $\tld{z}^*\in \un{W}_a$ that occurs in Proposition \ref{prop:minimal type formula}, which controls the minimal type $\tau'$ of the weight given by the third column with respect to $\rhobar$. Thus $\tau'=\tau(s z^*,\mu+s \tld{z}^*(0))$.
Note that $\tld{w}(\rhobar,\tau')=\tld{w}(\rhobar,\tau)\tld{z}^{-1}$. Finally, the structure constants that feature in the presentation of $\ovl{R}_{\ovl{\fM},\tld{w}_{f-1-i}}^{\expl,\nabla}$ are given by $(a,b,c)\in\Fp^3$ with $(a,b,c)\equiv s_i^{-1}(\mu_i)\mod p$.
}
\end{table}
\begin{table}[h]
\centering
\adjustbox{max width=\textwidth}{
\begin{tabular}{| c | c | c | c | c |}
\hline
$\tld{w}_{f-1-i}t_{-\un{1}}$&$A^{(f-1-i)}$&$(\omega_i,a_i)\in \Sigma_{\tld{w}^*_{i}}$&$\fW_{\omega_i}\defeq\fC_{(\omega_i,0)}\cap \fC_{(\omega_i,1)}$ &$\tld{z}_i^*$\\
\hline
&&&&\\
{$\alpha\beta$}&{$
\begin{matrix}\begin{pmatrix}c_{13}c_{12}(c_{32}^*)^{-1}&c_{12}&c_{13}+v c_{13}^*\\v c_{21}^*&c_{22}&c_{23}+vd_{23}\\v c_{31}&vc^*_{32}&c_{31}c_{23}(c_{21}^*)^{-1}+v d_{33}\end{pmatrix}\\
\\
c_{12}c_{23}-c_{22}c_{13}=0;\\
c_{22}c_{31}=0;\\
c_{32}^*c_{13}-d_{33}c_{12}=0;\\
c_{12}((a-b)c_{31}d_{23}-(b-c)d_{33}c_{21}^*)=0;\\
(-1-a+c)c_{23}c_{32}^*=(-1-a+b)c_{22}d_{33}
\end{matrix}$
}&$(0,1),\ (0,0)$&
$(c_{22})$&$\alpha$\\
&&&&\\
\hline
\hline
&&&&\\
{$\beta\alpha$}&{$\begin{matrix}\begin{pmatrix}c_{11}&(c_{31}^*)^{-1}c_{11}c_{32}+v c_{12}^*&c_{13}\\0&vd_{22}&vc_{23}^*\\v c^*_{31}&vc_{32}&c_{33}+v d_{33}\end{pmatrix}\\
\\
c_{11}c_{32}=0;\\
d_{22}(c_{13}c_{31}^*-c_{11}d_{33})=0;\\
c_{11}((a-b)c_{32}c^*_{23}-(a-c)d_{22}d_{33})=0;\\
(1+a-c)c_{33}c_{23}^*c_{12}^*=c_{13}((a-b)c_{32}c^*_{23}-(a-c)d_{22}d_{33})
\end{matrix}$}&$(0,1),\ (0,0)$&
$(c_{13}c_{31}^*-c_{11}d_{33})$&$\beta$\\
&&&&\\
\hline
\hline
&&&&\\
\multirow{2}{*}{$\alpha$}&\multirow{2}{*}{$\begin{matrix}
\begin{pmatrix}
c_{11}&c_{12}+vc_{12}^*& c_{13}\\
vc_{21}^*&c_{22}+vd_{22}&c_{23}\\
vc_{31}&vc_{32}&(c_{21}^*)^{-1}c_{31}c_{23}+vc_{33}^*
\end{pmatrix}
\\
\text{see \S \ref{subsec:explicitcomputation} for the relations among coefficients}\end{matrix}$}&$(\eps_2,1),\ (\eps_2,0)$&
$(c_{32}c_{21}^*-d_{22}c_{31},c_{11})$&$\gamma^+\alpha$\\
&&&&\\
\cline{3-5}&&&&\\
&&$(0,1),\ (0,0)$&
$(c_{11}c_{33}^*-c_{13}c_{31},\ c_{23})$&$\beta\alpha$\\
&&&&\\
\hline
\hline
&&&&\\
\multirow{3}{*}{$\id$}&\multirow{3}{*}{$\begin{matrix}
\begin{pmatrix}
c_{11}+vc_{11}^*&c_{12}& c_{13}\\
vc_{21}&c_{22}+vc_{22}^*&c_{23}\\
vc_{31}&vc_{32}&c_{33}+vc_{33}^*
\end{pmatrix}
\\
\text{see \S \ref{subsec:explicitcomputation} for the relations among coefficients}\end{matrix}$}&$(\eps_1,1),\ (\eps_1,0)$&
$(c_{33},\ c_{11},\  c_{31},\ c_{11}^*c_{23}-c_{21}c_{13})$&$\beta\gamma^+\beta$\\
&&&&\\
\cline{3-5}&&&&\\
&&$(\eps_2,1),\ (\eps_2,0)$&
$(c_{11},\ c_{22},\  c_{12},\ c_{22}^*c_{31}-c_{32}c_{21})$&$\alpha\gamma^+\alpha$\\
&&&&\\
\cline{3-5}&&&&\\
&&$(0,1),\ (0,0)$&
$(c_{22},\ c_{33},\  c_{23},\ c_{33}^*c_{12}-c_{13}c_{32})$&$\alpha\beta\alpha$\\
&&&&\\
\hline
\end{tabular}}
\caption*{The table records further data relevant to Theorem \ref{thm:matching}.
The meaning of the columns is the same as in Table \ref{Table:intsct}.}
\end{table}
We record the following lemma, which follows from Theorem \ref{thm:matching}, for future use.
Let $\tld{w} = \tld{w}(\rhobar,\tau)$.
Suppose that $\ell(\widetilde{w}_i) \geq 2$ for all $0\leq i\leq f-1$. 
Let $N \defeq \sum_i (4-\ell(\widetilde{w}_i))$, and $R_N \defeq \widehat{\otimes}_{j=1}^N\cO[\![x_j,y_j]\!]/(x_jy_j-p)$.
 In this case, there exists a formally smooth map $R_N \ra {R}^{\expl,\nabla}_{\overline{\fM},\tld{w}}$, which we fix.
We think of $R_N$ as a subalgebra of ${R}^{\expl,\nabla}_{\overline{\fM},\tld{w}}$.
For $\sigma \in W^?(\rhobar, \tau)$, let $\mathfrak{p}^{\expl}(\sigma) \subset {R}^{\expl,\nabla}_{\overline{\fM},\tld{w}}$ be the prime ideal corresponding to the weight $\sigma$.
The minimal prime ideals of ${R}^{\expl,\nabla}_{\overline{\fM},\tld{w}}$ containing $\varpi$ are of the form $((z_j)_{j=1}^N+(\varpi)){R}^{\expl,\nabla}_{\overline{\fM},\tld{w}}$ where $z_j \in \{x_j,y_j\}$ for all $1\leq j \leq N$.
Define $z_j(\sigma)$ so that $\mathfrak{p}^{\expl}(\sigma) = ((z_j(\sigma))_{j=1}^N+(\varpi)){R}^{\expl,\nabla}_{\overline{\fM},\tld{w}}$.

\begin{lemma}\label{lemma:expl4wt}
Assume that $\rhobar$ is $10$-generic and that $\ell(\widetilde{w}_i) \geq 2$ for all $0\leq i\leq f-1$.
For $\sigma_1$ and $\sigma_2 \in W^?(\rhobar,\tau)$,
\[\#\big(\{z_j(\sigma_1)\}_{j=1}^N \Delta \{z_j(\sigma_2)\}_{j=1}^N\big) = 2\dgr{\sigma_1}{\sigma_2}\]
where $\Delta$ denotes the symmetric difference.
\end{lemma}
\begin{proof}
From the $\alpha\beta\alpha$, $\alpha\beta$ and $\beta\alpha$ row of Table \ref{Table:intsct}, we check that if $\sigma_1,\ \sigma_2 \in W^?(\rhobar,\tau)$, then $\mathfrak{p}^{\expl}(\sigma_1) + \mathfrak{p}^{\expl}(\sigma_2)$ is a prime ideal of ${R}^{\expl,\nabla}_{\overline{\fM},\tld{w}}$ of height
\begin{equation*}
\label{height}
\dgr{\sigma_1}{\sigma_2}+1.
\end{equation*}
On the other hand, the height of the intersection of $\big((z_j)_j+(\varpi)\big) + \big((z_j')_j+(\varpi)\big)$ is easily seen to be 
\[\frac{1}{2}\#\big(\{z_j\}_{j=1}^N \Delta \{z_j'\}_{j=1}^N\big)+1.\]
\end{proof}

For convenience, we list in the Table \ref{table:lifts} the length 4 shapes and their universal families over $\overline{R}^{\expl,\nabla}_{\overline{\fM},\tld{w}_j}$.%

\begin{table}[h]
\caption{}
\label{table:lifts}
\centering
\adjustbox{max width=\textwidth}{
\begin{tabular}{| c | c | | c | c |}
\hline
& &&\\
$\widetilde{w}'_{f-1-i}t_{-\un{1}}$ & ${A}^{\prime\, (f-1-i)}$ & $\widetilde{w}'_{f-1-i}t_{-\un{1}}$ & ${A}^{\prime\, (f-1-i)}$ \\ 
 &&&\\
\hline
 &&&\\
$\alpha \beta \alpha \gamma$ &   $\begin{pmatrix} v^2 c_{11}^{\prime\,*} & 0 & 0 \\ v^2 c'_{21} & v c_{22}^{\prime\,*} & 0 \\ v^2 d'_{31} & v c'_{32} & c_{33}^{\prime\,*}\end{pmatrix}$
&
$\beta \gamma \alpha \gamma$ & $\begin{matrix}  \begin{pmatrix} v  c_{11}^{\prime\,*} & v c'_{12} & 0 \\ 0 & v^2 c_{22}^{\prime\,*} & 0 \\ v c'_{31} & v c'_{32}+v^2 d'_{32}& c_{33}^{\prime\,*}\end{pmatrix}
\\
\\
(-1-b'+c')c'_{32} c_{11}^{\prime\,*}-(-1-a'+c')c'_{12}c'_{31}=0
\end{matrix}
$
\\
 &&&\\
\hline
 &&&\\
$ \beta\gamma\beta \alpha$ &   $\begin{pmatrix} c_{11}^{\prime\,*} & v d'_{12} & c'_{13} \\ 0 & v^2 c_{22}^{\prime\,*} & 0 \\ 0 & v^2 c'_{32} & v c_{33}^{\prime\,*}\end{pmatrix}$&
$\gamma\alpha\beta\alpha$
&
$\begin{matrix}  \begin{pmatrix} c_{11}^{\prime\,*} &  c'_{12} & c'_{13}+vd_{13}' \\ 0 & v c_{22}^{\prime\,*} & v c'_{23} \\ 0& 0& v^2 c_{33}^{\prime\,*}\end{pmatrix}
\\
\\
(a'-c')c'_{13} c_{22}^{\prime\,*}-(a'-b')c'_{23}c'_{12}=0
\end{matrix}
$
\\
&&&\\
\hline
&&&\\
$\alpha \gamma \alpha \beta$ &   $\begin{pmatrix} v c_{11}^{\prime\,*} & 0 & v c'_{13} \\ v c'_{12} & c_{22}^{\prime\,*} &v d_{23}' \\  0 & 0 & v^2 c_{33}^{\prime\,*}\end{pmatrix}$&
$\alpha\beta\gamma\beta$
&
$\begin{matrix}  \begin{pmatrix} v^2 c_{11}^{\prime\,*} &  0 & 0 \\ v(c'_{21}+v d_{21}') & c_{22}^{\prime\,*} &  c'_{23} \\ v^2 c'_{31}& 0& v c_{33}^{\prime\,*}\end{pmatrix}
\\
\\
(-1-a'+b')c'_{21} c_{33}^{\prime\,*}-(b'-c')c'_{31}c'_{23}=0
\end{matrix}
$
\\
&&&\\
\hline
&&&\\
$\alpha \beta\alpha$ &   $\begin{matrix}
\begin{pmatrix} c'_{11} & c'_{11}c'_{32}(c_{31}^{\prime\,*})^{-1} & c'_{11}d'_{33}(c_{31}^{\prime\,*})^{-1}+v c_{13}^{\prime\,*} \\ 0 & vc_{22}^{\prime\,*} &v c_{23}' \\  v c_{31}^{\prime\,*} & vc_{32}' & v d'_{33}\end{pmatrix}
\\
\\
c'_{11}\big((a'-b')c_{23}'c'_{32}-(a'-c')c_{22}^{\prime\,*}d'_{33}\big)=0
\end{matrix}$&
$\beta\gamma\beta$
&
$\begin{matrix}  \begin{pmatrix} v d'_{11} &  c_{12}^{\prime\,*} & c_{13}' \\ v(c'_{22}d'_{11}(c_{12}^{\prime\, *})^{-1}+v c_{21}^{\prime\,*}) & c'_{22} &  c'_{22}c'_{13}(c_{12}^{\prime\,*})^{-1} \\ v^2 c'_{31}& 0& v c_{33}^{\prime\,*}\end{pmatrix}
\\
\\
c'_{22}\big((b'-c')c_{31}'c'_{13}-(-1-a'+b')c_{33}^{\prime\,*}d'_{11}\big)=0
\end{matrix}
$
\\
&&&\\
\hline
 &&&\\
 $\gamma\alpha \gamma$ &   $\begin{matrix}
\begin{pmatrix} v c_{11}^{\prime\,*} & vc'_{12} & 0\\ v c'_{21} & v d'_{22} & c_{23}^{\prime\,*} \\  v c'_{33}c'_{21}(c_{23}^{\prime\,*})^{-1} & v(c'_{33}d'_{22}(c_{23}^{\prime\, *})^{-1}+v c_{32}^{\prime\,*}) &  c'_{33}\end{pmatrix}
\\
\\
c'_{33}\big((-1-a'+c')c_{12}'c'_{21}-(1-b'+c')c_{11}^{\prime\,*}d'_{22}\big)=0
\end{matrix}$&
&
\\
 &&&\\
\hline
\hline
\end{tabular}}
\caption*{This Table records the relevant data for the type $\tau'$ occurring in the proof of Theorem \ref{thm:matching}. The first column records the components of the shape $\tld{w}'=\tld{w}(\rhobar,\tau')$. The second column records the form of the matrix of partial Frobenius $A^{\prime\,(f-1-i)}$, and the presentation of the ring $\ovl{R}_{\ovl{\fM}',\tld{w}'_{f-1-i}}^{\expl,\nabla}$ in terms of the entries of $A^{\prime\,(f-1-i)}$. The structure constants that feature in the presentation of $\ovl{R}_{\ovl{\fM}',\tld{w}'_{f-1-i}}^{\expl,\nabla}$ are given by $(a',b',c')\in\Fp^3$ where $(a',b',c')\equiv (s'_i)^{-1}(\mu'_i)\mod p$. Note that $(s'_i)^{-1}(\mu'_i)\equiv\tld{z}_{f-1-i}(a,b,c) \mod p$.}%
\end{table}
\subsubsection{Explicit computations}\label{subsec:explicitcomputation}%
In this section, we record the matching of matrices of partial Frobenii needed in the proof of Theorem \ref{thm:matching}. We are always in the setting of the theorem. In particular, we have two types $\tau$, $\tau'$ with chosen presentations related by the element $\tld{z}^*$ together with matrices of partial Frobenii $A^{(f-1-i)}$, $A'^{(f-1-i)}$. We will fix $i$ throughout.

We will frequently recall presentations of rings from \cite{LLLM}, and since we work in characteristic $p$, all occurrences of the symbol $e$ in \emph{loc.~cit.}~will become $-1$ here. We let $(a,b,c)$ and $(a',b',c')\in \Fp^3$ be such that $(a,b,c)=s_{i}^{-1}(\mu_{i}) \mod p$, $(a',b',c')=(s^{\prime}_{i})^{-1}(\mu^{\prime}_{i})\mod p$.
Note that $(a',b',c')=\tld{z}_{f-1-i}(a,b,c)$.
These are the structure constants that feature in the presentation of our explicit rings. %
We will also replace occurrences of the symbols $c'_{ij}$ in \emph{loc.~cit.}~by $d_{ij}$, as we wish to decorate objects associated to $\tau'$ with a prime superscript.

\vspace{2mm}

\paragraph{\emph{Case $\tld{w}_{f-1-i}=\alpha \beta t_{\un{1}}$.}}
From \cite[Table 5]{LLLM}, the matrix of the partial Frobenius $A^{(f-1-i)}$ has the form
\begin{align*}
\begin{pmatrix}
c_{31}c_{12}(c_{32}^*)^{-1}&c_{12}&c_{13}+v c_{13}^*\\
vc_{21}^*&c_{22}&c_{23}+v d_{23}\\
vc_{31}&vc_{32}^*&\left(c_{31}c_{23}(c_{21}^*)^{-1}+vd_{33}\right)
\end{pmatrix}
\end{align*}
The ring $\ovl{R}_{\ovl{\fM},\tld{w}_{f-1-i}}^{\expl,\nabla}$ is the quotient of $\F[\![c_{12},c_{13},c_{22}, c_{23}, d_{23}, c_{31}, d_{33}, c_{13}^*-[\ovl{c}_{13}^*], c_{21}^*-[\ovl{c}_{21}^*],c_{32}^*-[\ovl{c}_{32}^*]]\!]$ by the relations
\begin{align}
\label{eq:list:rel}
c_{12}c_{23}-c_{22}c_{13}&=0;\\
c_{22}c_{31}&=0;&\nonumber\\
c_{32}^*c_{13}-d_{33}c_{12}&=0;&\nonumber\\
c_{12}((b-c)d_{33}c_{21}^*+(a-b)c_{31}d_{23})&=0;\nonumber\\
(-1-a+c)c_{23}c^*_{32}&=(-1-a+b)c_{22}d_{33}.\nonumber
&%
\end{align}
Note that all equations except for the last one are mod $\varpi$ reductions of equations in \cite[\S 5.3.3]{LLLM}.
We explain how to justify the last equation from the computations in \cite[\S 5.3.3]{LLLM}, which was implicit in \cite[Table 7]{LLLM}. For the remainder of this paragraph, we adopt the notation of \emph{loc.~cit.} 
We have $c_{12}((a-b)c_{31}c'_{23}+(b-c)c^*_{21}c'_{33})=pz^*$. 
From the equation $c_{32}^*c_{13}-pc^*_{32}c^*_{13}-c'_{33}c_{12}=0$, replacing $p$ by $(z^*)^{-1}(c_{12}((a-b)c_{31}c'_{23}+(b-c)c^*_{21}c'_{33}))$ we deduce that $c_{13}$ is a certain multiple of $c_{12}$. Finally using the equation $c_{12}c_{23}=c_{22}c_{13}$ and canceling out $c_{12}$ (as it is a unit after inverting $p$), we get an equation solving $c_{23}$ in terms of the remaining variables. The mod $\varpi$ reduction of this equation gives the last equation above.

The minimal primes of $\overline{R}^{\expl,\nabla}_{\overline{\fM},\tld{w}_{f-1-i}}$ are 
\begin{align*}
\fC_{(0,0)}=(c_{12},\ c_{22});&& \fC_{(\eps_1,1)}=(c_{12},\ c_{31});&& \fC_{(\eps_1-\eps_2,0)}=(c_{31},\ d_{33}); && \fC_{(0,1)}=(c_{22},\ (b-c)d_{33}c_{21}^*+(a-b)c_{31}d_{23}).
\end{align*}
(Note that in $\overline{R}^{\expl,\nabla}_{\overline{\fM},\tld{w}_{f-1-i}}$ there are no more relation than those listed in (\ref{eq:list:rel}).
Indeed, let $\tld{R}$ be the quotient of $\F[\![c_{12},c_{13},c_{22}, c_{23}, d_{23}, c_{31}, d_{33}, c_{13}^*-[\ovl{c}_{13}^*], c_{21}^*-[\ovl{c}_{21}^*],c_{32}^*-[\ovl{c}_{32}^*]]\!]$ by the relations (\ref{eq:list:rel}).
Then the discussion above proves that there is a surjection $\tld{R}\onto\overline{R}^{\expl,\nabla}_{\overline{\fM},\tld{w}_{f-1-i}}$.
A direct check on the relations (\ref{eq:list:rel}) shows that $\tld{R}$ is reduced, equidimensional of the same dimension as $\overline{R}^{\expl,\nabla}_{\overline{\fM},\tld{w}_{f-1-i}}$, and with the same Hilbert--Samuel multiplicity.
Therefore the surjection is an isomorphism, see Lemma \ref{lem:iso:ring} below.)
We need to perform matching for the ideals $\fC_{(0,0)}, \fC_{(\eps_1,1)},\fC_{(\eps_1-\eps_2,0)}, \fW_{(0)}\defeq \fC_{(0,1)}\cap \fC_{(0,0)}$.
We provide details for the matching of the ideal $
\fC_{(0,0)}$ in Table \ref{Table:intsct}.
We have
\begin{align*}
\overline{R}^{\expl,\nabla}_{\ovl{\fM},\tld{w}_{f-1}}\big/
\fC_{(0,0)}=\F[\![d_{23}, c_{31}, d_{33}, c_{13}^*-[\ovl{c}_{13}^*], c_{21}^*-[\ovl{c}_{21}^*],c_{32}^*-[\ovl{c}_{32}^*]]\!]
\end{align*}
and
\begin{align*}
A^{(f-1-i)}\mod  \fC_{(0,0)}=\begin{pmatrix}
0&0&v c_{13}^*\\
v c_{21}^*&0&v d_{23}\\
v c_{31}&v c_{32}^*&v d_{33}
\end{pmatrix}
\end{align*}
On the other hand, $\tld{w}'_{f-1-i}=\alpha\beta\alpha\gamma t_{\un{1}}$,  $\tld{z}^*_i=\alpha\gamma^+$, $(a',b',c')=(c-1,a,b+1)$ and $A^{\prime\,(f-1-i)}$ has the form
\begin{align*}
\begin{pmatrix} v^2 c_{11}^{\prime\,*} & 0 & 0 \\ v^2 c'_{21} & v c_{22}^{\prime\,*} & 0 \\ v^2 d'_{31} & v c'_{32} & c_{33}^{\prime\,*}
\end{pmatrix}.
\end{align*}
Now we note that $\tld{z}_{f-1-i}=t_{(-1,0,1)}(123)$ and 
\begin{align*}
\begin{pmatrix}
0&0&v c_{13}^*\\
v c_{21}^*&0&v d_{23}\\
v c_{31}&v c_{32}^*&v d_{33}
\end{pmatrix}= \begin{pmatrix} v^2 c_{11}^{\prime\,*} & 0 & 0 \\ v^2 c'_{21} & v c_{22}^{\prime\,*} & 0 \\ v^2 d'_{31} & v c'_{32} & c_{33}^{\prime\,*} \end{pmatrix} 
\begin{pmatrix} 0 & 0 & v^{-1} \\ 1 &0 & 0 \\ 0& v &0
\end{pmatrix}
\end{align*}
under the isomorphism
\begin{align*}\overline{R}^{\expl,\nabla}_{\ovl{\fM},\tld{w}_{f-1-i}}\big/
\fC_{(0,0)}\cong \ovl{R}^{\expl,\nabla}_{\ovl{\fM}',\tld{w}'_{f-1-i}}
\end{align*}
given by the change of variable
\begin{align*}
c_{21}^*&=c_{22}^{\prime\,*},& c_{13}^*&=c_{11}^{\prime\,*},&d_{23}&=c'_{21},
\\
c_{31}&=c'_{32},&c^*_{32}&=c_{33}^{\prime\,*},& d_{33}&=d'_{31}.\end{align*}
Such a change of variable is allowed provided the units that are matched with each other agree modulo the maximal ideal, because $\ovl{c}^{\prime\,\ast}_{22}=\ovl{c}^*_{21}$, $\ovl{c}^{\prime\,\ast}_{11}=\ovl{c}^*_{13}$ and $\ovl{c}^{\prime\,\ast}_{33}= \ovl{c}^*_{32}$ as elements of $\F$ (see Remark \ref{matching gauge}).
Thus \ref{cor:matching} applies and we are done with this case.
The matching for $\fC_{(\eps_1-\eps_2,0)}$ can performed in a similar fashion.

Next we provide details for the matching of $\fC_{(\eps_1,1)}$.
We have
\[
\overline{R}^{\expl,\nabla}_{\ovl{\fM},\tld{w}_{f-1-i}}\big/
\fC_{(\eps_1,1)}=\F[\![c_{22},c_{23},d_{23}, d_{33}, c_{13}^*-[\ovl{c}_{13}^*], c_{21}^*-[\ovl{c}_{21}^*],c_{32}^*-[\ovl{c}_{32}^*]]\!]\big/\big(  (-1-a+c)c_{23}c^*_{32}-(-1-a+b)c_{22}d_{33}
 \big).
\]
Moreover
\begin{align*}
A^{(f-1-i)}\mod  \fC_{(\eps_1,1)}=\begin{pmatrix}
0&0&v c_{13}^*\\
v c_{21}^*&c_{22}&c_{23}+v d_{23}\\
0&v c_{32}^*&v d_{33}
\end{pmatrix}.
\end{align*}
On the other hand, $\tld{w}'_{f-1-i}=\alpha\beta\gamma\beta t_{\un{1}}$, $\tld{z}^*_i=\gamma^+\beta$, $(a',b',c')=(c-1,a+1,b)$ and $A^{\prime\,(f-1-i)}$ has the form
\begin{align*}
A^{\prime\,(f-1-i)}=\begin{pmatrix}
v^2c^{\prime\,*}_{11}&0&0\\
vc^{\prime}_{21}+v^2d^{\prime}_{21} &c^{\prime\,*}_{22}& c'_{23} \\
v^2 c^{\prime}_{31}&0&vc^{\prime\,*}_{33}
\end{pmatrix}.
\end{align*}
The ring $\overline{R}^{\expl,\nabla}_{\ovl{\fM}',\tld{w}'_{f-1-i}}$ is the quotient of $\F[\![ c'_{21}, d'_{21},c'_{23}, c'_{31}, c_{11}^{\prime\,*}-[\ovl{c}_{11}^{\prime\,*}], c_{22}^{\prime\,*}-[\ovl{c}_{22}^{\prime\,*}],c_{33}^{\prime\,*}-[\ovl{c}_{33}^{\prime\,*}]]\!]$ by the relation
\begin{equation*}
(-1-a'+b')c'_{21}c^{\prime\,*}_{33}-(b'-c')c'_{31}c'_{23}=0.
\end{equation*}

Now we note that %
\begin{equation*}
\begin{pmatrix}
0&0&v c_{13}^*\\
v c_{21}^*&c_{22}&c_{23}+v d_{23}\\
0 &v c_{32}^*&v d_{33}
\end{pmatrix}=\begin{pmatrix}
v^2c^{\prime\,*}_{11}&0&0\\
vc^{\prime}_{21}+v^2d^{\prime}_{21} &c^{\prime\,*}_{22}& c'_{23} \\
v^2 c^{\prime}_{31}&0&vc^{\prime\,*}_{33}
\end{pmatrix}
\begin{pmatrix}
0&0&v^{-1}\\
v&0&0\\
0&1&0
\end{pmatrix}
\end{equation*}
under the isomorphism 
\begin{align*}
\label{iso:ring:AB}
\overline{R}^{\expl,\nabla}_{\ovl{\fM},\tld{w}_{f-1-i}}\big/
\fC_{(\eps_1,1)}\cong \ovl{R}^{\expl,\nabla}_{\ovl{\fM}',\tld{w}'_{f-1-i}}
\end{align*}
given by the change of variable 
\begin{align*}
c_{22}&=c'_{23},&
c_{23}&=c'_{21},&
d_{23}&=d'_{21},&\\
d_{33}&=c'_{31},&
c_{13}^*&=c^{\prime\,*}_{11},&
c_{21}^*&=c^{\prime\,*}_{22},&
c_{32}^*&=c^{\prime\,*}_{33}.
\end{align*}
We finish the matching in this case.

Finally, we explain the matching of the ideal $\fW_{0}\defeq\fC_{(0,1)}\cap \fC_{(0,0)}=(c_{22})$.
We have
\[
\overline{R}^{\expl,\nabla}_{\ovl{\fM},\tld{w}_{f-1-i}}\big/
\fW_{0}=\F[\![c_{12},d_{23}, c_{31}, d_{33}, c_{13}^*-[\ovl{c}_{13}^*], c_{21}^*-[\ovl{c}_{21}^*],c_{32}^*-[\ovl{c}_{32}^*]]\!]\big/\big(c_{12}((b-c)d_{33}c_{21}^*+(a-b)c_{31}d_{23})\big).
\]
Moreover
\begin{align*}
A^{(f-1-i)}\mod  \fW_{0}=\begin{pmatrix}
c_{31}c_{12}(c_{32}^*)^{-1}&c_{12}&d_{33}c_{12}(c_{32}^*)^{-1}+v c_{13}^*\\
v c_{21}^*&0&v d_{23}\\
v c_{31}&v c_{32}^*&v d_{33}
\end{pmatrix}.
\end{align*}
On the other hand, $\tld{w}'_{f-1-i}=\alpha\beta\alpha t_{\un{1}}$, $\tld{z}^*_i=\alpha$, $(a',b',c')=(b,a,c)$ and $A^{\prime\,(f-1-i)}$ has the form
\begin{align*}
A^{\prime\,(f-1-i)}=\begin{pmatrix}
c'_{11}&c'_{11}c'_{32}(c^{\prime\,*}_{31})^{-1}&d'_{33}c'_{11}(c^{\prime\, *}_{31})^{-1}+vc_{13}^{\prime\,*}\\
0&v c^{\prime\,*}_{22}& vc'_{23} \\
v c^{\prime\,*}_{31}&v c'_{32}&v d'_{33}
\end{pmatrix}.
\end{align*}
The ring $\overline{R}^{\expl,\nabla}_{\ovl{\fM}',\tld{w}'_{f-1-i}}$ is the quotient of $\F[\![ c'_{11}, c'_{23},c'_{32}, d'_{33}, c_{31}^*-[\ovl{c}_{31}^*], c_{22}^*-[\ovl{c}_{22}^*],c_{13}^*-[\ovl{c}_{13}^*]]\!]$ by the relations
\begin{equation*}
c'_{11}((a'-b')c'_{23}c'_{32}-(a'-c')c^{\prime\,*}_{33}d'_{33})=0.
\end{equation*}
(We remark that the form of the $(3,3)$ entry of $A^{\prime\,(f-1-i)}$ is as above because in the notation of \cite[\S 5.3.1]{LLLM} we have $c_{33}=-y'_{33}c_{13}c_{31}^{*}=0 \mod p$ since $c_{13}c^*_{31}=c_{11}c'_{33}\mod p$ and $y'_{33}c_{11}=0 \mod p$.)

Now we note that %
\begin{equation*}
\begin{pmatrix}
c_{31}c_{12}(c_{32}^*)^{-1}&c_{12}&d_{33}c_{12}(c_{32}^*)^{-1}+v c_{13}^*\\
v c_{21}^*&0&v d_{23}\\
v c_{31}&v c_{32}^*&v d_{33}
\end{pmatrix}=
\begin{pmatrix}
c'_{11}&c'_{11}c'_{32}(c^{\prime\,*}_{31})^{-1}&d'_{33}c'_{11}(c^{\prime\, *}_{31})^{-1}+vc_{13}^{\prime\,*}\\
0&v c^{\prime\,*}_{22}& vc'_{23} \\
v c^{\prime\,*}_{31}&v c'_{32}&v d'_{33}
\end{pmatrix}
\begin{pmatrix}
0&1&0\\
1&0&0\\
0&0&1
\end{pmatrix}
\end{equation*}
under the isomorphism 
\begin{align*}
\label{iso:ring:AB}
\overline{R}^{\expl,\nabla}_{\ovl{\fM},\tld{w}_{f-1-i}}\big/
\fW_{0}\cong \ovl{R}^{\expl,\nabla}_{\ovl{\fM}',\tld{w}'_{f-1-i}}
\end{align*}
given by the change of variable 
\begin{align*}
c_{12}&=c'_{11},&
c_{31}&=c_{32}',&
d_{23}&=d'_{23},&\\
d_{33}&=d'_{33},&
c_{32}^*&=c^{\prime\,*}_{31},&
c_{21}^*&=c^{\prime\,*}_{22},&
c_{13}^*&=c^{\prime\,*}_{13}.
\end{align*}

\paragraph{\emph{Case $\tld{w}_{f-1-i}=\alpha t_{\un{1}}$.}}

From \cite[Table 5]{LLLM}, the matrix of the partial Frobenius $A^{(f-1-i)}$ has the following form:
\begin{align*}
\begin{pmatrix}
c_{11}&c_{12}+vc_{12}^*& c_{13}\\
vc_{21}^*&c_{22}+vd_{22}&c_{23}\\
vc_{31}&vc_{32}&c_{33}+vc_{33}^*
\end{pmatrix}.
\end{align*}
Set $\tld{c}_{32}\defeq \frac{c_{32}c_{21}^*-d_{22}c_{31}}{c_{21}^*}$. 
Recall from item (\ref{expl:alpha}) above that $\ovl{R}^{\expl,\nabla}_{\ovl{\fM},\tld{w}_{f-1-i}}$ is the quotient of 
\[
\F[\![c_{11},c_{12},c_{13}, c_{22},c_{23}, c_{31}, \tld{c}_{32},c_{33},d_{22}, c^{*}_{12}-[\ovl{c}^*_{12}], c^{*}_{21}-[\ovl{c}^*_{21}], c^{*}_{33}-[\ovl{c}^*_{33}]]\!]
\] by the relations:
\begin{align*}
&c_{11}c_{23}=0,&&
c_{33}^*c_{11}\tld{c}_{32}=c_{13}c_{31}\tld{c}_{32},\\
&c_{11}d_{22}c^*_{33}=\frac{b-c}{a-b}c^*_{21}c_{13}\tld{c}_{32},&&
c_{13}c_{23}\tld{c}_{32}=0,\\
&c_{23}c_{31}\tld{c}_{32}=0,&&
(a-b)c_{13}c_{31}d_{22}+(c-b)c_{13}\tld{c}_{32}c_{21}^*+(-1-a+c)c_{23}c_{31}c_{12}^*=0\\
&c_{12}c^*_{33}=\frac{a-c}{a-b}c_{13}\tld{c}_{32}
&&
c_{22}c_{33}^*=\frac{(-1-a+c)}{(-1-a+b)}c_{23}\tld{c}_{32}
\\
&c_{21}^*c_{33}=c_{31}c_{23}.
&&
\end{align*}
(Note that $\ovl{R}^{\expl,\nabla}_{\ovl{\fM},\tld{w}_{f-1-i}}$ is formally smooth of relative dimension $3$ over the ring $\tld{R}$ defined in \cite[Proposition 8.11]{LLLM}.)

We provide detail for the matching of the ideal $\fC_{(\eps_1,1)}$. 
We have 
\[
\ovl{R}^{\expl,\nabla}_{\ovl{\fM},\tld{w}_{f-1-i}}\big/ \fC_{(\eps_1,1)}=\F[\![c_{22},c_{23},c_{32},d_{22}, c_{21}^*-[\ovl{c}^*_{21}], c_{12}^*-[\ovl{c}^*_{12}],c_{33}^*-[\ovl{c}^*_{33}]]\!]\big/\big((-1-a+b)c_{22}c_{33}^*-(-1-a+c)c_{23}c_{32}\big)
\]
and
\[
A^{(f-1-i)}\mod \fC_{(\eps_1,1)}=\begin{pmatrix}
0&vc_{12}^*&0\\
v c_{21}^*&c_{22}+v d_{22}&c_{23}\\
0&v c_{32}&v c_{33}^*
\end{pmatrix}.
\]
On the other hand, $\tld{w}'_{f-1-i}=\alpha\beta\gamma\beta$, $\tld{z}^*_{i}=\beta\gamma^+\beta$, $(a',b',c')=(b-1,a+1,c)$ and $A^{\prime\,(f-1-i)}$ has the form
\[
\begin{pmatrix} v^2 c_{11}^{\prime\,*} &  0 & 0 \\ v(c'_{21}+v d_{21}') & c_{22}^{\prime\,*} &  c'_{23} \\ v^2 c'_{31}& 0& v c_{33}^{\prime\,*}\end{pmatrix}.
\]
The ring $\ovl{R}^{\expl,\nabla}_{\ovl{\fM}',\tld{w}'_{f-1-i}}$ is the quotient of $\F[\![c'_{21},d'_{21},c'_{23},c'_{31},c_{11}^{\prime\,*}-[\ovl{c}_{11}^{\prime\,*}],c_{22}^{\prime\,*}-[\ovl{c}_{22}^{\prime\,*}], c_{33}^{\prime\,*}-[\ovl{c}_{33}^{\prime\,*}]]\!]$ by the relation
\[(-1-a'+b')c'_{21} c_{33}^{\prime\,*}-(b'-c')c'_{31}c'_{23}=0.\]
We now note that
\[
\begin{pmatrix}
0&vc_{12}^*&0\\
v c_{21}^*&c_{22}+v d_{22}&c_{23}\\
0&v c_{32}&v c_{33}^*
\end{pmatrix}=
\begin{pmatrix} v^2 c_{11}^{\prime\,*} &  0 & 0 \\ v(c'_{21}+v d_{21}') & c_{22}^{\prime\,*} &  c'_{23} \\ v^2 c'_{31}& 0& v c_{33}^{\prime\,*}\end{pmatrix}\begin{pmatrix} 0 &  v^{-1} & 0 \\ v& 0&  0\\0 & 0& 1\end{pmatrix}
\]
under the isomorphism
\[
\ovl{R}^{\expl,\nabla}_{\ovl{\fM},\tld{w}_{f-1-i}}\big/\fC_{(\eps_1,1)}\cong \ovl{R}^{\expl,\nabla}_{\ovl{\fM}',\tld{w}'_{f-1-i}}
\]
given by the change of variables
\begin{align*}
c_{12}^*=c_{11}^{\prime\,*},&&c_{21}^*=c_{22}^{\prime\, *},&&
d_{22}=d'_{21}\\
c_{23}=c'_{23}&&
c_{32}=c'_{31}
&&
c_{33}^*=c_{33}^{\prime\,*}.
\end{align*}

We now explain the matching of the ideal $\fW_{\eps_2}\defeq\fC_{(\eps_2,0)}\cap \fC_{(\eps_2,1)}=(c_{11},\tld{c}_{32})$.
We have
\[
\ovl{R}^{\expl,\nabla}_{\ovl{\fM},\tld{w}_{f-1-i}}\big/\fW_{\eps_2}=
\F[\![c_{13}, c_{23}, c_{31},d_{22}, c^{*}_{12}-[\ovl{c}^*_{12}], c^{*}_{21}-[\ovl{c}^*_{21}], c^{*}_{33}-[\ovl{c}^*_{33}]\!]\big/c_{31}((a-b)c_{13}d_{22}+(-1-a+c)c_{12}^*c_{23}).
\]
Moreover
\[
A^{(f-1-i)}\mod \fW_{\eps_2}=
\begin{pmatrix}
0&v c_{12}^*&c_{13}\\
v c_{21}^*& v d_{22}&c_{23}\\
v c_{31}&v d_{22}c_{31}(c_{21}^*)^{-1}&c_{23}c_{31}(c_{21}^*)^{-1}+v c_{33}^*
\end{pmatrix}.
\]
On the other hand, $\tld{w}'_{f-1-i}=\alpha\gamma\alpha$, $\tld{z}^*_i=\gamma^+\alpha$ and $(a',b',c')=(b,c-1,a+1)$ and $A^{\prime\,(f-1-i)}$ has the form
\[
\begin{pmatrix} v c_{11}^{\prime\,*} & vc'_{12} & 0\\ v c'_{21} & v d'_{22} & c_{23}^{\prime\,*} \\  v c'_{33}c'_{21}(c_{23}^{\prime\,*})^{-1} & v(c'_{33}d'_{22}(c_{23}^{\prime\, *})^{-1}+v c_{32}^{\prime\,*}) &  c'_{33}\end{pmatrix}.
\]
The ring $\ovl{R}^{\expl,\nabla}_{\ovl{\fM}',\tld{w}'_{f-1-i}}$ is the quotient of $\F[\![c'_{12},c_{21}',d_{22}',c_{33}',c_{11}^{\prime\,*}-[\ovl{c}_{11}^{\prime\,*}],c_{23}^{\prime\,*}-[\ovl{c}_{23}^{\prime\,*}], c_{32}^{\prime\,*}-[\ovl{c}_{32}^{\prime\,*}]]\!]$ by the relation
\[
c'_{33}\big((-1-a'+c')c_{12}'c'_{21}-(1-b'+c')c_{11}^{\prime\,*}d'_{22}\big)=0.
\]
Now we note that
\[
\text{\small{$
\begin{pmatrix}
0&v c_{12}^*&c_{13}\\
v c_{21}^*& v d_{22}&c_{23}\\
v c_{31}&v d_{22}c_{31}(c_{21}^*)^{-1}&c_{23}c_{31}(c_{21}^*)^{-1}+v c_{33}^*.
\end{pmatrix}=
\begin{pmatrix} v c_{11}^{\prime\,*} & vc'_{12} & 0\\ v c'_{21} & v d'_{22} & c_{23}^{\prime\,*} \\  v c'_{33}c'_{21}(c_{23}^{\prime\,*})^{-1} & v(c'_{33}d'_{22}(c_{23}^{\prime\, *})^{-1}+v c_{32}^{\prime\,*}) &  c'_{33}\end{pmatrix}
\begin{pmatrix} 0 & 1& 0\\ 0 &0 &v^{-1} \\  v& 0 &  0\end{pmatrix}
$}}
\]
under the isomorphism 
\[
\ovl{R}^{\expl,\nabla}_{\ovl{\fM},\tld{w}_{f-1-i}}\big/\fW_{\eps_2}\cong \ovl{R}^{\expl,\nabla}_{\ovl{\fM}',\tld{w}'_{f-1-i}}
\]
given by the change of variables
\begin{align*}
c_{21}^*=c^{\prime\,*}_{23},&&c_{31}=c_{33}',	&&c_{12}^*=c_{11}^{\prime\,*}&&\\
d_{22}=c_{21}'&&c_{13}=c_{12}',&&
c_{23}=d_{22}'&&c_{33}^*=c_{32}^{\prime\,*}.
\end{align*}

\paragraph{\emph{Case $\tld{w}_{f-1-i}=t_{\un{1}}$.}}

From \cite[Table 5]{LLLM}, the matrix of the partial Frobenius $A^{(f-1-i)}$ has the following form:
\begin{align*}
\begin{pmatrix}
c_{11}+vc_{11}^*&c_{12}& c_{13}\\
vc_{21}&c_{22}+vc_{22}^*&c_{23}\\
vc_{31}&vc_{32}&c_{33}+vc_{33}^*
\end{pmatrix}.
\end{align*}
Recall from item (\ref{expl:id}) above that $\ovl{R}^{\expl,\nabla}_{\ovl{\fM},\tld{w}_{f-1-i}}$ is the quotient of $\F[\![c_{ij},\  c^{*}_{kk}-[\ovl{c}_{kk}^*],\, 1\leq i,j,k\leq 3]\!]$ by the relations
\begin{align*}
&c_{ii}c_{jj}=0\,\text{for $i\neq j$},&&
c_{11}c_{23}=0&&c_{31}c_{22}=0,&&
c_{33}c_{12}=0,\\
&c_{12}c_{23}=c_{22}c_{13},&&
c_{11}c_{32}=c_{12}c_{31},&&
c_{21}c_{33}=c_{31}c_{23}&&
\end{align*}
and
\begin{align*}
&(-1-a+c)c^*_{22}c_{33}+(-1-a+b)c_{22}c^*_{33}-(-1-a+c)c_{23}c_{32}=0
\\
&(a-b)c^*_{33}c_{11}+(-1-b+c)c_{33}c^*_{11}-(a-b)c_{13}c_{31}=0
\\
&(b-c)c^*_{11}c_{22}+(a-c)c_{11}c^*_{22}-(b-c)c_{12}c_{21}=0\\
&c_{11}c_{22}^*c_{33}^*+c_{22}c_{11}^*c_{33}^*+c_{33}c_{11}^*c_{22}^*-c_{11}^*c_{23}c_{32}-c_{22}^*c_{13}c_{31}-c_{33}^*c_{12}c_{21}+c_{13}c_{32}c_{21}=0.
\end{align*}
(Note that $\ovl{R}^{\expl,\nabla}_{\ovl{\fM},\tld{w}_{f-1-i}}$ is formally smooth of relative dimension $3$ over the ring $\tld{R}$ defined in \cite[Corollary 8.4]{LLLM}.)

We provide detail for the matching of $\fC_{(0,0)}=(c_{11}, c_{22}, c_{33}, c_{13},c_{23},c_{12})$.

We have 
\[
\ovl{R}^{\expl,\nabla}_{\ovl{\fM},\tld{w}_{f-1-i}}\big/\fC_{(0,0)}=\F[\![c_{ij},\ c^{*}_{kk}-[\ovl{c}_{kk}^*],\, 1\leq j<i\leq 3,\ 1\leq k\leq 3]\!]
\]
and 
\[
A^{(f-1-i)}\mod \fC_{(0,0)}=\begin{pmatrix}
v c_{11}^*&0&0\\
v c_{21}&v c_{22}^*& 0 \\
v c_{31}&v c_{32}&v c_{33}^*
\end{pmatrix}.
\]
On the other hand, $\tld{w}'_{f-1-i}=\alpha\beta\alpha\gamma t_{\un{1}}$ and $A^{\prime\,(f-1-i)}$ has the form
\[
\begin{pmatrix}
v^2 c_{11}^{\prime\,*}&0&0\\
v^2 c'_{21}&v c_{22}^{\prime\,*}& 0 \\
v^2 c'_{31}&v c'_{32}& c_{33}^{\prime\,*}
\end{pmatrix}.
\]
We note that $\tld{z}_{f-1-i}=t_{(1,0,-1)}$ and 
\[
\begin{pmatrix}
v c_{11}^*&0&0\\
v c_{21}&v c_{22}^*& 0 \\
v c_{31}&v c_{32}&v c_{33}^*
\end{pmatrix}=\begin{pmatrix}
v^2 c_{11}^{\prime\,*}&0&0\\
v^2 c'_{21}&v c_{22}^{\prime,*}& 0 \\
v^2 c'_{31}&v c'_{32}& c_{33}^{\prime\,*}
\end{pmatrix}
\begin{pmatrix}
v^{-1}&0&0\\
0&1& 0 \\
0&0&v
\end{pmatrix}
\]
under the isomorphism
\[
\ovl{R}^{\expl,\nabla}_{\ovl{\fM},\tld{w}_{f-1-i}}\big/\fC_{(0,0)}\cong \ovl{R}^{\expl,\nabla}_{\ovl{\fM}',\tld{w}'_{f-1-i}}
\]
given by the change of variables 
\begin{align*}
&c_{ij}=c'_{ij}\, \text{for $1\leq j<i\leq 3$},&& c^*_{kk}=c^{\prime\,*}_{kk}\, \text{for $1\leq k\leq 3$}.
\end{align*}

We now explain the matching for the ideal $\fW_{0}\defeq\fC_{(0,0)}\cap\fC_{(0,1)}=(c_{22},\ c_{33},\  c_{23},\ c_{33}^*c_{12}-c_{13}c_{32})$.
We have
\[
\ovl{R}^{\expl,\nabla}_{\ovl{\fM},\tld{w}_{f-1-i}}\big/\fW_{0}=\F[\![c_{13},c_{21},c_{31},c_{32},c^*_{11}-[\ovl{c}^*_{11}],c^*_{22}-[\ovl{c}^*_{22}],c^*_{33}-[\ovl{c}^*_{33}]]\!]\big/(c_{13}((a-c)c_{22}^*c_{31}-(b-c)c_{32}c_{21}))
\]
(we have used the relations
\begin{align*}
&c_{33}^*c_{12}-c_{13}c_{32}=0,&&c_{33}^*c_{11}-c_{13}c_{31}=0,&&
(a-c)c_{22}^*c_{11}-(b-c)c_{12}c_{21}=0
\end{align*}
holding in $\ovl{R}^{\expl,\nabla}_{\ovl{\fM},\tld{w}_{f-1-i}}\big/\fW_{0}$).
Moreover
\[
A^{(f-1-i)}\mod\fW_{0}=\begin{pmatrix}
c_{13}c_{31}(c_{33}^*)^{-1}+v c_{11}^*&c_{13}c_{32}(c_{33}^*)^{-1}&c_{13}\\
v c_{21}&v c_{22}^*& 0 \\
v c_{31}&v c_{32}&v c_{33}^*
\end{pmatrix}.
\]
On the other hand $\tld{w}_{f-1-i}=\alpha\beta\alpha=\tld{z}^*_i$, $(a',b',c')=(c,b,a)$ and $A^{\prime\,(f-1-i)}$ has the form
\[
\begin{pmatrix} c'_{11} & c'_{11}c'_{32}(c_{31}^{\prime\,*})^{-1} & c'_{11}d'_{33}(c_{31}^{\prime\,*})^{-1}+v c_{13}^{\prime\,*} \\ 0 & vc_{22}^{\prime\,*} &v c_{23}' \\  v c_{31}^{\prime\,*} & vc_{32}' & v d'_{33}\end{pmatrix}.
\]
The ring $R^{\expl,\nabla}_{\ovl{\fM}',\tld{w}'_{f-1-i}}$ is the quotient of $\F[\![c'_{11},c'_{23},c'_{32},d'_{33},c_{13}^{\prime\,*}-[\ovl{c}_{13}^{\prime\,*}],c_{22}^{\prime\,*}-[\ovl{c}_{22}^{\prime\,*}],c_{31}^{\prime\,*}-[\ovl{c}_{31}^{\prime\,*}]]\!]$ by the relation
\[
c'_{11}\big((a'-b')c_{23}'c'_{32}-(a'-c')c_{22}^{\prime\,*}d'_{33}\big)=0.
\]
We now note that 
\[
\begin{pmatrix}
c_{13}c_{31}(c_{33}^*)^{-1}+v c_{11}^*&c_{13}c_{32}(c_{33}^*)^{-1}&c_{13}\\
v c_{21}&v c_{22}^*& 0 \\
v c_{31}&v c_{32}&v c_{33}^*
\end{pmatrix}=\begin{pmatrix} c'_{11} & c'_{11}c'_{32}(c_{31}^{\prime\,*})^{-1} & c'_{11}d'_{33}(c_{31}^{\prime\,*})^{-1}+v c_{13}^{\prime\,*} \\ 0 & vc_{22}^{\prime\,*} &v c_{23}' \\  v c_{31}^{\prime\,*} & vc_{32}' & v d'_{33}\end{pmatrix}\begin{pmatrix}0&0&1\\0&1&0\\1&0&0 \end{pmatrix}
\]
under the isomorphism
\[
\ovl{R}^{\expl,\nabla}_{\ovl{\fM},\tld{w}_{f-1-i}}\big/\fW_{0}\cong\ovl{R}^{\expl,\nabla}_{\ovl{\fM'},\tld{w}'_{f-1-i}}
\]
given by the change of variables
\begin{align*}
&
c_{13}=c'_{11},&&c_{21}=c'_{23},&&
c_{31}=d'_{33},&&\\
&c_{32}=c'_{32},&&c^*_{11}=c_{13}^{\prime\,*},&&c^*_{22}=c_{22}^{\prime\,*},&&c^*_{33}=c^{\prime\,*}_{31}.
\end{align*}

\paragraph{\emph{Case $\tld{w}_{f-1-i}= \beta\alpha t_{\un{1}}$ and $\alpha\beta\alpha t_{\un{1}}$}}: The computations of these two cases are very similar to those we already performed and are left to the reader.

\subsubsection{Ideal relations in deformation rings} \label{sec:idealrel}

In this subsection, we collect some results about sums of intersections of minimal primes in the potentially crystalline deformation rings $R^{\tau}_{\rhobar}$. These computations play a crucial role in \S \ref{sec:cyc}, where they are used to compute the value of a patching functor on certain representations as the limit of the value of the patching functor on simpler pieces of the representation.

Thanks to Theorem \ref{thm:matching}, all computations that we need to perform can be done on the rings $\ovl{R}_{\ovl{\fM},\tld{w}_{f-1-i}}^{\expl,\nabla}$ given in Table \ref{Table:intsct}. We continue to adopt the notation and setting of Theorem \ref{thm:matching}.

We will frequently make use of the following:
\begin{lemma} 
\label{lem:iso:ring}
Suppose we have a surjection of rings $g: S \onto R$. Assume that $R$ and $S$ are equidimensional of dimension $d$, have the same number of minimal primes, and $S$ is reduced. Then $g$ is an isomorphism.
\end{lemma}
\begin{proof} The first two hypotheses imply that the kernel of $g$ is nilpotent, since $g$ induces and isomorphism between the underlying topological spaces of $\Spec (R)$ and $\Spec(S)$. But since $S$ is reduced, this kernel must in fact be $0$.
\end{proof}

\vspace{2mm}

\paragraph{\emph{Ideal relations in $\ovl{R}_{\ovl{\fM},\id}^{\expl,\nabla}$}.}

\begin{lemma}
\label{lem:ideal}
In the ring $\ovl{R}^{\expl,\nabla}_{\ovl{\fM},\id}$, we have 
\begin{equation*}
(\fW_{0}\cap\fC_{(\eps_1,0)})+(\fW_{0}\cap\fC_{(\eps_2,0)})=\fW_{0}.
\end{equation*}
\end{lemma}
\begin{proof} 
An elementary check on the list of generators of the ideals $\fW_{0}$, $\fC_{(\eps_1,0)}$ and $\fC_{(\eps_2,0)}$ (cf.~ Table \ref{Table:intsct}) shows that
\begin{align*}
\fW_{0}\cap \fC_{(\eps_1,0)}\supseteq(c_{33}, c_{23}, c_{22}),&\qquad
\fW_{0}\cap \fC_{(\eps_2,0)}\supseteq(c_{33},  c_{22}, c_{13}c_{32}-c_{12}c_{33}^*).
\end{align*}
In particular (again by looking at the list of generators of $\fW_{0}$), we have $(\fW_{0}\cap \fC_{(\eps_1,0)})+(\fW_{0}\cap \fC_{(\eps_2,0)})\supseteq\fW_{0}$.
The reverse inclusion is obvious.
\end{proof}

\begin{lemma}
\label{lem:ideal:0}
In the ring $\overline{R}_{\overline{\fM},\id}^{\expl,\nabla}$, we have
\begin{equation*}
\label{int:three} \fW_{0}\cap \fC_{(\eps_1,0)}\cap \fC_{(\eps_2,0)}=(c_{22},c_{33}).
\end{equation*}
\end{lemma}
\begin{proof}
We let $\tld{R}$ be the ring with the same presentation as $\overline{R}_{\overline{\fM},\id}^{\expl,\nabla}$ except that all $c^*_{ij}$ are set to $1$. As explained in \cite[Corollary 8.4]{LLLM}, there is a natural identification of $\overline{R}_{\overline{\fM},\id}^{\expl,\nabla}$ with the power series over $\tld{R}$ with $3$ variables, and we can work with the ring $\tld{R}$ instead of $\overline{R}_{\overline{\fM},\id}^{\expl,\nabla}$. All the ideals that we consider come from $\tld{R}$, and are given by generators with the same name.

From Table \ref{Table:intsct}, we immediately obtain $(c_{22},c_{33})\subseteq \fW_{0}\cap \fC_{(\eps_1,0)}\cap \fC_{(\eps_2,0)}$, and hence we need to prove that the surjection
\begin{equation}
\label{surj:4:weight}
\tld{R}/(c_{22},c_{33})\onto \tld{R}/( \fW_{0}\cap \fC_{(\eps_1,0)}\cap \fC_{(\eps_2,0)})
\end{equation}
is an isomorphism. The ring on the right-hand side is equidimensional of dimension $3$ and has $4$ minimal primes. By Lemma \ref{lem:iso:ring}, it suffices to show that $\tld{R}/(c_{22},c_{33})$ is reduced, is equidimensional of dimension $3$, and has $4$ minimal primes.

Now $\tld{R}/(c_{22},c_{33})$ is the quotient of the power series ring $\F[\![c_{11},\ c_{ij},\ 1\leq i,j,\leq 3,\ i\neq j]\!]$ by the ideal generated by the following elements:
\begin{align*}
&c_{11}c_{23},&& c_{12}c_{23},&& c_{11}c_{32}-c_{12}c_{31},\\
&c_{31}c_{23},&&c_{23}c_{32},&&\kappa c_{11}=c_{21}c_{12},\\
&c_{11}-c_{13}c_{31},&&c_{11}-c_{12}c_{21}-c_{13}c_{31}-c_{23}c_{32}+c_{21}c_{13}c_{32}
\end{align*}
for some $\kappa\in \F^{\times}$ depending on $(a,b,c)$.
By standard manipulations, we conclude that $\tld{R}/(c_{22},c_{33})$ is isomorphic to the quotient of the power series ring $\F[\![c_{ij},\ 1\leq i,j,\leq 3,\ i\neq j]\!]$ by the ideal generated by the following elements:
\begin{align*}
&c_{31}c_{23},&& c_{23}c_{32},&& c_{12}c_{23},\\
&c_{31}(c_{12}-c_{13}c_{32}),&&c_{13}(\kappa c_{31}-c_{21}c_{32}),&&c_{21}(c_{12}-c_{13}c_{32})
\end{align*}
hence, by writing $\tld{c}_{12}\defeq c_{12}-c_{13}c_{32}$, $\tld{c}_{31}\defeq \kappa c_{31}-c_{21}c_{32}$ we obtain
\begin{equation*}
\tld{R}/(c_{22},c_{33})\cong 
\F[\![\tld{c}_{12},\tld{c}_{31},c_{23},c_{13},c_{21},c_{32}]\!]/(\tld{c}_{12}\tld{c}_{31},\,\tld{c}_{12}c_{23},\,\tld{c}_{12}c_{21},\,\tld{c}_{31}c_{23},\,\tld{c}_{31}c_{13},\,c_{23}c_{32}).
\end{equation*}
This latter ring is easily seen to be equidimensional of dimension $3$ and has $4$ minimal primes. It is furthermore reduced since it is the quotient of a power series ring by an ideal generated by squarefree monomials.
\end{proof}

From Lemma \ref{lem:ideal:0}, the same argument used in the proof of Lemma \ref{lem:ideal} gives the following:
\begin{lemma}\label{lem:ideal:1:id}
In the ring $\ovl{R}^{\expl,\nabla}_{\ovl{\fM},\id}$, we have
\begin{equation*}
(\fW_{0}\cap\fW_{\eps_1}\cap\fC_{(\eps_2,0)})+(\fW_{0}\cap\fW_{\eps_2}\cap\fC_{(\eps_1,0)})=\fW_{0}\cap\fC_{(\eps_1,0)}\cap\fC_{(\eps_2,0)}
\end{equation*}
\end{lemma}
\begin{proof}
From Table \ref{Table:intsct}, we immediately deduce that $(\fW_{0}\cap\fW_{\eps_1}\cap\fC_{(\eps_2,0)})\supseteq (c_{33})$ and $(\fW_{0}\cap\fW_{\eps_2}\cap\fC_{(\eps_1,0)})\supseteq (c_{22})$. The conclusion now follows as in Lemma \ref{lem:ideal} by noting that $(c_{22},c_{33})=\fW_{0}\cap\fC_{(\eps_1,0)}\cap\fC_{(\eps_2,0)}$.
\end{proof}

\vspace{2mm}

\paragraph{\emph{Ideal relations in $\ovl{R}_{\ovl{\fM},\alpha}^{\expl,\nabla}$}.}

\begin{lemma}
\label{lem:ideal:0:alpha}
Consider the ring $\overline{R}^{\expl,\nabla}_{\overline{\fM},\alpha}$ (cf.~ Table\ref{Table:intsct}) and let $I_{\Lambda_0}\defeq \fC_{(0,0)}\cap\fC_{(0,1)}\cap\fC_{(\eps_2,0)}$.

Then
\begin{enumerate}
	\item\label{id:int:1} $I_{\Lambda_0}=(c_{11}c_{33}^*-c_{13}c_{31},c_{23}c_{31},c_{23}(c_{32}c_{21}^*-d_{22}c_{31}))$,
	\item\label{id:int:2} $\fC_{(\eps_1,1)}\cap\fC_{(0,0)}\cap\fC_{(\eps_2,0)}=(c_{11},c_{13}c_{31})$,
\item \label{id:int:3}$\fC_{(0,1)}\cap\fC_{(0,0)}=(c_{23},c_{11}c_{33}^*-c_{13}c_{31})$.
\end{enumerate}
\end{lemma}
\begin{proof}
As in the proof of Lemma \ref{lem:ideal:0}, it suffices to work in the ring $\tld{R}$ which has the same presentation as $\overline{R}_{\overline{\fM},\alpha}^{\expl,\nabla}$ except that all $c^*_{ij}$ are set to $1$.
Recall that $\tld{c}_{32}\defeq c_{32}-d_{22}c_{31}$ in $\tld{R}$.

We start with item (\ref{id:int:1}).
From Table (\ref{Table:intsct}), we easily deduce 
\begin{equation*}
I_{\Lambda_0}\supseteq (c_{11}-c_{13}c_{31},c_{23}c_{31},c_{23}\tld{c}_{32})
\end{equation*}
i.e. a surjection
\begin{equation}
\label{cl:im:ring}
\tld{R}\big/ (c_{11}-c_{13}c_{31},c_{23}c_{31},c_{23}\tld{c}_{32})\onto
\tld{R}/ I_{\Lambda_0}
\end{equation}
which we claim it is an isomorphism. Note that by construction the ring $\tld{R}/ I_{\Lambda_0}$ has three minimal primes, and it is equidimensional of dimension 3.

From \cite[Proposition 8.11]{LLLM}, we immediately deduce that the ring $\tld{R}/(c_{11}-c_{13}c_{31},c_{23}c_{31},c_{23}\tld{c}_{32})$ is isomorphic to the quotient of the power series ring $\F[\![c_{31},c_{13},d_{22},\tld{c}_{32},c_{23}]\!]$ by the ideal generated by the following elements:
\begin{align*}
&c_{13}c_{31}c_{23},&&  c_{13}c_{31}d_{22}-\frac{b-c}{a-b}c_{13}\tld{c}_{32},&&c_{23}c_{31},\\
&c_{23}\tld{c}_{32},&&(a-b)c_{13}c_{31}d_{22}+(c-b)c_{13}\tld{c}_{32}+(-1-a+c)c_{23}c_{31}&&
\end{align*}
i.e. %
by the ideal generated by
\begin{align*}
c_{13}(c_{31}d_{22}-\frac{b-c}{a-b}\tld{c}_{32}),&&c_{23}c_{31},&&c_{23}\tld{c}_{32},
\end{align*}
or equivalently, the ideal generated by
\begin{align*}
c_{13}(c_{31}d_{22}-\frac{b-c}{a-b}\tld{c}_{32}),&&c_{23}c_{31},&&c_{23}(c_{31}d_{22}-\frac{b-c}{a-b}\tld{c}_{32}).
\end{align*}
In other words, by an evident change of variables we have
\begin{equation*}
\tld{R}\big/ (c_{11}-c_{13}c_{31},c_{23}c_{31},c_{23}\tld{c}_{32})\cong \F[\![X,Y,Z,W,c_{22}']\!]\big/(XY,YZ,WZ)
\end{equation*}
and the latter ring is equidimensional of dimension 3 and has 3 irreducible components. It is moreover reduced, since the ideal of relations is generated by square-free monomials.%

The argument to prove (\ref{id:int:2}) is completely analogous and we only  sketch it.
It is immediate to obtain from Table \ref{Table:intsct} the inclusion $\fC_{(\eps_1,1)}\cap \fC_{(0,0)}\cap\fC_{(\eps_2,0)}\supseteq (c_{11}, c_{13}c_{31})$. From \cite[Proposition 8.11]{LLLM}, we see that $\tld{R}/(c_{11}, c_{13}c_{31})$ is isomorphic to 
\begin{align*}
&\F[\![c_{13},c_{23},c_{31},\tld{c}_{32},d_{22}]\!]/(c_{13}c_{31},c_{13}\tld{c}_{32},c_{31}c_{23}\tld{c}_{32},c_{31}(\kappa c_{23}-c_{13}d_{22}))=\\
&\qquad=\F[\![c_{13},c_{23},c_{31},\tld{c}_{32},d_{22}]\!]/(c_{13}c_{31},c_{13}\tld{c}_{32},c_{31}c_{23}),
\end{align*}
for some $\kappa\in \F^\times$ depending on $(a,b,c)$.
The conclusion is obtained by an analogous argument to the previous case.

The proof of item (\ref{id:int:3}) is similar and left to the reader.
(Note that $\fC_{(0,1)}\cap\fC_{(0,0)}\supseteq (c_{23},c_{11}-c_{13}c_{31})$ and use the explicit presentation of $\tld{R}$ given by \cite[Proposition 8.11]{LLLM} to check that $\tld{R}/(c_{23},c_{11}-c_{13}c_{31})$ is reduced with $2$ irreducible component of dimension $3$).
\end{proof}
\begin{lemma}
\label{lem:ideal:1:alpha}
In $\overline{R}^{\expl,\nabla}_{\overline{\fM},\alpha}$, we have
\begin{align}
\label{it:ideal:1:alpha:1}
&(I_{\Lambda_0}\cap \fC_{(\eps_1,1)})+(I_{\Lambda_0}\cap \fC_{(\eps_2,1)}\cap \fC_{(\eps_2-\eps_1,0)})=I_{\Lambda_0},\\
\label{it:ideal:1:alpha:2}
&(\fC_{(\eps_1,1)}\cap \fC_{(0,0)} \cap \fC_{(\eps_2,0)} \cap \fC_{(0,1)}) +(\fC_{(\eps_1,1)} \cap \fC_{(0,0)} \cap \fC_{(\eps_2,0)} \cap \fC_{(\eps_2,1)}) = \fC_{(\eps_1,1)} \cap \fC_{(0,0)} \cap \fC_{(\eps_2,0)},\\
\label{it:ideal:1:alpha:3}
&(\fC_{(0,1)}\cap\fC_{(0,0)}\cap\fC_{(\eps_2,0)})+(\fC_{(0,1)}\cap\fC_{(0,0)}\cap\fC_{(\eps_2-\eps_1,0)})=\fC_{(0,1)}\cap\fC_{(0,0)}.
\end{align}
\end{lemma}
\begin{proof}
We work in $\tld{R}$ and set $\tld{c}_{32}\defeq c_{32}-d_{22}c_{31}$.
An elementary check on the list of generators of the ideals $I_{\Lambda_0}$, $\fC_{(\eps_2,1)}$,  $\fC_{(\eps_2-\eps_1,0)}$ and $\fC_{(\eps_1,1)}$ shows that
\begin{equation*}
(I_{\Lambda_0}\cap \fC_{(\eps_1,1)})\supseteq (c_{11}-c_{13}c_{31}, c_{23}c_{31}),\qquad
(I_{\Lambda_0}\cap \fC_{(\eps_2,1)}\cap \fC_{(\eps_2-\eps_1,0)})\supseteq (c_{23}\tld{c}_{32})
\end{equation*}
and we deduce (again by looking at the list of generators of $I_{\Lambda_0}$) that
\begin{equation*}
(I_{\Lambda_0}\cap \fC_{(\eps_1,1)})+(I_{\Lambda_0}\cap \fC_{(\eps_2,1)}\cap \fC_{(\eps_2-\eps_1,0)})\supseteq I_{\Lambda_0}.
\end{equation*}
The reverse inclusion is obvious.

Similarly, $(c_{11}) \subseteq \fC_{(\eps_1,1)} \cap \fC_{(0,0)} \cap \fC_{(\eps_2,0)} \cap \fC_{(\eps_2,1)}$ and $(c_{11}-c_{13}c_{31}) \subseteq \fC_{(\eps_1,1)} \cap \fC_{(0,0)} \cap \fC_{(\eps_2,0)} \cap \fC_{(0,1)}$.
Hence $(\fC_{(\eps_1,1)}\cap \fC_{(0,0)} \cap \fC_{(\eps_2,0)} \cap \fC_{(0,1)}) + (\fC_{(\eps_1,1)}\cap \fC_{(0,0)} \cap \fC_{(\eps_2,0)} \cap \fC_{(\eps_2,1)}) \supseteq \fC_{(\eps_1,1)}\cap \fC_{(0,0)} \cap \fC_{(\eps_2,0)}$.

The proof of item (\ref{it:ideal:1:alpha:3}) 
is similar and left to the reader (note that $c_{11}-c_{13}c_{31}\in (\fC_{(0,1)}\cap\fC_{(0,0)}\cap\fC_{(\eps_2,0)})$ and $c_{23}\in (\fC_{(0,1)}\cap\fC_{(0,0)}\cap\fC_{(\eps_2-\eps_1,0)})$).
\end{proof}

\begin{rmk}
The ideal relation appearing in Lemmas \ref{lem:ideal:0:alpha}, \ref{lem:ideal:1:alpha} are compatible with the outer automorphism of $\tld{W}^\vee$.
Explicitly, define $\tld{\delta}\defeq (123)t_{(0,0,-1)}$ and let $\tld{w}\mapsto \tld{\delta}\tld{w}\tld{\delta}^{-1}$ be the corresponding outer automorphism on $\tld{W}^\vee$.
Then $\delta$ acts on $\Sigma_0$ and the ideal relations of Lemmas \ref{lem:ideal:0:alpha}, \ref{lem:ideal:1:alpha} holds for shape $\beta$ with $\fC_{(\omega_i,a_i)}$ replaced by $\fC_{\big((\tld{\delta}^*)^{-1}(\omega_i),a_i\big)}$.
As an example, for shape $\beta$ the relation 
(\ref{it:ideal:1:alpha:3}) becomes $(\fC_{(\eps_1,1)}\cap\fC_{(\eps_1,0)}\cap\fC_{(0,0)})+(\fC_{(\eps_1,1)}\cap\fC_{(\eps_1,0)}\cap\fC_{(\eps_1-\eps_2,0)})=\fC_{(\eps_1,1)}\cap\fC_{(\eps_1,0)}$.
\end{rmk}

\section{Lattices in generic Deligne--Lusztig representations}
\label{sec:lct}

The aim of this section is to classify lattices with irreducible cosocle in generic $\GL_3(\Fq)$ Deligne--Lusztig representations, providing the crucial representation-theoretic input to deduce Breuil's lattice conjecture from the weight part of Serre conjecture. 
The main result (Theorem \ref{thm structure}) states that the submodule structure of lattices with irreducible cosocle can be predicted using the extension graph introduced in \S \ref{sec:graph}. 

\paragraph{\it{Outline of the proof}.}
Let $R$ be a generic Deligne--Lusztig representation.
We have two main steps in the proof of the classification theorem: first, by local algebraic methods, we describe the reduction of lattices with irreducible cosocles isomorphic to a lower alcove weight with {\it defect zero} (Theorem \ref{thm:soc:simple}).
Second, in \S \ref{subsec:dist}, we introduce another notion of distance related to the second part of Theorem \ref{thm structure}, which we call {\it saturation distance}.
It turns out that this notion is closely related to the submodule structure of the reduction of lattices.

Using a crucial global input coming from the geometry of Galois deformation rings, we show that
\begin{enumerate}
\item if measured from $\sigma$ the graph and saturation distance coincide, then the reduction of the lattice with cosocle isomorphic to $\sigma$ is as predicted in Theorem \ref{thm structure} (Proposition \ref{prop:dcosoc}); and
\item measured from a lower alcove weight of defect zero, the saturation distance and the graph distance coincide (Proposition \ref{cor:max:sat1}, which uses Theorem \ref{thm:soc:simple});
\end{enumerate}
In contrast to the other notions of distance that we introduce, the saturation distance involves lattices in characteristic zero, making it far more flexible.
Taking advantage of this flexibility, we finally show by an induction on defect that the saturation distance and the graph distance coincide, completing the proof of Theorem \ref{thm structure}.

\paragraph{\it{Structure of \S \ref{sec:lct}}}

For $R$ as above and $\sigma \in \JH(\ovl{R})$, let $R^{\sigma}$ denote the unique (up to homothety) $\cO$-lattice in $R$ with irreducible cosocle $\sigma$ and write $\ovl{R}^\sigma$ to denote its reduction modulo $\varpi$.
The first main step is in \S \ref{sec:inj:env}. 
The argument uses the modular representation theory of algebraic groups to embed $\ovl{R}^{\sigma}$ into the $\rG\defeq \un{G}_0(\Fp)$-restriction of a tensor product $V_{\mu}$ of algebraic Weyl modules with non $p$-restricted highest weight (see \S \ref{subsubsec:GL3} for the definition of $V_{\mu}$).
The content of \S \ref{subsub:Weyl} is the description of $V_{\mu}|_{\rG}$ provided by Theorem \ref{thm:Vdecomp}.
This theorem describes the Jordan-H\"older constituents of $V_{\mu}|_{\rG}$ and \emph{the existence of non-trivial extensions} between constituents at graph distance one (cf.~the key technical result Proposition \ref{prop:nonsplit}).
The embedding of $\ovl{R}^{\sigma}$ in $V_{\mu}$ is constructed in \S \ref{subsub:emb:arg}.
One first proves the existence of a non-zero (and unique up to scalar) morphism $\ovl{R}^{\sigma}\ra V_{\mu}$ (Proposition  \ref{cor:map}); an inductive argument, using the description of the submodule structure of $V_{\mu}|_{\rG}$, then shows that the image of this morphism contains all the constituents of $\ovl{R}^{\sigma}$.
The submodule structure of $\ovl{R}^{\sigma}$ is then obtained from that of $V_{\mu}$.

The second part of the proof of Theorem \ref{thm structure} is the content of \S \ref{subsec:dist}. The key insight is the introduction  of the auxiliary notion of saturation distance on $\JH(\ovl{R})$  in \S \ref{sub:sub:def:dist} which relates the position of saturated lattices in $R^{\sigma}$.
Using a \emph{global} input, we  give in \S \ref{sub:sub:def:dist} a first coarse relation between the saturation and the graph distance (cf.~ Corollary \ref{ineq:sat:gr}). %
Subsequently, we prove in \S \ref{sec:DE} that the three distances are actually equal provided that $\sigma$ verifies an appropriate condition relating its defect, its graph distance and its saturated distance (we say that $\sigma$ is \emph{maximally saturated} in $R$). In particular, if $\sigma$ is maximally saturated in $R$, the structure of $\ovl{R}^{\sigma}$ is predicted by the extension graph (Proposition \ref{prop:dcosoc}).

We are hence left to prove that \emph{all} constituents of $R$ are maximally saturated.
This is shown by an induction argument in \S \ref{sub:red:MS}.
The proof of Theorem \ref{thm structure} concludes this section.

\subsection{The classification statement}

\subsubsection{Some generalities}\label{subsec:gencat}

Let $\cC$ be a nonzero finite abelian category over $\F$.
Let $M$ be a nonzero object of $\cC$.
A \emph{decreasing} (resp. ~\emph{increasing}) \emph{filtration} $\mathscr{F}$ on $M$ is a collection of subobjects $\mathscr{F}^n (M) \subset M$ (resp. $\mathscr{F}_n (M) \subset M$) for $n\in \Z$ such that $\mathscr{F}^{n+1}(M)\subset \mathscr{F}^n (M)$ (resp.~ $\mathscr{F}_n (M)\subset \mathscr{F}_{n+1} (M)$) for all $i$.
A filtration $\mathscr{F}$ is \emph{exhaustive} and \emph{separated} if $\mathscr{F}^n (M) = 0$ for $i$ sufficiently large (resp.~ small) and $\mathscr{F}^n (M) = M$ for $n$ sufficiently small (resp.~ large), and this property will always be assumed to hold.
We write $\gr^n_{\mathscr{F}}(M)\defeq \frac{\mathscr{F}^n(M)}{\mathscr{F}^{n+1}(M)}$ (resp.~ $\gr_n^{\mathscr{F}}(M)\defeq \frac{\mathscr{F}_n(M)}{\mathscr{F}_{n-1}(M)}$), and omit $\mathscr{F}$ from the notation if it is clear from context.
A filtration is \emph{semisimple} if $\gr^n_{\mathscr{F}}(M)$ (resp. $\gr_n^{\mathscr{F}}(M)$) is semisimple for all $n\in \Z$.
By shifting the filtration, we will assume that $\gr^n_{\mathscr{F}}(M) = 0$ (resp. $\gr_n^{\mathscr{F}}(M) = 0$) for $n<0$ and that $\gr^0_{\mathscr{F}}(M) \neq 0$ (resp.~ $\gr_0^{\mathscr{F}}(M) \neq 0$).
The \emph{length} of a filtration is the maximal $\ell\in \Z$ such that $\gr^{\ell-1}_{\mathscr{F}}(M) \neq 0$ (resp.~ $\gr_{\ell-1}^{\mathscr{F}}(M) \neq 0$).

The \emph{socle} of $M$, denoted $\soc (M)$ is defined to be the maximal (with respect to inclusion) semisimple subobject of $M$.
The \emph{radical} of $M$, denoted $\rad (M)$, is the minimal (with respect to inclusion) subobject of $M$ whose corresponding quotient is semisimple.
The \emph{cosocle} of $M$ is $\cosoc(M) \defeq M/\rad (M)$.
We inductively define the \emph{radical} and \emph{socle filtration} on $M$: we set $\rad^0(M) = M$ and let $\rad^n(M)\defeq \rad\left(\rad^{n-1}(M)\right)$, and set $\soc_{-1} (M) = 0$ and let $\soc_n (M)$ be the inverse image, via the canonical projection $M\onto M/\soc_{n-1}(M)$, of $\soc\left(M/\soc_{n-1}(M)\right)\subseteq M/\soc_{n-1}(M)$. 
Then the radical (resp. ~socle) filtration is a decreasing (resp. ~increasing) semisimple filtration.
Moreover, $\gr_{\rad}^0(M)=\cosoc(M)$ and $\gr^{\soc}_0(M)=\soc(M)$.
Since formation of cosocle (resp. ~socle) is right (resp. ~left) exact, the filtration induced from the radical (resp.~ socle) filtration on a quotient object (resp. ~subobject) is the radical (resp. ~socle) filtration.
The lengths of the radical and socle filtrations coincide, and we call this value the \emph{Loewy length} of $M$, and denote it by $\ell\ell(M)$.
Any semisimple filtration has length at least $\ell\ell(M)$, and we say that it is a \emph{Loewy series} if its length equals $\ell\ell(M)$.
If $\mathscr{F}$ is a decreasing Loewy series than we necessarily have 
\begin{equation}\label{eqn:Loewy}
\rad^n(M)\subseteq \mathscr{F}^n(M)\subseteq \soc_{\ell\ell(M)-n-1}(M)
\end{equation}
for all $n\in \{0,\dots,\ell\ell(M)\}$.
We say that $M$ is \emph{rigid} if $\rad^n(M) = \soc_{\ell\ell(M)-n-1}(M)$ for all $n$.

We say that $M$ is \emph{multiplicity free} if every Jordan--H\"older factor of $M$ appears with multiplicity one.
We now suppose that $M$ is multiplicity free.
We say that $\sigma\in \JH(M)$ points to $\sigma'\in\JH(M)$ if there exists a subquotient of $M$ which is isomorphic to a nontrivial extension of $\sigma$ by $\sigma'$.
We say that a subset $S\subset \JH(M)$ is closed if $\sigma\in S$ and $\sigma$ points to $\sigma'$ imply that $\sigma' \in S$.

\begin{prop}\label{prop:subclosed}
The assignment of $\JH(N)$ to a subobject $N\subset M$ gives a bijection between subobjects of $M$ and closed subsets of $\JH(M)$.
\end{prop}
\begin{proof}
It is easy to see that $\JH(N)$ is a closed subset of $\JH(M)$.
Suppose that $S\subset \JH(M)$ is a closed subset.
Let $N$ be the minimal subobject of $M$ with $S\subset \JH(N)$.
Suppose that $\JH(N) \setminus S$ is nonempty and contains $\sigma$.
If $\sigma'\in \JH(N)$ points to $\sigma$, then $\sigma'$ is not in $S$ since $S$ is closed.
By replacing $\sigma$ by $\sigma'$ repeatedly, we can assume without loss of generality that $\sigma'$ does not point to $\sigma$ for all $\sigma' \in \JH(N)$.
Let $N'$ be the maximal subobject of $N$ such that $\sigma\notin \JH(N')$.
This maximality implies that the socle of $N/N'$ must be isomorphic to $\sigma$.
By the assumption above, there is no subobject of $N/N'$ which is an extension by $\sigma$, and therefore $N/N'$ is isomorphic to $\sigma$.
Then the existence of $N'$ contradicts the minimality of $N$.
\end{proof}

Suppose that $M$ is multiplicity free and that $\mathscr{F}$ is a decreasing (resp.~ increasing) filtration on $M$.
For $\sigma \in \JH(M)$, we define $\mathrm{d}_{\mathscr{F}}^M(\sigma)$ (resp. $\mathrm{d}^{\mathscr{F}}_M(\sigma)$) to be the unique value $n$ such that $\Hom_{\cC}(\sigma,\gr_{\mathscr{F}}^n(M))$ (resp.~ $\Hom_{\cC}(\sigma,\gr^{\mathscr{F}}_n(M))$) is nonzero.
For $\mathscr{F}$ semisimple, we say that $\sigma\in \JH(M)$ $\mathscr{F}$\emph{-points to} $\sigma'\in \JH(M)$ if $\sigma$ points to $\sigma'$ and the (shifted) induced filtration on the subquotient which is isomorphic to a nontrivial extension of $\sigma$ by $\sigma'$ has length two.

If $M$ is multiplicity free, we can attach a directed acyclic graph $\Gamma(M)$ and a subgraph $\Gamma_{\mathscr{F}}(M)$ of $\Gamma(M)$ (resp. ~$\Gamma^{\mathscr{F}}(M)$) of $\Gamma(M)$) where the vertices are in bijection with the Jordan--H\"older factors of $M$ and there is an arrow $\sigma\ra \sigma'$ in $\Gamma(M)$ if $\sigma$ points to $\sigma'$ and an arrow $\sigma\ra \sigma'$ in $\Gamma_{\mathscr{F}}(M)$ (resp.~in $\Gamma^{\mathscr{F}}(M)$) if $\sigma$ $\mathscr{F}$-points to $\sigma'$.
An \emph{extension path} (in $M$) is a directed path in $\Gamma(M)$, and an \emph{extension path} in $\mathscr{F}$ is a directed path in $\Gamma_{\mathscr{F}}(M)$ (resp. ~$\Gamma^{\mathscr{F}}(M)$).
The following proposition is immediate from the definitions.

\begin{prop}\label{prop:subgraph}
Let $N$ be a subquotient of $M$.
If $\mathscr{F}$ is a semisimple filtration on $M$, then it naturally induces a semisimple filtration on $N$.
Moreover, $\Gamma(N)$ \emph{(}resp.~$\Gamma_{\mathscr{F}}(N)$ and $\Gamma^{\mathscr{F}}(N)$\emph{)} is the maximal subgraph of $\Gamma(M)$ \emph{(}resp.~$\Gamma_{\mathscr{F}}(M)$ and $\Gamma^{\mathscr{F}}(M)$\emph{)} with vertices corresponding to $\JH(N)$.
\end{prop}

\begin{lemma}\label{lemma:point}
Suppose that $M$ is multiplicity free and $\sigma\in \JH(M)$.
If $\mathrm{d}_{\rad}^M(\sigma)>0$ \emph{(}resp. $\mathrm{d}^{\soc}_M(\sigma)>0$\emph{)}, there is a $\sigma'$ such that $\mathrm{d}_{\rad}^M(\sigma')=\mathrm{d}_{\rad}^M(\sigma)-1$ \emph{(}resp. $\mathrm{d}^{\soc}_M(\sigma')=\mathrm{d}^{\soc}_M(\sigma)-1$\emph{)} and 
 $\sigma'$ $\rad$-points to $\sigma$ \emph{(}resp. $\sigma$ $\soc$-points to $\sigma'$\emph{)}.
\end{lemma}
\begin{proof}
We consider the radical filtration; the socle filtration is analyzed similarly.
Let $d = \mathrm{d}_{\rad}^M(\sigma)$.
Then the radical filtration on $\rad^{d-1} (M)$ is a shift of the radical filtration on $M$ by definition and induces the radical filtration on $\rad^{d-1} (M)/\rad^{d+1}(M)$.
Then there is a $\sigma'\in \JH(\gr^{d-1}_{\rad}(M))$ which $\rad$-points to $\sigma$, otherwise $\sigma$ would not be in the radical of $\rad^{d-1} (M)/\rad^{d+1}(M)$ which is $\gr^d_{\rad}(M)$.
\end{proof}

\begin{cor}\label{cor:path}
Suppose that $M$ is multiplicity free and $\sigma\in \JH(M)$.
Then there is an extension path in the radical \emph{(}resp. ~socle\emph{)} filtration of length $\mathrm{d}_{\rad}^M(\sigma)$ \emph{(}resp. ~$\mathrm{d}^{\soc}_M(\sigma)$\emph{)} ending \emph{(}resp.~beginning\emph{)} with $\sigma$.
\end{cor}
\begin{proof}
The case $\mathrm{d}_{\rad}^M(\sigma) = 0$ (resp. $\mathrm{d}^{\soc}_M(\sigma) = 0$) is trivial.
The induction step follows from Lemma \ref{lemma:point}.
\end{proof}

\begin{lemma}\label{lemma:pathdist}
Suppose that $M$ is multiplicity free and $\mathscr{F}$ is an increasing semisimple filtration on $M$.
If $\sigma$ points to $\sigma'$ then $d^{\mathscr{F}}_M(\sigma)>d^{\mathscr{F}}_M(\sigma')$.
\end{lemma}
\begin{proof}
Let $d$ be $d^{\mathscr{F}}_M(\sigma)$.
By Proposition \ref{prop:subclosed}, $\mathscr{F}_d (M)$ contains $\sigma'$ as a Jordan--H\"older factor.
Thus $d^{\mathscr{F}}_M(\sigma)\geq d^{\mathscr{F}}_M(\sigma')$.
If $d^{\mathscr{F}}_M(\sigma)= d^{\mathscr{F}}_M(\sigma')$, then there is a subquotient of $M$ which is isomorphic to a direct sum of $\sigma$ and $\sigma'$.
This contradicts the fact that $\sigma$ points to $\sigma'$.
\end{proof}

\begin{prop}\label{prop:rigcrit}
Suppose that $M\in \cC$ is multiplicity free.
Then $M$ is rigid if and only if for every $\sigma \in \JH(M)$, there is an extension path in the radical \emph{(}resp. socle\emph{)} filtration of length $\ell\ell(M)-1 - \mathrm{d}_{\rad}^M(\sigma)$ \emph{(}resp. $\ell\ell(M)-1 - \mathrm{d}^{\soc}_M(\sigma)$\emph{)} beginning \emph{(}resp.~ending\emph{)} with $\sigma$.
\end{prop}
\begin{proof}
First suppose that $M$ is rigid.
There is an extension path $P_{\rad}$ in the radical filtration of length $\mathrm{d}_{\rad}^M(\sigma)$ ending at $\sigma$ by Corollary \ref{cor:path}.
By rigidity, this is an extension path in the socle filtration of length $\mathrm{d}_{\rad}^M(\sigma) = \ell\ell(M)-1 - \mathrm{d}^{\soc}_M(\sigma)$.

Now suppose that there is an extension path in the radical filtration of length $\ell\ell(M)-1 - \mathrm{d}_{\rad}^M(\sigma)$ starting with $\sigma$.
Then 
\[\mathrm{d}^{\soc}_M(\sigma) \geq \ell\ell(M)-1 - \mathrm{d}_{\rad}^M(\sigma)\]
by Lemma \ref{lemma:pathdist}.
The reverse inequality is implied by (\ref{eqn:Loewy}), and we conclude that $\gr_{\rad}^n(M)$ and $\gr^{\soc}_{\ell\ell(M)-1-n}(M)$ are isomorphic, and thus that $M$ is rigid.
\end{proof}

The following is self evident.

\begin{prop}\label{prop:dualgraph}
Let $M\in \cC$ be multiplicity free.
Then the dual object $M^*$ in the dual abelian category $\cC^*$ is also multiplicity free.
A decreasing filtration $\mathscr{F}$ on $M$ gives rise to an increasing filtration $\mathscr{F}^*$ on $M^*$.
Then the map sending $\sigma\in \JH(M)$ to $\sigma^* \in \JH(M^*)$ extends to isomorphisms of directed graphs $\Gamma(M) \risom \Gamma(M^*)$ and $\Gamma_{\mathscr{F}}(M)\risom \Gamma^{\mathscr{F}^*}(M)^*$ where $-^*$ denotes the transpose of a directed graph.
In particular, if $\sigma_0 \ra \sigma_1 \ra \cdots \ra \sigma_n$ is an extension path in $\mathscr{F}$, then $\sigma_n^* \ra \sigma_{n-1}^* \ra \cdots \ra \sigma_0^*$ is an extension path in $\mathscr{F}^*$.
\end{prop}

\subsubsection{The main result} \label{sec:mainresult}
We now use the notation from \S \ref{subsec:gencat} with $\cC$ the category of finite $\F[\rG]$-modules.
Let $R$ be a 2-generic Deligne--Lusztig representation of $\rG$ over $E$.
By \cite[Appendix, Theorem 3.4]{herzig-duke} (see the proof of Proposition \ref{prop:typedecomp}), $R$ is residually multiplicity free.
We will show that the elements in $\JH(\ovl{R})$ are $p$-regular in Lemma \ref{lem:deep:type}.
If $\sigma\in \JH(\ovl{R})$, then there exists a unique (up to homothety) $\cO$-lattice $R^\sigma\subseteq R$ with irreducible cosocle isomorphic to $\sigma$ by \cite[Lemma 4.1.1]{EGS}, and we write $\ovl{R}^\sigma$ to denote its reduction modulo $\varpi$. 
In \S \ref{sec:graph}, we defined a distance function $\mathrm{d}_{\mathrm{gph}}$ on $p$-regular Serre weights. 
We now want to relate the graph distance to the submodule structure of $\ovl{R}^\sigma$. 
To simplify notation, we fix $R$ and write $\dcosoc{\sigma}{\sigma'}$ for $\mathrm{d}_\rad^{\ovl{R}^{\sigma}}(\sigma')$.

\begin{defn}
Let $V$ be a set of (isomorphism classes of) weights $\sigma = F(\mu)$ with $\mu$ $p$-regular.
Let $\sigma\in V$.
Let $\Gamma$ be a directed graph with vertex set $V$.
Then we say that $\Gamma$ is \emph{predicted by the extension graph with respect to $\sigma$} if there is an edge from $\kappa_1$ to $\kappa_2$ if and only if $\dgr{\kappa_1}{\kappa_2} = 1$ and $\dgr{\sigma}{\kappa_1} \leq \dgr{\sigma}{\kappa_2}$ (see Definition \ref{defn:dgph}).
\end{defn}

Our main result on the representation theory side is the following:
\begin{thm}
\label{thm structure}
Let $R$ be as above and $13$-generic.
Let $\sigma\in \JH(\ovl{R})$. Then 
\begin{enumerate}
\item \label{item:loewy} $\dgr{\sigma}{\kappa} = \dcosoc{\sigma}{\kappa}$ for all $\kappa\in \JH(\ovl{R})$, in particular, 
\[
\ell\ell(\ovl{R}^\sigma)=2\Def_R(\sigma)+3(f-\Def_R(\sigma))+1;
\]
\item \label{item:radpt} $\Gamma_{\rad}(\ovl{R}^{\sigma})$ is predicted by the extension graph with respect to $\sigma$;
\item \label{item:point} $\Gamma(\ovl{R}^{\sigma})$ is predicted by the extension graph with respect to $\sigma$; 
\item \label{item:rigid} $\ovl{R}^\sigma$ is rigid; and
\item \label{item:sat} if $\kappa\in\JH(\ovl{R})$ and $R^{\kappa}\into R^{\sigma}$ is a saturated inclusion, then 
$p^{\dgr{\kappa}{\sigma}}R^{\sigma}\into R^{\kappa}$ is a saturated inclusion.
\end{enumerate}
\end{thm}

\begin{rmk} \label{rmk:point}
By Proposition \ref{prop:typedecomp}, every maximal geodesic in $\JH(\ovl{R})$ starting from $\sigma$ has the same length.
Then item (\ref{item:rigid}) follows from items (\ref{item:loewy}), (\ref{item:radpt}), and Proposition \ref{prop:rigcrit}.
Furthermore, (\ref{item:point}) and Proposition \ref{prop:subclosed} gives a classification of submodules of $\ovl{R}^\sigma$, from which one can easily deduce items (\ref{item:loewy}) and (\ref{item:radpt}).
\end{rmk}

The proof of Theorem \ref{thm structure} will be carried out in the following subsections.

\subsection{Injective envelopes}
\label{sec:inj:env}

We now relax our hypotheses on $\un{G}_0$, but keep much of the related notation.
Let $\un{G}_0$ be a connected reductive group over $\F_p$ and let $\un{G}$ be the base change $\un{G}_0 \times_{\Fp} \F$.
Assume that $\un{G}$ is split, and isomorphic to $\un{G}_s \times_{\F_p} \F$ where $\un{G}_s$ is a connected split reductive group over $\F_p$.
Let $\un{G}_0^{\der}$ be the derived subgroup of $\un{G}_0$ and let $\un{G}^{\der}$ be the base change $\un{G}_0^{\der} \times_{\F_p} \F$.
Assume that $\un{G}_0^{\der}$ is \emph{simply connected}.
Let $\rG$ (resp.~$\rG^{\der}$) be the finite group $\un{G}_0(\F_p)$ (resp. ~$\un{G}_0^{\der}(\F_p)$).
Let $F:\un{G} \ra \un{G}$ denote the relative Frobenius with respect to $\un{G}_s$.
There is an automorphism $\pi$ of $\un{G}_s$, and hence its based root datum, so that $F \circ \pi: \un{G} \ra \un{G}$ is the relative Frobenius with respect to $\un{G}_0$.
This definition of $\pi$ is consistent with the special case introduced in \S \ref{subsection:Notation}.
Let $h$ be the Coxeter number of $\un{G}$.
We will eventually specialize to the case where $\un{G}$ is a product of copies of $\GL_3$, so that $h=3$.

We define $\Proj(\sigma)$ to be the projective hull of $\sigma$ in the category of $\F[\rG]$-modules. As $\F[\rG]$ is a Frobenius algebra, we have an isomorphism $\Proj(\sigma)\cong \Inj(\sigma)$ where $\Inj(\sigma)$ denotes the injective envelope of $\sigma$ (again in the category of $\F[\rG]$-modules, cf.~ \cite[\S 6, Theorems 4 and 6]{alperin}).

\subsubsection{Algebraic groups, Frobenius kernels, and finite groups}

In this section, we compare injective envelopes of  weights for representations of Frobenius kernel and of finite groups.
One goal is to prove that, under genericity conditions, the graph distance introduced in \S \ref{sec:gph:1} can be characterized in terms of non-vanishing of $\mathrm{Ext}^1_{\rG}$-groups (Lemma \ref{lem:dic}).
We moreover introduce and describe a non $p$-restricted Weyl module $V_\mu$ which will play a key role in the proof of a particular case of Theorem \ref{thm structure}

Let $\un{G}_1$ denote the Frobenius kernel of $\un{G}$ and $\un{G}_1\un{T}$ denote the product of $\un{T}$ and $\un{G}_1$ in $\un{G}$.
We similarly define $\un{G}^{\der}_1$ and $\un{G}^{\der}_1\un{T}^{\der}$.
The $\un{G}$-representation $L(\mu)$ remains irreducible when restricted to $\un{G}_1\un{T}$ and $\un{G}_1$ (cf.~ \cite[II.3.10 and II.9.6]{RAGS}) %
and we will use the standard notations $L_1(\mu)=L(\mu)|_{\un{G}_1}$, $\widehat{L}_1(\mu)=L(\mu)|_{\un{G}_1\un{T}}$ (cf.~ \emph{loc.~ cit.}).
We write $\widehat{Q}_1(\mu)$ to denote the injective envelope of the irreducible representation $\widehat{L}_1(\mu)$  in the category of $\un{G}_1\un{T}$-modules.
It restricts to an injective envelope of $L_1(\mu)$ in the category of $\un{G}_1$-modules. %
As in the case of the finite group $\rG$, it is isomorphic to a projective cover of $L_1(\mu)$ as well.
We recall the following important result.

\begin{thm}
\label{thm:dec:tens}
Assume that $p> 2(h-1)$ and that $\un{G}$ has no factors of type $A_1$.
Then $\wdht{Q}_1(\mu)$ has a unique $\un{G}$-module structure which will be denoted by $Q_1(\mu)$ in what follows. 
In particular, $\soc_{\un{G}} Q_1(\mu)$ is isomorphic to $L(\mu)$ since $L(\mu)$ is the unique extension of $\wdht{L}_1(\mu)$.
Assume that $\mu\in X_1(\un{T})$ is $h-2$-deep. Then
\begin{align}
\label{res:G}
Q_1(\mu)|_G\cong \Proj(F(\mu))\cong \Inj(F(\mu)).
\end{align}
\end{thm}
\begin{proof}
The first part of the theorem is well known, cf.~ \cite[II.11.11]{RAGS}.
In what follows we deduce the isomorphism (\ref{res:G})  from  \cite[Lemma 6.1]{PillenDL} (where it is stated when $\un{G}$ is semisimple). %
It suffices to show that $Q_1(\mu)|_{\rG}$ is injective and its socle is $F(\mu)$.

\paragraph{\textbf{Claim 1.}} 
Let $M$ be a $\un{G}$-module. Then $\soc_{\un{G}^{\der}}(M|_{\un{G}^{\der}})=\soc_{\un{G}}(M)|_{\un{G}^{\der}}$.
The analogous statement holds true for $\un{G}_1\un{T}$ and $\un{G}^{\der}_1\un{T}^{\der}$ and for the finite groups $\rG$ and $\rG^{\der}$.
\begin{proof}[Proof of Claim 1]
Let  $\un{Z}\defeq \un{G}/\un{G}^{\der}$ (resp. ~$\un{Z}_1\defeq \un{G}_1/\un{G}^{\der}_1$).
For $\bullet\in\{\emptyset,\ 1\}$ the group $\un{Z}_\bullet$ is diagonalizable and hence, by the Hochschild--Serre spectral sequence \cite[I.6.9(3)]{RAGS}, the restriction functor $\Res^{\un{G}_\bullet}_{\un{G}^{\der}_\bullet}$  (which is exact) induces a canonical isomorphism
\begin{equation}
\label{iso:HS1}
\Ext_{\un{G}_\bullet}^i(M,\ N)\stackrel{\sim}{\longrightarrow}H^0\big(\un{Z}_\bullet,\Ext^i_{\un{G}^{\der}_\bullet}(M|_{\un{G}^{\der}_\bullet},\ N|_{\un{G}^{\der}_\bullet}\big)\big)
\end{equation}
for all $i\in\N$ and all $\un{G}_\bullet$-modules $M,\ N$. Since the $\un{G}^{\der}_\bullet$-restriction of an irreducible $\un{G}_\bullet$-module remains irreducible, we conclude that $\Res^{\un{G}_\bullet}_{\un{G}^{\der}_\bullet}$ commutes with formation of socles and cosocles.
By \cite[II.9.6 (11)]{RAGS},  
this implies the required statement for the groups $\un{G}_1\un{T}$, $\un{G}^{\der}_1\un{T}^{\der}$. 

Since $Z\defeq \rG/\rG^{\der}$ has order prime to $p$, we also have a canonical isomorphism
\begin{equation}
\label{iso:HS1:fin}
\Ext_{\rG}^i(M,\ N)\stackrel{\sim}{\longrightarrow}H^0\big(\mathrm{Z},\Ext^i_{\rG^{\der}}(M|_{\rG^{\der}},\ N|_{\rG^{\der}}\big)
\end{equation}
for all $i\in\N$ and all $\rG$-modules $M,\ N$ which implies the statement in the case of finite groups.
\end{proof}

\paragraph{\textbf{Claim 2.}}
Let $\nu$ be a $p$-restricted weight.
Then $\wdht{Q}_1(\nu)|_{\un{G}^{\der}_1\un{T}^{\der}}$ is the injective envelope of $\wdht{L}_1(\nu)|_{\un{G}^{\der}_1\un{T}^{\der}}$ as a $\un{G}^{\der}_1\un{T}^{\der}$-module. 
\begin{proof}[Proof of Claim 2.]
By Claim 1, the socle of $\wdht{Q}_1(\nu)|_{\un{G}^{\der}_1\un{T}^{\der}}$ is isomorphic to $\wdht{L}_1(\nu)|_{\un{G}^{\der}_1\un{T}^{\der}}$.
It suffices to prove injectivity.
By \cite[II.9.4]{RAGS}, it is enough to prove injectivity for the the restriction to $\un{G}^{\der}_1$. %
As $\un{G}^{\der}_1$ and $\un{G}_1$ are both finite, and $\un{G}^{\der}_1$ is closed in $\un{G}_1$, we deduce that $\Res^{\un{G}_1}_{\un{G}^{\der}_1}$ maps injectives to injectives 
(since it has an exact left adjoint cf.~\cite[I.3.5, I.8.16]{RAGS}). 
\end{proof}
 
We are now ready to prove (\ref{res:G}).
We show that $Q_1(\mu)|_{\rG}$ is an injective $\rG$-module with socle isomorphic to $F(\mu)$.
By \cite[Lemma 6.1]{PillenDL} (which holds also holds in the nonsplit case, cf.~the final remark of \cite[\S 11]{PillenDL}) and Claim 2, we have
\begin{equation}
\label{iso:res}
Q_1(\mu)|_{\rG^{\der}}\cong \Inj(F(\mu)|_{\rG^{\der}}),
\end{equation}
so that $Q_1(\mu)|_{\rG^{\der}}$ is an injective $\rG^{\der}$-module with socle isomorphic to $F(\mu)|_{\rG^{\der}}$.

We first show injectivity.
The functor of $\rG/\rG^{\der}$-invariants is exact on the category of $\rG/\rG^{\der}$-representations.
We hence obtain a canonical isomorphism
\[
\Hom_{\rG}(\bullet,Q_1(\mu)|_{\rG})\cong 
\big(\Hom_{\rG'}(\bullet,Q_1(\mu)|_{\rG^{\der}})\big)^{\rG/\rG^{\der}}
\]
and $\Hom_{\rG}(\bullet,Q_1(\mu)|_{\rG})$, being the composite of two exact functors, is therefore exact.

The socle of $Q_1(\mu)|_{\rG}$ contains a submodule isomorphic to $F(\mu)$ and its restriction to $\rG^{\der}$ is isomorphic to $F(\mu)|_{\rG^{\der}}$.
Thus, the socle of $Q_1(\mu)|_{\rG}$ is isomorphic to $F(\mu)$.
(We are grateful to the referee for simplifying the argument in our first version).
\end{proof}

Recall that an irreducible $\un{G}$-module $L(\kappa)$, with $\kappa\in X^*_+(\un{T})$, is said to be $p$-bounded if $\langle \kappa, \alpha^\vee\rangle<2(h-1)p$ for all coroots $\alpha\in\un{R}^\vee$; a $\un{G}$-module is $p$-bounded if all its Jordan-H\"older factors are $p$-bounded. 
Similarly, a $\un{G}$-module is defined to be $m$-deep if the highest weights of all its Jordan--H\"older factors are $m$-deep.
The following lemmas will be used several times in the rest of this section.

\begin{lemma}[\cite{PillenPS}, Lemma 3.1]
\label{lem:res:soc}
Let $M$ be a $\un{G}$-module. If $M$ is $3(h-1)$-deep with $p$-bounded highest weight, then
\begin{align*}
\soc^{\un{G}}_i(M)|_{\rG}=\soc^{\rG}_i(M|_{\rG}),&& \rad_{\un{G}}^i(M)|_{\rG}=\rad_{\rG}^i(M|_{\rG})
\end{align*}
\end{lemma}
\begin{proof}
The statement on the socle filtration for $\un{G}^{\der}$ and $\rG^{\der}$ follows from the proof of \cite[Lemma 3.1]{PillenPS}.
While \emph{loc. ~cit.} ~assumes that $\un{G}^{\der}_0$ is split over $\F_p$, the proof applies setting $n=1$ and using that $L(p\lambda_1)|_{\rG^{\der}}$ is isomorphic to $L(\pi\lambda_1)|_{\rG^{\der}}$ rather than $L(\lambda_1)|_{\rG^{\der}}$.
By duality, noting that $M$ is $3(h-1)$-deep if and only if  its linear dual $M^*$ is $3(h-1)$-deep, and that $\soc^{\bullet}_{\ell\ell(M)-i}(M^*)=(M/\rad_{\bullet}^i(M))^*$ for $\bullet\in\{\rG^{\der},\ \un{G}^{\der}\}$, we obtain the analogous statement for the radical filtration (recall that $\ell\ell(M)$ is the Loewy length of $M$).
The general case follows from Claim 1 in the proof of Theorem \ref{thm:dec:tens}.
\end{proof}

\begin{cor}\label{cor:projfil}
Let $\mu \in X_1(\un{T})$ such that $Q_1(\mu)$ is $3(h-1)$-deep.
Then
\begin{align*}
\soc^{\un{G}}_i(Q_1(\mu))|_{\rG}=\soc^{\rG}_i(\Inj(F(\mu))),&& \rad_{\un{G}}^i(Q_1(\mu))|_{\rG}=\rad_{\rG}^i(\Inj(F(\mu))).
\end{align*}
\end{cor}
\begin{proof}
This follows from Theorem \ref{thm:dec:tens} and Lemma \ref{lem:res:soc}.
\end{proof}

If $\nu,\kappa\in X^*(\un{T})$ we let $m_{\kappa}(\nu)\defeq \dim_\F(L(\kappa))_{\nu}$.
Moreover we write $\nu\in L(\kappa)$ as a shorthand for $(L(\kappa))_{\nu}\neq 0$.

\begin{lemma}[Translation principle]
\label{lem:trans:princ}
Assume that $p\geq2(h-1)$. 
Let $\lambda,\ \xi\in X_1(\un{T})$. Assume that for all weights $\nu\in L(\xi)$, the weights $\lambda+\nu$ belongs to the same alcove as $\lambda$. Then we have the following isomorphism of $\un{G}$-modules:
\begin{enumerate}
	\item\label{it:tr:princ:1} $L(\lambda)\otimes_{\F}L(\xi)=\underset{\nu\in L(\xi)}{\bigoplus}L(\lambda+\nu)^{\oplus m_{\xi}(\nu)}$,
	\item\label{it:tr:princ:2} $Q_1(\lambda)\otimes_{\F}L(\xi)=\underset{\nu\in L(\xi)}{\bigoplus}Q_1(\lambda+\nu)^{\oplus m_{\xi}(\nu)}$.
\end{enumerate}
\end{lemma}
\begin{proof}
We first prove item (\ref{it:tr:princ:1}).
The isomorphism holds upon restriction to $\un{G}^{\der}_1$ by \cite[Lemma 5.1]{PillenDL}. As the LHS and RHS of (\ref{it:tr:princ:1}) have the same central characters, the $\un{G}^{\der}_1$-isomorphisms extend to $\un{G}$-isomorphisms by (\ref{iso:HS1}).

We now switch to item (\ref{it:tr:princ:2}). 
By the same argument as above, we deduce from \cite[Lemma 5.1]{PillenDL} a $\un{G}_1\un{T}$-equivariant isomorphism 
\begin{equation}
\label{eq:iso:GTf}
\widehat{Q}_1(\lambda)\otimes_{\F}\widehat{L}_1(\xi) \cong \underset{\nu\in \widehat{L}_1(\xi)}{\bigoplus} \widehat{Q}_1(\lambda+\nu)^{\oplus m_{\xi}(\nu)}.
\end{equation}
By item (\ref{it:tr:princ:1}) and  the isomorphism $\soc_{\un{G}}(Q_1(\lambda))\cong L(\lambda)$ from Theorem \ref{thm:dec:tens},
we have a $\un{G}$-equivariant injection
\begin{equation}
\label{eq:inj:soc}
\soc_{\un{G}}\left(\underset{\nu\in L(\xi)}{\bigoplus}Q_1(\lambda+\nu)^{\oplus m_{\xi}(\nu)} \right)=
\underset{\nu\in L(\xi)}{\bigoplus}L(\lambda+\nu)^{\oplus m_{\xi}(\nu)}=L(\lambda)\otimes_{\F}L(\xi)\into Q_1(\lambda)\otimes_{\F}L(\xi).
\end{equation}
We claim that in the full subcategory of $p$-bounded $\un{G}$-modules the functor $\Hom_{\un{G}}\left(\bullet,Q_1(\lambda)\otimes_{\F}L(\xi)\right)$ is exact.
Granting the claim we deduce from (\ref{eq:inj:soc}) a $\un{G}$-equivariant morphism 
\[
\underset{\nu\in L(\xi)}{\bigoplus} Q_1(\lambda+\nu)^{\oplus m_{\xi}(\nu)}\into Q_1(\lambda)\otimes_{\F}L(\xi)
\]
which is injective since it is injective on socles.
The morphism is hence an isomorphism by  (\ref{eq:iso:GTf}) (note that $\widehat{L}_1(\xi)_{\nu}=L(\xi)_{\nu}$ for all $\nu\in X_1(\un{T})$ since $\xi\in X_1(\un{T})$) and item (\ref{it:tr:princ:2}) follows.

We prove the claim. It will be enough to prove that for any irreducible, $p$-bounded $\un{G}$-module $L(\kappa)$ one has \begin{equation} \label{eqn:extvanish}
\Ext^1_{\un{G}}(L(\kappa), Q_1(\lambda)\otimes_{\F}L(\xi))=0.
\end{equation}
As $L(\kappa)$ is $p$-bounded we can write $\kappa=\kappa^{(0)}+p\omega_{\kappa}$ where $\kappa^{(0)}\in X_1(\un{T})$ and $\omega_\kappa\in X^*_+(\un{T})$ satisfies $\langle \omega_{\kappa},\alpha^{\vee}\rangle<2(h-1)$ for all coroots $\alpha^\vee\in\un{R}^\vee$. 
Recall that $Q_1(\lambda)|_{\un{G}_1\un{T}}\cong\widehat{Q}_1(\lambda)$.
As $\widehat{Q}_1(\lambda)$ is injective as a $\un{G}_1$-module the Lyndon--Hochschild--Serre spectral sequence \cite[I.6.6(3), I.6.5(2)]{RAGS}, together with (\ref{eq:iso:GTf}), provides us with an isomorphism
\begin{align*}
&\Ext_{\un{G}}^1(L(\kappa), Q_1(\lambda)\otimes_{\F}L(\xi))\\
\cong
& \, \Ext_{\un{G}/\un{G}_1}^1\big(L(p\omega_{\kappa}), \Hom_{\un{G}_1}\big(L_1(\kappa^{(0)}), \widehat{Q}_1(\lambda)\otimes_{\F}L_1(\xi)\big)\big)&\\
\cong
& \, \Ext_{\un{G}/\un{G}_1}^1\bigg(L(p\omega_{\kappa}), \Hom_{\un{G}_1}\big(L_1(\kappa^{(0)}), \underset{\nu\in \widehat{L}_1(\xi)}{\bigoplus} \widehat{Q}_1(\lambda+\nu)\big)^{\oplus m_{\xi}(\nu)}\bigg)&\\
\cong& \, \left\{\begin{matrix}
 \Ext_{\un{G}}^1\Big(L(\omega_{\kappa}), L\big(\frac{\lambda+\nu-\kappa^{(0)}}{p}\big)\Big)^{\oplus m_{\xi}(\nu)} &\text{if}&\lambda+\nu-\kappa^{(0)}\in pX^0(\un{T})
\\
0&\text{else.}&
\end{matrix}\right.
\end{align*}
As $\langle \omega_{\kappa},\alpha^{\vee}\rangle<2(h-1)\leq p-2$ for all coroots $\alpha^\vee\in\un{R}^\vee$ it follows that $\omega_{\kappa}$ lies in the lower $p$-restricted alcove; in particular, there are no algebraic extensions between $L(\omega_{\kappa})$ and $L(\omega)$ for any $\omega\in X^0(\un{T})$.
This establishes (\ref{eqn:extvanish}).
\end{proof}

\subsubsection{The case of $\GL_3$}
\label{subsubsec:GL3}

We now describe the modules $Q_1(\mu)|_{\rG}$ in more detail in the case $\un{G}_0$ is $\prod_{\tld{v}\in \mathcal{S}} \Res_{k_{\tld{v}}/\FF_p} \GL_3$ as in \S \ref{subsection:Notation} and $\mu\in X_1(\un{T})$.
We recall the alcove labeling for $\SL_3$ in \cite[\S 13.9]{HumphreysBook} and write \[\mathscr{A}\defeq\{A,\ B,\ C,\ D,\ E,\ F,\ G,\ H,\ I,\ J\}.\]
If $X\in\mathscr{A}$ let $\tld{w}_X\in W_a$ be the unique element such that $\tld{w}_X\cdot C_0=X$.
For $i\in\cJ$ we define $\tld{w}_{X,i}\in\un{W}_a$ in the evident way.
For $\un{X}\in\mathscr{A}^{\cJ}$ we also define $\tld{w}_{\un{X}}\in\un{W}_a$ in the evident way.
In what follows we let $f\defeq \#\cJ$.

\emph{Assume from now on that $p\geq 5$ and $\mu$ is $2$-deep.}
Let $\mu^{\op}$ be ${\tld{w}}_{\un{B}}\cdot \mu$ and write $\mu^{\op}$ as the sum $\sum_i \mu_i^{\op}$.
Let $Q_1(\mu_i)$ be the ${\GL_3}_{/\F}$-module defined in Theorem \ref{thm:dec:tens}.
It is rigid with Loewy length $\ell\ell(Q_1(\mu))=6+1$ and is endowed with a Weyl filtration (cf.~ ~\cite[II.11.13, II.11.5(5), II.4.19]{RAGS} see also \cite[\S 13.9]{HumphreysBook}) with submodule $V_{\mu_i}\defeq V(\mu_i^{\op}+p\eta'_i)$, the ${\GL_3}_{/\F}$-module obtained by extension of scalars from the Weyl module for ${\GL_3}_{/\Fp}$ with highest weight $\mu_i^{\op}+p\eta'_i$.
Moreover, the socle filtration of $Q_1(\mu_i)$ (cf.~ \cite[\S 13.9]{HumphreysBook} for a concise reference; see also \cite[Proposition 8.4]{Andersen-Kaneda} and its proof and \cite[\S II.D.4]{RAGS}) and $V_{\mu_i}$ (cf.~ \cite[\S 4]{BDM}, see Table \ref{TableWeyl1} below) are known.
(The condition that $\mu$ is $2$-deep is to guarantee that all $V_{\mu_i}$ has maximal length; their Loewy length in this case is $3+1$.)
In particular, one sees that $V_{\mu_i}$ is a multiplicity free submodule of $Q_1(\mu_i)$.
Then the $\un{G}$-modules $Q_1(\mu)$ and $V_\mu$ are defined to be the tensor products $\otimes_iQ_1(\mu_i)$ and $\otimes_i V_{\mu_i}$, respectively.

The module $Q_1(\mu)$ is rigid with Loewy length $\ell\ell(Q_1(\mu))=6f+1$.
The socle filtration on $Q_1(\mu)$ is the tensor product of the socle filtrations $\Fil$ on $Q_1(\mu_i)$ for $i \in\cJ$, and the graph $\Gamma(V_\mu)$ is the product $\prod_i \Gamma(V_{\mu_i})$.
In particular, $V_\mu$ is rigid.
Let $\Fil$ be the unique increasing Loewy series for $Q_1(\mu)$; its restriction to $V_\mu$ is the unique Loewy series for $V_\mu$.

Recall that an irreducible $\F[\rG]$-module $F$ is said to be \emph{$n$-deep} if we can write $F\cong L(\mu)|_{\un{G}(\Fp)}$ for some $\mu\in X_1(\un{T})$ which is $n$-deep.
A  $\F[\rG]$-module is defined to be $n$-deep if all its Jordan-H\"older constituents are $n$-deep.

\begin{lemma}
\label{lem:deep}
Let $\mu\in X^*(\un{T})$ and $n\in\N$ and $\tld{w}\in \tld{\un{W}}$.
If $\mu$ is $n$-deep in alcove $a$, then $\tld{w}\cdot\mu$ is $n$-deep in alcove $\tld{w}\cdot a$.
In particular, if $\mu\in X_1(\un{T})$ is $n$-deep, then $Q_1(\mu)$ and $V_{\mu}$ are $n$-deep.
If $\mu\in X_1(\un{T})$ is $n+2$-deep, then $Q_1(\mu)|_{\rG}$ and $V_{\mu}|_{\rG}$ are all $n$-deep.
\end{lemma}
\begin{proof}
The first two claims are easy.
To prove the final claim, first note that the Jordan--H\"older factors of $Q_1(\mu)|_{\rG}$ and $V_{\mu}|_{\rG}$ are the same, and so it suffices to prove the claim for $L|_{\rG}$ where $L \in \JH(V_\mu)$.
Suppose that $L$ is isomorphic to $L(\lambda)$ and $\lambda = \lambda^0+p\omega_\lambda$.
Then 
\[
L|_{\rG} \cong L(\lambda^0+\pi\omega_\lambda) \cong \oplus_{\eps \in L(\pi\omega_\lambda)} L(\lambda+\eps)^{m_{\pi\omega_\lambda}(\eps)}
\]
by Lemma \ref{lem:trans:princ}(\ref{it:tr:princ:1}).
The result now follows from the fact that $|\langle {\eps},\alpha^{\vee}\rangle|\leq 2$ for any ${\eps}\in L(\pi\omega_\lambda)$  and any positive root $\alpha\in R^+$ of a simple factor of $\un{G}$.
\end{proof}

Assume $\mu$ is 6-deep.
By Corollary \ref{cor:projfil}, the socle filtration on $\Inj(F(\mu))$ is given by $(\Fil_n Q_1(\mu))|_{\rG}$.
Since $\Inj(F(\mu))$ is rigid (it is isomorphic to $\Proj(F(\mu))$), this is the unique increasing Loewy series.
One can use Lemma \ref{lem:trans:princ} to compute $\gr_n \Inj(F(\mu))$.
We do this in the case $n=1$ to compute $\rG$-extensions, justifying the name ``extension graph'' introduced in Section \ref{sec:graph}.
\begin{lemma}
\label{lem:dic}
Assume that $\mu_i$ is $6$-deep for all $i$ and let $\sigma\defeq F(\mu)$.
Then $\dgr{\sigma}{\kappa}=1$ if and only if $\Ext^1_{\rG}(\kappa,\sigma)\neq 0$, in which case the dimension of the $\Ext^1$ group is $1$.
\end{lemma}
\begin{proof}
Since $\mu$ is $6$-deep, $Q_1(\mu)$ is $6$-deep by Lemma \ref{lem:deep}, so that $Q_1(\mu)|_{\rG} \cong \Inj(F(\mu))$ by Corollary \ref{cor:projfil}.
It suffices to show that $[\gr_1 Q_1(\mu)|_\rG:\kappa] \leq 1$ and that $\dgr{\sigma}{\kappa}=1$ if and only if $[\gr_1 Q_1(\mu)|_\rG:\kappa] = 1$.
Note that $\gr_1 Q_1(\mu) = \oplus_i (\gr_1 Q_1(\mu_i)) \otimes \bigotimes_{j\neq i} L(\mu_j)$.
Let $\tld{w}_\mu$ be the element of $\un{W}_a$ so that $\lambda \defeq \tld{w}_\mu^{-1}\cdot \mu$ is in $\un{A}$.
The length of $\gr_1 Q_1(\mu)$ is $3f$ with Jordan--H\"older factors of the form $L(\tld{w}\cdot \mu)$ for $3f$ choices of $\tld{w}$.
Writing $\tld{w}$ as $t_{\omega_-}\tld{w}_+$ with $\tld{w}_+\in\tld{\un{W}}^+_1$, the $3f$ choices of $\tld{w}$ correspond to $\omega_- = 0$, $\eps'_{1,i}$, or $\eps'_{2,i}$ for some $i \in \cJ$, with $\tld{w}_+$ the unique element in $\tld{\un{W}}^+_1$ so that $t_{\omega_-}\tld{w}_+$ is in $\un{W}_a$ and $\tld{w}_+ \tld{w}_\mu \cdot \un{A}=\tld{w}_{B,i} \tld{w}_\mu \cdot \un{A}$.

It suffices to show that 
\[L(\tld{w}\cdot \mu)\vert_{\rG}\]
is multiplicity free and contains exactly the weights of the form $F(\Trns_{\lambda+\eta}(\omega,\pi(\tld{w}_{B,i}\tld{w}_\mu \cdot \un{A})))$ with $\omega$ a permutation of $\pi\omega_-$. 
(Here, $\omega$ as the first argument of $\Trns_{\lambda+\eta}$ is understood to be the image of $\omega$ in $\un{\Lambda}_W$.)

We have isomorphisms
\begin{equation}\label{eqn:tp}
L(\tld{w}\cdot \mu)\vert_{\rG}\cong L(\tld{w}_+\cdot \mu) \otimes L(\pi\omega_{-,i}) \vert_{\rG}\cong\oplus_{\omega\in L(\pi\omega_{-,i})} F(\tld{w}_+\cdot \mu +\omega)^{\oplus m_{\pi\omega_{-,i}}(\omega)},
\end{equation}
where the second isomorphism follows from Lemma \ref{lem:trans:princ}.
On the other hand, the pair $(\omega,\pi(\tld{w}_{+,i} \tld{w}_\mu \cdot \un{A}))$ is $\beta(\omega,\tld{w}_{+,i}\tld{w}_\mu)$ with $\beta$ as in Lemma \ref{lem:eq:Weyl} since we have that $t_\omega \pi(\tld{w}_{+,i}\tld{w}_\mu) \in \un{W}_a$.
Then by definition, we have 
\begin{equation}\label{eqn:trans}
F(\Trns_{\lambda+\eta}(\omega,\pi(\tld{w}_{+} \tld{w}_\mu \cdot \un{A})))\cong F(\tld{w}_+ \tld{w}_\mu \cdot (\lambda+\omega)) \cong F(\tld{w}_+\cdot \mu +w_+w_\mu \omega), 
\end{equation}
where $w_+w_\mu$ is the image of $\tld{w}_+\tld{w}_\mu$ in $\un{W}$.
We conclude by combining (\ref{eqn:tp}) and (\ref{eqn:trans}).
\end{proof}

\subsubsection{Study of the Weyl module $V_\mu|_{\rG}$}
\label{subsub:Weyl}

We now assume that $\mu\in X_1(\un{T})$ is such that $\mu\in \un{B}$ and push the analysis in Lemma \ref{lem:dic} further to describe $V_\mu|_{\rG}$ in this case.
Recall that $\underline{G}_0 \defeq \Res_{k_{\tld{v}}/\Fp} \GL_3$ and that $\un{G}
\defeq \un{G}_0 \times_{\F_p} \F$.
Recall that $\underline{T} \subset \un{G}$ is the diagonal torus, and $\un{\La}_W$ is the weight lattice of $\un{G}^{\der}$.
Let $\un{\La}_{\eta'}\subseteq \un{\La}_W$ be the convex hull of the $\un{W}$-orbit of $\eta'$.
Explicitly, we have
\begin{align*}
\un{\La}_{\eta'}=\left\{
\begin{aligned}
(\nu_i)_i\in \un{\La}_W,\ \nu_i\in \{0,\ \pm\eps_{1,i},\ \pm\eps_{2,i},\ \pm \eta'_i,\  \pm(\eps_{1,i}-\eps_{2,i}),\ \pm(2\eps_{1,i}-\eps_{2,i}),\ \pm(\eps_{1,i}-2\eps_{2,i})\}
\end{aligned}
\right\}
\end{align*}
We define the subgraph $\un{\La}_{\preceq(\eta',0)}\subseteq \un{\La}_{\eta'}\times \cA$ as follows:
\begin{align*}
\un{\La}_{\preceq{\eta',0}}\defeq \big\{(\omega,a)\in \un{\La}_{\eta'}\times \cA \ : \ \text{$a_{i}=0$ if $\omega_i=w \eta'_i$ for some $w\in W$} \big\}.
\end{align*}
The main result concerning $V_\mu$ is the following.

\begin{thm}\label{thm:Vdecomp}
Let $\mu\in X_1(\un{T})$ be a $p$-restricted weight such that $\mu$ is $2$-deep in alcove $\un{B}$. %
\begin{enumerate}
\item \label{item:Vbij} The translation map $\Trns_{\mu^{\op}+\eta}:\un{\La}_{W}^{(\mu^{\op}+\eta)}\times \cA\rightarrow X_1(\un{T})/\langle (p-\pi)X^0(\un{T})\rangle$ induces a bijection:
\begin{align*}
\Trns_{V_\mu}:\un{\La}_{\preceq(\eta',0)}&\rightarrow \JH(V_{\mu}|_{\rG})\\
(\omega,a)&\mapsto \sigma_{(\omega,a)}.
\end{align*}
\item \label{item:multone} We have
\[
[\gr_d(V_{\mu})|_{\rG}:\sigma_{(\omega,a)}] =
\begin{cases} 
      1 & \textrm{ if } d = \dgr{\sigma_{(0,\un{1})}}{\sigma_{(\omega,a)}} \\
      0 & \textrm{ if } d < \dgr{\sigma_{(0,\un{1})}}{\sigma_{(\omega,a)}}
   \end{cases}
\]
for all $(\omega,a)\in \un{\La}_{\preceq(\eta',0)}$.
\item \label{item:maxext} 
Assume that $\mu$ 6-deep.
Then there exists a $\rG$-submodule $U \subset V_{\mu}|_{\rG}$ such that:
\begin{enumerate}
	\item $U$ is multiplicity free;
	\item $\JH(U)=\JH(V_{\mu}|_{\rG})$;  
	\item if we denote the restriction $\Fil|_U$ by $\Fil$, for any $\sigma\in \JH(U)$ we have $[\gr_d (U):\sigma] =1$
if and only if $d=\dgr{\sigma_{(0,\un{1})}}{\sigma}$; and
	\item if $(\omega,a),\ (\omega',a')\in \un{\La}_{\preceq(\eta',0)}$, then $\sigma_{(\omega,a)}$ points to $\sigma_{(\omega',a')}$ \emph{(}with respect to $U$\emph{)} if and only if $\dgr{\sigma_{(0,\un{1})}}{\sigma_{(\omega,a)}} \geq \dgr{\sigma_{(0,\un{1})}}{\sigma_{(\omega',a')}}$ and $\dgr{\sigma_{(\omega,a)}}{\sigma_{(\omega',a')}} = 1$.
	In particular $\Gamma(U)^\op$ (i.e.~ the graph obtained from $\Gamma(U)$ by reversing the direction of the edges) is predicted by the extension graph with respect to $\sigma_{(0,\un{1})}$.
\end{enumerate}
\end{enumerate}
\end{thm}
\begin{proof}[Proof of \ref{thm:Vdecomp}(\ref{item:Vbij})-(\ref{item:multone})]
We first show that the image of $\Trns_{V_\mu}$ contains $\JH(V_\mu|_{\rG})$.
Let $\un{\tld{w}} \in \un{W}_a$ such that $[V_\mu:L(\tld{w}\cdot \mu)]\neq 0$.
It suffices to show that $\Trns_{V_\mu}$ contains 
\begin{equation} \label{eqn:JHV}
\JH(L(\tld{w}\cdot \mu)|_{\rG}).
\end{equation}
The proof is similar to that of Lemma \ref{lem:dic}.
There is a decomposition $\tld{w} = t_{\omega_-} \tld{w}_+$ where $\tld{w}_+\in \tld{\un{W}}^+_1$.
Again, $L(\tld{w}\cdot \mu)|_{\rG}$ is isomorphic to $L(\tld{w}_+\cdot \mu) \otimes L(\pi\omega_-)|_{\rG}$, which is isomorphic to
\[\underset{\varepsilon\in L(\pi\omega_-)}\bigoplus F(\tld{w}_+\cdot \mu+\varepsilon)^{\oplus m_{\pi\omega_-}(\varepsilon)}\]
by Lemma \ref{lem:trans:princ}.
As in the proof of Lemma \ref{lem:dic}, the summand $F(\tld{w}_+\cdot \mu+\varepsilon)$ is $\sigma_{(w_0\varepsilon,\pi(\tld{w}_+\cdot \tld{w}_h \cdot \un{A}))}$.
Then (\ref{eqn:JHV}) is contained in the image of $\Trns_{V_\mu}$ by an analysis of the weights of $L(\pi\omega_-)$.
Indeed, $\tld{w} \cdot \mu$ is in one of the alcoves in the set $\{A,B,C,D,E,F,G\}^{\cJ}$ so that $\omega_{-,i}$ can be taken to be one of $0$, $\eps_{1,i}$, $\eps_{2,i}$, and $\eta'_i$ for all $i$ and if $\omega_{-,i}$ is $\eta'_i$ then $\tld{w}_{+,i} = \tld{w}_{B,i}$.

Item (\ref{item:Vbij}) follows from (\ref{item:multone}) and the above paragraph.
We now prove item (\ref{item:multone}).
With $\tld{w}$ as above, we define $n_i(\tld{w}_i)\in \N$ by 
\[
[\gr_{n_i(\tld{w}_i)}(V_{\mu_i}):L(\tld{w}_i\cdot \mu_i)]\neq 0.
\]
Let $0\leq d\leq 3f$.
Since $V_\mu$ is multiplicity free, it suffices to show that 
\begin{equation} \label{eqn:Vdist}
\underset{\sum_i n_i(\tld{w}_i) = d}\bigoplus L(\tld{w}\cdot \mu)|_{\rG}
\end{equation}
contains weights of the form $\sigma_{(\omega,a)}$ with multiplicity one if $(\omega,a)\in \un{\La}_{\preceq(\eta',0)}$ and the distance between $(\omega_i,a_i)$ and $(0_i,1_i)$ is $n_{\pi^{-1}i}(\tld{w}_{\pi^{-1}i})$ for all $i$.
If $\nu \in L(\pi\omega_-)$ is a permutation of $\pi\omega_-$, then it appears in $L(\pi\omega_-)$ with multiplicity one.
Thus $\sigma_{(\varepsilon,\pi(\tld{w}_+ \tld{w}_h) \cdot \un{A})}$ appears in (\ref{eqn:Vdist}) with multiplicity one.
A casewise analysis, using the fifth column of Table \ref{TableWeylAlpha} and the description of the socle layers of $V_{\mu_i}$ in \cite{BDM}, shows that the distance from $(\nu_i,\pi(\tld{w}_{+,\pi^{-1}i}\cdot \tld{w}_{B,\pi^{-1}i}) \cdot A)$ to $(0_i,1_i)$ is $n_{\pi^{-1}i}(\tld{w}_{\pi^{-1}i})$.
For example, if $n_{\pi^{-1}i}(\tld{w}_{\pi^{-1}i}) = 2$, then $\omega_{-,\pi^{-1}i}$ is $\eps'_{1,\pi^{-1}i}$ or $\eps'_{2,\pi^{-1}i}$ and $\tld{w}_{+,\pi^{-1}i}$ is trivial.
Then $\nu_i$ is a permutation of $\eps_{1,i}$ or $\eps_{2,i}$.
\end{proof}

We now move to the proof of Theorem \ref{thm:Vdecomp}(\ref{item:maxext}). We start with the following preliminary lemma.

\begin{lemma}
\label{lem:max:multfree}
Let $\mu\in X_1(\un{T})$ be a $p$-restricted weight which is $2$-deep in alcove $\un{B}$.
There exists a $\rG$-submodule $U\subset V_\mu|_{\rG}$ such that $\JH(U)=\JH(V_\mu|_{\rG})$ and 
\begin{equation}\label{eqn:exactdist}
[\gr_d(U)|_{\rG}:\sigma_{(\omega,a)}] =
\begin{cases} 
      1 & \textrm{ if } d = \dgr{\sigma_{(0,\un{1})}}{\sigma_{(\omega,a)}} \\
      0 & \textrm{ if } d \neq \dgr{\sigma_{(0,\un{1})}}{\sigma_{(\omega,a)}},
   \end{cases}
\end{equation}
where $\gr$ is with respect to $\Fil \defeq \Fil|_U$.
In particular, $U$ is multiplicity free.
\end{lemma}
\begin{proof}
The notation in this proof is complicated by necessity.
To illustrate the simple underlying idea, we first present the proof in the case that $f= \#\cJ = 1$.
We have the following: 
\begin{align*}
\gr_3(V_\mu) &\cong L(\tld{w}_G\tld{w}_B\cdot \mu) \\
\gr_2(V_\mu) &\cong L(\tld{w}_E\tld{w}_B\cdot \mu) \oplus L(\tld{w}_F \tld{w}_B\cdot \mu)\\
\gr_1(V_\mu) &\cong L(\tld{w}_C\tld{w}_B\cdot \mu) \oplus L(\tld{w}_B\cdot \mu) \oplus L(\tld{w}_D \tld{w}_B\cdot \mu)\\
\gr_0(V_\mu) &\cong L(\mu).
\end{align*}
For the alcove $a$, there is a decomposition $\tld{w}_a$ as the product $t_{\omega_{a,-}}\tld{w}_{a,+}$ where $\tld{w}_{a,+} \in \tld{W}^+_1$ and $\omega_{a,-} \in X^*(T)$.
Let $\lambda$ be $\tld{w}_B \cdot \mu$.
Then 
\begin{align*}
L(\tld{w}_a \tld{w}_B \cdot \mu)|_{\rG} &\cong L(\tld{w}_a \cdot \lambda) \cong L(\tld{w}_{a,+} \cdot \lambda+p\omega_{a,-})|_{\rG} \cong (L(\tld{w}_{a,+} \cdot \lambda)\otimes L(p\omega_{a,-}))|_{\rG}\\
&\cong (L(\tld{w}_{a,+} \cdot \lambda)\otimes L(\omega_{a,-}))|_{\rG} \cong \oplus_{\omega \in L(\omega_{a,-})} L(\tld{w}_{a,+} \cdot \lambda+\omega)|^{\oplus m_{\omega_{a,-}}(\omega)}_{\rG}\\
&\cong \oplus_{\omega \in L(\omega_{a,-})} F(\Trns_{\lambda+\eta}(\omega,\ovl{a}))^{\oplus m_{\omega_{a,-}}(\omega)},
\end{align*}
where $\ovl{a}$ is the unique $pX^*(T)$-translate of $a$ lying in $\{A,B\}$.
The third isomorphism above follows from the Steinberg tensor product theorem, and the fifth isomorphism follows from Lemma \ref{lem:trans:princ}(\ref{it:tr:princ:1}).
We see then that $\gr_1(V_\mu)|_{\rG}$ is isomorphic to the multiplicity-free direct sum of $F(\Trns_{\lambda+\eta}(\omega,0))$ where $\omega$ is a permutation of $0,\, \eps_1,$ or $\eps_2$.
Similarly, $\gr_2(V_\mu)|_{\rG}$ is isomorphic to the multiplicity-free direct sum of $F(\Trns_{\lambda+\eta}(\omega,1))$ where $\omega$ is a permutation of $\eps_1$ or $\eps_2$.
Finally, $\gr_3(V_\mu)|_{\rG}$ is isomorphic to $\oplus_{\omega\in L(\eta')} F(\Trns_{\lambda+\eta}(\omega,0))^{m_{\eta'}(\omega)}$.
Then, we can take $U$ to be the preimage of $\oplus_{\omega \in L(\eta')} F(\Trns_{\lambda+\eta}(\omega,0))$ in $V_\mu|_{\rG}$.
It is then easy to check that $U$ satisfies the required properties.

We now proceed to the general case.
Let $i\in\cJ$.
We have an exact sequence of $\un{G}$-modules
\[0\ra \rad(V_{\mu_i})\ra V_{\mu_i}\stackrel{p_i}{\ra} L(\mu_i^{\op}+p\eta'_i)\ra 0,\]
which gives the exact sequence
\[0\ra \rad(V_{\mu_i})\otimes \bigotimes_{j\neq i} V_{\mu_j} \ra V_\mu \ra L(\mu_i^{\op}+p\eta'_i)\otimes \bigotimes_{j\neq i} V_{\mu_j}\ra 0.\]
Then $L(\mu_i^{\op}+p\eta'_i)\otimes \bigotimes_{j\neq i} V_{\mu_j}|_{\rG}$ is isomorphic to 
\[L(\mu_i^{\op})\otimes \big( V_{\mu_{\pi i}}\otimes L(\eta'_{\pi i}) \big) \otimes \bigotimes_{j\neq i,\, \pi i} V_{\mu_j}|_{\rG}.\]
Then we claim that \[V_{\mu_{\pi i}}\otimes L(\eta'_{\pi i}) \cong \oplus_{\omega\in L(\eta'_{\pi i})} V_{\mu_{\pi i}+\omega}^{\oplus m_{\eta'_{\pi i}}(\omega)}\]
as $\un{G}$-representations.
We have an embedding 
\begin{equation}\label{eqn:VQembedding}
V_{\mu_{\pi i}}\otimes L(\eta'_{\pi i}) \into Q_1(\mu_{\pi i}) \otimes L(\eta'_{\pi i}) \cong \oplus_{\omega\in L(\eta'_{\pi i})} Q_1(\mu_{\pi i}+\omega)^{\oplus m_{\eta'_{\pi i}}(\omega)},
\end{equation}
where the last isomorphism follows from Lemma \ref{lem:trans:princ}.
Since:
\begin{enumerate}
\item $V_{\mu_{\pi i}}\otimes L(\eta'_{\pi i})$ contains all the Jordan--H\"older factors of the right-hand side of (\ref{eqn:VQembedding}) with highest weights in alcove $G$ (using Lemma \ref{lem:trans:princ}), and
\item $\oplus_{\omega\in L(\eta'_{\pi i})} V_{\mu_{\pi i}+\omega}^{\oplus m_{\eta'_{\pi i}}(\omega)}$ is the minimal submodule of the right-hand side of (\ref{eqn:VQembedding}) containing all the Jordan--H\"older factors (counted with multiplicities) with highest weights in alcove $G$ (using the cosocle filtration of $V_{\mu_{\pi i}+\omega}$),
\end{enumerate}
there is an injective map $\oplus_{\omega\in L(\eta'_{\pi i})} V_{\mu_{\pi i}+\omega}^{\oplus m_{\eta'_{\pi i}}(\omega)} \into V_{\mu_{\pi i}}\otimes L(\eta'_{\pi i})$ which must be an isomorphism since the domain and codomain have the same length.

Let $U_i$ be the preimage of 
\[\underset{\omega \in L(\eta'_i), \, \omega \neq 0}\bigoplus L(\mu_i^{\op})\otimes V_{\mu_{\pi i}+\omega} \otimes \bigotimes_{j\neq i,\, \pi i} V_{\mu_j}|_{\rG}\]
in $V_\mu|_{\rG}$.
Let $U$ be the intersection $\cap_i U_i$.
We claim that $U$ has the desired properties.
If $d<\dgr{\sigma_{(0,\un{1})}}{\sigma_{(\omega,a)}}$, then (\ref{eqn:exactdist}) holds by Theorem \ref{thm:Vdecomp}(\ref{item:multone}).
If $d=\dgr{\sigma_{(0,\un{1})}}{\sigma_{(\omega,a)}}$, then $[\gr_d(V_\mu|_{\rG}):\sigma_{(\omega,a)}] =1$ and $[\gr_d(U_i):\sigma_{(\omega,a)}] =1$ for all $i$ by the proof of Theorem \ref{thm:Vdecomp}(\ref{item:multone}).
We conclude that $[\gr_d(U):\sigma_{(\omega,a)}] =1$.

We now suppose that $d>\dgr{\sigma_{(0,\un{1})}}{\sigma_{(\omega,a)}}$ and use the notation of the proof of Theorem \ref{thm:Vdecomp}(\ref{item:multone}).
Suppose that $\sigma_{(\omega,a)}$ is a Jordan--H\"older factor of $L(\tld{w}\cdot \mu)|_{\rG}\cap \gr_d(U)$.
Then there is some $i$ such that $n_{\pi^{-1}i}(\tld{w}_{\pi^{-1}i}) > d_i$ where $d_i$ is the distance from $(\omega_i,a_i)$ to $(0_i,1_i)$.
More precisely, $n_{\pi^{-1}i}(\tld{w}_{\pi^{-1}i})$ is $3$ and $(\omega_i,a_i)$ is $(0_i,0_i)$.
However, one can check as in the proof of Theorem \ref{thm:Vdecomp}(\ref{item:multone}) that $L(\tld{w}\cdot \mu)|_{\rG}\cap\gr_d (U_{\pi^{-1}i})$ does not contain any weight of the form $\sigma_{(\omega,a)}$ with $(\omega_i,a_i)=(0_i,0_i)$.
\end{proof}

\begin{prop}
\label{prop:Ext:dim}
Let $\lambda,\, \theta$ be $6$-deep in an alcove in $\{A,\ B,\ C,\ D,\ E,\ F,\ G\}^\cJ$. %
Then $\Ext^1_{\un{G}}(L(\theta), L(\lambda))$ is at most one dimensional, and it is one dimensional if and only if $\lambda$ and $\theta$ are linked and lie in adjacent alcoves
(i.e.~there exists $i_0 \in \cJ$ such that $\lambda_i=\theta_i$ for all $i\neq i_0$ and $\lambda_{i_0}$ and $\theta_{i_0}$ lie in different alcoves sharing a face).
\end{prop}
\begin{proof}
We immediately reduce to the case of $\GL_3$.
Using (\ref{iso:HS1}), it suffices to consider the case of $\mathrm{SL}_3$, where the result follows from \cite[\S 4.1]{Andersen87}.
\end{proof}

Let $\mu \in X_1(\un{T})$ be a $p$-restricted weight $6$-deep in alcove $\un{B}$. 
Let $\tld{w},\ \tld{y}\in \un{W}_a$ be elements such that the alcoves containing $\lambda\defeq \tld{w}\cdot \mu, \theta\defeq \tld{y}\cdot \mu$ are in $\{A,\ B,\ C,\ D,\ E,\ F,\ G\}^\cJ$ and $\Ext^1_{\un{G}}(L(\theta), L(\lambda))\neq 0$. 
Note that $\lambda,\, \theta$ are both $6$-deep in their alcove by Lemma \ref{lem:deep}.

Let $i_0 \in \cJ$ be as in Proposition \ref{prop:Ext:dim}.
Let $M$ be a non-split extension of $L(\theta)$ by $L(\lambda)$, which is unique up to isomorphism by Proposition \ref{prop:Ext:dim}.

\begin{prop}
\label{prop:nonsplit}
If $F_0$ and $F_1$ are in the socle and cosocle of $M|_{\rG}$, respectively, with $[M|_{\rG}:F_0]=1=[M|_{\rG}:F_1]$ and $\dgr{F_0}{F_1} = 1$, then there is a subquotient of $M|_{\rG}$ which is a nonsplit extension, unique up to isomorphism by Lemma \ref{lem:dic}, of $F_1$ by $F_0$.
\end{prop}

Proposition \ref{prop:nonsplit} will be proven in several steps. 
Let $\lambda$ be $\lambda^0+p\omega_\lambda$ and $\theta = \theta^0 + p\omega_\theta$ so that $\lambda^0$ and $\theta^0$ are in $X_1(\un{T})$ and $\omega_\lambda$ and $\omega_\theta$ belong to $\{0,\eps'_{1,i}, \eps'_{2,i},\eta'_i\}^{\cJ}$.
We begin with an algebraization lemma.

\begin{lemma}\label{lem:algemb}
With $M$ and $i_0$ defined as above, assume moreover that $\lambda_i$ is $p$-restricted for all $i \neq i_0$.
Then there is an injection $M \into Q_1(\lambda^0) \otimes L(p\omega_\lambda) \cong \big(\otimes_{i\neq i_0} Q_1(\lambda_i)\big) \otimes Q_1(\lambda^0_{i_0}) \otimes L(p\omega_\lambda)$ whose restriction to $\rG$ is an injective hull.
\end{lemma}
\begin{proof}
The proof follows closely the argument of \cite[Lemma 3.1]{PillenPS} (which is in turn based on \cite[Lemma 2.2]{Andersen87}).
There is an injection 
\[\soc\, M \cong L(\lambda^0) \otimes L(p\omega_\lambda) \into Q_1(\lambda^0) \otimes L(p\omega_\lambda).\]
This extends to an injection $M \into Q_1(\lambda^0) \otimes L(p\omega_\lambda)$
since
\[\Ext_{\un{G}}^1(L(\theta), Q_1(\lambda^0) \otimes L(p\omega_\lambda))=0\]
by the proof of (\ref{eqn:extvanish}).
Both $Q_1(\lambda^0) \otimes L(p\omega_\lambda)$ and $M$ are $6$-deep by Lemma \ref{lem:deep}. 
By Lemma \ref{lem:res:soc} the restriction to $\rG$ of the map $M\into Q_1(\lambda^0) \otimes L(p\omega_\lambda)$ is an isomorphism on socles, and is thus essential.
The restriction $Q_1(\lambda^0) \otimes L(p\omega_\lambda)|_{\rG}$ is an injective object by Lemma \ref{lem:trans:princ} and Theorem \ref{thm:dec:tens}.
\end{proof}

We begin with the following special case of Proposition \ref{prop:nonsplit}, from which the general case will follow by the translation principle.

\begin{prop}
\label{prop:nonsplit:res}
If $\lambda_i$ is $p$-restricted for all $i\neq i_0$, the conclusion of Proposition \ref{prop:nonsplit} holds.
\end{prop}
\begin{proof}
We assume that $\dgr{F_0}{F_1} = 1$.
Fix nonzero maps $\Proj(F_1) \ra M|_{\rG}$ and $M|_{\rG} \ra \Inj(F_0)$, unique up to scalar.
If suffices to show that the composition 
\begin{equation}\label{eqn:projinj}
\Proj(F_1) \ra M|_{\rG} \ra \Inj(F_0)
\end{equation}
is nonzero.
The first map factors through $\Proj(F_1)/\rad^2 (\Proj(F_1))$ since the Loewy length of $M|_{\rG}$ is two.
The second map factors as $M|_{\rG}\into \Inj(M|_{\rG}) \surj \Inj(F_0)$ where the second map is a projection to a direct summand.
Applying $\gr_1$ to the composition 
\[\Proj(F_1)/\rad^2 (\Proj(F_1)) \ra M|_{\rG}\into \Inj(M|_{\rG}) \surj \Inj(F_0),\] 
we obtain 
\begin{equation}\label{eqn:gr1}
F_1 \into \gr_1(M|_{\rG}) \into \gr_1 \Inj(M|_{\rG}) \surj \gr_1 \Inj(F_0),
\end{equation}
where the grading is with respect to the socle filtration.
It suffices to show that the composition (\ref{eqn:gr1}) is nonzero.

By Lemma \ref{lem:trans:princ}, $F_0$ (resp.~$F_1$) is isomorphic to $F(\lambda')$ where $\lambda' = \lambda^0+\varepsilon_\lambda$ (resp.~$F(\theta^0+\varepsilon_\theta)$) for some $\varepsilon_\lambda \in L(\pi\omega_\lambda)$ (resp.~$\varepsilon_\theta\in L(\pi\omega_\theta)$).
Then using Lemma \ref{lem:algemb}, (\ref{eqn:gr1}) can be rewritten as
\[
L(\theta^0+\varepsilon_\theta)|_{\rG} \into L(\theta)|_{\rG} \into \big(\gr_1 \big(Q_1(\lambda^0)\otimes L(p\omega_\lambda)\big)\big)|_{\rG} \surj \big(\gr_1 Q_1(\lambda^0+\varepsilon_\lambda)\big)|_{\rG},
\]
where the second map is the restriction of a map of $\un{G}$-modules.
By Lemma \ref{lem:res:soc}, the socle filtrations of 
$Q_1(\lambda^0)\otimes L(p\omega_\lambda)$ and $Q_1(\lambda^0)\otimes L(\pi\omega_\lambda)$ both induce the socle filtration on $Q_1(\lambda^0)\otimes L(p\omega_\lambda)|_{\rG}$.
Applying Lemma \ref{lem:trans:princ} to $Q_1(\lambda^0)\otimes L(\pi\omega_\lambda)$ and using the description of $Q_1(\lambda^0+\nu)$ in \cite[\S 13.9]{HumphreysBook}, we see that the graded pieces of the socle filtration of $Q_1(\lambda^0)\otimes L(\pi\omega_\lambda)$ are $\gr_k\big(Q_1(\lambda^0)\big) \otimes L(\pi\omega_\lambda)$.
We claim that the graded pieces of the socle filtration of $Q_1(\lambda^0)\otimes L(p\omega_\lambda)$ are $\gr_k\big(Q_1(\lambda^0)\big) \otimes L(p\omega_\lambda)$.
Taking the tensor product of the socle filtration on $Q_1(\lambda^0)$ with $L(p\omega_\lambda)$ gives a semisimple filtration $\cF$ (as follows from Steinberg's tensor product theorem).
Moreover, since the restrictions $Q_1(\lambda^0)\otimes L(\pi\omega_\lambda)|_\rG$ and $Q_1(\lambda^0)\otimes L(p\omega_\lambda)|_\rG$ are isomorphic, Lemma \ref{lem:res:soc} implies that the dimensions of the graded pieces of the socle filtrations of $Q_1(\lambda^0)\otimes L(\pi\omega_\lambda)$ and $Q_1(\lambda^0)\otimes L(p\omega_\lambda)$ agree.
A dimension consideration implies that $\cF$ is the socle filtration.

In particular, we have 
\[
\gr_1 \big(Q_1(\lambda^0)\otimes L(p\omega_\lambda)\big) \cong \bigoplus_i \Big(L(p\omega_\lambda) \otimes \big(\gr_1 Q_1(\lambda_i^0)\big) \otimes \bigotimes_{j\neq i}  L(\lambda_j^0)\Big).
\]
The algebraic map $L(\theta) \into \gr_1 (Q_1(\lambda^0) \otimes L(p\omega_\lambda))$ factors through the direct summand
\[L(\lambda^{i_0}) \otimes \big(\gr_1 Q_1(\lambda_{i_0}^0)\big) \otimes L(p\omega_\lambda)\]
by alcove considerations, where $\lambda^{i_0} \defeq \sum_{i\neq i_0} \lambda_i$.
(Note that $\lambda_i=\lambda^0_i$ and $\theta_i=\lambda_i$ for all $i\neq i_0$.)
Additionally, the map $F_1 \ra \gr_1 \Inj(F_0) \cong (\gr_1 Q_1(\lambda'))|_{\rG}$ factors through the direct summand
\[
L(\lambda^{\prime,i_0}) \otimes (\gr_1 Q_1(\lambda'_{i_0}))|_{\rG},
\]
where $\lambda^{\prime,i_0} = \sum_{i\neq i_0} \lambda'_i$, since if $j\neq i_0$, then the highest weights of the Jordan--H\"older factors of 
\[
\Big((\otimes_{i\neq j} L(\lambda_i')) \otimes\big(\gr_1 Q_1(\lambda_j')\big)\Big)|_{\rG}
\]
lie in the same alcove as the highest weight of $L(\lambda^0)$ except exactly at embedding $j$.
Thus it suffices to show that the composition
\begin{equation}\label{eqn:gr1'}
L(\theta^0+\varepsilon_\theta)|_{\rG} \into L(\theta)|_{\rG} \into L(\lambda^{i_0}) \otimes (\gr_1 Q_1(\lambda_{i_0}^0))\otimes L(p\omega_\lambda)|_{\rG} \surj L(\lambda^{\prime,i_0}) \otimes (\gr_1 Q_1(\lambda'_{i_0}))|_{\rG}
\end{equation}
is nonzero.

The $G$-module $\gr_1 Q_1(\lambda_{i_0}^0)$ is the direct sum of three irreducible modules in alcoves $A$, $C$,  and $D$ (resp.~$B$, $E$, and $F$) if $\lambda_{i_0}^0$ is in alcove $B$ (resp.~$A$).
Let $\tld{w}_1$, $\tld{w}_2$, and $\tld{w}_3$ be the elements of $W_a$ such that $\tld{w}_1\cdot A$, $\tld{w}_2\cdot A$, and $\tld{w}_3\cdot A$ are the alcoves $A$, $C$,  and $D$ (resp.~$B$, $E$, and $F$), respectively.
Let $\tld{w}_B\in W_a$ be the element such that $B = \tld{w}_B\cdot A$.
The natural identification $X^*(T)\backslash \tld{W}/\Omega \cong \{A,B\}$ where $\Omega = \Stab_{\tld{W}}(A)$ and the fact that $\theta_{i_0}^0$ is in alcove $A$ (resp.~$B$) if $\lambda_{i_0}^0$ is in alcove $B$ (resp.~$A$) shows that $X^*(T)\tld{y}_{i_0}\tld{w}_B \Omega = X^*(T)\tld{w}_j\Omega$ for any $j\in \{1,2,3\}$.
Moreover, it is not difficult to check that 
\[
X^*(T)\tld{y}_{i_0}\tld{w}_B \Omega = \coprod_{j=1}^3 X^*(T)\tld{w}_j.
\]
This implies that there exists unique $j \in \{1,2,3\}$ and $\omega\in X^*(T)$ such that 
\begin{equation} \label{eqn:omega}
t_{-\omega_{\theta,i_0}} \tld{y}_{i_0}\tld{w}_B(\tld{w}_\lambda t_{-\omega_{\lambda,i_0}} \tld{w}_{i_0}\tld{w}_B)^{-1} = t_{-\omega}\tld{w}_j,
\end{equation}
where $\tld{w}_\lambda \in W_a$ is the unique element such that $\tld{w}_\lambda \cdot \lambda_{i_0}^0 \in A$, or equivalently that $\tld{w}_\lambda t_{-\omega_{\lambda,i_0}} \tld{w}_{i_0}\tld{w}_B \in \Omega$.
Then we have that
\[
\theta^0_{i_0}+p\omega = t_\omega t_{-\omega_{\theta,i_0}}\tld{y}_{i_0} \cdot \mu_{i_0} = \tld{w}_j \tld{w}_\lambda t_{-\omega_{\lambda,i_0}}\tld{w}_{i_0} \cdot \mu_{i_0} = \tld{w}_j \tld{w}_\lambda \cdot \lambda_{i_0}^0,
\]
so that by construction, $\omega\in X^*(T)$ is the unique weight such that $L(\theta^0_{i_0}+p\omega)$ is a Jordan--H\"older factor of $\gr_1 Q_1(\lambda_{i_0}^0)$.
We now consider $\omega$ as an element of $X^*(\un{T})$ in the $i_0$-embedding.

We now proceed casewise.
If $\omega_\lambda$ (resp. $\omega_\theta$) is $0$ then $\gr_0(M|_{\rG})$ (resp. $\gr_1(M|_{\rG})$) is irreducible, and so the map $M|_{\rG} \ra \Inj(F_0)$ is injective (resp. the map $\Proj(F_1) \ra M|_{\rG}$ is surjective) and the composition (\ref{eqn:projinj}) is nonzero.

Now assume that $\omega_\lambda$ and $\omega_\theta$ are both nonzero.
At most one of $\omega_\lambda$ and $\omega_\theta$ can be $\eta'_{i_0}$.
By duality, we can and will assume that $\omega_\lambda\neq \eta'_{i_0}$.
Then $\omega_\lambda$ is either $\eps'_{1,i_0}$ or $\eps'_{2,i_0}$.
We assume that $\omega_\lambda$ is $\eps'_{2,i_0}$, the other case being symmetric.
So if $\lambda_{i_0}$ is in alcove $D$ (resp.~$F$), then $\theta_{i_0}$ is in alcove $E$ or $F$ (resp.~$C$, $D$, or $G$) by Proposition \ref{prop:Ext:dim}.
We claim that if $\omega_\theta$ is $\eps'_{1,i_0}$ (resp.~$\eps'_{2,i_0}$ or $\eta'_{i_0}$), then the $\omega$ defined above is $\eps'_{2,i_0}+(1,1,1)_{i_0}$ (resp.~$0$ or $\eps'_{1,i_0}$).
Indeed, since $t_{-\omega}\tld{w}_j \cdot A \in \{A,B\}$ by (\ref{eqn:omega}) and $\tld{w}_j\cdot A \in \{A,B,C,D,E,F\}$ by definition, $\omega_{i_0}$ is in $\{0,\eps'_1,\eps'_2\}+X^0(T)$.
From (\ref{eqn:omega}), we see that $\omega \equiv \omega_\theta-\omega_\lambda \pmod{\un{\Lambda}_R}$. 
These two facts determine $\omega$.

A morphism 
\[
L(\theta_{i_0}) \cong L(\theta_{i_0}^0+p\omega_\theta) \ra (\gr_1 Q_1(\lambda_{i_0}^0)) \otimes L(p\omega_\lambda)
\]
must factor through $L(\theta_{i_0}^0+p\omega')\otimes L(p\omega_\lambda)$ for some $\omega'$.
By construction, $\omega'$ must be $\omega$.
We conclude that the map 
\[
L(\theta) \into L(\lambda^{i_0}) \otimes (\gr_1 Q_1(\lambda_{i_0}^0)) \otimes L(p\omega_\lambda)
\]
factors through $L(\theta^0 + p\omega) \otimes L(p\omega_\lambda)$.
It suffices to show that the composition 
\[
L(\theta^0+\varepsilon_\theta)|_{\rG} \into L(\theta)|_{\rG} \into L(\theta^0)\otimes L(p\omega) \otimes L(p\omega_\lambda)|_{\rG} \surj L(\lambda^{\prime,i_0}) \otimes (\gr_1 Q_1(\lambda'_{i_0}))|_{\rG}
\]
is nonzero.

Let $\varepsilon'_\lambda\in L(\pi\omega_\lambda)$ be such that $\lambda^0+\varepsilon_\lambda$ and $\theta^0+\varepsilon'_\lambda$ are in the same $\un{\tld{W}}$-orbit under the $p$-dot action.
Similar to an earlier argument, of the three simple Jordan--H\"older factors of $L(\lambda^{\prime,i_0}) \otimes (\gr_1 Q_1(\lambda'_{i_0}))$, only the restriction to $\rG$ of one of the form $L(\theta^0+\varepsilon'_\lambda) \otimes L(p\omega')$ for some $\omega'$ contains $F_1 \cong F(\theta^0+\eps_\theta)$ as a Jordan--H\"older factor.
By construction, $\omega'$ must be $\omega$.
It now suffices to show that the composition 
\begin{align}\label{eqn:gr1''}
L(\theta^0+\varepsilon_\theta)|_{\rG} \into L(\theta^0+\pi\omega_\theta)|_{\rG} \cong L(\theta)|_{\rG} \into &L(\theta^0)\otimes L(p\omega) \otimes L(p\omega_\lambda)|_{\rG} \\
\nonumber \cong &L(\theta^0)\otimes L(\pi\omega_\lambda) \otimes L(p\omega)|_{\rG} \surj L(\theta^0+\varepsilon'_\lambda) \otimes L(p\omega)|_{\rG}
\end{align}
is nonzero.
Moreover, the multiplicity $[L(\theta^0+\varepsilon'_\lambda) \otimes L(p\omega)|_{\rG}:F_1]$ is one.

Assume now that $\omega_\theta$ is not $\eps'_{1,i_0}$. 
Then using the assumption that $[M|_{\rG}:F_1]=1$ so that $\varepsilon_\theta \neq 0$, one can check that $F_1$ appears in 
$L(\theta^0) \otimes L(p\omega)\otimes L(p\omega_\lambda)|_{\rG}$ with multiplicity one, we see that (\ref{eqn:gr1''}) is nonzero.

Finally, we assume that $\omega_\theta$ is $\eps'_{1,i_0}$. 
Let $\varepsilon$ be $\varepsilon_\theta - \varepsilon'_\lambda$ which is in $L(\pi\omega)$ by the composition (\ref{eqn:gr1''}).
Note that $\varepsilon \not\equiv \varepsilon'_\lambda \pmod{X^0(\un{T})}$.
The weight $F_1$ appears with multiplicity two in 
$\oplus_{\nu_\lambda\in L(\pi\omega_\lambda)} L(\theta^0+\nu_\lambda) \otimes L(p\omega)|_{\rG}$,
with multiplicity one in each of 
$L(\theta^0+\varepsilon'_\lambda) \otimes L(p\omega)|_{\rG}$
and 
$L(\theta^0+\varepsilon-\pi(1,1,1)_{i_0}) \otimes L(p\omega)|_{\rG}$.
Assume for the sake of contradiction that the composition (\ref{eqn:gr1''}) is zero.
Then the image of the composition of the first four maps of (\ref{eqn:gr1''}) is contained in $L(\theta^0+\varepsilon-\pi(1,1,1)_{i_0}) \otimes L(p\omega)|_{\rG}$.
In particular, since $\varepsilon \not\equiv \varepsilon'_\lambda \pmod{X^0(\un{T})}$, the image of the composition of the first three maps of (\ref{eqn:gr1''}) is not stable under the involution action on $L(\theta^0) \otimes L(p\omega)\otimes L(p\omega_\lambda)|_{\rG} \cong L(\theta^0+p(1,1,1)_{i_0}) \otimes L(p\omega_\lambda)\otimes L(p\omega_\lambda)|_{\rG}$ which permutes the last two tensor factors.
Since $F_1$ appears with multiplicity one in $L(\theta)|_{\rG}$, this implies that the image of 
\[
L(\theta^0+p(1,1,1)_{i_0})\otimes L(p\omega_\theta-p(1,1,1)_{i_0}) \into L(\theta^0+p(1,1,1)_{i_0}) \otimes L(p\omega_\lambda) \otimes L(p\omega_\lambda),
\]
which induces the second map of (\ref{eqn:gr1''}), is not stable under this involution action (note the key role played by Lemma \ref{lem:algemb}).
However, the unique submodule of $L(p\omega_\lambda)\otimes L(p\omega_\lambda)$ isomorphic to $L(p\omega_\theta-p(1,1,1)_{i_0})$ is the submodule where this involution acts by $-1$.
This is a contradiction.
\end{proof}

\begin{proof}[Proof of Proposition \ref{prop:nonsplit}]
We write $\lambda$ and $\theta$ as $\lambda^0+p\omega_\lambda$ and $\theta^0 + p\omega_\theta$, respectively, where $\lambda^0$ and $\theta^0$ are $p$-restricted weights.
Write $\omega_\lambda$ as $\omega_{\lambda,i_0}+\omega_\lambda^{i_0}$ where $\omega_{\lambda,i_0}$ (resp. ~$\omega_\lambda^{i_0}$) is $0$ away from (resp. ~at) embedding $i_0$.
Similarly write $\omega_\theta$ as $\omega_{\theta,{i_0}}+\omega_\theta^{i_0}$.
Then $\omega_\lambda^{i_0}$ equals $\omega_\theta^{i_0}$ by assumption, and so we set $\omega^{i_0} \defeq \omega_\lambda^{i_0}=\omega_\theta^{i_0}$.
Let $M_{i_0}$ be the unique up to isomorphism nontrivial extension of $L(\theta_{i_0})$ by $L(\lambda_{i_0})$ and let $M^{i_0}$ be $\otimes_{i\neq i_0} L(\lambda_i)$ so that $M$ is isomorphic to $M_{i_0} \otimes M^{i_0}$.
Then $M|_{\rG}$ is isomorphic to 
\begin{align*}
M_{i_0}\otimes M^{i_0}|_{\rG} &\cong M_{i_0}\otimes \big(\underset{i\neq i_0} \bigotimes L(\lambda^0_i)\big) \otimes L(p\omega^{i_0}) |_{\rG} \\
&\cong M_{i_0}\otimes \big(\underset{i\neq i_0} \bigotimes L(\lambda^0_i)\big) \otimes L(\pi\omega^{i_0}) |_{\rG}. 
\end{align*}
Let $M'$ be $M_{i_0}\otimes \big(\underset{i\neq i_0}\bigotimes L(\lambda^0_i)\big) \otimes L(\pi\omega^{i_0})$.
By Lemma \ref{lem:res:soc}, we have that 
$\soc(M') |_{\rG}\cong \soc (M'|_{\rG}) \cong \soc(M|_{\rG}) \cong \soc(M)|_{\rG}$,
so that 
\[\soc (M') \cong L(\lambda_{i_0}) \otimes \big(\underset{i\neq i_0} \bigotimes L(\lambda^0_i)\big) \otimes L(\pi\omega^{i_0}).\]
Similarly, we have that
\[\cosoc (M') \cong L(\theta_{i_0}) \otimes \big(\underset{i\neq i_0} \bigotimes L(\lambda^0_i)\big) \otimes L(\pi\omega^{i_0}).\]
The socle and cosocle of $M'$ can be decomposed using Lemma \ref{lem:trans:princ}.
Fix a direct sum decomposition $\cosoc (M') = \oplus_j M'_{1,j}$ into simple modules.
Using the known dimensions of $\Ext^1_{\un{G}}$ groups between simple modules and the fact that $M'$ is rigid of Loewy length two, one sees that the minimal submodule $M_j'$ of $M'$ whose projection to $\cosoc (M')$ contains $M'_{1,j}$ has length exactly two.
Thus the natural surjection $\oplus M_j' \onto \sum M_j' = M'$ is an isomorphism by length considerations.
More explicitly, if $M_{1,j}'$ is isomorphic to $L(\lambda^0+p\omega_{\lambda,i_0}+\pi\varepsilon^{i_0})$ for some $\varepsilon^{i_0} \in L(\omega^{i_0})$, then $M_{0,j}'$ is isomorphic to $L(\theta^0+p\omega_{\theta,i_0}+w\pi\varepsilon^{i_0})$ for the unique element $w\in \un{W}$ such that these weights are linked.

The alcoves of $F_1$ and $F_0$ differ only in embedding $i_0$.
Then by Lemma \ref{lem:dic}, the corresponding elements of $\un{\Lambda}^\nu_W$ under $\Trns_\nu$ for any appropriate $\nu$ must be the same for all embeddings except for $\pi(i_0)$.
By the above explicit description, we conclude that $F_1$ is in $\JH(M_{1,j}'|_{\rG})$ if and only if $F_0$ is in $\JH(M_{0,j}'|_{\rG})$.
So it suffices to prove the proposition for $M'_j$ in place of $M$.
This now follows from Proposition \ref{prop:nonsplit:res}.
\end{proof}

\begin{proof}[Proof of Theorem \ref{thm:Vdecomp}(\ref{item:maxext})]
Let $\Gamma$ be the directed graph with vertices $\JH(U) = \JH(V_\mu|_{\rG})$ so that there is a directed edge $\sigma_1 \ra \sigma_2$ for $\sigma_1$ and $\sigma_2 \in \JH(V_\mu|_{\rG})$ if and only if $\dgr{\sigma_{(0,\un{1})}}{\sigma_1} > \dgr{\sigma_{(0,\un{1})}}{\sigma_2}$ and $\dgr{\sigma_1}{\sigma_2} = 1$. Note that the first condition ensures that $\Gamma$ is acyclic.

We claim that $\Gamma$ is a subgraph of $\Gamma(U)$.
Let $\sigma_1$ and $\sigma_2$ be in $\JH(U)$ with $\dgr{\sigma_1}{\sigma_2} = 1$.
Let $d_j = \dgr{\sigma_{(0,\un{1})}}{\sigma_j}$ for $j = 1$ and $2$ and suppose that $d_1>d_2$, so that $d_1 = d_2+1$.
Then by Lemma \ref{lem:max:multfree}, $d_U^{\soc}(\sigma_j)$ is $d_j$ for $j=1$ and $2$.
Moreover, by Theorem \ref{thm:Vdecomp}(\ref{item:multone}) and parity reasons, $[\gr_{d_j}(V_\mu)|_{\rG}:\sigma_k] = \delta_{j k}$ for $j$ and $k$ in $\{1,\, 2\}$.
Using that every non-trivial extension which can occur in the layers of $V_\mu$ does occur (see Table \ref{TableWeyl1}), it is easy to check that there is a unique length two subquotient $M$ of $\Fil_{d_1}(V_\mu)/\Fil_{d_2-1}(V_\mu)$ such that $\sigma_1$ and $\sigma_2$ appear in $\JH(M|_{\rG})$ with multiplicity one.
By Proposition \ref{prop:nonsplit}, there is a subquotient of $M|_{\rG}$ which is a nonsplit extension of $\sigma_1$ by $\sigma_2$.
For multiplicity reasons, this must also be a subquotient of $U$.
Hence, there is a directed edge from $\sigma_1$ to $\sigma_2$ in $\Gamma(U)$.

We now claim that if $\Gamma$ is a subgraph of a directed graph $\Gamma'$ such that:
\begin{enumerate}
	\item $\Gamma$ and $\Gamma'$ have the same vertices;
	\item $\Gamma'$ is acyclic; and
	\item $\sigma_1\ra \sigma_2$ is a subgraph of $\Gamma'$ only if $\dgr{\sigma_1}{\sigma_2}=1$
\end{enumerate}
then $\Gamma'=\Gamma$.
Since $\Gamma(U)$ satisfies these conditions ((3) follows from Lemmas \ref{lem:deep} and \ref{lem:dic}), this would complete the proof.

Assume $\sigma_1\ra \sigma_2$ is a subgraph of $\Gamma'$.
By (3) and the geometry of the extension graph, we have that $\dgr{\sigma_{(0,\un{1})}}{\sigma_1} = \dgr{\sigma_{(0,\un{1})}}{\sigma_2}\pm1$, and hence it is enough to prove that $\dgr{\sigma_{(0,\un{1})}}{\sigma_1} = \dgr{\sigma_{(0,\un{1})}}{\sigma_2}+1$.
Suppose otherwise.
Then $\sigma_2\ra\sigma_1$ is a subgraph of $\Gamma$ by definition and thus a subgraph of $\Gamma'$.
But this contradicts (2).
\end{proof}

\subsubsection{The embedding construction}
\label{subsub:emb:arg}

We start with the following observation.
\begin{lemma}
\label{lem:deep:type}
Let $R$ be an $n$-generic Deligne--Lusztig representation. 
Then $R$ is $n-2$-deep.
\end{lemma}
\begin{proof}
If we write $R=R_s(\mu+\eta)$ with $\mu$ being $n$-deep it is hence enough to prove that $\sigma_{r(\omega,a)}$ is $n-2$-deep for any obvious weight $\sigma_{(\omega,a)}$ in $\JH(\ovl{R})$. This follows from Proposition \ref{prop:serrewts}.
\end{proof}

From now onwards, we assume that $\mu$ is $6$-deep in alcove $\un{B}$. 
(By Lemma \ref{lem:deep:type}, the Deligne--Lusztig representation $R\defeq R_s(\mu^{\op}+\eta')$ is therefore $4$-deep.) 
We write $\sigma^{\op}\defeq F(\mu) = F(\Trns_{\mu^{\op}+\eta}(0,1))$. 
Note from Proposition \ref{prop:typedecomp} that the unique element $\sigma\in \JH(\ovl{R})$ satisfying $\dgr{\sigma}{\sigma^{\op}}=3f$ (i.e.~having maximal graph distance from $\sigma^{\op}$) is 
\begin{equation*}
\sigma\defeq F(\mu^{\op}+s(\eta')) = F(\Trns_{\mu^{\op}+\eta}(s(\eta'),0)).
\end{equation*}
To ease notation, we write $\lambda\defeq \mu^{\op}+s(\eta')$ (hence $\sigma= F(\lambda)$).

\begin{lemma}
\label{lem:mult1:G}
Let $\mu\in X_1(\un{T})$ be a weight such that $\mu\in \un{B}$ is $6$-deep  and $R\defeq R_s(\mu^{\op}+\eta')$. %
Set $\sigma^{\op}\defeq F(\mu)$ and let $\sigma\defeq F(\lambda)\in \JH(\ovl{R})$ be the unique constituent at maximal graph distance.
We have:
\begin{equation*}
[\Inj(\sigma^{\op}): \sigma]= [V_\mu|_{\rG}:\sigma]=[U:\sigma] = [\gr_{3f}(V_\mu|_{\rG}):\sigma]=1.
\end{equation*}
\end{lemma}
\begin{proof}
The equations $[V_\mu|_{\rG}:\sigma]= [U:\sigma] = [\gr_{3f}(V_\mu|_{\rG}):\sigma]=1$ follow from Theorem \ref{thm:Vdecomp}(\ref{item:multone}).
The equation $[\Inj(\sigma^{\op}): \sigma] = 1$ is proved similarly using Theorem \ref{thm:dec:tens}.
\end{proof}

\begin{prop}
\label{cor:map}
Let $\mu\in X_1(\un{T})$ be a weight such that $\mu\in \un{B}$ is $6$-deep and $R\defeq R_s(\mu^{\op}+\eta')$.
Set $\sigma^{\op}\defeq F(\mu)$ and let $\sigma\defeq F(\lambda)\in \JH(\ovl{R})$ be the unique constituent at maximal graph distance.
Let $R^{\sigma}$ be an $\cO$-lattice in $R$ with irreducible cosocle isomorphic to $\sigma$.
There is an injection $\iota:\overline{R}^{\sigma}\rightarrow U$.
\end{prop}
Note that the map $\iota$ is unique up to scalar, as there is a unique up to scalar non-zero morphism $\Proj(\sigma)\ra  U$ by Lemma \ref{lem:mult1:G}.
\begin{proof}
Since $\sigma^\op \in \JH(\ovl{R}^\sigma)$, there is a unique up to scalar nonzero map 
\begin{align}\label{map:non:zero}
\ovl{R}^{\sigma}\rightarrow\Inj(\sigma^{\op}).
\end{align}
We first prove that the map (\ref{map:non:zero}) factors through the embedding \begin{align}
\label{exact:seq}
U&\stackrel{\alpha}{\into} \Inj(\sigma^{\op})
\end{align}
or, equivalently, that the composite of (\ref{map:non:zero}) with the natural projection $\Inj(\sigma^{\op})\onto \coker(\alpha)$ is zero.
As formation of cosocle is right exact, and $\ovl{R}^{\sigma}$ has irreducible cosocle isomorphic to $\sigma$, the image of the composite map $\ovl{R}^{\sigma}\ra \coker(\alpha)$ is either zero, or has irreducible cosocle $\sigma$.
By Lemma \ref{lem:mult1:G}, we know that $\sigma$ is not a Jordan-H\"older constituent  of $\coker(\alpha)$, and therefore $\Hom_{\rG}(\ovl{R}^{\sigma}, \coker(\alpha))=0$.
Now the image of this nonzero map contains $\sigma$ as a Jordan--H\"older factor.
Since the minimal submodule of $U$ containing $\sigma$ as a Jordan--H\"older factor has length $9^f$, which is the length of $\ovl{R}^\sigma$, by Theorem \ref{thm:Vdecomp}(\ref{item:maxext}) and Proposition \ref{prop:subclosed}, this map must be an injection.
\end{proof}

\begin{thm}
\label{thm:soc:simple}
With the hypotheses of Theorem \ref{thm structure}, assume moreover that $\sigma$ is a lowest alcove weight with $\Def_R(\sigma)=0$.
Then Theorem \ref{thm structure}(\ref{item:loewy})-(\ref{item:rigid}) hold.
\end{thm}
\begin{proof}
By Remark \ref{rmk:point}, it suffices to prove Theorem \ref{thm structure}(\ref{item:point}).
This follows from Theorem \ref{thm:Vdecomp}(\ref{item:maxext}) and Propositions \ref{cor:map} and \ref{prop:subgraph}.
\end{proof}

\subsection{Proof of the structure theorem in the general case}
\label{subsec:dist}

The aim of this section is to deduce Theorem \ref{thm structure} from the particular case in Theorem \ref{thm:soc:simple}. 
For this, we introduce a third notion of distance called the saturation distance (see Definition \ref{df:sat:dist}).
The first step is to show an inequality between the graph and saturation distances (Corollary \ref{ineq:sat:gr}) and then proceed further in \S \ref{sec:DE} to prove that under appropriate conditions on $\sigma\in\JH(\ovl{R})$ (i.e.~when $\sigma$ is {\it maximally saturated}) the notions of graph, saturation, and \emph{cosocle} distances ($\dcosoc{\sigma}{\kappa}$) actually coincide (Proposition \ref{prop:dcosoc}). 
The agreement of the three distances is equivalent to Theorem \ref{thm structure}, taking into account Lemma \ref{lem:simplify:def}.
In \S \ref{sub:red:MS}, we show that all weights are maximally saturated, thus concluding the proof of Theorem \ref{thm structure}.

\subsubsection{Notions of distances}
\label{sub:sub:def:dist}
In what follows,  $R$ is a $3$-generic Deligne--Lusztig representation.
Recall that if $\sigma\in \JH(\ovl{R})$ we write $\mathrm{d}^{\sigma}_{\mathrm{rad}}$ as a shorthand for $\mathrm{d}^{\ovl{R}^{\sigma}}_{\mathrm{rad}}.$
We establish some properties of $\mathrm{d}^{\sigma}_{\mathrm{rad}}$.
We deduce the following from Lemma \ref{lem:dic}.

\begin{cor}
\label{cor:comp:gr:cosoc}
Assume that $\ovl{R}$ is $6$-deep and $\sigma \in \JH(\ovl{R})$.
Then $\dgr{\sigma}{\sigma'}\leq \dcosoc{\sigma}{\sigma'}$ for all $\sigma'\in \JH(\ovl{R})$.
\end{cor}
\begin{proof}
The assumption that $\ovl{R}$ is $6$-deep implies %
that all the elements in $\JH(\ovl{R})$ satisfy the hypotheses of Lemma \ref{lem:dic}.
If $\kappa\in \JH(\ovl{R})$ and $k+1 = \dcosoc{\sigma}{\kappa}$, then 
$\Ext^1_{\rG}\big(\gr^k(\ovl{R}^{\sigma}),\kappa\big) \neq 0$.
Hence $\Ext^1_{\rG}\big(\kappa',\kappa) \neq 0$ for some $\kappa'\in \JH(\ovl{R})$ with $\dcosoc{\sigma}{\kappa'} = k$.
Equivalently, $\dgr{\kappa}{\kappa'} = 1$ by Lemma \ref{lem:dic}.
If $d = \dcosoc{\sigma}{\sigma'}$, we conclude that there is a sequence of weights $\sigma_i$ for $0\leq i\leq d-1$ such that $\sigma_0 = \sigma$, $\sigma_d = \sigma'$, and $\dgr{\sigma_i}{\sigma_{i+1}} = 1$ for all $i$.
Hence $\dgr{\sigma}{\sigma'} \leq d$.
\end{proof}

\begin{lemma}
\label{lem mult 1}
Let $\sigma,\ \sigma_1\in \JH(\ovl{R})$. There is a unique minimal subrepresentation $Q_{\sigma}(\sigma_1)\subseteq \ovl{R}^{\sigma}$ containing $\sigma_1$ as a Jordan--H\"{o}lder factor.
Moreover, we have that 
\begin{enumerate}
\item $\cosoc(Q_{\sigma}(\sigma_1))\cong \sigma_1$; and \label{item:cosocle}
\item $Q_{\sigma}(\sigma_1)$ is the image of any nonzero map $\ovl{R}^{\sigma_1} \ra \ovl{R}^\sigma$.
\end{enumerate}
\end{lemma}
\begin{proof}
Let $Q_{\sigma}(\sigma_1)$ be the image of any nonzero map $\ovl{R}^{\sigma_1} \ra \ovl{R}^\sigma$.
It is clear that $\cosoc(Q_{\sigma}(\sigma_1))\cong \sigma_1$.
Suppose that $M\subset \ovl{R}^\sigma$ is a subrepresentation containing $\sigma_1$ as a Jordan--H\"{o}lder factor.
As $\ovl{R}^\sigma$ is multiplicity free, $\ovl{R}^\sigma/M$ does not contain $\sigma_1$ as a Jordan--H\"{o}lder factor, and hence the composition $Q_{\sigma}(\sigma_1) \hookrightarrow \ovl{R}^\sigma \surj \ovl{R}^\sigma/M$ is $0$, or equivalently, $Q_{\sigma}(\sigma_1)\subset M$.
\end{proof}

\begin{lemma}
\label{lem length submodule}
Let $\sigma$ and $\sigma_1\in \JH(\ovl{R})$ and let $\sigma_2\in \JH(Q_{\sigma}(\sigma_1))$. 
If $n=\dcosoc{\sigma_1}{\sigma_2}$, then we have $\Hom_{\rG}\big(\sigma_2,\gr^n_{\rad}(Q_{\sigma}(\sigma_1))\neq 0$.
Moreover $\dcosoc{\sigma}{\sigma_2}\geq \dcosoc{\sigma}{\sigma_1}+\dcosoc{\sigma_1}{\sigma_2}$.
\end{lemma}
\noindent Note that, \emph{a priori}, the inequality may be strict, so that what we call the \emph{cosocle distance} $\dcosoc{\sigma}{\kappa}$ from $\sigma$ to $\kappa$ is not \emph{a priori} necessarily a metric. 
\begin{proof}
Since formation of cosocle is right exact, we see that the map $\ovl{R}^{\sigma_1} \surj Q_{\sigma}(\sigma_1)$ is strictly compatible with the radical filtrations.
This proves the first claim.

Let $\mathscr{F}=\{\Fil^k(\ovl{R}^\sigma)\}_k$ be the decreasing filtration defined by $\Fil^k(\ovl{R}^\sigma)=\rad^{k+\dcosoc{\sigma}{\sigma_1}}(\ovl{R}^\sigma)$ (i.e.~
$\mathscr{F}$ is the radical filtration of $\ovl{R}^\sigma$ shifted by $\dcosoc{\sigma}{\sigma_1}$).
Under the inclusion $Q_{\sigma}(\sigma_1) \hookrightarrow \ovl{R}^\sigma$, $\mathscr{F}$ induces a semisimple filtration on $Q_{\sigma}(\sigma_1)$, beginning at $0$.
So, with the radical filtration on the domain and $\mathscr{F}$ on the codomain, this inclusion is compatible with the filtrations.
This immediately gives $\dcosoc{\sigma}{\sigma_2}-\dcosoc{\sigma}{\sigma_1}\geq \dcosoc{\sigma_1}{\sigma_2}$ for any $\sigma_2\in \JH(Q_{\sigma}(\sigma_1))$.
\end{proof}

We now introduce a third notion of distance on the set $\mathrm{JH}(\ovl{R})$ and give a comparison result with the graph distance (Corollary \ref{ineq:sat:gr}).
To do this, we make crucial use of a \emph{global} input (Proposition \ref{prop saturation 1}).
We also use these results to deduce some basic results about $\Gamma_{\rad}(\ovl{R}^{\sigma})$ (see, for example, Lemma \ref{lem:simplify:def}).

\begin{defn}
\label{df:sat:dist}
Let $\sigma_1,\ \sigma_2\in \JH(\ovl{R})$. Let $R^{\sigma_1}$ be an $\cO$-lattice with irreducible cosocle $\sigma_1$ and fix a saturated inclusion of lattices $R^{\sigma_2}\subseteq R^{\sigma_1}$.
The \emph{saturation distance} (\emph{with respect to $R$}) between $\sigma_1$ and $\sigma_2$, noted by $\dsat{\sigma_1}{\sigma_2}$, is defined to be the unique integer $d$ such that $p^dR^{\sigma_1}\subseteq R^{\sigma_2}$ is a saturated inclusion.
\end{defn}

\begin{rmk}
\label{rmk:unr:coeff}
We indicate a justification for the existence of the integer $d$ in Definition \ref{df:sat:dist}.
Given two $\cO$-lattices $\Lambda_1$ and $\Lambda_2$ in an $E$-vector space $V$, there is a unique $d\in \Z$ such that $\varpi^d \Lambda_2 \subset \Lambda_1$ is a saturated inclusion.
If $V$, $\Lambda_1$, and $\Lambda_2$ are obtained from base change from an unramified subfield of $E$, then in fact $\varpi^d$ is a power of $p$ up to units.
Since $R$, $R^{\sigma_1}$, and $R^{\sigma_2}$ are defined over an unramified extension of $\Q_p$, the integer $d$ in Definition \ref{df:sat:dist} exists.
\end{rmk}

\noindent The following lemma is clear:
\begin{lemma}
\label{lem:dsat:dist}
The saturation distance is a metric on $\JH(\ovl{R})$.
\end{lemma}

\noindent The following proposition is the main global input.

\begin{prop}
\label{prop saturation 1}
Let $R$ be a $13$-generic Deligne--Lusztig representation.
Let $\sigma_1,\sigma_2\in \JH(\ovl{R})$ be such that $\dgr{\sigma_1}{\sigma_2}=1$.
Then $\dsat{\sigma_1}{\sigma_2}=1$.
\end{prop}
\begin{proof}
Let $\tau_{\cS}$ be the collection of tame inertial type such that $\sigma(\tau_{\cS})=R$.
We can and do fix a collection $\rhobar_{\cS}$ of semisimple $10$-generic Galois representation such that if $\tld{w} \defeq \tld{w}(\rhobar_{\cS},\tau_{\cS})$, then $\ell(\tld{w}^*_i)\geq 2$ for all $i\in \cJ$ and $\sigma_1$ and $\sigma_2 \in W^?(\rhobar_\cS,\tau_\cS)$.
(Note that the condition that $R$, hence $\tau_{\cS}$, is $13$-generic guarantees that $\rhobar_{\cS}$ is $10$-generic by the argument of Proposition \ref{telim2}.)
Fix a weak minimal patching functor $M_\infty$ for $\rhobar_\cS$, which exists by Corollary \ref{cor:multipatch}. Recall $R^{\tau_{\tld{v}},\ovl{\beta}_{\tld{v}},\Box}_{\ovl{\fM}_{\tld{v}},{\rhobar_{\tld{v}}}}$ from \cite[Definition 5.10]{LLLM}.

For $i\in\{1,2\}$ define the modules
\begin{align*}
M'_\infty (R^{\sigma_i})&\defeq M_\infty (R^{\sigma_i})\widehat{\otimes}_{\big(\widehat{\otimes}_{\tld{v}\in\cS}R^\Box_{\rhobar_{\tld{v}}}\big)} \Big(\widehat{\otimes}_{\tld{v}\in\cS}R^{\tau_{\tld{v}},\ovl{\beta}_{\tld{v}},\Box}_{\ovl{\fM}_{\tld{v}},\rhobar_{\tld{v}}}\Big)\\
M'_\infty (\sigma_i)&\defeq M_\infty (\sigma_i)\widehat{\otimes}_{\big(\widehat{\otimes}_{\tld{v}\in\cS}R^\Box_{\rhobar_{\tld{v}}}\big)} \Big(\widehat{\otimes}_{\tld{v}\in\cS}R^{\tau_{\tld{v}},\ovl{\beta}_{\tld{v}},\Box}_{\ovl{\fM}_{\tld{v}},\rhobar_{\tld{v}}}\Big).
\end{align*}
Similarly define $R'_\infty(\tau_\cS)$ and $\ovl{R}'_\infty(\tau_\cS)$.
Then $M'_\infty (R^{\sigma_i})$ is $p$-torsion free and maximally Cohen-Macaulay over $R_\infty'(\tau_\cS)$, and generically of rank one, and similarly $M'_\infty (\sigma_i)$ is maximally Cohen-Macaulay over $\ovl{R}_\infty'(\tau_\cS)$.

We recall the setup of Lemma \ref{lemma:expl4wt}.
The ring $R'_\infty(\tau_\cS)$ is formally smooth over $\widehat{\otimes}_{\tld{v}\in\cS} R^{\expl,\nabla}_{\ovl{\fM}_{\tld{v}},\tld{w}_{\tld{v}}}$, where we write $\ovl{\fM}_{\tld{v}}\in Y^{\eta,\tau_{\tld{v}}}(\F)$ for the unique Kisin module corresponding to $\rhobar_{\tld{v}}$ as in Theorem \ref{fixedpoints}.
By letting $N\defeq \sum_{i\in \cJ}(4-\ell(\tld{w}^*_i))$, the latter ring is formally smooth over $R_N\defeq \widehat{\otimes}_{j=1}^N\cO[\![x_j,y_j]\!]/(x_jy_j-p)$ by \cite[\S 5.3.2]{LLLM} using that $\ell(\tld{w}^*_i)>1$ for all $i\in\cJ$.
Fix an isomorphism $R_N[\![t_1,\ldots,t_k]\!] \risom R'_\infty(\tau_\cS)$.
Let $S\subset R_N[\![t_1,\ldots,t_k]\!]$ be the subring $\Z_p[\![(x_j,y_j)_{j=1}^N, (t_j)_{j=1}^k]\!]/(x_jy_j-p)_{j=1}^N$.
If $M'$ is a maximal Cohen--Macaulay $R'_\infty(\tau_\cS)$-module, it is a maximal Cohen--Macaulay $S$-module as well by \cite[Corollaire 5.7.10]{EGAIV} and using that the maximal ideal of $R'_\infty(\tau_\cS)$ is the only prime above the maximal ideal of $S$. The maximality follows from the fact that $R'_\infty(\tau_\cS)$ is finite over $S$.
It is convenient to work over $S$ below since $\Spec S/pS$ is reduced.

For $\sigma \in \JH(\ovl{R})$, $\mathrm{Ann}_{R_\infty'(\tau_\cS)}(M_\infty'(\sigma))$ is $\mathfrak{p}(\sigma)R_\infty'(\tau_\cS)$ by Proposition \ref{prop:identify:cmpt}(\ref{item:patchcomp}).
For $\sigma\in W^?(\rhobar_\cS,\tau_\cS)$, let $\mathfrak{p}(\sigma)$ be $(z_j(\sigma))_{j=1}^N+(\varpi)$ where $z_j(\sigma) \in \{x_j,y_j\}$ for each $1\leq j\leq N$.
Let $\fp_S(\sigma) \subset S$ be the preimage of $\mathfrak{p}(\sigma)R_\infty'(\tau_\cS)$ for all $\sigma \in \JH(\ovl{R})$ so that $\fp_S(\sigma) = (z_j(\sigma))_{j=1}^N+(p)$ and $\supp_S M_\infty'(\sigma) = \Spec (S/\fp_S(\sigma)) $.

For all $1\leq j\leq N$, assume without loss of generality that $z_j(\sigma_1)=x_j$ and let $z_j \defeq z_j(\sigma_2)$.
Then by Lemma \ref{lemma:expl4wt}, $\#(\{x_j\}_j \Delta \{z_j\}_j) = 2$.
We assume without loss of generality that $z_1 = y_1$ and $z_j = x_j$ for $j\neq 1$.
To simplify notation, let $R_i = R^{\sigma_i}$ for $i = 1,\ 2$. 
We fix a chain of saturated inclusion of lattices $p^kR_1 \subseteq R_2\subseteq R_1$ with $k\geq 1$.
Since $R$ is residually multiplicity free, $C \defeq \coker(R_2+pR_1\hookrightarrow R_1)$ does not contain $\sigma_2$ as a Jordan--H\"{o}lder factor (as can be seen from descent to an unramified coefficient ring).
Thus, 
\[
\supp_S M_{\infty}'(C) \subset \bigcup_{\sigma \in \JH(\ovl{R}),\, \sigma \neq \sigma_2} \Spec S/\fp_S(\sigma).
\]
The scheme theoretic support of $M_{\infty}'(C)$ in $\Spec S$ is contained in $\Spec S/pS$ and is thus generically reduced, so that by the proof of Lemma \ref{lem:supp}, the scheme-theoretic support of $M_{\infty}'(C)$ in $\Spec S$ is a closed subscheme of 
\[
\Spec \big(S/ \bigcap_{\sigma \in \JH(\ovl{R}),\, \sigma \neq \sigma_2} \fp_S(\sigma) \big).
\]
Since $x_1 y_2 \cdots y_N\in \fp_S(\sigma)$ for all $\sigma \in \JH(\ovl{R})$ with $\sigma \neq \sigma_2$, $x_1 y_2 \cdots y_N$ annihilates $M_{\infty}'(C)$, or equivalently%
\[x_1 y_2 \cdots y_{N} M'_\infty(R_1) \subset M'_\infty( R_2) + pM'_\infty(R_1).\]
Symmetrically, we have
\[y_1 y_2 \cdots y_{N} M_\infty'(R_2) \subset M_\infty'(p^k R_1) + pM_\infty'(R_2).\]
Combining these, we have 
\begin{align*}
x_1 y_2 \cdots y_{N} y_1 y_2 \cdots y_{N} M_\infty'(R_1) &\subset y_1 y_2 \cdots y_{N} M_\infty'(R_2) + y_1 y_2 \cdots y_{N} pM_\infty'(R_1) \\
&\subset M_\infty'(p^k R_1) + pM_\infty'(R_2) + y_1 y_2 \cdots y_{N} pM_\infty'(R_1).
\end{align*}
Simplifying and canceling $p$, we have
\[y_2^2 \cdots y_{N}^2 M_\infty'(R_1) \subset p^{k-1} M_\infty'(R_1) + M_\infty'(R_2) + y_1 y_2 \cdots y_{N} M_\infty'(R_1).\]
Assume that $k>1$.
Then projecting via $M_\infty'(R_1)\onto M_\infty'(\ovl{R}_1)\onto M_\infty'(\sigma_1)$, we have the inclusion
\[y_2^2 \cdots y_{N}^2 M_\infty'(\sigma_1) \subset y_1 y_2 \cdots y_{N} M_\infty'(\sigma_1).\]
This is a contradiction, since $M_\infty'(\sigma_1)$ is free over (a power series ring over) $\F[\![y_1,\cdots y_{N}]\!]$, being maximal Cohen-Macaulay over it.
\end{proof}

We deduce the following inequality:

\begin{cor}
\label{ineq:sat:gr}
Assume that $R$ is $13$-generic and let $\sigma_1,\ \sigma_2\in \JH(\ovl{R})$. Then $\dgr{\sigma_1}{\sigma_2}\geq \dsat{\sigma_1}{\sigma_2}$.
\end{cor}
\begin{proof}
The proof proceeds by induction on $\dgr{\sigma_1}{\sigma_2}$, the case $\dgr{\sigma_1}{\sigma_2}=1$ being covered by Proposition \ref{prop saturation 1}.
Pick a weight $\kappa \in \JH(\ovl{R})$ distinct from $\sigma_1$ and $\sigma_2$ such that $\dgr{\sigma_1}{\kappa}+\dgr{\kappa}{\sigma_2}=\dgr{\sigma_1}{\sigma_2}$.
Then we have
\[\dgr{\sigma_1}{\sigma_2}=\dgr{\sigma_1}{\kappa}+\dgr{\kappa}{\sigma_2}\geq\dsat{\sigma_1}{\kappa}+\dsat{\kappa}{\sigma_2}\geq\dsat{\sigma_1}{\sigma_2}.\]
\end{proof}

We conclude this section showing that, under appropriate genericity conditions on the Deligne--Lusztig representation $R$, the graph $\Gamma_{\rad}(\ovl{R}^{\sigma})$ is predicted by the extension graph if its Loewy strata are predicted by the extension graph (which is a weaker assumption, \emph{a priori}).

\begin{lemma}
\label{lem:dyc}
Assume that $R$ is $13$-generic.
Let $\sigma_1,\ \sigma_2\in \JH(\ovl{R})$ be such that $\dgr{\sigma_1}{\sigma_2}=1$. Let us fix $\sigma\in \JH(\ovl{R})$ as well as two saturated inclusions of lattices $R^{\sigma_2}\subseteq R^{\sigma}$ and $R^{\sigma_1}\subseteq R^{\sigma}$.
Then either $R^{\sigma_2}\subseteq R^{\sigma_1}$ or $R^{\sigma_1}\subseteq R^{\sigma_2}$.
\end{lemma}
\begin{proof}
We have $\dgr{\sigma_1}{\sigma_2}=\dsat{\sigma_1}{\sigma_2}=1$ by Proposition \ref{prop saturation 1}.
Hence there exists an integer $k\in \Z$ such that
\[p^{k+1}R^{\sigma_1}\subseteq R^{\sigma_2}\subseteq p^{k}R^{\sigma_1}\]
is chain of saturated inclusions of lattices (see Remark \ref{rmk:unr:coeff}). 
The first inclusion implies that $p^{k+1}R^{\sigma_1}\subseteq R^{\sigma}$ so that $k+1 \geq 0$.
The second inclusion implies that $p^{-k}R^{\sigma_2} \subset R^{\sigma}$ so that $-k \geq 0$.
Hence $k = 0$ or $-1$ and the result follows.
\end{proof}

\begin{lemma}\label{lem:criterion}
Assume that $R$ is $13$-generic.
Let $\sigma,\,\sigma_1,\,\sigma_2\in \JH(\ovl{R})$ and fix saturated inclusions $R^{\sigma_i}\subseteq R^{\sigma}$ for $i=1,2$. 
If $\dcosoc{\sigma}{\sigma_2}\geq\dcosoc{\sigma}{\sigma_1}$ and $\dgr{\sigma_1}{\sigma_2}=1$ then $R^{\sigma_2}\subseteq R^{\sigma_1}$.
\end{lemma}
\begin{proof}
By Lemma \ref{lem:dyc}, either $R^{\sigma_1} \subset R^{\sigma_2}$ or $R^{\sigma_2} \subset R^{\sigma_1}$.
Suppose that the former holds.
Then $Q_\sigma(\sigma_1) \subset Q_\sigma(\sigma_2)$, and thus $\dcosoc{\sigma}{\sigma_1}> \dcosoc{\sigma}{\sigma_2}$ by Lemma \ref{lem length submodule}.
This is a contradiction.
\end{proof}

\begin{defn}
\label{def:prediction}
Assume that $R$ is $2$-generic and let $\sigma\in \JH(\ovl{R})$.
We say that the radical strata of $\ovl{R}^\sigma$ are \emph{predicted by the extension graph} 
if $\dgr{\sigma}{\sigma'} = \dcosoc{\sigma}{\sigma'}$ for all $\sigma' \in \JH(\ovl{R})$.
\end{defn}

\begin{lemma}
\label{lem:simplify:def}
Assume that $R$ is $13$-generic and let $\sigma\in\JH(\ovl{R})$.
If the radical strata of $\ovl{R}^{\sigma}$ are predicted by the extension graph, then $\Gamma_{\rad}(\ovl{R}^\sigma)$ is predicted by the extension graph with respect to $\sigma$.
\end{lemma}
\begin{proof}
Let $\sigma_i,$ and $\sigma_{i+1}$ be elements of $\JH(\ovl{R})$ such that $\dgr{\sigma}{\sigma_i}=i$ and $\dgr{\sigma}{\sigma_{i+1}}=i+1$.
As the radical strata of $\ovl{R}^{\sigma}$ are predicted by the extension graph we have $\dgr{\sigma}{\sigma_i}=\dcosoc{\sigma}{\sigma_i}$ and $\dgr{\sigma}{\sigma_{i+1}}=\dcosoc{\sigma}{\sigma_{i+1}}$.

If $\dgr{\sigma_i}{\sigma_{i+1}}=1$, then $\sigma_{i+1}\in \JH(Q_{\sigma}(\sigma_{i}))$ by Lemma \ref{lem:criterion}.
By Lemma \ref{lem length submodule}, we have 
\[i+1 = \dcosoc{\sigma}{\sigma_{i+1}}\geq \dcosoc{\sigma}{\sigma_i}+\dcosoc{\sigma_i}{\sigma_{i+1}} = i+\dcosoc{\sigma_i}{\sigma_{i+1}}.\]
This implies that $\sigma_{i+1}$ appears in the second layer of the radical filtration of $Q_{\sigma}(\sigma_i)$, whose cosocle is isomorphic to $\sigma_i$ by Lemma \ref{lem mult 1}(\ref{item:cosocle}).
Then there is an edge from $\sigma_i$ to $\sigma_{i+1}$ in $\Gamma_{\rad}(\ovl{R}^\sigma)$.

Conversely, if there is an edge from $\sigma_i$ to $\sigma_{i+1}$ in $\Gamma_{\rad}(\ovl{R}^\sigma)$, then $\dgr{\sigma_i}{\sigma_{i+1}}$ is $1$ by Lemma \ref{lem:dic}.
\end{proof}

\subsubsection{Distance equalities}
\label{sec:DE}
In this subsection, we define when a weight $\sigma\in \JH(\ovl{R})$ is maximally saturated.
Crucially using Lemma \ref{lem:dic}, we show that if $\sigma$ is maximally saturated, then the graph, saturation, and cosocle distances from $\sigma$ are equal, and therefore the structure Theorem \ref{thm structure} holds for $\ovl{R}^{\sigma}$. From now on, \emph{we assume that the Deligne--Lusztig representation $R$ is $13$-generic}.%

From Proposition \ref{prop:typedecomp} and the definition of defect (Definition \ref{df:defc}), it follows that
\begin{equation}
\max\big\{\dgr{\sigma}{\kappa},\  \kappa\in \JH(\ovl{R})\big\}=3f-\Def_R(\sigma).
\end{equation}
Furthermore, if $\Def_R(\sigma)=0$, then there is a unique $\sigma^{\mathrm{op}}\in\JH(\ovl{R})$ which has maximal graph distance from $\sigma$.

\begin{defn}
\label{df:max:sat}
Let $\sigma\in \JH(\ovl{R})$. We say that the weight $\sigma$ is \emph{maximally saturated in $R$} if
the following property holds:
\begin{equation}
\label{eq:max:sat}
\text{If}\ \kappa\in \JH(\ovl{R})\ \text{verifies}\ \dgr{\sigma}{\kappa}=3f-\Def_R(\sigma)\ \text{then}\ 
\dsat{\sigma}{\kappa}=3f-\Def_R(\sigma).
\end{equation}
\end{defn}

The following proposition motivates Definition \ref{df:max:sat}.

\begin{prop}
\label{prop:max:sat}
A weight $\sigma\in \JH(\ovl{R})$ is maximally saturated if and only if $\dgr{\sigma}{\kappa}=\dsat{\sigma}{\kappa}$ for all $\kappa\in \JH(\ovl{R})$.
\end{prop}
\begin{proof}
The ``if" part is clear.
Let $\kappa\in \JH(\ovl{R})$ be any weight and write $d\defeq \dgr{\sigma}{\kappa}$ and $D \defeq 3f-\Def_R(\sigma)$.
There are weights $\sigma = \sigma_0,\, \sigma_1,\ldots,\, \sigma_D$ so that $\kappa \in \{\sigma_0,\sigma_1,\ldots,\sigma_D\}\subset \JH(\ovl{R})$, $\dgr{\sigma_i}{\sigma_{i+1}} = 1$ for all $0\leq i \leq D-1$ and $\dgr{\sigma}{\sigma_D} = D$.
It is then easy to see that $\dgr{\sigma_i}{\sigma_{j}} = j-i$ if $0\leq i\leq j \leq D$.
We have the following chain of \emph{a priori} inequalities
\begin{align*}
\dsat{\sigma}{\sigma_{D}}&\leq \dsat{\sigma}{\kappa}+\dsat{\kappa}{\sigma_{D}}\\
&\leq \dgr{\sigma}{\kappa}+\dgr{\kappa}{\sigma_{D}}\\
&=\dgr{\sigma}{\sigma_{D}} \\
&=\dsat{\sigma}{\sigma_{D}},
\end{align*}
using Corollary \ref{ineq:sat:gr} where the last equality holds by assumption.
We conclude that the above inequalities are in fact equalities and that $\dsat{\sigma}{\kappa} = \dgr{\sigma}{\kappa}$.
\end{proof}

The following lemma will be the key in relating the notions of saturation and cosocle distance.

\begin{lemma}
\label{covering lemma}
Let $d< \ell\ell (\ovl{R}^{\sigma})$ and let %
$$
\sigma_d\lar\sigma_{d-1}\lar\dots\lar\sigma_{1}\lar\sigma_0\defeq \sigma
$$
be an extension path in $\Gamma_{\rad}(\ovl{R}^{\sigma})$ \emph{(}note that $\dcosoc{\sigma}{\sigma_i}=i$ for all $i\in\{0,\dots,d\}$\emph{)}.
For each $i\in \{0,\dots,d\}$ let us fix a saturated inclusion $R^{\sigma_i}\subseteq R^{\sigma}$.
Then we have a chain of (saturated) inclusions 
$$
R^{\sigma_{d}}\subseteq R^{\sigma_{d-1}}\subseteq\dots\subseteq
R^{\sigma_{1}}\subseteq R^{\sigma_0}.
$$
\end{lemma}
\begin{proof}
Since $\dcosoc{\sigma}{\sigma_{i+1}} = i+1 \geq i = \dcosoc{\sigma}{\sigma_i}$ for $0 \leq i \leq d-1$, the result follows from Lemma \ref{lem:criterion}.
\end{proof}

We can use Lemma \ref{covering lemma} to show that all three notions of distance from $\sigma$ agree in some particular situations.

\begin{prop}
\label{prop:dcosoc}
If $\sigma \in \JH(\ovl{R})$ is maximally saturated, then $\Gamma_{\rad}(\ovl{R}^\sigma)$ is predicted by the extension graph with respect to $\sigma$.
\end{prop}
\begin{proof}
Assume that $\sigma$ is maximally saturated.
By Lemma \ref{lem:simplify:def}, it suffices to show that $\dgr{\sigma}{\sigma'}=\dcosoc{\sigma}{\sigma'}$ for all $\sigma'\in \JH(\ovl{R})$.
Let $\sigma' \in \JH(\ovl{R})$ and $d = \dgr{\sigma}{\sigma'}$.
We have that $\Hom_{\rG}(\sigma', \gr^{k}(\ovl{R}^{\sigma}))\neq 0$ for some $k\geq d$ by Corollary \ref{cor:comp:gr:cosoc}. Assume for the sake of contradiction that $k>d$.
By the definition of the radical filtration we may, and do, fix an extension path of length $k+1$ in $\Gamma_{\rad}(\ovl{R}^{\sigma})$:
\[ \sigma' = \sigma_{k}\lar\sigma_{k-1}\lar\dots\lar\sigma_{1}\lar\sigma_{0} = \sigma. \]
Since $\dgr{\sigma_0}{\sigma_{k}}=d<k$ and $\dgr{\sigma_0}{\sigma_i}\le \dcosoc{\sigma_0}{\sigma_i} = i$ for all $i$ by Corollary \ref{cor:comp:gr:cosoc}, there exists an index $i\in\{0,\dots,k-1\}$ such that
\begin{enumerate} 
	\item \label{cnd1} $\dgr{\sigma_0}{\sigma_{i}}=i$ and
	\item \label{cnd2} $\dgr{\sigma_0}{\sigma_{i+1}}<i+1$.
\end{enumerate}
Moreover, since the extension graph is bipartite, we have that $\dgr{\sigma_0}{\sigma_{i+1}}\leq i-1$. By Lemma \ref{covering lemma}, there is a chain of  saturated inclusions:
$$
R^{\sigma_{i+1}}\subseteq R^{\sigma_i}\subseteq \dots\subseteq 
R^{\sigma_{1}}\subseteq R^{\sigma}
$$ 
where, in particular, $R^{\sigma_{i}}\subseteq R^{\sigma}$ is saturated as well.
As $\dgr{\sigma_0}{\sigma_{i+1}}=\dsat{\sigma_0}{\sigma_{i+1}}$ by Proposition \ref{prop:max:sat}, we further have $p^{i-1}R^{\sigma}\subseteq R^{\sigma_{i+1}}$.
We conclude that $p^{i-1}R^{\sigma}\subseteq R^{\sigma_i}\subseteq R^{\sigma}$, and hence that $\dsat{\sigma_0}{\sigma_i} \leq i-1$.
This contradicts (\ref{cnd1}) since $\dgr{\sigma_0}{\sigma_{i}} = \dsat{\sigma_0}{\sigma_i}$ by Proposition \ref{prop:max:sat}.
\end{proof}

We now use Theorem \ref{thm:soc:simple} to prove that lower alcove weights of defect zero are maximally saturated.

\begin{prop} 
\label{cor:max:sat1}
Let $R$ be a $13$-generic Deligne--Lusztig representation and let $\sigma\in \JH(\ovl{R})$ be a constituent with $\Def_R(\sigma)=0$. 
Assume that $\sigma\cong F(\lambda)$ where $\lambda\in \ X_1(\un{T})$ is in alcove $\un{A}$.
Then $\sigma$ is maximally saturated in $\ovl{R}$.
\end{prop}
\begin{proof}
By Theorem \ref{thm:soc:simple}, we can and do fix an extension path in $\Gamma_{\rad}(\ovl{R}^{\sigma})$ with starting point $\sigma_0\defeq \sigma$: $\sigma^{\op}=\sigma_{3f}\lar\sigma_{3f-1}\lar\dots\lar\sigma_{1}\lar\sigma_0$. By Lemma \ref{covering lemma} we have a sequence of saturated inclusions:
\begin{equation*}
R^{\sigma}\supseteq R^{\sigma_1}\supseteq\dots\supseteq R^{\sigma_{3f-1}}\supseteq R^{\sigma^{\op}}
\end{equation*}
where $R^{\sigma^{\op}}\subseteq R^{\sigma}$ is itself saturated.
For each $0 \leq i\leq 3f-1$, let $n_i$ be $\dsat{\sigma_i}{\sigma^{\op}}$. 
It suffices to show that $n_0 = \dgr{\sigma_0}{\sigma^{\op}} = 3f$.

By Theorem \ref{thm:soc:simple}, the reduction of the lattice in the dual Deligne--Lusztig representation $R^*$ with cosocle $\sigma^\vee$ is rigid and $\Gamma(\ovl{\big(R^*\big)}^{\sigma^\vee})$ is predicted by the extension graph with respect to $\sigma^\vee$.
Noting that the reduction of the dual of a lattice is the dual of the reduction of a lattice and using Proposition \ref{prop:dualgraph}, we see that $\Gamma_\rad(\ovl{R}^{\sigma^{\op}})$ is predicted by the extension graph with respect to $\sigma^{\op}$.
In particular, $\sigma_0\lar\sigma_1\lar\dots\lar\sigma_{3f-1}\lar{\sigma_{3f}}=\sigma^{\op}$ is an extension path in $\Gamma_{\rad}(\ovl{R}^{\sigma})$ and hence, from Lemma \ref{covering lemma}  we deduce a chain of (saturated) inclusions
\begin{equation*}
R^{\sigma^{\op}}\supseteq p^{n_{3f-1}}R^{\sigma_{3f-1}}\supseteq\dots\supseteq p^{n_1}R^{\sigma_{1}}\supseteq p^{n_0}R^{\sigma}
\end{equation*}
where we necessarily have $n_{i-1}>n_i$ for all $i$ (as $R^{\sigma_{i+1}}\subseteq R^{\sigma_{i}}$ is saturated). In particular $n_0\geq 3f$.
On the other hand, Corollary \ref{ineq:sat:gr} implies that $n_0\leq 3f$, so that $n_0=3f$. 
\end{proof}

\subsubsection{Induction on defect}
\label{sub:red:MS}

In this subsection, we show inductively that all weights are maximally saturated, starting from lower alcove weights as in Proposition \ref{cor:max:sat1}.
We conclude the section with the proof of Theorem \ref{thm structure}.
We first start with the defect zero case.

\begin{lemma} \label{lem:propagate1}
Let $R$ be a $13$-generic Deligne--Lusztig representation.
If $\sigma\in \JH(\ovl{R})$ and $\Def_R(\sigma)=0$, then $\sigma$ is maximally saturated.
\end{lemma}
\begin{proof}
We claim that if $\sigma \defeq \sigma_0,\ \sigma_1 \in \JH(\ovl{R})$ such that $\Def_R(\sigma_0)=0$, $\Def_R(\sigma_1)=0$, $\dgr{\sigma_0}{\sigma_1} = 1$, and $\sigma_0$ is maximally saturated, then $\sigma_1$ is maximally saturated.
The result then follows from Proposition \ref{cor:max:sat1} and an easy induction argument.

By Proposition \ref{prop:dcosoc}, the graph $\Gamma_{\rad}(\ovl{R}^{\sigma})$ is predicted by the extension graph. 
By duality (cf.~ the proof of Proposition \ref{cor:max:sat1}), the graph $\Gamma_{\rad}\big(\ovl{R}^{\sigma^{\op}}\big)$ is also predicted by the extension graph.
Hence we may and do choose two extension paths, of starting point $\sigma$, in the graph $\Gamma_{\rad}(\ovl{R}^{\sigma_0})$ having the form:
$\sigma_0^{\op}\lar\sigma'_{3f-1}\lar\dots\lar\sigma_2'\lar\sigma_1\lar\sigma_0$ and 
$\sigma_0^{\op}\lar\sigma_{1}^{\op}\lar\sigma''_{3f-2}\lar\dots\lar\sigma''_1\lar\sigma_0$
where $\sigma_{i}',\ \sigma''_{i}\in \JH(\ovl{R}^{\sigma})$. As the graph $\Gamma_{\rad}(\ovl{R}^{\sigma^{\op}})$ is predicted by the extension graph, the extension paths above induce extensions paths in $\ovl{R}^{\sigma^{\op}}$ by ``reversing the arrows and the endpoints''.
Let us fix saturated inclusions of lattices $R^{\sigma_i'}, R^{\sigma''_i}\subseteq R^{\sigma}$, $R^{\sigma_1}, R^{\sigma_1^{\op}}\subseteq R^{\sigma^{\op}}$.
Since $\sigma$ and $\sigma^{\op}$ are maximally saturated, we deduce that $p^{3f-i}R^{\sigma_i'}, p^{3f-i}R^{\sigma''_i}\subseteq R^{\sigma^{\op}}$, $p^{3f-1}R^{\sigma_1}, pR^{\sigma_1^{\op}}\subseteq R^{\sigma^{\op}}$ are saturated inclusions as well.

By Lemma \ref{covering lemma} we deduce the following chain of saturated inclusions:
\begin{equation}
\label{big graph}
\xymatrix@=1pc{
&&p^{3f-1}R^{\sigma_1}\ar@{}[r]|-*[@]{\subseteq} 
&%
\dots\ar@{}[r]|-*[@]{\subseteq}
&pR^{\sigma_{3f-1}'}\ar@{}[rd]|-*[@]{\subseteq}
&
&
R^{\sigma_{3f-1}'}\ar@{}[r]|-*[@]{\subseteq} 
&\dots\ar@{}[r]|-*[@]{\subseteq}
&R^{\sigma_{1}}\ar@{}[rd]|-*[@]{\subseteq}
&
\\
p^{3f}R^{\sigma_1}\ar@{}[r]|-*[@]{\subseteq}&p^{3f}R^{\sigma}\ar@{}[ru]|-*[@]{\subseteq}\ar@{}[rd]|-*[@]{\subseteq}
&&&&R^{\sigma^{\op}}\ar@{}[ur]|-*[@]{\subseteq}\ar@{}[rd]|-*[@]{\subseteq}
&&&&R^{\sigma}
  \\
&&p^{3f-1}R^{\sigma''_{1}}\ar@{}[r]|-*[@]{\subseteq} 
&%
\dots\ar@{}[r]|-*[@]{\subseteq}
&pR^{\sigma_{1}^{\op}}\ar@{}[ru]|-*[@]{\subseteq}
&
&R^{\sigma_1^{\op}}\ar@{}[r]|-*[@]{\subseteq} 
&\dots\ar@{}[r]|-*[@]{\subseteq}
&R^{\sigma''_{1}}\ar@{}[ru]|-*[@]{\subseteq}
&
}
\end{equation}
We claim that the inclusions $p^{3f}R^{\sigma_1}\subseteq pR^{\sigma_1^{\op}}$ and 
$pR^{\sigma_1^{\op}}\subseteq R^{\sigma_1}$, obtained by composing the saturated inclusions above, are saturated. %

We first show that $p^{3f}R^{\sigma_1}\subseteq pR^{\sigma_1^{\op}}$ is saturated.
If not, then we obtain a chain of saturated inclusion $p^{3f-1}R^{\sigma_1}\subseteq pR^{\sigma_1^{\op}}\subseteq R^{\sigma^{\op}}$. 
We deduce that $\sigma_1$ is a constituent in image  $Q_{\sigma^{\op}}(\sigma_1^{\op})$ and hence $\dcosoc{\sigma^{\op}}{\sigma_1}\geq \dcosoc{\sigma^{\op}}
{\sigma_1^{\op}}+\dcosoc{\sigma_1^{\op}}{\sigma_1}\geq \dgr{\sigma^{\op}}
{\sigma_1^{\op}}+\dgr{\sigma_1^{\op}}{\sigma_1}\geq 3f+1$ by Corollary \ref{cor:comp:gr:cosoc}
 and Lemma \ref{lem length submodule}.
On the other hand, as the graph $\Gamma_{\rad}\big(\ovl{R}^{\sigma^{\op}_0}\big)$ is predicted by the extension graph we have that $\dcosoc{\sigma^{\op}}{\sigma_1}=\dgr{\sigma^{\op}}{\sigma_1}=3f-1$ contradiction.

The evident analogue of the previous argument shows that $pR^{\sigma_1^{\op}}\subseteq R^{\sigma_1}$ is saturated as well. Hence $\dsat{\sigma_1}{\sigma_1^{\op}}=3f$.
\end{proof}

We now give the induction argument.

\begin{prop}
\label{prop step 3}
Let $R$ be a $13$-generic Deligne--Lusztig representation.
Then any constituent $\sigma\in \JH(\ovl{R})$ is maximally saturated in $\ovl{R}$.
\end{prop}
\begin{proof}
We induct on the defect $\delta \defeq \Def_R(\sigma)$ for $\sigma\in \JH(\ovl{R})$.
The case $\delta=0$ holds by Lemma \ref{lem:propagate1}.
Suppose that $\delta>0$.
To ease notation, let $d\defeq 3f-\delta$ and pick a weight $\sigma_{d}\in \JH(\ovl{R})$ such that  $\dgr{\sigma}{\sigma_{d}}=d$ (in other words $\sigma_d\in \JH(\ovl{R})$ is at maximal graph distance from $\sigma$); we will show that $d=\dsat{\sigma}{\sigma_d}$.  Note that $\Def_R(\sigma_d)\leq \Def_R(\sigma)$.
If $\Def_R(\sigma_d)< \Def_R(\sigma)$ then $\sigma_d$ is maximally saturated (by induction on $\Def_R(\sigma_d)$) and hence $\dgr{\sigma}{\sigma_d}=\dsat{\sigma}{\sigma_d}$ by Proposition \ref{prop:max:sat}.

We now consider the case $\Def_R(\sigma_d)=\Def_R(\sigma)$. 
By a direct check on the extension graph, we see that there exists $\sigma_{d-1}\in \JH(\ovl{R})$ with $\dgr{\sigma_d}{\sigma_{d-1}}=1$ and $\Def_R(\sigma_{d-1})= \Def_R(\sigma)-1$.
Note that $\dgr{\sigma_{d-1}}{\sigma} = d-1$.
By induction, the weight $\sigma_{d-1}$ is maximally saturated, hence $\dsat{\sigma_{d-1}}{\kappa}=\dgr{\sigma_{d-1}}{\kappa}=\dcosoc{\sigma_{d-1}}{\kappa}$ for all $\kappa\in \JH(\ovl{R})$ by Propositions \ref{prop:max:sat} and \ref{prop:dcosoc}.

Let $R^{\sigma_{d-1}}\subseteq R^{\sigma}$, $R^{\sigma_d}\subseteq R^{\sigma}$ be saturated inclusions of lattices.
By Lemma \ref{lem:dyc}, we are in one of the following situations:
\begin{enumerate}
	\item\label{case1} $R^{\sigma_d}\subseteq R^{\sigma_{d-1}}\subseteq R^{\sigma}$;
\item\label{case2} $R^{\sigma_{d-1}}\subseteq R^{\sigma_{d}}\subseteq R^{\sigma}$.
\end{enumerate}
where the inclusions are all saturated.

\emph{Case (\ref{case1}).} Taking $\kappa = \sigma$ in the above, we have that $\dsat{\sigma}{\sigma_{d-1}}=d-1$, and hence obtain chains of saturated inclusions:
\begin{equation*} 
\xymatrix{
p^{d-1}R^{\sigma}\ar@{}[r]|{\subseteq}&R^{\sigma_{d-1}}\ar@{}[r]|{\subseteq}& R^{\sigma}\\
&R^{\sigma_d}\ar@{}[u]|{\rotatebox{90}{$\subseteq$}}&
}
\end{equation*}
Assume for the sake of contradiction that $\dsat{\sigma}{\sigma_d}<d$, i.e.~that we have a factorization
\begin{equation*} 
\xymatrix{
p^{d-1}R^{\sigma}\ar@{}[r]|{\subseteq}\ar@{^{(}.>}[rd]&R^{\sigma_{d-1}}\ar@{}[r]|{\subseteq}& R^{\sigma}.\\
&R^{\sigma_d}\ar@{}[u]|{\rotatebox{90}{$\subseteq$}}&
}
\end{equation*}
Then, we necessarily have that $\dsat{\sigma}{\sigma_d} = d-1$.
We obtain a commutative diagram:
\begin{equation*}
\xymatrix{
p^{d-1}R^{\sigma}\ar@{^{(}->}[r]\ar@{->>}[d]&R^{\sigma_d}\ar@{^{(}->}[r]\ar@{->>}[d]&R^{\sigma_{d-1}}\ar@{->>}[d]\\
(p^{d-1}R^{\sigma})\otimes_{\cO}\F\ar@/_2pc/[rr]_{\neq 0}\ar_>>>>{\neq 0}[r]&\ovl{R}^{\sigma_d}\ar_{\neq 0}[r]&\overline{R}^{\sigma_{d-1}}
}
\end{equation*}
where the lower arrows are all \emph{non-zero}.

In particular, $\sigma$ is a constituent of 
$Q_{\sigma_{d-1}}(\sigma_d)$.
By Corollary \ref{cor:comp:gr:cosoc} and Lemma \ref{lem length submodule}, we have $\dcosoc{\sigma_{d-1}}{\sigma}\geq \dcosoc{\sigma_{d-1}}{\sigma_d}+\dcosoc{\sigma_d}{\sigma}\geq \dgr{\sigma_{d-1}}{\sigma_d}+\dgr{\sigma_d}{\sigma}= d+1$. On the other hand, as  $\sigma_{d-1}$ is maximally saturated, $\dcosoc{\sigma_{d-1}}{\sigma}=d-1$, a contradiction.

\emph{Case  (\ref{case2}).} We now have a commutative diagram 
\begin{equation*}
\xymatrix{
R^{\sigma_{d-1}}\ar@{^{(}->}[r]\ar@{->>}[d]&R^{\sigma_d}\ar@{^{(}->}[r]\ar@{->>}[d]&R^{\sigma}\ar@{->>}[d]\\
\ovl{R}^{\sigma_{d-1}}\ar@/_2pc/[rr]_{\neq 0}\ar_{\neq 0}[r]&\ovl{R}^{\sigma_d}\ar_{\neq 0}[r]&\overline{R}^{\sigma}
}
\end{equation*}
where again the lower arrows are \emph{non-zero}. 
Exactly as in the previous case, we deduce that $\dcosoc{\sigma}{\sigma_{d-1}}\geq \dcosoc{\sigma}{\sigma_d}+1\geq d+1$.

We therefore may, and do, fix an extension path in the radical filtration of $\ovl{R}^{\sigma}$:
\begin{equation*}
\sigma_{d-1} = \sigma'_{k}\lar\sigma'_{k-1}\lar\dots\lar\sigma'_{1}\lar\sigma'_0 = \sigma.
\end{equation*}
For notational convenience, we set $\kappa_i\defeq \sigma'_{k-i}$.
As $k\geq d+1$, and as $\sigma_{d-1}$ is maximally saturated, we deduce as in the proof of Proposition \ref{prop:dcosoc} the existence of an index $i\in \{0,\dots, k-1\}$ such that
\begin{enumerate} 
	\item\label{it:4.3.13:1} $\dgr{\kappa_0}{\kappa_{i}}=\dsat{\kappa_0}{\kappa_{i}}=i$ and
	\item\label{it:4.3.13:2} $\dgr{\kappa_0}{\kappa_{i+1}}=\dsat{\kappa_0}{\kappa_{i+1}}<i+1$, and actually $\dgr{\kappa_0}{\kappa_{i+1}}\leq i-1$.
\end{enumerate}
Fix a chain of saturated inclusions $R^{\kappa_0}\subseteq R^{\kappa_1}\subseteq \dots\subseteq R^{\kappa_k}$.
Item (\ref{it:4.3.13:2}) implies that
$R^{\kappa_{0}}\subseteq R^{\kappa_{i+1}}\subseteq p^{-i+1}R^{\kappa_{0}}$.
The induced inclusion $R^{\kappa_i} \subseteq R^{\kappa_{i+1}} \subseteq p^{-i+1}R^{\kappa_0}$ contradicts item (\ref{it:4.3.13:1}).
\end{proof}

\begin{proof}[Proof of Theorem \ref{thm structure}]
(\ref{item:sat}) follows from Propositions \ref{prop:max:sat} and \ref{prop step 3}.
(\ref{item:loewy}) and (\ref{item:radpt}) follow from Propositions \ref{prop:dcosoc} and \ref{prop step 3}.
Note that $\Gamma_{\rad}(\ovl{R}^\sigma)$ is a subgraph of $\Gamma(\ovl{R}^\sigma)$.
Using Lemma \ref{lem:dic} and the fact that $\Gamma(\ovl{R}^\sigma)$ is acyclic, $\Gamma(\ovl{R}^\sigma)$ must in fact be $\Gamma_{\rad}(\ovl{R}^\sigma)$.
This implies (\ref{item:point}).
\end{proof}

\begin{table}[h]
\caption{\textbf{Comparison between alcoves $C_i$, affine Weyl group elements $\tld{w}_i$ and the graph.}}\label{TableWeylAlpha}
\centering
\adjustbox{max width=\textwidth}{
\begin{tabular}{| c | c | c | c | c |}
\hline
$\tld{w}_i$&$\tld{w}_{+,i}$&$\omega_{-,i}$&
$\tld{w}_i\cdot A$&$(\omega,a)_{\pi i}$\\
\hline
$\id$&$\id$&$t_{0}$&$A$&$(0,0)$
\\
\hline
$(13)t_{-(\eps'_1+\eps'_2)}$&$(13)t_{-(\eps'_1+\eps'_2)}$&$t_{0}$&$B$&$(0,1)$\\
\hline
$(123)t_{-\alpha_{2}}$&$(123)t_{-\eps'_{2}}$&$t_{\eps'_{2}}$&$C$&
$(s(\eps'_{2}),0)$; $s\in S_3$
\\
\hline
$(12)t_{\alpha_{2}}$&$(12)t_{-(\eps'_{1}-\eps'_{2}-\un{1})}$&$t_{\eps'_{2}}$&$E$&
$(s(\eps'_{2}),1)$; $s\in S_3$\\
\hline
$(132)t_{-\alpha_{1}}$&$(132)t_{-\eps'_{1}}$&$t_{\eps'_{1}}$&$D$&
$(s(\eps'_{1}),0)$; $s\in S_3$\\
\hline
$(23)t_{\alpha_{1}}$&$(23)t_{-(\eps'_{2}-\eps'_{1}+\un{1})}$&$t_{\eps'_{1}}$&$F$&
$(s(\eps'_{1}),1)$; $s\in S_3$\\
\hline
$(13)t_{-2(\eps'_1+\eps'_2)}$&$(13)t_{-(\eps'_1+\eps'_2)}$&$t_{(\eps'_1+\eps'_2)}$&$J$&$(s(\eps'_1+\eps'_2),1)$; $s\in S_3$ and $(0,1)$\\
\hline
\hline
\hline
$\tld{w}_i$&$\tld{w}_{i,+}$&$\tld{w}_{i,-}$&
$\tld{w}_i\cdot B$&$(\omega,a)_{\pi i}$\\
\hline
$\id$&$\id$&$t_{0}$&$B$&$(0,1)$
\\
\hline
$(13)t_{-(\eps'_1+\eps'_2)}$&$(13)t_{-(\eps'_1+\eps'_2)}$&$t_{0}$&$A$&$(0,0)$\\
\hline
$(132)t_{-(2\alpha_{1}+\alpha_{2})}$&$(132)t_{-(2\eps'_{1}-\un{1})}$&$t_{\eps'_{2}}$&$E$&
$(s(\eps'_{2}),1)$; $s\in S_3$
\\
\hline
$(23)t_{-\alpha_{2}}$&$(23)t_{-\eps'_{2}}$&$t_{\eps'_{2}}$&$C$&
$(s(\eps'_{2}),0)$; $s\in S_3$\\
\hline
$(123)t_{-(\alpha_{1}+2\alpha_{2})}$&$(123)t_{-(2\eps'_{2}+\un{1})}$&$t_{\eps'_{1}}$&$F$&
$(s(\eps'_{1}),1)$; $s\in S_3$\\
\hline
$(12)t_{-\alpha_{1}}$&$(12)t_{-\eps'_{1}}$&$t_{\eps'_{1}}$&$D$&
$(s(\eps'_{1}),0)$; $s\in S_3$\\
\hline
$(13)t_{-2(\eps'_1+\eps'_2)}$&$(13)t_{-(\eps'_1+\eps'_2)}$&$t_{(\eps'_1+\eps'_2)}$&$G$&$(s(\eps'_1+\eps'_2),0)$; $s\in S_3$ and $(0,0)$\\
\hline
\hline
\end{tabular}}
\caption*{For each element $\widetilde{w}^*$ in the first column, we consider the decomposition $\tld{w}_i=\omega_{i,-}\cdot \tld{w}_{i,+}$ with  $(\tld{w}_{i,+})_i\in \widetilde{\un{W}}^+$. In the fourth column we write the alcove containing $\tld{w}_i\cdot A$. In the fifth column we write the $\pi i$-th coordinate of the points in the graph (with origin $\mu+\eta'$) corresponding to an irreducible constituent of $(\otimes_i L(\tld{w}_i\cdot \mu_i))|_{\rG}$.

Similar comments apply to the second half of the table. Note that in this case, we consider the decomposition $\tld{w}_i=\omega_{-,i}\cdot \tld{w}_{+,i}$ with  $\tld{w}_{+,i}\in \widetilde{\un{W}}_1^+\tld{w}_B$ and the graph has origin in $\mu^{\op}+\eta$.

Finally, we have set $\alpha_1$ and $ \alpha_2 \in X^*(T)$ to be $(1,-1,0)$ and $(0,1,-1)$ respectively.  
}
\end{table}

\begin{table}[h]
\caption{\textbf{Graph of $V(\mu^{\op}+p\eta')$}}\label{TableWeyl1}
\centering
\adjustbox{max width=\textwidth}{
\begin{tabular}{ c }
$
\xymatrix{
&&
\underset{L(z+p-2,y,x-p+2)\otimes L(\eta')}{\mathbf{G}}\ \ar@{-}[dr]\ar@{-}[dl]&
&&
\\
&
\underset{L(z+p-2,x-p+1,y-p+1)\otimes L(\eps'_1)}{\mathbf{E}}\ar@{-}[dl]\ar@{-}[dr]\ar@{-}[drrr]
&
&
\underset{L(y+p-1,z+p-1,x-p+2)\otimes L(\eps'_2)}{\mathbf{F}}\ar@{-}[dr]\ar@{-}[dl]\ar@{-}[dlll]
&
\\
\underset{L(y-1,x-p+1,z)\otimes L(\eps'_1)}{\mathbf{C}}\ar@{-}[drr]
&
&
\underset{L(z+p-2,y,x-p+2)}{\mathbf{A}}\ar@{-}[d]
&
&
\underset{L(x,z+p-1,y+1)\otimes L(\eps'_2)}{\mathbf{D}}\ar@{-}[dll]
\\
&&\underset{(L(x,y,z))}{\mathbf{B}}&&
}
$\\
\hline
\hline
$
\xymatrix{
&&&
\underset{L(z+p-2,y,x-p+2)\otimes L(\eta')}{\mathbf{J}}\ar@{-}[drrr]\ar@{-}[dr]\ar@{-}[d]\ar@{-}[dl]\ar@{-}[dlll]&
&&&
\\
\underset{L(z+p-2,x+1,y+1)\otimes L(2\eps'_1-\un{1})}{\mathbf{H}}\ar@{-}[dr]&&
\underset{L(y-1,z-1,x-p+2)\otimes L(\eps'_1)}{\mathbf{C}}\ar@{-}[drrr]\ar@{-}[dl]
&
\underset{L(x,y,z)\otimes L(\eta')}{\mathbf{G}}\ar@{-}[drr]\ar@{-}[dll]
&
\underset{L(z+p-2,x+1,y+1)\otimes L(\eps'_2)}{\mathbf{D}}\ar@{-}[dlll]\ar@{-}[dr]
&&\underset{L(y-1,z-1,x-p+2)\otimes L(2\eps'_2+\un{1})}{\mathbf{I}}\ar@{-}[dl]
\\
&\underset{L(x,z-1,y-p+1)\otimes L(\eps'_1)}{\mathbf{E}}\ar@{-}[drr]
&
&
&
&
\underset{L(y+p-1,x+1,z)\otimes L(\eps'_2)}{\mathbf{F}}\ar@{-}[dll]
&
\\
&&&\underset{L(x,y,z)}{\mathbf{A}}&&&
}
$
\end{tabular}
}
\caption*{The graphs of $V(\mu^{\op}+p\eta')$ when $f=1$ (and $\mu$ is $2$-deep). In the first (resp. second) diagram $\mu\in X_1(\un{T})$ is upper alcove (resp. lower alcove). 
For each alcove we write below the unique weight $L$ such that $L(x,y,z)\uparrow L$.
We remark that $\gr_1(V(\mu^{\op}+p\eta')=\gr_1(Q_1(\mu))$ if $\mu=(x,y,z)$ is upper alcove, while $\gr_1(V(\mu^{\op}+p\eta')\oplus L(\mu^{\op})=\gr_1(Q_1(\mu))$ if $\mu=(x,y,z)$ is lower alcove.
}
\end{table}

\section{Breuil's Conjectures}
\label{sec:BC}

In this section, we deduce generalizations of Breuil's  conjectures on mod $p$ multiplicity one and lattices from the results in \S \ref{sec:weights} - \ref{sec:lct}.
In \S \ref{sec:cyc} - \ref{sec:gauge}, we prove a version of Breuil's conjectures for abstract patching functors.
Finally, we deduce the main results in \S \ref{subsec:global}.

\subsection{Cyclicity for patching functors}
\label{sec:cyc}

In this subsection, we show that certain patched modules for tame types are locally free of rank one over the corresponding local deformation space using several inductive steps. The base case is Lemma \ref{lemma:4weight}. Each inductive step uses one of two arguments in \cite[\S 10]{EGS}.
It is here that the results of \S \ref{sec:idealrel} enter.

Recall from \S \ref{subsection:Notation} that $\cS$ is a finite set so that for each $\tld{v} \in \cS$, $F_{\tld{v}}$ is a finite unramified extension of $\Q_p$ of degree $f_{\tld{v}}$.
For each $\tld{v}$, let $\rhobar_{\tld{v}}:G_{F_{\tld{v}}} \ra \GL_3(\F)$ be a $10$-generic 
semisimple continuous Galois representation and let $\rhobar_{\cS}$ be $(\rhobar_{\tld{v}})_{\tld{v}\in\cS}$.
Recall the weak minimal patching functor setting of \S \ref{sec:SWC}.
Suppose that $\rhobar_{\cS}|_{I_\cS} \cong \taubar_\cS(s,\mu)$ for $s\in \un{W}$ and $\mu \in X^*(\un{T})$.

Let $K$ be $\prod_{\tld{v}\in\cS} K_{\tld{v}}$.
Suppose that $M_\infty$ is a weak minimal patching functor for $\rhobar_{\cS}$ (cf.~ Definition \ref{minimalpatching}).
For each $\tld{v}\in\cS$, let $\tau_{\tld{v}}$ be an $13$-generic tame inertial type for $F_{\tld{v}}$.
Recall $R^{\tau_{\tld{v}},\ovl{\beta}_{\tld{v}},\Box}_{\ovl{\fM}_{\tld{v}},{\rhobar_{\tld{v}}}}$ from \cite[Definition 5.10]{LLLM}.
We let $\tau_\cS$ be $(\tau_{\tld{v}})_{\tld{v}}$ so that $\sigma(\tau_\cS)$ is the $K$-module $\otimes_{\tld{v}\in \cS} \sigma(\tau_{\tld{v}})$. 
Let
\[R'_\infty(\tau_\cS) \defeq R_\infty(\tau_\cS) \widehat{\otimes}_{(\widehat{\otimes}_{\tld{v}\in \cS} R^{\tau_{\tld{v}}}_{\tld{v}})} \Big(\widehat{\otimes}_{\tld{v}\in \cS} R^{\tau_{\tld{v}},\ovl{\beta}_{\tld{v}},\Box}_{\ovl{\fM}_{\tld{v}},{\rhobar_{\tld{v}}}}\Big)\]
and 
\[M'_\infty(V) \defeq M_\infty(V) \widehat{\otimes}_{(\widehat{\otimes}_{\tld{v}\in \cS} R^{\tau_{\tld{v}}}_{\tld{v}})} \Big(\widehat{\otimes}_{\tld{v}\in \cS} R^{\tau_{\tld{v}},\ovl{\beta}_{\tld{v}},\Box}_{\ovl{\fM}_{\tld{v}},{\rhobar_{\tld{v}}}}\Big)\]
for any subquotient $V$ of a lattice in $\sigma(\tau_\cS)$.

We assume that $R_\infty(\tau_\cS)$ is nonzero.
Then there exists $\tld{w} = (\tld{w}_{\tld{v}})_{\tld{v}\in \cS}=(\tld{w}_i)_{i\in \cJ} \in \Adm^\vee(\eta)$ such that $\tau_\cS \cong \tau_\cS(s(w^*)^{-1},\mu+s(\tld{w}^*)^{-1}(0))$ with the lowest alcove presentation $(s(w^*)^{-1},\mu+s(\tld{w}^*)^{-1}(0)-\eta)$ by Theorem \ref{fixedpoints} and Proposition \ref{typereflection}. Note that $\tld{w}_{\tld{v}}= \tld{w}(\rhobar_{\tld{v}},\tau_{\tld{v}})$.
Then, as exlpained in Section \ref{sec:match}, $\Spf \ovl{R}'_\infty(\tau_\cS)$ (resp. $\Spf R'_\infty(\tau_\cS)$) is formally smooth over $\Spf \Big(\widehat{\otimes}_{i\in\cJ}\ovl{R}_{\ovl{\fM}_{\tld{v}},\tld{w}_{i}}^{\expl,\nabla}\Big)$ (resp. $\Spf \Big(\widehat{\otimes}_{i\in \cJ}{R}_{\ovl{\fM}_{\tld{v}},\tld{w}_{i}}^{\expl,\nabla}\Big)$ if $\ell(\tld{w}_{i})>1$ for all $i\in \cJ$), where $i=(\tld{v},i_{\tld{v}})$.
We consider $\ovl{R}'_\infty(\tau_\cS)$ (resp. ${R}'_\infty(\tau_\cS)$) both as a $\widehat{\otimes}_{i\in\cJ}\ovl{R}_{\ovl{\fM}_{\tld{v}},\tld{w}_{i}}^{\expl,\nabla}$-algebra and a $\widehat{\otimes}_{\tld{v}\in \cS}\ovl{R}^\Box_{\rhobar_{\tld{v}}}$-algebra (resp. a $\widehat{\otimes}_{i\in \cJ}{R}_{\ovl{\fM}_{\tld{v}},\tld{w}_{i}}^{\expl,\nabla}$-algebra and a $\widehat{\otimes}_{\tld{v}\in \cS}{R}^\Box_{\rhobar_{\tld{v}}}$-algebra).
It is easy to see that $M'_\infty(V)$ is a cyclic $R'_\infty(\tau_\cS)$-module if and only if $M_\infty(V)$ is a cyclic $R_\infty(\tau_\cS)$-module.
Moreover, $M'_\infty(V)$ is always maximal Cohen--Macaulay over its support.
It will often be more convenient to prove that $M'_\infty(V)$ is a cyclic $R'_\infty(\tau_\cS)$-module, and we will switch to $M'_\infty(V)$ without comment.
If $\sigma \in \JH(\ovl{\sigma(\tau_\cS)})$, then recall from \S \ref{sec:mainresult} that $\sigma(\tau_\cS)^\sigma$ is the unique $\cO$-lattice up to homothety in $\sigma(\tau_\cS)$ with cosocle $\sigma$.

\begin{thm}\label{thm:cyclic}
Let $\tau_\cS$ be a $13$-generic tame type and $\sigma \defeq F(\lambda) \in W^?(\rhobar_{\cS},\tau_\cS)$ 
such that for all $i \in \cJ$, 
\begin{equation}
\label{up:each:emb}
\lambda_{\pi^{-1}(i)}\in X_1(T) \text{ is in alcove $C_1$ if } \ell(\tld{w}_i^*)\leq 1.
\end{equation} 
Then $M_\infty(\sigma(\tau_\cS)^\sigma)$ is free of rank one as a $R_\infty(\tau_\cS)$-module.
\end{thm}

\noindent The proof of Theorem \ref{thm:cyclic} proceeds by proving cases of increasing complexity.
We distinguish five steps in the argument.

\begin{lemma} \label{lemma:diamond}
If  $\sigma \in W^{?}(\rhobar_{\cS})$, then $M_\infty(\sigma)$ is a cyclic $R_\infty(\tau_\cS)$-module.
\end{lemma}
\begin{proof}
The proof follows the methods of Diamond and Fujiwara (\cite{diamond}).
Note that the support of $M_\infty(\sigma)$ is formally smooth as can be checked from Theorem \ref{SWC}.%
Since $M_\infty(\sigma)$ is maximal Cohen--Macaulay over its formally smooth support by Definition \ref{minimalpatching}(\ref{dimd}), it is free over its support by the Auslander--Buchsbaum--Serre theorem and the Auslander--Buchsbaum formula.
Since $e(M_\infty(\sigma)) =1$ by Theorem \ref{SWC}, the free rank is one.
\end{proof}

If $V$ is a $\varpi$-torsion $K$-module satisfying the hypotheses of Lemma \ref{lem:supp}, we write $R_{\infty}(V)$ to denote the quotient $R_\infty/I(V)$.
\emph{For the rest of this subsection, we assume that the weight $\sigma$ satisfies the conditions of Theorem \ref{thm:cyclic}.}

\begin{lemma}\label{lemma:4weight}
Assume that $\ell(\widetilde{w}^*_i) > 1$ for all $i\in\cJ$.
Let $\sigma\in W^?(\rhobar_{\cS},\tau_\cS)$ 
and let $V$ be a nonzero quotient of $\ovl{\sigma}(\tau_\cS)^\sigma$.
Then $M_\infty(V)$ is a cyclic $R_\infty(\tau_\cS)$-module.
\end{lemma}
\begin{proof}
If $R'_\infty(\tau_\cS)$ is formally smooth over $\cO$, then the result follows from Lemma \ref{lemma:diamond}.
Suppose otherwise.
We use the notation of the proof of Lemma \ref{prop saturation 1} with $R = \sigma(\tau_\cS)$.
Recall that for $\sigma \in \JH(\ovl{\sigma(\tau_\cS)})$, $\fp_S(\sigma) = (z_j(\sigma))_{j=1}^N + (p)$.
By Lemma \ref{lemma:expl4wt}, for $\sigma_1$ and $\sigma_2\in W^?(\rhobar_{\cS},\tau_\cS)$, we have
\begin{equation}\label{eqn:comp}
\#\big(\{z_j(\sigma_1)\}_j \Delta \{z_j(\sigma_2)\}_j\big) = 2\dgr{\sigma_1}{\sigma_2}.
\end{equation}

Let $\sigma_1 \in W^?(\rhobar_{\cS},\tau_\cS)$ be such that $\dgr{\sigma}{\sigma_1} =1$.
Furthermore, fix a saturated inclusion $\sigma(\tau_\cS)^{\sigma_1} \into \sigma(\tau_\cS)^\sigma$. 
Letting $M\defeq M'_\infty(\sigma(\tau_\cS)^\sigma)$ and $M_1\defeq M'_\infty(\sigma(\tau_\cS)^{\sigma_1})$, we have a map $M_1\ra M$.
By Proposition \ref{prop saturation 1}, $\sigma(\tau_\cS)^\sigma/\sigma(\tau_\cS)^{\sigma_1}$ is $p$-torsion, 
with Jordan--H\"{o}lder factors determined by Theorem \ref{thm structure}. 
By the proofs of Lemmas \ref{lem:supp} and \ref{prop saturation 1}, the scheme-theoretic support of $M/M_1$ in $\Spec S$ is
\begin{equation}\label{eqn:sts}
\Spec \big( S/\bigcap_{\sigma' \in \JH(\sigma(\tau_\cS)^\sigma/\sigma(\tau_\cS)^{\sigma_1})} \fp_S(\sigma') \big).
\end{equation}
The sets $\{z_j(\sigma_1)\}_j$ and $\{z_j(\sigma)\}_j$ differ at exactly one component, say $j_1$.
The equation (\ref{eqn:comp}) determines $\fp_S(\sigma')$ for all $\sigma'\in \JH(\ovl{\sigma(\tau_\cS)})$ from which one checks that $z_{j_1}(\sigma)$ is in $\fp_S(\sigma')$ for any $\sigma' \in \JH(\sigma(\tau_\cS)^\sigma/\sigma(\tau_\cS)^{\sigma_1})$.
By (\ref{eqn:sts}), we see that $\coker(M_1 \ra M)$ is killed by $z_{j_1}(\sigma)$.
Similarly, we have that $\coker(pM \ra M_1)$ is annihilated by $z_{j_1}(\sigma_1)$.
Hence we have inclusions $z_{j_1}(\sigma)M\subset M_1$ and $z_{j_1}(\sigma_1) M_1 \subset pM$.
Combining these, we have that $pM = z_{j_1}(\sigma)z_{j_1}(\sigma_1)M \subset z_{j_1}(\sigma_1)M_1\subset pM$.
We conclude that the above inclusions are equalities.

Applying the above argument for all $\sigma_1\in W^?(\rhobar_{\cS},\tau_\cS)$ at distance one from $\sigma$, we have that $M'_\infty(\sigma) \cong M'_\infty(\sigma(\tau_\cS)^\sigma)/(\varpi,\{z_i(\sigma)\}_i)$.
The left hand side, and hence the right hand side, is cyclic by Lemma \ref{lemma:diamond}.
We conclude that $M'_\infty(\ovl{\sigma}(\tau_\cS)^\sigma)$ is cyclic by Nakayama's lemma.
Since $M'_\infty(V)$ is a nonzero quotient of $M'_\infty(\ovl{\sigma}(\tau_\cS)^\sigma)$, it too is cyclic.
\end{proof}

As in \S \ref{subsec:comb}, let $\sigma_{(\omega,a)}\defeq F(\Trns_\mu(s\omega,a))$ (where $(s\omega,a) \in \un{\Lambda}_W^\mu \times \mathcal{A}$).
Then $W^?(\rhobar_\cS) = \{ \sigma_{(\omega,a)} : (\omega,a) \in r(\Sigma)\}$.
Similarly, for $\tld{w}\in \Adm^\vee(\eta)$, we have a bijection 
\begin{align*}
(\tld{w}^*)^{-1}(\Sigma) &\risom \JH(\ovl{\sigma(\tau_\cS)})\\
(\omega,a) &\mapsto \sigma_{(\omega,a)}
\end{align*}
where $\tau_{\cS}\defeq\tau(s (w^*)^{-1},\mu-s(\tld{w}^*)^{-1}(0))$.
In what follows, $\sigma\in W^?(\rhobar_\cS,\tau_\cS)$ is a Serre weight satisfying (\ref{up:each:emb}), and $V$ is a nonzero quotient of $\ovl{\sigma}(\tau_\cS)^\sigma$.
We assume that for each $i\in \cJ$, there exists a subset $\Sigma_{V,i}\subseteq (\tld{w}_i^*)^{-1}(\Sigma_0)$ such that
\begin{align}
\prod_{i\in \cJ} \Sigma_{V,i}&\risom \JH(V) \label{eqn:product}\\
(\omega,a)&\mapsto \sigma_{(\omega,a)}=F(\Trns_\mu(s\omega,a))\nonumber
\end{align}
is a bijection.
All representations we consider below satisfy this assumption.

In the following lemmas, we use a gluing procedure to show that $M_\infty(V)$ is cyclic in cases where the set $W^?(\rhobar_{\cS},\tau_\cS)$ is of increasing complexity.
In figures \ref{TableBroom}, \ref{TableL6W},\ref{TableL8W}, \ref{TableL8W1}, and \ref{TableL9W} we give pictorial realization of the gluing procedure employed in the proofs of Lemma \ref{lemma:broom},\ref{lemma:6weight}, \ref{lemma:8weight}, and \ref{lemma:9weight}, respectively.

\begin{lemma}\label{lemma:broom}
Suppose that for each $i\in\cJ$ either $\ell(\widetilde{w}^*_i) > 1$ or $\Sigma_{V,i} \subset \{(\omega,1),\,(0,0),\,(\eps_1,0),\,(\eps_2,0)\}$ for some $\omega\in\{0,\, \eps_1,\,\eps_2\}$.
Then $M_\infty(V)$ is a cyclic $R_\infty(\tau_\cS)$-module.
\end{lemma}
\begin{proof}
Note that by condition (\ref{up:each:emb}) one has $(\omega,1) \in \Sigma_{V,i}$ whenever $\ell(\tld{w}^*_i)\leq 1$.
In this case, we assume without loss of generality that $(\omega,0) \in \Sigma_{V,i}$ (because the quotient of a cyclic module is cyclic).
Let $\overline{\Sigma}_{V,i}$ be the image of $\Sigma_{V,i}$ in $\un{\Lambda}_W$ under the natural projection.
We proceed by induction on 
\begin{equation*}
n = \#\big\{ i\in \cJ:\ell(\widetilde{w}^*_i) \leq 1 \textrm{ and } \#\overline{\Sigma}_{V,i} = 3\big\}
\end{equation*}
If $n=0$ then a casewise check shows that there exists $\tld{w}' \in \Adm(\eta)$ such that for all $i\in \cJ$, $\ell(\tld{w}'_i) > 1$ and $\Sigma_{V,i} \subset \Sigma_{(\tld{w}'_i)^{-1}}$. 
For example, if $\Sigma_{V,i} \subset \{(0,1),(0,0),(\eps_1,0)\}$, then one takes $\tld{w}'_i = \beta\alpha t_{\un{1}}$. %
Let $\tau_\cS'$ be $\tau_\cS(s(w')^{-1},\mu+s(\tld{w}')^{-1}(0))$ where $w'\in \un{W}$ is the image of $\tld{w}'$.
Then $M_\infty(V)$ is a cyclic $R_\infty(\tau_\cS')$-module by Lemma \ref{lemma:4weight}.
Hence $M_\infty(V)$ is a cyclic as a $R_\infty$-module, and thus it is a cyclic as an $R_\infty(\tau_\cS)$-module as well.

Suppose now that $n>0$ so that there is an $i'\in \cJ$ such that $\#\overline{\Sigma}_{V_{i'}} = 3$ and $\ell(\tld{w}^*_{i'})\leq 1$.
By Theorem \ref{thm structure} and Propositions \ref{prop:subclosed} and \ref{prop:subgraph}, there are quotients $V^1$ and $V^2$ of $V$ satisfying (\ref{eqn:product}) with the following additional properties %
\begin{enumerate}
\item $\Sigma_{V,i'} = {\Sigma}_{V^1,i'}\cup{\Sigma}_{V^2,i'}$, \label{property:union}
\item $(\omega,0)\in \Sigma_{V^1,i'}\cap\Sigma_{V^2,i'}$,
\item $\#{\overline{\Sigma}}_{V^1,i'},\#{\overline{\Sigma}}_{V^2,i'} = 2$, and
\item ${\Sigma}_{V^j,i} = \Sigma_{V,i}$ for all $i \neq i'$ and $j=1,2$.
\end{enumerate}
For example, if $\Sigma_{V,i'} = \{(0,1),(0,0),(\eps_1,0),(\eps_2,0)\}$, then one takes $\Sigma_{V^1,i'} = \{(0,1),(0,0),(\eps_1,0)\}$ and ${\Sigma}_{V^2,i'} = \{(0,1),(0,0),(\eps_2,0)\}$ (cf.~ Table \ref{TableBroom}).

Let $V^3$ be the quotient of $V^1$ and $V^2$ satisfying (\ref{eqn:product}) such that $\Sigma_{V^3,i} = \Sigma_{V^1,i}\cap\Sigma_{V^2,i}$ for all $i$.
By the inductive hypothesis, $M'_\infty(V^1)$, $M'_\infty(V^2)$, and $M'_\infty(V^3)$ are cyclic.
Let $I_j$ be $\mathrm{Ann}_{R'_\infty(\tau_\cS)} M'_\infty(V^j)$ for $j = 1,2,3$.
By (\ref{property:union}), we have an injection $V \into V^1 \oplus V^2$.
In fact, this injection lands in $V^1 \times_{V^3} V^2$.
Comparing lengths, we see see that the injection $V \into V^1 \times_{V^3} V^2$ is in fact an isomorphism, which, using the exactness of $M_\infty$ (see Definition \ref{minimalpatching}), gives the first isomorphism below:
\begin{align*}
M'_\infty(V) &\cong M'_\infty(V^1) \times_{M'_\infty(V^3)} M'_\infty(V^2) \\
&\cong R'_\infty(\tau_\cS)/I_1 \times_{R'_\infty(\tau_\cS)/I_3} R'_\infty(\tau_\cS)/I_2 \\
&\cong R'_\infty(\tau_\cS)/I_1 \times_{R'_\infty(\tau_\cS)/(I_1+I_2)} R'_\infty(\tau_\cS)/I_2 \\
&\cong R'_\infty(\tau_\cS)/(I_1\cap I_2).
\end{align*}
The second isomorphism follows from Lemma \ref{lem:supp}. 
The third isomorphism follows from the fact that $I_1+I_2 = I_3$ by Theorem \ref{thm:matching}, Table \ref{Table:intsct}, Lemmas \ref{lem:ideal} and \ref{lem:ideal:1:alpha}(\ref{it:ideal:1:alpha:3}).
We deduce in particular that $M'_\infty(V)$ is cyclic. %
\end{proof}
We remind the reader that our parametrization of Serre weights $\sigma_{(\omega,a)}$ is ``centered at $\rhobar_\cS$''. Thus $\sigma_{(\omega,a)}\in W^?(\rhobar_\cS)$ exactly when $(\omega,a)\in r(\Sigma)$, and $\sigma_{(\omega,a)}\in \JH(\ovl{\sigma(\tau_\cS)})$ exactly when $(\omega,a)\in(\tld{w}^*)^{-1}(\Sigma)$.

\begin{lemma} \label{lemma:6weight}
Suppose that for all $i\in\cJ$ either $\ell(\widetilde{w}^*_i) > 1$ or 
\[\Sigma_{V,i} \subset (\tld{w}_i^*)^{-1} (\Sigma_0 \setminus \{(\nu_1,0),(\nu_2,1),(\nu_3,0)\})\]
where $(\nu_1,\nu_2,\nu_3)$ is $(\eps_1-\eps_2,\eps_1, \eps_1+\eps_2)$, $(\eps_2-\eps_1, \eps_2, \eps_1+\eps_2)$, or $(\eps_1-\eps_2,0,\eps_2-\eps_1)$.
Then $M_\infty(V)$ is a cyclic $R_\infty(\tau_\cS)$-module.
\end{lemma}
\begin{proof}
We note that $\#(\Sigma_{V,i}\cap r(\Sigma_0)) \leq 5$ for all $i\in \cJ$.
We proceed by induction on $n = \#\{ i\in \cJ:\#(\Sigma_{V,i}\cap r(\Sigma_0)) = 5\}$.
There is a unique $\tld{w}'_i \in \Adm(2,1,0)$ such that 
\begin{equation}
\label{eq:intsct:types}
(\tld{w}_i^*)^{-1} (\Sigma_0 \setminus \{(\nu_1,0),(\nu_2,1),(\nu_3,0)\})= (\tld{w}_i^*)^{-1} (\Sigma_0) \cap (\tld{w}'_i)^{-1} (\Sigma_0).
\end{equation}
One can check that if $\#((\tld{w}_i^*)^{-1} (\Sigma_0 \setminus \{(\nu_1,0),(\nu_2,1),(\nu_3,0)\}) \cap r(\Sigma_0)) \leq 4$, then one of $\ell(\tld{w}^*_i)$ and $\ell(\tld{w}'_i)$ is strictly greater than one.
If we change the type $\tau_\cS$ so that $\tld{w}^*_i$ is replaced by $\tld{w}'_i$, but $\tld{w}^*_{i'}$ are unchanged for $i' \neq i$, there is still a surjection $\ovl{\sigma}(\tau_\cS)^\sigma\onto V$ by (\ref{eq:intsct:types}) and Theorem \ref{thm structure}.
We can therefore assume without loss of generality that for each $i$, either $\ell(\tld{w}^*_i)>1$ or $\#((\tld{w}_i^*)^{-1} (\Sigma_0 \setminus \{(\nu_1,0),(\nu_2,1),(\nu_3,0)\}) \cap r(\Sigma_0)) = 5$.

Suppose that $n=0$.
If $\ell(\tld{w}^*_i) \leq 1$ for some $i\in \cJ$, then $\#((\tld{w}_i^*)^{-1} (\Sigma_0 \setminus \{(\nu_1,0),(\nu_2,1),(\nu_3,0)\}) \cap r(\Sigma_0)) = 5$ by assumption.
On the other hand, $\Sigma_{V,i}\cap r(\Sigma_0) < 5$, which implies that $\Sigma_{V,i}\cap r(\Sigma_0) \subset \{(\omega,1),(0,0),(\eps_1,0),(\eps_2,0)\}$ where $\omega = 0, \eps_1,$ or $\eps_2$ by considerations of submodule structure.
Hence, there is a quotient $V'$ of $V$ such that $V'$ is of the form in Lemma \ref{lemma:broom} and the induced map $M_\infty(V) \ra M_\infty(V')$ is an isomorphism by exactness of $M_\infty$ and Theorem \ref{SWC}.
We are then done by Lemma \ref{lemma:broom}.

Suppose now that $n > 0$.
Suppose that $\#(\Sigma_{V,i'}\cap r(\Sigma_0)) = 5$ for some $i'\in \cJ$.
This in particular implies that $\#(\Sigma_{V,i'})=6$, by consideration of submodule structure given in Theorem \ref{thm structure}.
Then $V$ has a unique quotient $V^1$ such that $\Sigma_{V^1,i} = \Sigma_{V,i}$ if $i \neq i'$ and $\Sigma^1_{V,i'} = {\Sigma}_{V,i'}\setminus \{(\tld{w}^*_{i'})^{-1}(\nu_4,1)\}$ where $\nu_4 = 0,\eps_1,$ or $\eps_2$ and $\nu_4 \neq \nu_2$ (cf.~ Figure \ref{TableL6W} for an example in the particular case where $\tld{w}_{i'}=t_{\un{1}}$).
Then $M_\infty(V^1)$ is cyclic by the induction hypothesis.
There is also a submodule $V^2 \subset V$ such that $\Sigma_{V^2,i} = \Sigma_{V,i}$ if $i \neq i'$ and $\# \Sigma_{V^2,i'} = \#\big(\Sigma_{V^2,i'} \cap r(\Sigma_0) \big)= 2$ (cf.~ Figure \ref{TableL6W}; note also that $(\tld{w}^*_{i'})^{-1}(\nu_4,1)\in {\Sigma}_{V^2,i'}$ since $\soc\, V^2\subseteq \soc\,V$).
One can check that $M_\infty(V^2)$ is cyclic by the inductive hypothesis.
Then letting $M'' = \ker(M_\infty(V) \surj M_\infty(V^1))$ and $M' = M_\infty(V^2)$, \cite[Lemma 10.1.13]{EGS} implies that $M_\infty(V)$ is cyclic.
\end{proof}

\begin{lemma}\label{lemma:8weight}
Suppose that for all $i\in\cJ$ either $\ell(\widetilde{w}_i) > 1$ or $\Sigma_{V,i} \subset (\tld{w}_i^*)^{-1} (\Sigma_0 \setminus \{(\nu,0)\})$ where $\nu$ is $\eps_1-\eps_2,\eps_2-\eps_1,$ or $\eps_1+\eps_2$.
Then $M_\infty(V)$ is a cyclic $R_\infty(\tau_\cS)$-module.
\end{lemma}
In the setting of Lemma \ref{lemma:8weight} note that when $\ell(\tld{w}_i)\leq 1$ the condition $\Sigma_{V,i} \subset (\tld{w}_i^*)^{-1} (\Sigma_0 \setminus \{(\nu,0)\})$, $\nu\in\{\eps_1-\eps_2,\eps_2-\eps_1,\eps_1+\eps_2\}$ is equivalent to $\# \Sigma_{V,i}<9$.
\begin{proof}
We proceed by induction on 
\begin{align*}
n=\#\{i\in \cJ : &\ell(\tld{w}^*_i) \leq 1 \textrm{ and }\Sigma_{V,i} \not\subset (\tld{w}_i^*)^{-1}(\Sigma_0 \setminus \{(\nu_1,0),(\nu_2,1),(\nu_3,0)\}) \textrm{ for } \\
&(\nu_1,\nu_2,\nu_3) = (\eps_1-\eps_2,\eps_1, \eps_1+\eps_2), (\eps_2-\eps_1, \eps_2, \eps_1+\eps_2), \textrm{ and } (\eps_1-\eps_2,0,\eps_2-\eps_1)\}.
\end{align*}
The case $n=0$ is covered by Lemma \ref{lemma:6weight}.
Suppose that $n>0$, and that $i' \in \cJ$ with $\ell(\tld{w}^*_{i'})\leq 1$ and
\begin{equation*}
\Sigma_{V,i'} \not\subset (\tld{w}_{i'}^*)^{-1}(\Sigma_0 \setminus \{(\nu_1,0),(\nu_2,1),(\nu_3,0)\})
\end{equation*} for $(\nu_1,\nu_2,\nu_3)\in\{ (\eps_1-\eps_2,\eps_1, \eps_1+\eps_2),\ (\eps_2-\eps_1, \eps_2, \eps_1+\eps_2),\ (\eps_1-\eps_2,0,\eps_2-\eps_1)\}$.
Assume without loss of generality that $\Sigma_{V,i'} = (\tld{w}_{i'}^*)^{-1}(\Sigma_0 \setminus \{(\nu,0)\})$, as the quotient of a cyclic module is again cyclic.
We can further assume without loss of generality that $\nu = \eps_1+\eps_2$ (which implies that if $\sigma = \sigma_{(\omega,a)}$, then $(\omega_{i'},a_{i'}) = (\tld{w}^*_{i'})^{-1}(0,1)$).
Then there are quotients $V^1$ and $V^2$ of $V$ such that 
\begin{enumerate}
\item ${\Sigma}_{V,i'} = \Sigma_{V^1,i'} \cup \Sigma_{V^2,i'}$,
\item $\Sigma_{V^j,i} = \Sigma_{V,i}$ for $i \neq i'$ and $j=1,2$,
\item $\Sigma_{V^1,i'} \subset (\tld{w}_{i'}^*)^{-1}(\Sigma_0 \setminus \{(\eps_1-\eps_2,0), (\eps_1,1), (\eps_1+\eps_2,0)\})$, and
\item $\Sigma_{V^2,i'} \subset (\tld{w}_{i'}^*)^{-1}(\Sigma_0 \setminus \{(\eps_2-\eps_1,0), (\eps_2,1), (\eps_1+\eps_2,0)\})$
\end{enumerate}
(cf.~ Figure \ref{TableL8W} for an example when $\tld{w}_{i'}=\alpha$).
We now argue by induction as in the proof of Lemma \ref{lemma:broom} using Theorem \ref{thm:matching}, Lemmas \ref{lem:ideal:1:id} and \ref{lem:ideal:1:alpha}(\ref{it:ideal:1:alpha:1}) and \ref{lem:ideal:1:alpha}(\ref{it:ideal:1:alpha:2}).
Note that the analogue of Lemma \ref{lem:ideal:1:alpha} for the shapes $\beta$ and $\gamma$ hold symmetrically (see Remark \ref{rmk:chg:var}). 
\end{proof}

\begin{lemma}\label{lemma:9weight}
With $V$ as described before Lemma \ref{lemma:broom}, $M_\infty(V)$ is a cyclic $R_\infty(\tau_\cS)$-module.
\end{lemma}
\begin{proof}
We proceed by induction on 
\begin{align*}
n=\#\{i\in \cJ:\ &\ell(\tld{w}^*_i) \leq 1 \textrm{ and }\#\Sigma_{V,i}=9\}.
\end{align*}
(In the definition of $n$, note that the condition $\#\Sigma_{V,i}=9$ can be replaced by: for all $\nu\in\{ \eps_1-\eps_2, \eps_2-\eps_1, \eps_1+\eps_2\}$ one has $\Sigma_{V,i} \not\subset (\tld{w}_i^*)^{-1}(\Sigma_0 \setminus \{(\nu,0)\})$.)
The case $n=0$ is covered by Lemma \ref{lemma:8weight}.
Suppose that $n>0$.
Let $i'\in \cJ$ be such that $\ell(\tld{w}^*_{i'}) \leq 1$ and $\#\Sigma_{V,i'}=9$ (equivalently, $\ell(\tld{w}^*_{i'}) \leq 1$ and $\Sigma_{V,i'} = (\tld{w}_{i'}^*)^{-1}(\Sigma_0)$).
Then $V$ has a quotient $V^1$ such that $\Sigma_{V^1,i} = \Sigma_{V,i}$ if $i \neq i'$ and $\Sigma_{V^1,i'} = \Sigma_{V,i'}\setminus \{(\tld{w}_{i'}^*)^{-1}(\nu,0)\}$ where $\nu = \eps_1-\eps_2,\eps_2-\eps_1,$ or $\eps_1+\eps_2$ (cf.~ Figure \ref{TableL9W} for an example when $\tld{w}^*_{i'}=\gamma^+ t_{\un{1}}$).
If the map $M_\infty(V) \ra M_\infty(V^1)$ is an isomorphism, we are done.
Otherwise, there is a submodule $V^2$ of $V$ such that $\Sigma_{V^2,i} = {\Sigma}_{V,i}$ if $i \neq i'$, $\#\Sigma_{V^2,i'} =\#\big(\Sigma_{V^2,i'} \cap r(\Sigma_0) \big)= 2$, and $(\tld{w}_{i'}^*)^{-1}(\nu,0) \in \Sigma_{V^2,i'}\subset r(\Sigma_0)$.
Then one argues as in the proof of Lemma \ref{lemma:6weight}.
\end{proof}

\begin{proof}[Proof of Theorem \ref{thm:cyclic}]
Lemma \ref{lemma:9weight} implies that $M_\infty(\ovl{\sigma}(\tau_\cS)^\sigma)$ is a cyclic $\ovl{R}_\infty(\tau)$-module.
Nakayama's Lemma implies that $M_\infty(\sigma(\tau_\cS)^\sigma)$ is a cyclic $R_\infty(\tau_\cS)$-module.
By Theorem \ref{SWC}, $M_\infty(\sigma)$, and thus $M_\infty(\sigma(\tau_\cS)^\sigma)$, is nonzero.
By Theorem \ref{thm:BM}, $R_\infty(\tau_\cS)$ is irreducible and reduced, so that Definition \ref{minimalpatching}(\ref{dimd}) implies that $M_\infty(\sigma(\tau_\cS)^\sigma)$ is a faithful $R_\infty(\tau_\cS)$-module.
Since a faithful cyclic module is free of rank one, we are done.
\end{proof}

Finally, we show that the hypothesis of Theorem \ref{thm:cyclic} is actually neccessary:
\begin{prop}
Let $\tau_\cS$ be a $13$-generic tame type and $\sigma \defeq F(\lambda) \in W^?(\rhobar_{\cS},\tau_\cS)$. Assume that for some $i \in \cJ$, 
\begin{equation}
\label{low:each:emb}
\lambda_{\pi^{-1}(i)}\in X_1(T) \text{ is in alcove $C_0$ and } \ell(\tld{w}_i^*)\leq 1.
\end{equation} 
Then $M_\infty(\sigma(\tau_\cS)^\sigma)$ is not free as an $R_\infty(\tau_\cS)$-module.
\end{prop}
\begin{proof} As we will not make use of this, we only give a sketch of proof. Unlike the rest of this section we will parametrize the weights by centering around $\tau_\cS$, so that $\tau_\cS=\tau(s,\mu)$ and $\sigma_{(\omega,a)}$ means $\sigma_{(\omega,a)}\defeq F(\Trns_\mu(s\omega,a))$. In particular, $W^?(\rhobar_{\cS},\tau_\cS)$ consists of $\sigma_{(\omega,a)}$ such that $(\omega,a)\in \Sigma_{\tld{w}^*}$.

We write $\sigma=\sigma_{(\omega,a)}$, so $(\omega,a)\in \Sigma$. For each $i$ satisfying (\ref{low:each:emb}), we have $a_i=0$. If moreover $(\omega_i,a_i)\in \Sigma_0^{\mathrm{inn}}$, we set $(\omega'_i,a'_i)=(\omega_i,1)$, otherwise we set $(\omega'_i,a'_i)$ to be one of the two elements in $\Sigma_0$ that are adjacent to $(\omega_i,a_i)$. For $i$ not satisfying (\ref{low:each:emb}), we set $(\omega'_i,a'_i)=(\omega_i,a_i)$. Thus we obtain a weight $\sigma'=\sigma_{(\omega',a')}$ which satisfies the hypothesis of Theorem \ref{thm:cyclic}.

We now fix a saturated inclusion $\imath:\sigma(\tau_\cS)^{\sigma} \into \sigma(\tau_\cS)^{\sigma'}$, and let $C$ denote the cokernel. Since $M_{\infty}( \sigma(\tau_\cS)^{\sigma'})\cong R_\infty(\tau_\cS)$, $M_\infty(\imath)$ identifies $M_{\infty}( \sigma(\tau_\cS)^{\sigma})$ with an ideal $I(C)$ of $R_\infty(\tau_\cS)$. Thus we need to show that $I(C)$ is not a principal ideal, and to do so it suffices to show that the image of $I(C)$ in $\ovl{R}_\infty(\tau_\cS)$ is not principal. 

Let $V$ denote the cokernel of $\imath \mod \varpi$. Using Theorem \ref{thm structure}, we see that $V$ is multiplicity free, and $\JH(V)\cap W^?(\rhobar_\cS)$ consists of $\sigma_{(\nu,c)}$ in $(\nu,c)\in \Sigma_{\tld{w}^*}\setminus\prod_i \Sigma_{V^c,i}$, where: %
\begin{itemize}
\item if $i$ does not satisfy (\ref{low:each:emb}), $\Sigma_{V^c,i}=  \Sigma_{\tld{w}_i^*}$; 
\item if $i$ satisfies (\ref{low:each:emb}) and  $(\omega_i,a_i)\in \Sigma_0^{\mathrm{inn}}$, 
$\Sigma_{V^c,i}=\{(\omega_i,0),(\mu_1,1),(\mu_2,1)\}$
where $\{\omega_i,\mu_1,\mu_2\}=\{0,\eps_1,\eps_2\}$; and
\item if $i$ satisfies (\ref{low:each:emb}) and $(\omega_i,a_i)\notin \Sigma_0^{\mathrm{inn}}$, $\Sigma_{V^c,i}=\big\{(\omega_i,0), (\omega''_i,1)\big\}$ where $\omega''_i$ is such that $\{\omega_i,\omega'_i,\omega''_i\}=\{\eps_2-\eps_1,0,\eps_2\}$, $\{\eps_1-\eps_2,0,\eps_1\}$, or $\{\eps_1+\eps_2,\eps_1,\eps_2\}$.
\end{itemize} 
Lemma \ref{lem:supp} shows that $M_\infty(V)\cong \ovl{R}_{\infty}(\tau_\cS)/I(V)$, where $I(V)$ is the intersection, over $\kappa\in \JH(V)\cap W^?(\rhobar_{\cS})$, of the ideals $\fp(\kappa)$ defined in Proposition \ref{prop:identify:cmpt}.
Note that the image of $I(C)$ in $\ovl{R}_{\infty}(\tau_\cS)$ is $I(V)\ovl{R}_{\infty}(\tau_\cS)$. 
Thus we need to show $I(V)\ovl{R}_{\infty}(\tau_\cS)$, or equivalently, $I(V)\ovl{R}'_{\infty}(\tau_\cS)$ is not principal.
Recall that there is a formally smooth map $\widehat{\otimes}_{i\in\cJ}\ovl{R}_{\ovl{\fM}_{\tld{v}},\tld{w}_{i}}^{\expl,\nabla}\to \ovl{R}'_{\infty}(\tau_\cS)$. 
We claim that the ideal $I(V)$ comes from $I \defeq\widehat{\bigotimes}_i\Big( \underset{(\nu_i,c_i)\in \Sigma_{\tld{w}_i^*}\setminus\Sigma_{V^c,i}}{\bigcap}\fC_{(\nu_i,c_i)} \ovl{R}_{\ovl{\fM}_{\tld{v}},\tld{w}_{f_{\tld{v}}-1-i}}^{\expl,\nabla}\Big)$. 
To check the claim, let $\ovl{R}_{\ovl{\fM},\tld{w}}^{\expl,\nabla}\defeq \widehat{\otimes}_{i\in\cJ}\ovl{R}_{\ovl{\fM}_{\tld{v}},\tld{w}_{i}}^{\expl,\nabla}$.
We first check that $\ovl{R}_{\ovl{\fM},\tld{w}}^{\expl,\nabla}/I$ has the expected cycle. To do this, we check using Table \ref{Table:intsct} that  $\ovl{R}_{\ovl{\fM}_{\tld{v}},\tld{w}_{f_{\tld{v}}-1-i}}^{\expl,\nabla}/\underset{(\nu_i,c_i)\in \Sigma_{\tld{w}_i^*}\setminus\Sigma_{V^c,i}}{\bigcap}\fC_{(\nu_i,c_i)} \ovl{R}_{\ovl{\fM}_{\tld{v}},\tld{w}_{f_{\tld{v}}-1-i}}^{\expl,\nabla}$ is filtered by  $\ovl{R}_{\ovl{\fM}_{\tld{v}},\tld{w}_{f_{\tld{v}}-1-i}}^{\expl,\nabla}/\fC_{(\nu_j,c_j)} $ for $(\nu_j,c_j)\in \Sigma_{\tld{w}_i^*}\setminus\Sigma_{V^c,i}$, and that $\underset{(\nu_i,c_i)\in \Sigma_{\tld{w}_i^*}\setminus\Sigma_{V^c,i}}{\bigcap}\fC_{(\nu_i,c_i)} \ovl{R}_{\ovl{\fM}_{\tld{v}},\tld{w}_{f_{\tld{v}}-1-i}}^{\expl,\nabla}$ is filtered by  $\ovl{R}_{\ovl{\fM}_{\tld{v}},\tld{w}_{f_{\tld{v}}-1-i}}^{\expl,\nabla}/\fC_{(\nu_j,c_j)}$ for $(\nu_j,c_j)\in \Sigma_{V^c,i}$. Thus we can filter the quotient $\ovl{R}_{\ovl{\fM},\tld{w}}^{\expl,\nabla}/I$ with factors of the form $\widehat{\otimes}_{i\in\cJ}\ovl{R}_{\ovl{\fM}_{\tld{v}},\tld{w}_{i}}^{\expl,\nabla}/\fp_i$, where the $\fp_i$ are minimal primes of $\ovl{R}_{\ovl{\fM}_{\tld{v}},\tld{w}_{i}}^{\expl,\nabla}$ (note that taking completion is an exact operation).
As $\Gamma(\mathrm{im}(\imath \pmod{\varpi}))$ decomposes as a product over $\cJ$, an inductive argument using Theorem \ref{thm:matching} and the above filtration on $\ovl{R}_{\ovl{\fM},\tld{w}}^{\expl,\nabla}/I$ shows that the cycle of $\ovl{R}_{\ovl{\fM},\tld{w}}^{\expl,\nabla}/I$ is given by the components of  $\ovl{R}_{\ovl{\fM},\tld{w}}^{\expl,\nabla}$ labeled by the Serre weights corresponding to $\Gamma(\mathrm{im}(\imath \pmod{\varpi}))$. 
In order to prove the claim we are left to show that $I$ is a radical ideal.
To see this, we make the following observation: If $R$, $S$ are two reduced local Noetherian rings over a perfect field $k$, and $J$, $L$ are radical ideals of $R$,$S$, then $J\otimes_k L$ is a radical ideal of $R\otimes_k S$. 
This is because $(R\otimes S)/(J \otimes L)$ embeds into $((R/J) \otimes S) \times (R \otimes (S/L))$ 
(as can be seen by choosing bases of $J$,$L$ as $k$-vector spaces and extending them to bases of $R$,$S$), and the latter ring is reduced.
Finally, we note that in the situation above the property of being a radical ideal is preserved by completion, so that $J\widehat{\otimes}_k L$ is also radical in $R\widehat{\otimes}_k S$.
The claim is proven.

By Nakayama's lemma, the size of a minimal set of generators of $I$ is given by the dimension over $\bF$ of
\[ I\otimes_{\widehat{\otimes}_{i\in\cJ}\ovl{R}_{\ovl{\fM}_{\tld{v}},\tld{w}_{i}}^{\expl,\nabla}} \bF=\bigotimes_i \big(\underset{(\nu_i,c_i)\in \Sigma_{\tld{w}_i^*}\setminus\Sigma_{V^c,i}}{\bigcap}\fC_{(\nu_i,c_i)} \otimes_{\ovl{R}_{\ovl{\fM}_{\tld{v}},\tld{w}_{f_{\tld{v}}-1-i}}^{\expl,\nabla}} \F \big)\]
Hence, it suffices to show that for $i$ satisfying (\ref{low:each:emb}), $\underset{(\nu_i,c_i)\in \Sigma_{\tld{w}_i^*}\setminus\Sigma_{V^c,i}}{\bigcap}\fC_{(\nu_i,c_i)} \otimes_{\ovl{R}_{\ovl{\fM}_{\tld{v}},\tld{w}_{f_{\tld{v}}-1-i}}^{\expl,\nabla}} \F$ has dimension greater than one, which can be checked from Table \ref{Table:intsct}.
\end{proof}

\subsection{Gauges for patching functors}\label{sec:gauge}

In this subsection, we compute the image of maps between patched modules for lattices in Deligne--Lusztig representations.
This can be viewed as a calculation of lattice gauges in families.
The main ingredient for this section is Theorem \ref{thm:cyclic}, after which algebro-geometric arguments using the projection formula and the Cohen--Macaulay property prove Theorem \ref{thm:gauge}. We learned these arguments from M. Emerton.

We continue to use the weak minimal patching functor setting of \S \ref{sec:SWC}.
Let $\cS$, $\rhobar_{\cS}$, $R_\infty$, $X_\infty$, and $M_\infty$ be as in \S \ref{sec:cyc}.
Let $\tau_\cS$ be $(\tau_{\tld{v}})_{\tld{v}\in\cS}$, where $\tau_{\tld{v}}$ is an $13$-generic tame type.
Let $X_{\infty}(\tau_\cS)\defeq \Spec R_{\infty}(\tau_\cS)$.
It is a normal and Cohen--Macaulay scheme by (the proof of) Theorem \ref{thm:BM}.
Let $Z \subset X_{\infty}(\tau_\cS)$ be the locus of points lying on two irreducible components of the special fiber of $X_{\infty}(\tau_\cS)$.
Note that the codimension of $Z\subset X_{\infty}(\tau_\cS)$ is two.
Let 
\begin{equation*}
j: U \defeq X_{\infty}(\tau_\cS) \setminus Z \into X_{\infty}(\tau_\cS)
\end{equation*}
be the natural open immersion. Note that $j$ and $U$ comes from pulling back the open immersion $j_0: U_0\into \Spf( \widehat{\underset{\tld{v}\in \cS}{\bigotimes}} R_{\tld{v}}^{\tau_{\tld{v}}})$ via the formally smooth map $X_\infty(\tau_{\cS})\to  \Spf(\widehat{\underset{\tld{v}\in \cS}{\bigotimes}} R_{\tld{v}}^{\tau_{\tld{v}}})$, where $U_0$ is defined in an analogous way as $U$.

\begin{lemma}\label{lemma:Uregular}
The scheme $U$ is regular.
\end{lemma}
\begin{proof}
The irreducible components of the special fiber of $X_\infty(\tau_\cS)$ are formally smooth over $\F$, hence the special fiber $\overline{U}$ of $U$ is regular by the last part of Theorem \ref{thm:BM}.
The dimension of the tangent space of $U$ at a characteristic $p$ point is at most one more than the dimension of the tangent space of $\ovl{U}$ at that point.
By $p$-flatness, the Krull dimension of $U$ is one more than the Krull dimension of $\overline{U}$, and so $U$ is regular at characteristic $p$ points.
Since the generic fiber of $U$, which is isomorphic to the generic fiber of $X_\infty(\tau_\cS)$, is regular, $U$ is regular.
\end{proof}

We now use the notation $j^*$ and $j_*$ which take quasi-coherent sheaves on $X_\infty(\tau_\cS)$ to those on $U$ and vice versa, respectively.

\begin{lemma} \label{lemma:Ugauge}
Let $\sigma,\,\kappa\in \JH(\ovl{\sigma(\tau_\cS)})$ and let $\iota: \sigma(\tau_\cS)^\kappa \into \sigma(\tau_\cS)^\sigma$ be a saturated injection.
For any $\theta = \otimes_{\tld{v}\in\cS} \theta_{\tld{v}}\in W^?(\rhobar_{\cS})$ let $m(\theta)$ be the multiplicity with which $\theta$ appears in the cokernel of $\iota$.
Then the induced injection $j^*M_\infty(\iota):  j^*M_\infty(\sigma(\tau_\cS)^\kappa) \into j^*M_\infty(\sigma(\tau_\cS)^\sigma)$ has image 
\[j^* \Big(\prod_{\theta\in W^?(\rhobar_{\cS})} \mathfrak{p}(\theta)^{m(\theta)}M_\infty(\sigma(\tau_\cS)^\sigma)\Big),\]
where $\mathfrak{p}(\theta) = \sum_{\tld{v}\in\cS} \mathfrak{p}(\theta_{\tld{v}})$.
\end{lemma}
\begin{proof}
Note that $j^*( M_\infty(\sigma(\tau_\cS)^\kappa))$ and $j^*( M_\infty(\sigma(\tau_\cS)^\sigma))$ are locally free (of rank one), since they are Cohen--Macaulay of full support over the regular scheme $U$ (see Definition \ref{minimalpatching}).
Then the image of $j^*M_\infty(\iota)$ is $J \otimes_{\cO_U} j^*( M_\infty(\sigma(\tau_\cS)^\sigma))$ where $J$ is the ideal sheaf of the Cartier divisor corresponding to the cokernel of $j^*M_\infty(\iota)$.
It is easily seen that $J =  j^* \Big(\prod_{\theta\in W^?(\rhobar_{\cS})} \mathfrak{p}(\theta)^{m(\theta)}R_\infty(\tau_\cS)\Big)$.
Finally, we have that
\begin{align*}
j^* \Big(\prod_{\theta\in W^?(\rhobar_{\cS})} \mathfrak{p}(\theta)^{m(\theta)}R_\infty(\tau_\cS)\Big) \otimes_{\cO_U} j^*\Big( M_\infty(\sigma(\tau_\cS)^\sigma)\Big) \cong & j^* \Big(\prod_{\theta\in W^?(\rhobar_{\cS})} \mathfrak{p}(\theta)^{m(\theta)}R_\infty(\tau_\cS)\Big) j^*\Big( M_\infty(\sigma(\tau_\cS)^\sigma)\Big) \\
= &j^* \Big(\prod_{\theta\in W^?(\rhobar_{\cS})} \mathfrak{p}(\theta)^{m(\theta)}M_\infty(\sigma(\tau_\cS)^\sigma)\Big)
\end{align*}
where the isomorphism follows from the fact that $j^*( M_\infty(\sigma(\tau_\cS)^\sigma))$ is locally free.
\end{proof}

Let $\sigma$ be as in Theorem \ref{thm:cyclic}, so that $M_\infty(\sigma(\tau_\cS)^\sigma)$ is cyclic.

\begin{thm}\label{thm:gauge}
With the notation of Lemma \ref{lemma:Ugauge}, suppose further that $\sigma$ is as in Theorem \ref{thm:cyclic}.
Then the induced injection $M_\infty(\iota):  M_\infty(\sigma(\tau_\cS)^\kappa) \into M_\infty(\sigma(\tau_\cS)^\sigma)$ has image 
\[j_*j^* \Big(\prod_{\theta\in W^?(\rhobar_{\cS})} \mathfrak{p}(\theta)^{m(\theta)}R_\infty(\tau_\cS)\Big) M_\infty(\sigma(\tau_\cS)^\sigma).\]
\end{thm}
\begin{proof}
If $\mathcal{M}$ is a Cohen--Macaulay sheaf on $X_\infty(\tau_\cS)$, then $j_*j^* \mathcal{M} = \mathcal{M}$ since the codimension of $Z$ is two (cf.~ \cite[Proposition 3.5]{hasset-kovacs}).
Hence the image of $M_\infty(\iota)$ is 
\[j_*j^* \Big(\prod_{\theta\in W^?(\rhobar_{\cS})} \mathfrak{p}(\theta)^{m(\theta)}M_\infty(\sigma(\tau_\cS)^\sigma)\Big)\]
by Lemma \ref{lemma:Ugauge}.
Since $M_\infty(\sigma(\tau_\cS)^\sigma)$ is free over $R_\infty(\tau_\cS)$, we have that 
\[j_*j^* \Big(\prod_{\theta\in W^?(\rhobar_{\cS})} \mathfrak{p}(\theta)^{m(\theta)}M_\infty(\sigma(\tau_\cS)^\sigma)\Big) = j_*j^* \Big(\prod_{\theta\in W^?(\rhobar_{\cS})} \mathfrak{p}(\theta)^{m(\theta)}R_\infty(\tau_\cS)\Big) M_\infty(\sigma(\tau_\cS)^\sigma).\]
\end{proof}
\begin{rmk}\label{rmk:descend} As $j$ comes from pulling back $j_0$ via $X_\infty(\tau_{\cS})\to  \Spf(\widehat{\underset{\tld{v}\in \cS}{\bigotimes}} R_{\tld{v}}^{\tau_{\tld{v}}})$ and $\fp(\theta)$ are ideals of $  \Spf(\widehat{\underset{\tld{v}\in \cS}{\bigotimes}} R_{\tld{v}}^{\tau_{\tld{v}}})$, we see that the ideal
\[j_*j^* \Big(\prod_{\theta\in W^?(\rhobar_{\cS})} \mathfrak{p}(\theta)^{m(\theta)}R_\infty(\tau_\cS)\Big)=j_{0*}j_0^* \Big(\prod_{\theta\in W^?(\rhobar_{\cS})} \mathfrak{p}(\theta)^{m(\theta)}\widehat{\underset{\tld{v}\in \cS}{\bigotimes}} R_{\tld{v}}^{\tau_{\tld{v}}}\Big)R_\infty(\tau_\cS)\]
comes from an ideal in $\widehat{\otimes}_{\tld{v}\in \cS} R_{\tld{v}}^\square$.
\end{rmk}
\subsection{Global applications}\label{subsec:global}

In this subsection, we deduce generalizations of conjectures of Demb\'el\'e and Breuil on mod $p$ multiplicity one and lattices, respectively (see \cite[Conjectures B.1 and 1.2]{breuil-buzzati}).
Theorem \ref{thm:modpmultone} follows immediately from Theorem \ref{thm:cyclic}.
While Theorem \ref{thm:lattice} also follows from Theorem \ref{thm:gauge}, it is crucial that the image of the map between two patched modules given in Theorem \ref{thm:gauge} is described by an ideal.
This is far from formal, and relies crucially on Theorem \ref{thm:cyclic}.

We use the setup in \cite[\S 7.1]{LLLM}.
Let $F/\Q$ be a CM field with maximal totally real subfield $F^+\neq \Q$ and let $\Sigma_p^+$ (resp.~ $\Sigma_p$) be the set of places of $F^+$ (resp.~ of $F$) lying above $p$.
Let $G_{/F^+}$ be a reductive group which is an outer form for $\GL_3$ which is quasi-split at all finite places of $F^+$ and which splits over $F$.
Suppose that $G_{/F^+}$ is definite, i.e.~that $G(F^+_v) \cong U_3(\R)$ for all $v|\infty$.
Recall from \cite[\S 7.1]{EGH} that $G$ admits a reductive model $\cG$ defined over $\cO_{F^+}[1/N]$, for some $N\in \N$ which is prime to $p$, together with an isomorphism
\begin{equation}
\label{iso integral}
\iota:\,\cG_{/\cO_{F}[1/N]} \stackrel{\iota}{\rightarrow}{\GL_3}_{/\cO_{F}[1/N]}
\end{equation}
which specializes to
$
\iota_w:\,\cG(\cO_{F^+_v})\stackrel{\sim}{\rightarrow}\cG(\cO_{F_w})\stackrel{\iota}{\rightarrow}\GL_3(\cO_{F_w})
$
for all places  $w\in \Sigma_p$ with $w|_{F^+} = v$.

Let $U=U^pU_p\leq G(\bA_{F^+}^{(\infty,p)})\times \cG(\cO_{F^+,p})$ be a compact open subgroup.
If $W$ is a finite $\cO$-module endowed with a continuous action of $U$ we write $S(U,W)$ to denote the space of algebraic automorphic forms with coefficients in $W$:
\[
S(U,W) \defeq \left\{f:\,G(F^{+})\backslash G(\A^{\infty}_{F^{+}})\rightarrow W\,|\, f(gu)=u^{-1}f(g)\,\,\forall\,\,g\in G(\A^{\infty}_{F^{+}}), u\in U\right\}.
\]
We define
\[S(U^p, W) \defeq \varinjlim_{U_{p}} S(U^pU_{p}, W) \textrm{ and } \tld{S}(U^p, W) \defeq \varprojlim_s S(U^p,W/\varpi^s)\]
where in the first limit the subgroups $U_{p}\leq \cG(\cO_{F^+,p})$ run over the compact open neighborhoods of $1\in \cG(\cO_{F^+,p})$.

For $U$ as above, let $\cP_U$ be the set of finite places $w$ in $F$ whose restriction $v \defeq w|_{F^+}$ is a place that splits in $F$ and at which $U$ is unramified.
Let $\cP \subset \cP_U$ be a subset of finite complement.
Then the universal Hecke algebra $\bT_{\cP}=\cO[T^{(i)}_w,\,\,w\in\cP,\,0 \leq i\leq n]$ on $\cP$ acts naturally on $S(U,W)$.
Let $\rbar: G_F \ra \GL_3(\F)$ be a continuous Galois representation.
We let $\fm\subseteq \bT_{\cP}$ be the maximal ideal which is the kernel of the system of Hecke eigenvalues $\overline{\alpha}:\bT_{\cP}\ra \F$ associated to $\overline{r}$, i.e.~ $\overline{\alpha}$ satisfies 
\[
\det\left(1-\overline{r}^{\vee}(\mathrm{Frob}_w)X\right)=\sum_{j=0}^3 (-1)^j(\mathbf{N}_{F/\Q}(w))^{\binom{j}{2}}\overline{\alpha}(T_w^{(j)})X^j
\]
for all $w$ as above.
Then we say that $\rbar$ is automorphic if $S(U,W)_{\fm}$ is nonzero for some $U$ and $W$.

For technical reasons, we choose a place $v_1$ of $F^+$ as in \cite[\S 2.3]{CEGGPS}.
We now fix $U^p\leq G(\bA_{F^+}^{(\infty,p)})$ to be the subgroup with:
\begin{enumerate}
\item $(U^p)_v=\cG(\cO_{F^+_{v}})$ for all finite places $v$ of $F^+$ which split in $F$ and do not belong to $\Sigma_p^+\cup\{v_1\}$;
\item $(U^p)_{v_1}$ is the preimage of the upper triangular matrices under the map
\[
\cG(\cO_{F^+_{v_1}})\ra \cG(k_{v_1})\stackrel{\iota_{\tld{v}_1}}{\ra}\GL_3(k_{\tld{v}_1})
\]
 \item $(U^p)_v$ is hyperspecial maximal compact in $G(F^+_v)$ if $v$ is inert in $F$.
\end{enumerate}

Let $\rbar$ be an automorphic Galois representation.
Let $\Sigma_0^+$ denote the set of finite places of $F^+$ which are the restriction of the finite places of $F$ away from $p$ where $\rbar$ ramifies.
For each $v\in \Sigma_0^+$, we let $\tau_{\tld{v}}$ be the minimally ramified type in the sense of \cite[Definition 2.4.14]{CHT} corresponding to $\rbar|_{G_{F_{\tld{v}}}}$ and  $\sigma(\tau_{\tld{v}})$ be the $\GL_3(\mathcal{O}_{F_{\tld{v}}})$-representation over $E$ associated to it (cf.~ the beginning of \S \ref{subsubsec:AL}).
We write $\sigma(\tau_v)\defeq \sigma(\tau_{\tld{v}})\circ\iota_{\tld{v}}$, which is a $\cG(\cO_{F^+_{v}})$-representation independent of the choice of $\tld{v}|v$.
For each $v\in \Sigma_0^+$ fix a $\cO$-lattice $\sigma(\tau_v)^\circ$ in $\sigma(\tau_v)$ and let $W_{\Sigma_0^+}$ be ${\otimes}_{v\in\Sigma_0^+}\sigma(\tau_v)^\circ$.

We let $\bT^{\univ}$ denote the abstract Hecke algebra over $\cO$ generated by the formal variables $T^{(j)}_{w}$, where $w$ runs over the finite places of $F$ such that $w|_{F+}$ is split in $F$ and $w|_{F^+}\notin \Sigma_0^+\cup\Sigma_p^+\cup\{v_1\}$,
and by $T^{(j)}_{v_1}$ for $j=1,2,3$.
For a $G(\bA_{F^+}^{(\infty,p)})$-module $V$ over $\cO$, $\bT^{\univ}$ acts naturally on $S(U^pU_p,W)$,  $S(U^p,W)$, and $\tld{S}(U^p,W)$ where $W = W_{\Sigma_0^+} \otimes V$ (cf.~ \cite[\S 2.3]{CEGGPS}), and we let $\fm\subseteq \bT^{\univ}$ be the maximal ideal as before.
We will now assume that $S(U,W)_{\fm}$ is nonzero for some choice of $V$ above.
In fact, one can show, by the proof of Proposition \ref{prop:levellower}, that this is a consequence of the hypothesis that $\rbar$ is automorphic. 

In the remainder of this section we let  $\rbar:G_F\ra \GL_3(\F)$ be an automorphic Galois representation that satisfies the Taylor--Wiles conditions in the sense of \cite[Definition 7.3]{LLLM}.
We assume furthermore that 
\begin{enumerate}[(i)]
\item the extension $F/F^+$ is unramified at all finite places;
\item (split ramification) if $\rbar:G_F \rightarrow \GL_3(\F)$ is ramified at a place $w$ of $F$, then $v = w|_{F^+}$ splits as $ww^c$;
\item $p$ is unramified in $F^+$ and all places in $F^+$ above $p$ split in $F$; and
\item $\rbar|_{G_{F_w}}$ is semisimple for all $w\in\Sigma_p$.
\end{enumerate}

For each $v\in \Sigma_p^+$, we choose a place $\tld{v}|v$ of $F$ and let $\cS$ be the set $\{\tld{v}|v\in \Sigma_p^+\}$.
Let $\rhobar_{\tld{v}}$ be $\rbar|_{G_{F_{\tld{v}}}}$ and $\rhobar_{\cS}$ be $(\rhobar_{\tld{v}})_{\tld{v}\in \cS}$.
We set $K\defeq \prod_{\tld{v}\in\cS}\GL_3(\cO_{F_{\tld{v}}})$.
Let $\tld{M}_\infty$ be the weak minimal patching functor for $\rbar$ in the sense of \cite[Definition 7.11]{LLLM} constructed in \cite[Proposition 7.15]{LLLM}.
As in \S \ref{sec:SWC}, we let $M_\infty$ be $\tld{M}_\infty \circ \prod_{\tld{v}\in \cS} \iota_{\tld{v}}$, which is a weak minimal patching functor for $\rhobar_\cS$ by Proposition \ref{prop:patchfunctor}.

Along with the construction of $\tld{M}_\infty$ (cf.~ \cite[\S 2.8]{CEGGPS}, \cite[\S 4.2]{Le}, \cite[\S 7.3]{LLLM}) one has a ring homomorphism $R_\infty=\underset{\tld{v}\in\cS}{\widehat{\otimes}}R_{\tld{v}}^{\Box}[\![x_1,\dots,x_h]\!]\ra 
\bT^{\univ}_{U^p\cG(\cO_{F^+,p})}(W_{\Sigma_0^+})_{\fm}$, where $\bT^{\univ}_{U^p\cG(\cO_{F^+,p})}(W_{\Sigma_0^+})$ is the image of $\bT^{\univ}$ in $\End_{\cO}\big(S(U^p\cG(\cO_{F^+,p}),W_{\Sigma_0^+})\big)$.
We also write $\fm$ for the pullback in $R_\infty$ of the maximal ideal of $\bT^{\univ}_{U^p\cG(\cO_{F^+,p})}(W_{\Sigma_0^+})_{\fm}$.
Then if $\ovl{W}_{\Sigma_p^+}$ is a smooth, finite dimensional $\cG(\cO_{F^+,p})$-representation over $\F$ one has
\begin{equation} \label{eqn:closedfiber}
(\tld{M}_\infty(\ovl{W}_{\Sigma_p^+})/\fm)^\vee=S(U^p\cG(\cO_{F^+,p}),\ovl{W}_{\Sigma_0^+}\otimes\ovl{W}^\vee_{\Sigma_p^+})[\fm].
\end{equation}
where $\cdot^\vee$ denotes Pontrjagin duals (cf.~ \cite[Theorem 4.1.5]{Le}).

\subsubsection{Automorphy lifting}
\label{subsubsec:AL}

Recall from Proposition \ref{prop:basic:Ktypes} that if $\tld{v}\in \cS$ and $\tau_{\tld{v}}$ is a $1$-generic tame inertial type for $I_{F_{\tld{v}}}$, we defined a $\GL_3(\cO_{F_{\tld{v}}})$-representation $\sigma(\tau_{\tld{v}})$ over $E$ corresponding to $\tau_{\tld{v}}$ by results towards inertial local Langlands.
We again let $\sigma(\tau_\cS)$ be $\otimes_{\tld{v}\in \cS} \sigma(\tau_{\tld{v}})$.

\begin{thm}\label{thm:automorphic}
Let $r:G_F \ra \GL_3(E)$ be an absolutely irreducible Galois representation such that 
\begin{enumerate}
	\item for all places $w\in\Sigma_p$, $\rbar|_{I_{F_w}}$ is semisimple and $10$-generic; 
	\item $r$ is unramified almost everywhere and satisfies $r^c\cong r^\vee \epsilon^{-2}$;
	\item for all places $\tld{v}\in \cS$, the representation $r|_{G_{F_{\tld{v}}}}$ is potentially crystalline, with parallel Hodge-Tate weights $(2,1,0)$ and with tame inertial type $\tau_{\tld{v}}$ $($see \cite[Definition 2.1]{LLLM}$)$;
	\item $\rbar$ satisfies the Taylor--Wiles conditions as above and $\rbar$ has split ramification; and
	\item $\rbar\cong \rbar_{\imath}(\pi)$ for a RACSDC representation $\pi$ of $\GL_3(\bA_{F})$ with trivial infinitesimal character such that $\otimes_{\tld{v}\in \cS}\sigma(\tau_{\tld{v}})$ is a $K$-type for $\otimes_{\tld{v}\in \cS}\pi_{\tld{v}}$.
\end{enumerate}
Then $r$ is automorphic in the sense of \cite[\S 7.2]{LLLM}.
\end{thm}
\begin{proof}
Given Theorems \ref{SWC2} and \ref{thm:BM}, the proof of \cite[Theorem 7.4]{LLLM} goes through unchanged.  %
\end{proof}
\begin{rmk} Compared to \cite[Theorem 7.4]{LLLM}, we relaxed the hypothesis that $p$ splits completely to $p$ being unramified. However, we also assumed that $\rbar$ is semisimple at all places above $p$. The reason is that we only established the connectedness of the generic fiber of $R_{\rhobar}^{\tau}$ when $\rhobar$ is semisimple (though we do know it for non-semisimple $\rhobar$ in the case that all shapes have length $\geq 2$). In work in progress, we will establish a counterpart of  \cite[Theorem 7.4]{LLLM} for non-semisimple representations. This will allow us to remove the semisimplicity hypothesis in a manner similar to  \cite[Theorem 7.4]{LLLM}.
\end{rmk}

\subsubsection{The Serre weight conjecture}

Let $\rbar:G_F\ra \GL_3(\F)$ be as in the beginning of \S \ref{subsec:global}.
Recall that for each $v\in \Sigma_p^+$, we chose a place $\tld{v}|v$ of $F$ and set $\cS$ to be the set $\{\tld{v}|v\in \Sigma_p^+\}$.
Furthermore, we let $\rhobar_{\tld{v}}$ be $\rbar|_{G_{F_{\tld{v}}}}$ and $\rhobar_{\cS}$ be $(\rhobar_{\tld{v}})_{\tld{v}\in \cS}$.
Recall that $K$ is the product $\prod_{\tld{v}\in \cS} \GL_3(\cO_{F_{\tld{v}}})$.
Let $W(\rbar)$ be the set of irreducible $K$-representations $\sigma$ over $\F$ such that 
\[
S(U^p\cG(\cO_{F^+,p}),\ovl{W}_{\Sigma_0^+}\otimes(\sigma^\vee \circ \prod_{\tld{v}\in \cS} \iota_{\tld{v}}))_\fm \neq 0.
\]
We have the following version of the weight part of Serre's conjecture.

\begin{thm} \label{SWC2}
Let $\rbar:G_F\rightarrow \GL_3(\F)$ be a continuous Galois representation, satisfying the Taylor-Wiles conditions.
Assume that $\rhobar_{\tld{v}} \defeq \rbar|_{G_{F_{\tld{v}}}}$ is semisimple and $10$-generic for all $\tld{v}\in \cS$, that $\rbar$ is automorphic, and that $\rbar$ has split ramification.
Then $W(\rbar) \circ \prod_{\tld{v}\in \cS}\iota_{\tld{v}}= W^?(\rhobar_\cS)$.
\end{thm}
\begin{proof}
We have that $W^\mathrm{BM}(\rhobar_\cS) = W(\rbar)\circ \prod_{\tld{v}\in \cS}\iota_{\tld{v}}$ by (\ref{eqn:closedfiber}). %
The result now follows from Theorem \ref{SWC}.
\end{proof}

\subsubsection{Mod $p$ multiplicity one}\label{subsec:multone}

We continue using the setup from the beginning of \S \ref{subsec:global}.
We have the following mod $p$ multiplicity one result.
\begin{thm}\label{thm:modpmultone}
Let $\tau_\cS$ and $\sigma\in W^?(\rhobar_\cS,\tau_\cS)$ be as in the statement of Theorem \ref{thm:cyclic}.
Then 
\[
S\Big(U^p\cG(\cO_{F^+,p}),\big(\ovl{\sigma}(\tau_\cS)^{\sigma} \circ \prod_{v\in \Sigma_p^+} \iota_{\tld{v}}\big)^\vee \otimes_{\cO_E}W_{\Sigma_0^+}\Big)[\mathfrak{m}]
\]
is one-dimensional over $\F$.
\end{thm}
\begin{proof}
By (\ref{eqn:closedfiber}), 
\[M_\infty(\ovl{\sigma}(\tau_\cS)^\sigma)/\mathfrak{m} \cong \tld{M}_\infty(\ovl{\sigma}(\tau_\cS)^\sigma \circ \prod_{v\in \Sigma_p^+} \iota_{\tld{v}})/\fm \cong \Big(S\Big(U^p\cG(\cO_{F^+,p}),\big(\ovl{\sigma}(\tau_\cS)^\sigma \circ \prod_{v\in \Sigma_p^+} \iota_{\tld{v}}\big)^\vee \otimes_{\cO} W_{\Sigma_0^+}\Big)[\mathfrak{m}]\Big)^\vee.\]
By Theorem \ref{thm:cyclic}, the dimension of $M_\infty(\ovl{\sigma}(\tau_\cS)^\sigma)/\mathfrak{m}$ is one.
\end{proof}

\subsubsection{Lattices in cohomology}

Let $r:G_F \ra \GL_3(E)$ be an automorphic Galois representation as in Theorem \ref{thm:automorphic}.
We say that $r$ is \emph{minimally ramified} if $r|_{G_{F_{\tld{v}}}}$ is minimally ramified in the sense of \cite[Definition 2.4.14]{CHT} for all $v\in \Sigma_0^+$.
Following the notation of \cite[\S 7.1]{LLLM}, let $\lambda$ be the kernel of the system of Hecke eigenvalues $\alpha:\bT^{\univ}\rightarrow \cO$ associated to $r$, i.e.~ $\alpha$ satisfies 
\[
\det\left(1-r^{\vee}(\mathrm{Frob}_w)X\right)=\sum_{j=0}^3 (-1)^j(\mathbf{N}_{F/\Q}(w))^{\binom{j}{2}}\alpha(T_w^{(j)})X^j
\]
for all $w$ as above. %
We now set $W \defeq W_{\Sigma_0^+}$ as in \S \ref{subsec:multone}.
By Theorem \ref{thm:automorphic}, $\tld{S}(U^p, W)[\lambda]$ is nonzero.
Since $r$ is minimally ramified, $r$ corresponds to a prime ideal of $R_\infty$ as in \cite[Theorem 5.2.1]{HLM}.
By an abuse of notation, we call this ideal $\lambda$.
Note that we have that $M_\infty/\lambda \isom \tld{S}(U^p, W)^{\mathrm{d}}/\lambda$ by (the proof of) \cite[Corollary 2.11]{CEGGPS}.

\begin{thm}\label{thm:lattice}
Let $r:G_F \ra \GL_3(E)$ be as in Theorem \ref{thm:automorphic}. Assume furthermore that $r$ is minimally ramified.
Let $\{\tau_{\tld{v}}\}_{\tld{v}\in\cS} $ be a $13$-generic tame  type.
The lattice $\sigma(\tau)^0 \defeq \sigma(\tau) \cap \tld{S}(U^p, W)[\lambda] \subset \sigma(\tau) \cap \tld{S}(U^p,W)[\lambda]\otimes_{\cO} E \cong \sigma(\tau)$ depends only on $\{r|_{G_{F_{\tld{v}}}}\}_{\tld{v}\in \cS}$.
\end{thm}
\begin{proof}
Let $\rhobar_\cS$ be $(\rbar|_{G_{F_{\tld{v}}}})_{\tld{v}\in \cS}$.
Fix $\sigma\in\JH(\ovl{\sigma(\tau_\cS)})$ as in Theorem \ref{thm:cyclic}, a saturated inclusion $\sigma(\tau_\cS)^\sigma \subset \sigma(\tau_\cS)^0$, and saturated inclusions $\sigma(\tau_\cS)^\kappa \subset \sigma(\tau_\cS)^\sigma$ for all $\kappa\in \JH(\ovl{\sigma(\tau_\cS)})$.
Let $\gamma(\kappa) \in \cO$ so that $\gamma(\kappa)^{-1} \sigma(\tau_\cS)^\kappa\subset \sigma(\tau_\cS)^0$ is saturated.
Then $\sigma(\tau_\cS)^0 = \sum_{\kappa\in \JH(\ovl{\sigma(\tau_\cS)})} \gamma(\kappa)^{-1} \sigma(\tau_\cS)^\kappa$ for some $\gamma(\kappa) \in \cO$ by \cite[Lemma 4.1.2]{EGS}.
It suffices to show that for each $\kappa \in \JH(\ovl{\sigma(\tau_\cS)})$, the ideal $(\gamma(\kappa)) \subset \cO$ depends only on $\rhobar_\cS$.

Observe that the image of the inclusion 
\[\cO \cong \Hom_{U_p}(\sigma(\tau_\cS)^\sigma, \tld{S}(U^p, W)[\lambda]) \ra \Hom_{U_p}(\sigma(\tau_\cS)^\kappa, \tld{S}(U^p, W)[\lambda]) \cong \cO,\] 
is given by the ideal $(\gamma(\kappa))$.
By Schikhov duality, the natural inclusion
\[\cO \cong \Hom_{U_p}((\tld{S}(U^p, W)[\lambda])^{\mathrm{d}}, (\sigma(\tau_\cS)^\sigma)^{\mathrm{d}}) \ra \Hom_{U_p}((\tld{S}(U^p, W)[\lambda])^{\mathrm{d}},(\sigma(\tau_\cS)^\kappa)^{\mathrm{d}}) \cong \cO,\]
is also given by the ideal $(\gamma(\kappa))$.
By another application of Schikhov duality, the natural inclusion
\[\cO \cong \Hom_{U_p}((\tld{S}(U^p, W)[\lambda])^{\mathrm{d}}, (\sigma(\tau_\cS)^\kappa)^{\mathrm{d}})^{\mathrm{d}} \ra \Hom_{U_p}((\tld{S}(U^p, W)[\lambda])^{\mathrm{d}},(\sigma(\tau_\cS)^\sigma)^{\mathrm{d}})^{\mathrm{d}} \cong \cO,\]
is also given by the ideal $(\gamma(\kappa))$.
Since $\tld{S}(U^p, W)[\lambda]^{\mathrm{d}}$ is canonically isomorphic to the torsion-free part of $\tld{S}(U^p, W)^{\mathrm{d}}/\lambda$,
 the ideal $(\gamma(\kappa))$ gives the inclusion
\[\Hom_{U_p}(M_\infty/\lambda,(\sigma(\tau_\cS)^\kappa)^{\mathrm{d}})^{\mathrm{d}} \ra \Hom_{U_p}(M_\infty/\lambda,(\sigma(\tau_\cS)^\sigma)^{\mathrm{d}})^{\mathrm{d}}.\]
Again, note that $\Hom_{U_p}(M_\infty/\lambda,(\sigma(\tau_\cS)^\sigma)^{\mathrm{d}})^{\mathrm{d}}$ is isomorphic to (the $p$-torsion-free part of) 
\[
\Hom_{U_p}(M_\infty,(\sigma(\tau_\cS)^\sigma)^{\mathrm{d}})^{\mathrm{d}}/\lambda = M_\infty(\sigma(\tau_\cS)^\sigma)/\lambda,\]
and similarly for $\kappa$.
So the image of
\begin{equation} \label{eqn:gauge}
\cO \cong M_\infty(\sigma(\tau_\cS)^\kappa)/\lambda \ra M_\infty(\sigma(\tau_\cS)^\sigma)/\lambda \cong \cO
\end{equation}
is given by the ideal $(\gamma(\kappa))$.

On the other hand, the image of (\ref{eqn:gauge}) is given by
\[
\Big(j_*j^* \Big(\prod_{\theta\in W^?(\rhobar_{\cS})} \mathfrak{p}(\theta)^{m(\theta)}R_\infty(\tau_\cS)\Big) \Big)
\big(M_\infty(\sigma(\tau_\cS)^\sigma)/\lambda\big)
\]
by Theorem \ref{thm:gauge}.
Then
 $(\gamma(\kappa))$ is generated by elements in the ideal
\[j_*j^* \Big(\prod_{\theta\in W^?(\rhobar_{\cS})} \mathfrak{p}(\theta)^{m(\theta)}R_\infty(\tau_\cS)\Big)\]
modulo the ideal $\lambda$. By Remark \ref{rmk:descend}, the above ideal comes from $\widehat{\otimes}_{\tld{v}\in \cS} R_{\tld{v}}^\square$, hence the ideal generated by its generators modulo $\lambda$ depends only on $\{r|_{G_{F_{\tld{v}}}}\}_{\tld{v}\in \cS}$.
\end{proof}
\begin{rmk}
In the hypotheses of Theorem \ref{thm:lattice}, assume further that $\tld{w}(\rhobar_{\tld{v}},\tau_{\tld{v}})_i$ has length at least $2$ for all $\tld{v}\in\cS$ and $i\in\Z/f_{\tld{v}}$.
Then the lattice $\sigma(\tau)^0$ can be described explicitly as in \cite[Theorem 14]{Le}.
\end{rmk}

\newpage

In the following figures we give a pictorial realization of the gluing procedure appearing in the proofs of Lemmas \ref{lemma:broom}--\ref{lemma:9weight}.
Recall that $\tld{w}^*_{i'}=\tld{w}(\rhobar_{\cS},\tau_\cS)^*_{i'}$.
\captionsetup[table]{name=Figure}
\begin{table}[h]
\caption{\textbf{Comparison of $\Sigma_{V^j,i'}$ in Lemma \ref{lemma:broom}}}\label{TableBroom}
\centering
\adjustbox{max width=\textwidth}{
\begin{tabular}{ c }
$
\xymatrix{
&
\sigma_{(0,1)} \ar@{-}[dl]\ar@{-}[d]&
&\sigma_{(0,1)}\ar@{-}[d]&&
\sigma_{(0,1)}\ar@{-}[dr]\ar@{-}[d]&
\\
\sigma_{(\eps_1, 0)} &\sigma_{(0,0)}&&
\sigma_{(0,0)}&&\sigma_{(0,0)}&\sigma_{(\eps_2,0)}
}
$\\
\end{tabular}}
\caption*{
From left to right with arrows pointing down, we have $\Sigma_{V^1,i'}$, $\Sigma_{V^3,i'}$ and $\Sigma_{V^2,i'}$.
The edges correspond to adjacent pairs.}%
\end{table}

\begin{table}[h]
\caption{\textbf{Comparison of $\Sigma_{V,i'}$ and $\Sigma_{V^j,i'}$ in Lemma \ref{lemma:6weight} when $\tld{w}_{i'}^*t_{-\un{1}} =\id$}}\label{TableL6W}
\centering
\adjustbox{max width=\textwidth}{
\begin{tabular}{ c }
$
\xymatrix{
&
\color{red}\sigma_{(0,1)} \ar@{-}@[red][dr]\ar@{-}@[red][d]\ar@{-}@[red][dl]\ar@{-}[drr]&&
&
&
\color{red}\sigma_{(0,1)} \ar@{-}@[red][dr]\ar@{-}@[red][d]\ar@{-}@[red][dl]\ar@{-}[drr]&&
&
\\
\color{red}\sigma_{(\eps_1,0)}\ar@{-}@[red][drr]&\color{red}\sigma_{(0,0)}\ar@{-}@[red][dr]&
\color{red} \sigma_{(\eps_2,0)}\ar@{-}@[red][d]&\sigma_{(\eps_2-\eps_1,0)}\ar@{-}[dl]
&
\color{red}\sigma_{(\eps_1,0)}&\color{red}\sigma_{(0,0)}&
\color{red} \sigma_{(\eps_2,0)}&\sigma_{(\eps_2-\eps_1,0)}
&
\color{red}\sigma_{(\eps_2,0)}\ar@{-}@[red][d]
\\
&&\color{red}\sigma_{(\eps_2,1)}&&&&&&\color{red}\sigma_{(\eps_2,1)}
}
$\\
\end{tabular}}
\caption*{
Assume $(\nu_1,\nu_2,\nu_3)=(\eps_1-\eps_2,\eps_1,\eps_1+\eps_2)$ and $\tld{w}_{i'}^*t_{-\un{1}}=\id$. Then %
$\nu_4=\eps_2$.  From left to right with arrows pointing down, we have $\Sigma_{V,i'}$, $\Sigma_{V^1,i'}$, and $\Sigma_{V^2,i'}$ (a Weyl segment). In red are the elements in $r(\Sigma_0)$.}
\end{table}

\begin{table}[h]
\caption{\textbf{Comparison of $\Sigma_{V^j,i'}$ in Lemma \ref{lemma:8weight} when  $\tld{w}_{i'}^*t_{-\un{1}}=\alpha$}}\label{TableL8W}
\centering
\adjustbox{max width=\textwidth}{
\begin{tabular}{ c }
$
\xymatrix{
&&
\color{red} \sigma_{(\alpha(0),1)} \ar@{-}@[red][dr]\ar@{-}@[red][d]\ar@{-}[dl]\ar@{-}[dll]&
&&
\color{red} \sigma_{(\alpha(0),1)} \ar@{-}@[red][dr]\ar@{-}@[red][d]\ar@{-}[dl]\ar@{-}@[red][drr]&&
\\
\sigma_{(\alpha(\eps_1-\eps_2), 0)} \ar@{-}[dr]&\sigma_{(\alpha(\eps_1),0)}\ar@{-}[d]&
\color{red}\sigma_{(\alpha(0),0)}\ar@{-}@[red][dl]&\color{red}\sigma_{(\alpha(\eps_2),0)}\ar@{-}@[red][dll]
&\sigma_{(\alpha(\eps_1),0)}\ar@{-}[drr]&
\color{red}\sigma_{(\alpha(0),0)}\ar@{-}@[red][dr]&\color{red}\sigma_{(\alpha(\eps_2),0)}\ar@{-}@[red][d]&
\color{red} \sigma_{(\alpha(\eps_2-\eps_1),0)}\ar@{-}@[red][dl]
\\
&\color{red} \sigma_{(\alpha(\eps_1),1)}&&
&&&\color{red} \sigma_{(\alpha(\eps_2),1)}&
}
$\\
\end{tabular}}
\caption*{In the notation of Lemma \ref{lemma:8weight} consider the case where $\nu=\eps_1+\eps_2$ and $\tld{w}_{i'}^*t_{-\un{1}}=\alpha$. On the left we have $\Sigma_{V^2,i'}$, where we write the elements in $r(\Sigma_0)$ in red.  %
Similarly on the right for $\Sigma_{V^1,i'}$.}

\end{table}
\begin{table}[h]
\caption{\textbf{Comparison of $\Sigma_{V^j,i'}$ in Lemma \ref{lemma:8weight} when  $\tld{w}_{i'}^*t_{-\un{1}}=\gamma^+$}}\label{TableL8W1}
\centering
\adjustbox{max width=\textwidth}{
\begin{tabular}{ c }
$
\xymatrix{
&&
\color{red} \sigma_{(\gamma^+(0),1)} \ar@{-}@[red][dr]\ar@{-}[d]\ar@{-}@[red][dl]\ar@{-}[dll]&
&&
\color{red} \sigma_{(\gamma^+(0),1)} \ar@{-}@[red][dr]\ar@{-}[d]\ar@{-}@[red][dl]\ar@{-}[drr]&&
\\
\sigma_{(\gamma^+(\eps_1-\eps_2), 0)} \ar@{-}[dr]&\color{red}\sigma_{(\gamma^+(\eps_1),0)}\ar@{-}@[red][d]&
\sigma_{(\gamma^+(0),0)}\ar@{-}[dl]&\color{red}\sigma_{(\gamma^+(\eps_2),0)}\ar@{-}@[red][dll]
&\color{red}\sigma_{(\gamma^+(\eps_1),0)}\ar@{-}@[red][drr]&
\sigma_{(\gamma^+(0),0)}\ar@{-}[dr]&\color{red}\sigma_{(\gamma^+(\eps_2),0)}\ar@{-}@[red][d]&
 \sigma_{(\gamma^+(\eps_2-\eps_1),0)}\ar@{-}[dl]
\\
&\color{red} \sigma_{(\gamma^+(\eps_1),1)}&&
&&&\color{red} \sigma_{(\gamma^+(\eps_2),1)}&
}
$\\
\end{tabular}}
\caption*{In the notation of Lemma \ref{lemma:8weight} consider the case where $\nu=\eps_1+\eps_2$ and $\tld{w}_{i'}^*t_{-\un{1}}=\gamma^+$. On the left we have $\Sigma_{V^2,i'}$, where we write the elements in $r(\Sigma_0)$ in red.  %
Similarly on the right for $\Sigma_{V^1,i'}$.}
\end{table}

\begin{table}[h]
\caption{\textbf{Comparison of $\Sigma_{V^j,i'}$ in Lemma \ref{lemma:9weight} when  $\tld{w}_{i'}^*t_{-\un{1}}=\gamma^+$}}\label{TableL9W}
\centering
\adjustbox{max width=\textwidth}{
\begin{tabular}{ c }
$
\xymatrix{
&&
\color{red} \sigma_{(\gamma^+(0),1)} \ar@{-}@[red][dr]\ar@{-}[d]\ar@{-}@[red][dl]\ar@{-}[dll]\ar@{-}[drr]&
&&&&&&
\\
\sigma_{(\gamma^+(\eps_1-\eps_2), 0)} \ar@{-}[d]&\color{red}\sigma_{(\gamma^+(\eps_1),0)}\ar@{-}@[red][dl]\ar@{-}@[red][drrr]&
\sigma_{(\gamma^+(0),0)}\ar@{-}[dll]\ar@{-}[drr]&\color{red} \sigma_{(\gamma^+(\eps_2),0)}\ar@{-}@[red][dlll]\ar@{-}@[red][dr]
&\sigma_{(\gamma^+(\eps_2-\eps_1),0)}\ar@{-}[d]&&&&
\\
\color{red} \sigma_{(\gamma^+(\eps_1),1)}\ar@{..}@[red][drr]&&
&&\color{red} \sigma_{(\gamma^+(\eps_2),1)}\ar@{..}@[red][dll]&
&\color{red} \sigma_{(\gamma^+(\eps_1),1)}\ar@{-}@[red][drr]&&&
\\
&&\color{red} \sigma_{(\gamma^+(\eps_1+\eps_2),0)}&&
&&&&\color{red} \sigma_{(\gamma^+(\eps_1+\eps_2),0)}
}
$\\
\end{tabular}}
\caption*{In the notation of Lemma \ref{lemma:9weight} consider the case where $\nu=\eps_1+\eps_2$ and $\tld{w}_{i'}^*t_{-\un{1}}=\gamma^+$. On the left, we have $\Sigma_{V,i'}$ (resp. $\Sigma_{V^1,i'}$ removing the dotted lower part), where we write the elements in $r(\Sigma_0)$ in red. On the right we have one of the possible choices for $\Sigma_{V^2,i'}$.}
\end{table}

\section{Addendum to \cite{LLLM}}
\label{sec:add}

\begin{enumerate}

\item In Theorem 1.1, the statement that ``its special fiber is as predicted by the geometric Breuil--M\'ezard conjecture'' means the following: under the assumptions of Theorem 1.1, the special fiber of $R^{\text{\tiny{$(2,1,0)$}},\tau}_{\rhobar}$ is reduced and the number of irreducible components is less or equal to $\#\big(W^?(\rhobar^{\mathrm{ss}})\cap \JH(\ovl{\sigma}(\tau))\big)$, with equality if $\rhobar$ is semisimple.
This can be checked directly using Theorem 6.14, Table 3, Table 7, Table 8, Propositions 8.5, 8.6, 8.12 and the results of section \S 8.2.2.

\item In Proposition 3.4 the codomain of $T_{\mathrm{tan}}$ should be replaced by
\[
\left\{
(\rho,\gamma_0)\ \mid\ \rho\in\mathrm{Rep}_{\F'[\varepsilon]/(\varepsilon^2)}(G_{K_\infty}),\ \gamma_0:\rho\mod\varepsilon\stackrel{\sim}{\ra} T_{dd}^*(\ovl{\fM})\right\}
\]

\item After equation (3.7), we remark that if $(\fM_A,\rho_A,\delta_A)\in D^{\tau,\Box}_{\ovl{\fM},\rhobar}$, then we have a canonical isomorphism $\fM_A\otimes_A\F\stackrel{\sim}{\ra}{\ovl{\fM}}$.%

\item After Definition 4.15, the symbol $Y^{[0,2],\tau}_{\ovl{\fM}}(R)$ denotes the category of pairs $(\fM_R,\jmath_R)$ where $\fM_R\in Y^{[0,h],\tau}(R)$ and $\jmath_R:\fM_R\otimes_R\F\stackrel{\sim}{\ra}\ovl{\fM}$ is an isomorphism in $Y^{[0,h],\tau}(\F)$. 
A similar comment applies to $Y^{\mu,\tau}_{\ovl{\fM}}(R)$.

\item Proof of Theorem 4.17: the ring $R$ is $p$-flat and reduced by \cite[Lemma 2.6]{calegariCrelle}.

\item The formula in Lemma 5.2 still converges in $\frac{1}{\lambda}\Mat(R[1/p][\![u]\!])$ for $R$ a complete local Noetherian flat $\cO$-algebra R. %
While it is possible to show it lies in $\cO^{\rig}_R$, for the computations and the arguments in this paper,
we only need that its formation is compatible with base change.

\item In Corollary 5.13, $T_1,\dots,T_8$ should be replaced by $T_1,\dots,T_9$. 
Similar comment applies to the displayed equation before Theorem 6.14.

\item In \S 5.3.2, ``$c_{11}\equiv 0$ modulo $\varpi$'' should read ``$c_{11}$ and $c_{13}\equiv 0$ modulo the maximal ideal''. 

\item In \S 5.3.3, line $-6$ and $-4$, the $c_{13}^*$ in the displayed equations must be replaced by $c_{13}$ .

\item Definition 7.1 should also define automorphic of weight $V$, level $U$, and coefficients $W$ as follows.
\begin{defn}
Let $\overline{r}:G_F\rightarrow \GL_3(\F)$ be a continuous Galois representation.
Let $V$ be a Serre weight for $\cG$, $U$ be a compact open subgroup of $G(\bA^{\infty,p}_F)\times \cG(\cO_{F^+,p})$ which is unramified at places $v|p$, and $W$ be an $\cO$-module with a $U$-action for which the factor $\cG(\cO_{F^+,p})$ acts trivially. 
We say that $\overline{r}$ is \emph{automorphic of weight $V$, level $U$, and coefficients $W$} if there exists a cofinite subset $\cP \subset \cP_U$ such that
$$
S(U,V\otimes W)_{\overline{\mathfrak{m}}}\neq 0
$$
where $\overline{\mathfrak{m}}$ is the kernel of the system of Hecke eigenvalues $\overline{\alpha}:\bT_{\cP}\rightarrow \F$ associated to $\overline{r}$, and $\overline{\alpha}$ satisfies the equality
$$
\det\left(1-\overline{r}^{\vee}(\mathrm{Frob}_w)X\right)=\sum_{j=0}^3 (-1)^j(\mathbf{N}_{F/\Q}(w))^{\binom{j}{2}}\overline{\alpha}(T_w^{(j)})X^j
$$
for all $w\in \cP$.
We say that $\overline{r}$ is \emph{automorphic of weight $V$} (or that $V$ is a Serre weight of $\overline{r}$) if $\overline{r}$ is \emph{automorphic of weight $V$, level $U$, and coefficients $W$} for some subgroup $U$ and coefficients $W$ as above.
We write $W(\rbar)$ for the set of all Serre weights of $\rbar$.
We say that $\rbar$ is \emph{automorphic} if $W(\rbar)\neq \emptyset$.
\end{defn}

\item
\label{item:def:mod} 
In Definition 7.11(2), ``automorphic of weight $V$" should be replaced with ``automorphic of weight $V$, level $U$, and coefficients $W$" where $U$ is a fixed compact open subgroup of $G(\bA^{\infty,p}_F)\times \cG(\cO_{F^+,p})$ which is unramified at places $v|p$ and $W$ is a fixed $\cO$-module with a $U$-action on which the factor $\cG(\cO_{F^+,p})$ acts trivially. This definition of patching functor depends on the implicit choices of $U$ and $W$.

\item The definition of $e(M)$, given before Proposition 7.14, is incorrect.
The correct definition of $e(M)$ is the following: given a finitely generated $R_\infty$-module $M$ with scheme-theoretic support $\Spec A$ of dimension at most $d$, define $e(M)$ to be $d!$ times the coefficient of the degree $d$-term of the Hilbert polynomial of $M$ (considered as an $A$-module).

\item In the paragraph following the proof of Proposition 7.14, the definition of $\Sigma_0$ should exclude primes dividing $p$.

\item \label{item:UW} In the proof of Proposition 7.15, ``automorphic of weight $V$" should be replaced with ``automorphic of weight $V$, level $U$, and coefficients $W$" where $U = \prod_{v\nmid \infty} U_v$ and $U_v$ is
\begin{itemize}
\item $\mathcal{G}(\cO_v)$ for $v$ which split in $F$ except for $v_1$,
\item the preimage of the upper triangular matrices under the map
\[
\mathcal{G}(\cO_v) \ra \mathcal{G}(k_v) \underset{\iota_{\tld{v}}}{\risom} \GL_3(k_{\tld{v}})
\]
if $v=v_1$, or
\item a maximal hyperspecial maximal compact open subgroup of $G(F_v)$ if $v$ is inert in $F$,
\end{itemize}
and $W$ is an $\cO$-lattice in $\otimes_{v\in \Sigma_0^+} \sigma(\tau_{\tld{v}}) \circ \iota_{\tld{v}}$.

\item
In the proof of Proposition 7.16 the tame inertial type $\tau'$ should read: $\tau'\defeq \omega_2^{-\big((b+1)+p(b+1)\big)}\oplus \omega_2^{-\big((c-1)+pa\big)}\oplus \omega_2^{-\big(a+p(c-1)\big)}$.

\item The proof of Theorem 7.8 also requires the following proposition, which is a level-lowering result based on techniques in \cite{Taylor}.
The proof of Theorem 7.4, which was omitted, uses the same techniques.
\begin{prop}\label{prop:levellower}
Let $\rbar:G_F\rightarrow \GL_3(\F)$ be a continuous Galois representation with split ramification outside $p$, which is automorphic and satisfies the Taylor-Wiles conditions.
If $\rbar$ is automorphic of a reachable weight, then it is automorphic of a reachable weight and level $U$ and coefficients $W$ where $U = \prod_{v\nmid \infty} U_v$ and $U_v$ is
\begin{itemize}
\item $\mathcal{G}(\cO_v)$ for $v$ which split in $F$ except for $v_1$,
\item the preimage of the upper triangular matrices under the map
\[
\mathcal{G}(\cO_v) \ra \mathcal{G}(k_v) \underset{\iota_{\tld{v}}}{\risom} \GL_3(k_{\tld{v}})
\]
if $v=v_1$, or
\item a maximal hyperspecial maximal compact open subgroup of $G(F_v)$ if $v$ is inert in $F$,
\end{itemize}
and $W$ is an $\cO$-lattice in $\otimes_{v\in \Sigma_0^+} \sigma(\tau_{\tld{v}}) \circ \iota_{\tld{v}}$.
\end{prop}
\begin{proof}
If $\rbar$ is automorphic of a reachable weight $V$, then $S(U,V\otimes W)_\mathfrak{m}$ is nonzero for some level $U$ and coefficients $W$.
Let $V$ be $\otimes_v V_v$ where $V_v = V_{\tld{v}} \circ \iota_{\tld{v}}$.
Choose tame types $\tau_v$ such that 
\begin{itemize}
\item $\rhobar_v$ is admissible with respect to $\tau_v$;
\item $\ell(\mathbf{w}(\rhobar_v,\tau_v)) \geq 3$; and
\item $V_{\tld{v}} \in \JH(\overline{\sigma}(\tau_v))$.
\end{itemize}
Letting $\sigma$ be an $\cO$-lattice in $\otimes_v \sigma(\tau_v) \circ \iota_{\tld{v}}$, we have that $S(U,\sigma \otimes W)_\mathfrak{m}$ is nonzero.
By \cite[Lemma 7.1.6 and Theorem 7.2.1]{EGH}, $\rbar \otimes_{\F} \ovl{\F}_p$ is isomorphic to the reduction of $r_\pi: G_F \ra \GL_n(\overline{\bQ}_p)$ for some $\pi$ as in \cite[Theorem 7.2.1]{EGH}.
Let $\pi_F$ be the RACSDC automorphic representation of $\GL_n(\bA_F)$ obtained from $\pi$ through base change.
Choose a totally real extension $L^+$ of $F^+$ such that
\begin{itemize}
\item $4|[L^+:\Q]$;
\item $L^+/F^+$ is Galois and solvable;
\item $L \defeq L^+ F$ is linearly disjoint from $\overline{F}^{\ker \rbar}(\zeta_p)$ over $F$;
\item $L/L^+$ is everywhere unramified;
\item $p$ is unramified in $L^+$;
\item $v_1$ splits completely in $L$;
\item if $\pi_L$ is the base change of $\pi_F$ to $L$ and $w$ is a place of $L$ lying above $\tld{v}$ for $v\in \Sigma_0^+$, then $\pi_{L,w}$ has Iwahori fixed vectors and $\tau_v|_{I_{L_w}}$ is trivial.
\end{itemize}

One can define analogues of $R_\tau^{univ}$ and $R_{\mathcal{S}}^{univ}$ from \cite[\S 5]{gee-annalen} as follows.
Let $S$ be the union of $\{v_1\}$ and $\Sigma^+$.
Let $\tld{S}$ be the union of $\{\tld{v}_1\}$ and $\{\tld{v}:v\in \Sigma^+\}$.
Let $S_L$ be the set of places in $L^+$ lying above places in $S$, and let $\tld{S}_L$ be the set of places in $L$ lying above places in $\tld{S}$.
Let $\tld{\Sigma}_{L,p}$ be the set of places in $L$ lying over a place in $\tld{\Sigma}_p^+$, $\tld{\Sigma}_{L,0}$ the set of places in $L$ lying over a place in $\tld{\Sigma}_0^+$, and $\tld{\Sigma}_{L,1}$ be the set of places in $L$ lying over $\tld{v}_1$.
Let $\tld{\Sigma}_L$ be the union of $\tld{\Sigma}_{L,p}$ and $\tld{\Sigma}_{L,0}$.
If $\tld{v} \in \tld{\Sigma}_{L,p}$, let $\tau_{\tld{v}}$ be $\tau_{\tld{v}|_{F^+}}|_{I_L}$.
If $\tld{v} \in \tld{\Sigma}_{L,0}$, let $R^{\square,\tau_{\tld{v}}}$ be the lifting ring for $\rbar|_{G_{L_{\tld{v}}}}$ parametrizing lifts whose characteristic polynomial is $(X-1)^3$.
We let $R_\tau^{univ}$ be the universal deformation ring corresponding to the deformation problem
\[
\left( F/F^+,S,\widetilde{S},\cO_E,\rbar, \epsilon^{-2} \delta_{F/F^+}, \{ R^\square_{\widetilde{v}_1} \} \cup \{ R^{\square,\tau_{\tld{v}}}_{\widetilde{v}} \}_{v \in \Sigma^+}\right),
\]
and we let $R_{\mathcal{S}}^{univ}$ be the universal deformation ring corresponding to the deformation problem
\[
\left(L/L^+,S_L,\widetilde{S}_L,\cO_E,\rbar|_{G_L}, \epsilon^{-2} \delta_{L/L^+}, \{ R^\square_{\widetilde{u}_1} \}_{\tld{u}_1\in \tld{S}_{L,1}} \cup \{ R^{\square,\tau_{\tld{v}}}_{\widetilde{v}} \}_{\tld{v} \in \tld{\Sigma}_L}\right),
\]
The proof of \cite[Theorem 5.1.4]{gee-annalen} shows that $R_\tau^{univ}$ is finite over $\cO_E$.
Indeed, $R_\tau^{univ}$ is finite over $R_{\mathcal{S}}^{univ}$ and $R_{\mathcal{S}}^{univ}$ is finite over $\cO_E$ since $(R_{\mathcal{S}}^{univ})^{\mathrm{red}}$ is isomorphic to an appropriate Hecke algebra by the proof of \cite[Theorem 3.4]{Guerberoff}.
One replaces Fontaine--Laffaille deformation rings with $R^{\square,\tau_{\tld{v}}}$ for $\tld{v} \in \tld{\Sigma}_{L,p}$, which is geometrically integral by \cite[\S 5.3]{LLLM}.

That $R_\tau^{univ}$ is finite over $\cO_E$ implies that there is a conjugate self-dual lift $r: G_F \ra \GL_3(\overline{\bQ}_p)$ of $\rbar$ which is minimally ramified outside $p$ and potentially crystalline of type $((0,1,2),\tau_v)$ at $\tld{v}$ for each $v\in \Sigma_0^+$.
Moreover, we have that the restriction $r|_{G_L}$, which corresponds to a point of $\Spec R_{\mathcal{S}}^{univ}$, is automorphic.
Solvable base change then implies that $r$ is automorphic.
Local-global compatibility implies that $S(U,\sigma \otimes W)_\mathfrak{m}$ is nonzero for $U$ and $W$ as in the statement of the proposition.
Then $S(U,V' \otimes W)_\mathfrak{m}$ is nonzero for some reachable $V' \in \JH(\overline{\sigma})$.
\end{proof}
This would then show that $M_{\infty}(\otimes_{v \in \Sigma^+_p} \overline{\sigma}(\tau_v)^\circ)$ is nonzero in the third paragraph (one cannot directly cite Definition 7.11(2) because of the change above).

\item In \S 8.1, after the proof of Lemma 8.2, the $\cO$-algebra $R^{\tau,\ovl{\beta},\Box}_{\ovl{\fM},\rhobar}$ has relative dimension $15$ over $\cO$.

\item In Corollary 8.4, there should be the further relation $c_{12}c_{33}=0$.

\item The final sentence of Corollary 8.4 should be replaced by the claim that the ring $R^{\tau,\ovl{\beta},\Box}_{\ovl{\fM},\rhobar}/\varpi$ is a quotient of $\tld{R}[\![c^*_{ii}-[\ovl{c}^*_{ii}], x_j,\ 1\leq i\leq 3,\ 1\leq j\leq 9]\!]$.
This comment applies also to Proposition 8.11.
These changes justify the dimension hypothesis in Lemma 8.8, which is used in the proof of Proposition 8.6.

\item In the proof of Proposition 8.11, we remark that $c_{33}=c_{23}c_{31}(c_{21}^*)^{-1}$.

\item In the caption of Table 4, the coefficients are in $\F$.

\item In the caption of Table 5, the coefficients are in $R$.

\item The entry $(1,3)$ of Table 6 should read ``Leading term of the monodromy condition''.

\item The missing entries in the second column of Table 6 can be read off from Table 5 and the missing entry in the third column of Table 6 can be read off from Proposition 8.3.

\item In Table 6, the Leading term of the monodromy condition for $\alpha\beta\alpha\gamma$, the second $c_{33}^*$ should be removed.
Moreover in the caption the $a_{s_{j+1}(i)}^{(j)}$ for $i=1,\,2,\,3$ should be bold.

\end{enumerate}

\newpage
\bibliography{Biblio}
\bibliographystyle{amsalpha}

\end{document}